\documentclass[12pt,fullpage,leqno]{amsart}
\usepackage{amsmath,amssymb,amsthm,amsfonts,amscd}
\usepackage{xr}
\usepackage{graphicx}
\usepackage{epsfig}
\usepackage{color}
\numberwithin{equation}{section}
\theoremstyle{plain}
\newtheorem{thm}{Theorem}[section]
\newtheorem*{thm0}{Theorem}
\newtheorem{pro}{Proposition}[section]
\newtheorem{lem}{Lemma}[section]

\theoremstyle{definition}
\newtheorem{defn}{Definition}[section]

\theoremstyle{remark}

\newtheorem*{rem0}{Remark}

\newtheorem*{acknow}{Acknowledgments}

\begin{document}
\title[Hamiltonian stability of the Gauss images]
{Hamiltonian stability of the Gauss images of homogeneous
isoparametric hypersurfaces}
\author{Hui Ma}
\address{Department of Mathematical Sciences, Tsinghua University,
Beijing 100084, P.R. CHINA} \email{hma@math.tsinghua.edu.cn}
\author{Yoshihiro Ohnita}
\address{Osaka City University Advanced Mathematical Institute \&
Department of Mathematics, Osaka City University, Sugimoto,
Sumiyoshi-ku, Osaka, 558-8585, JAPAN}
\email{ohnita@sci.osaka-cu.ac.jp}
\thanks{The first named author is partially supported by NSFC grant
No.~10501028 and No.~10971111, NKBRPC No.~2006CB805905 and a
scholarship from the China Scholarship Council. The second named
author is partially supported by JSPS Grant-in-Aid for Scientific
Research (A) No.~17204006, No.~19204006 and the Priority Research of
Osaka City University \lq\lq{Mathematics of knots and wide-angle
evolutions to scientific objects}\rq\rq.}

\maketitle

\begin{abstract}
The image of the Gauss map of any oriented isoparametric hypersurface of the unit
standard sphere $S^{n+1}(1)$
is a minimal Lagrangian submanifold in the complex hyperquadric
$Q_n({\mathbf C})$.
In this paper
we show that the Gauss image of a compact oriented isoparametric
hypersurface with $g$ distinct constant principal curvatures in $S^{n+1}(1)$
is a compact monotone and cyclic embedded Lagrangian submanifold
with minimal Maslov number $2n/g$.
The main result of this paper is to determine completely
the Hamiltonian stability of all compact minimal
Lagrangian submanifolds embedded in complex hyperquadrics
which are obtained as the images of the Gauss map
of homogeneous isoparametric hypersurfaces in the unit spheres,
by harmonic analysis on homogeneous spaces and fibrations on homogeneous isoparametric hypersurfaces.
In addition, the discussions on the exceptional Riemannian symmetric space
$(E_6, U(1)\cdot Spin(10))$ and the corresponding Gauss image have their own interest.
\end{abstract}


\section*{Introduction}
\label{Intro}

In 1990's Oh
initialized
the study of Hamiltonian minimality and Hamiltonian stability
of Lagrangian submanifolds in K\"{a}hler manifolds
(\cite{Oh90}, \cite{Oh91}, \cite{Oh93}).
It provides a
constrained volume variational problem of
Lagrangian submanifolds in K\"{a}hler manifolds under Hamiltonian deformations.
Thus it is natural to study
what Lagrangian submanifolds in specific K\"ahler manifolds are
Hamiltonian stable.
After Oh's pioneer papers, there has been extensive research done on
Hamiltonian stabilities of minimal or Hamiltonian minimal Lagrangian submanifolds
in various K\"{a}hler manifolds, such as complex Euclidean spaces,
complex projective spaces,
compact Hermitian symmetric spaces, certain toric K\"{a}hler manifolds and so on.
(See e.g.,
\cite{Amar-Ohn03, Castro-Urbano, Ono_TJM, Ono_AGAG, Palmer97, Urbano07}
and references therein.)
In particular, a compact minimal Lagrangian submanifold $L$ in a compact homogeneous Einstein-K\"{a}hler
manifold with positive Einstein constant $\kappa$ is Hamiltonian stable if and only if
the first (positive) eigenvalue $\lambda_1$ of the Laplacian of $L$ with respect to the induced metric
satisfies $\lambda_1=\kappa$.
Hence in this case, to determine the Hamiltonian stability reduces to
calculating the first eigenvalue of the Laplacian, which is an important problem in differential geometry.

On the other hand,
isoparametric hypersurfaces are next simple hypersurfaces in spheres after geodesic spheres.
The theory of isoparametric hypersurfaces in spheres was originated by \'{E}lie Cartan and well developed afterward.
Particularly great progress on the classification problem
of isoparametric hypersurfaces in spheres were made by the recent
work of Cecil-Chi-Jensen (\cite{Cecil-Chi-Jensen}), Immervoll
(\cite{Immervoll}), Chi (\cite{ChiII, ChiIII}) and Miyaoka (\cite{Miyaoka}).
Among all important results of isoparametric hypersurfaces in spheres,
M\"unzner (\cite{Muenzner1}, \cite{Muenzner2}) showed that the number $g$ of
distinct principal curvatures of  an isoparametric hypersurface $N^{n}$ in $S^{n+1}(1)$
must be $g=1,2,3,4,6$
and $N^{n}$ is always real algebraic in the sense that $N^{n}$ is defined by a certain
real homogeneous polynomial of degree $g$ which is called the \lq\lq
Cartan-M\"unzner polynomial\rq\rq.

It is known that the Gauss image of any compact oriented isoparametric hypersurface in the unit standard sphere is
a smooth compact embedded minimal Lagrangian submanifold in the complex hyperquadric
and the Gauss map is a covering map with covering transformation group ${\mathbf Z}_{g}$ (\cite{Palmer97, Ma-Ohnita1}).
Thus it can be expected that the Gauss images of isoparametric hypersurfaces in spheres provide
a nice class of compact Lagrangian submanifolds embedded in complex hyperquadrics
and moreover they should play certain roles in symplectic geometry.
Besides properties of Gauss images discussed in our previous paper \cite{Ma-Ohnita1},
in this paper we show
(see Theorem \ref{MinMaslovGaussImage})
\begin{thm0}
The Gauss image of a compact oriented isoparametric
hypersurface with $g$ distinct constant principal curvature in $S^{n+1}(1)$
is a compact monotone and cyclic embedded Lagrangian submanifold
with minimal Maslov number $2n/g$.
\end{thm0}

Recall that all isoparametric hypersurfaces in the unit standard sphere are
classified into homogeneous ones and non-homogeneous ones.
An isoparametric hypersurface $N^{n}$ in the unit standard sphere
$S^{n+1}(1)$ is called {\it homogeneous} if $N^{n}$ can be obtained
as an orbit of a compact Lie subgroup of $SO(n+2)$.
Every homogeneous isoparametric hypersurface in a sphere can be obtained
as a principal orbit of a linear isotropy representation of a
compact Riemannian symmetric pair $(U,K)$ of rank $2$, due to
Hsiang-Lawson (\cite{Hsiang-Lawson1971}) and Takagi-Takahashi
(\cite{Takagi-Takahashi1972}).
Only in the case of $g=4$ there are known to exist non-homogeneous
isoparametric hypersurfaces, which were discovered first by
Ozeki-Takeuchi (\cite{Ozeki-TakeuchiI}, \cite{Ozeki-TakeuchiII}) and
extensively generalized by Ferus-Karcher-M\"unzner (\cite{FKM}).
The purpose of this paper is to determine completely
the Hamiltonian stability of all compact
minimal Lagrangian embedded submanifolds in $Q_n({\mathbf C})$
which are obtained as the Gauss images of homogeneous
isoparametric hypersurfaces in $S^{n+1}(1)$.
This paper is a continuation of \cite{Ma-Ohnita1},
where we have already treated the cases of $g=1,2,3$.

The main result of this paper is as follows :
\begin{thm0}
Suppose that $(U,K)$ is not of type $\mathrm{EIII}$,
that is, $(U,K)\not=(E_{6},U(1)\cdot Spin(10))$.
Then the Gauss image $L={\mathcal G}(N)$ is not Hamiltonian stable if and only if
$m_{2}-m_{1}\geq{3}$.
Moreover if $(U,K)$ is of type $\mathrm{EIII}$,
then $(m_{1},m_{2})=(6,9)$ but
$L={\mathcal G}(N)$ is strictly Hamiltonian stable.
\end{thm0}

This paper is organized as follows:
In Section \ref{Hamilmin&Hamilsta} we recall the notions and fundamental properties
on Hamiltonian minimality, Hamiltonian stability and strictly Hamiltonian
stability of Lagrangian submanifolds in K\"ahler manifolds.
In Section \ref{Sec_Gauss maps} we briefly explain properties of minimal
Lagrangian submanifolds in complex hyperquadrics
as the Gauss images of isoparametric hypersurfaces in spheres.
In Section \ref{Sec_Method} we explain the method of eigenvalue computations of
our compact homogeneous spaces which are the Gauss images of
compact homogeneous isoparametric hypersurfaces in spheres,
and the fibrations on homogeneous isoparametric hypersurfaces
by homogeneous isoparametric hypersurfaces.
The fibrations are very useful for our computation.
In Sections \ref{Sec:G2xG2} and \ref{Section_G2 SO4}, we determine the strictly Hamiltonian stability
of the Gauss images of compact homogeneous isoparametric hypersurfaces
with $g=6$.
In Sections \ref{Sec_b2}-\ref{Sec_EIII}, we determine the strictly Hamiltonian stability
of the Gauss images of compact homogeneous isoparametric hypersurfaces
with $g=4$.
In particular, the discussions on the exceptional Riemannian symmetric space
$(E_6, U(1)\cdot Spin(10))$ and the corresponding Gauss image have their own interest.

\begin{acknow}
The main results were already announced in \cite{Ma-OhnitaCONM2010}.
This work was done while the first named author's stay at Osaka City
University Advanced Mathematical Institute (OCAMI) during 2008-2009
and the second named author's visits at Tsinghua University in
Beijing. The authors are grateful to both institutions for generous
supports and hospitalities. They also would like to thank Professors
Quo-Shin Chi, Josef Dorfmeister, Reiko Miyaoka and Zizhou Tang for
their continuous interest and helpful conversations on this work.
\end{acknow}

\section{Hamiltonian minimality and Hamiltonian stability}
\label{Hamilmin&Hamilsta}

Assume that $(M,\omega, J, g)$ is a K\"{a}hler manifold with the compatible complex structure $J$
and K\"{a}hler metric $g$.
Let $\varphi: L \rightarrow M$ be a Lagrangian immersion and $\rm H$ denote the mean curvature vector field of $\varphi$.
The corresponding $1$-form $\alpha_{\rm H}:=\omega({\rm H}, \cdot)\in \Omega^1(L)$ is called
the \emph{mean curvature form} of $\varphi$.
For simplicity, throughout this paper we assume that $L$ is compact
without boundary.

\begin{defn}
Let $M$ be a K\"{a}hler manifold. A Lagrangian immersion
$\varphi: L\rightarrow M$
is called \emph{Hamiltonian minimal} (shortly, H-minimal) or
\emph{Hamiltonian stationary},
if it is the critical point of the volume functional
for all Hamiltonian deformations $\{\varphi_t\}$.
\end{defn}

The corresponding Euler-Lagrange equation is $\delta \alpha_{\rm H}=0$,
where $\delta$ is the co-differential operator with
the respect to the induced metric on $L$.

\begin{defn}
An H-minimal Lagrangian immersion $\varphi$ is called \emph{Hamiltonian stable} (shortly,
\emph{H-stable}) if the second variation of the volume is nonnegative
under every Hamiltonian deformation $\{\varphi_t\}$.
\end{defn}

The second variational formula is given as follows (\cite{Oh93}):
\begin{equation*}
\begin{split}
&\displaystyle\frac{d^2}{dt^2}\mathrm{Vol}\left.(L, \varphi^*_tg)\right|_{t=0} \\
=&
\displaystyle\int_L \left(\langle \Delta^1_L\alpha, \alpha \rangle
- \langle \overline{R}(\alpha), \alpha \rangle
- 2\langle \alpha \otimes \alpha \otimes \alpha_{\mathrm{H}}, S \rangle
+ \langle\alpha_{\mathrm{H}}, \alpha \rangle^2\right) dv,
\end{split}
\end{equation*}
where $\Delta^1_L$ denotes the Laplace operator of
$(L,\varphi^{\ast}g)$ acting on the vector space
$\Omega^{1}(L)$ of smooth $1$-forms on $L$
and $\alpha:=\omega(V,\cdot) \in B(L)$ is the exact $1$-form
corresponding to an infinitesimal Hamiltonian deformation $V$.
Here
$$\langle \overline{R}(\alpha), \alpha \rangle := \displaystyle\sum^n_{i,j =1}
\mathrm{Ric}^M(e_i, e_j)\alpha(e_i)\alpha(e_j)$$
for a local orthonormal frame $\{e_i\}$ on $L$ and
$$S(X, Y, Z) := \omega(B(X, Y), Z)$$
for each $X,Y,Z \in C^{\infty}(TL)$,
which is a symmetric $3$-tensor field on $L$ defined by the second fundamental form $B$ of $L$ in $M$.

For an H-minimal Lagrangian immersion $\varphi: L\rightarrow M$,
we denote by $E_0(\varphi)$
the null space of the second variation on $B^1(L)$,
or equivalently
the solution space to the linearized H-minimal Lagrangian submanifold equation,
and we call
$n(\varphi):=\dim E_0(\varphi)$ the \emph{nullity} of $\varphi$.

If $H^1(M, {\mathbf R})=\{0\}$,
then any holomorphic Killing vector field on $M$ is a Hamiltonian vector field,
and thus it generates a volume-preserving Hamiltonian deformation of $\varphi$.
Namely,
\begin{equation*}
\{\varphi^{\ast}\alpha_{X}{\ }\vert{\ }
X \text{ is a holomorphic Killing vector field on }M\}
\subset
E_{0}(\varphi)\subset B^{1}(L).
\end{equation*}
Set $n_{hk}(\varphi):=\dim
\{\varphi^{\ast}\alpha_{X}{\ }\vert{\ }
X \mbox{ is a holomorphic Killing vector field on }M\}$,
which is called the \emph{holomorphic Killing nullity} of $\varphi$.

\begin{defn}
An H-minimal Lagrangian immersion $\varphi$ is called \emph{strictly Hamiltonian stable}
(shortly, strictly H-stable) if $\varphi$ is Hamiltonian stable and
$n_{hk}(\varphi)=n(\varphi)$.
\end{defn}
Note that if $L$ is strictly Hamiltonian stable, then $L$ has local minimum volume
under each Hamiltonian deformation.

In the case when $L$ is a compact minimal Lagrangian submanifold in
an Einstein-K\"{a}her manifold $M$ with Einstein constant $\kappa$,
the second variational formula becomes much simpler.
we see that $L$ is H-stable if and only
if the first (positive) eigenvalue $\lambda_1$ of the Laplacian of $L$
acting on smooth functions satisfies $\lambda_1\geq \kappa$ (\cite{Oh90}).
On the other hand, it is known that
the first eigenvalue $\lambda_1$ of
the Laplacian of any compact minimal Lagrangian submanifold $L$ in
a compact homogeneous Einstein-K\"{a}hler manifold with positive
Einstein constant $\kappa$
has the upper bound
$\lambda_1\leq \kappa$ (\cite{Ono_JMSJ}, \cite{Ono_TJM}).
In this case, $L$ is H-stable if and only if $\lambda_1=\kappa$.

Assume that $(M,\omega,J,g)$ is a K\"{a}hler manifold
and $G$ is an analytic subgroup of its automorphism group
$\mathrm{Aut}(M,\omega,J,g)$.
A Lagrangian orbit $L=G\cdot x \subset M$ of $G$ is called a
\emph{homogeneous Lagrangian submanifold} of $M$.
An easy but useful observation can be given as follows.
\begin{pro}
Any compact homogeneous Lagrangian submanifold
in a K\"ahler manifold is Hamiltonian minimal.
\end{pro}

\begin{proof}
Since $\alpha_{\mathrm{H}}$ is an invariant $1$-form on $L$,
$\delta\alpha_{\mathrm{H}}$ is a constant function on $L$.
Hence by the divergence theorem we obtain $\delta\alpha_{\mathrm{H}}=0$.
\end{proof}

Set
\begin{equation*}
\tilde{G}
:=\{a\in{\mathrm{Aut}(M,\omega,J,g)} \mid a(L)=L\}.
\end{equation*}
Then $G\subset\tilde{G}$ and $\tilde{G}$ is the maximal subgroup
of $\mathrm{Aut}(M,\omega,J,g)$ preserving $L$.
Moreover we have
$n_{hk}(\varphi)
=\dim(\mathrm{Aut}(M,\omega,J,g))-\dim(\tilde{G})$.

\section{Gauss maps of isoparametric hypersurfaces in a sphere}
\label{Sec_Gauss maps}

\subsection{Gauss maps of oriented hypersurfaces in spheres}

Let $N^n$ be an oriented hypersurface immersed in the unit standard sphere
$S^{n+1}(1) \subset {\mathbf R}^{n+2}$.
Denote by $\mathbf x$ its position vector of a point $p$ of $N$
and $\mathbf n$ the unit normal vector field of $N$ in $S^{n+1}(1)$.
It is a fundamental fact in symplectic geometry that
the \emph{Gauss map} defined by
\begin{equation*}
{\mathcal G}:N^{n}\ni{p}
\longmapsto
{\mathbf x}(p)\wedge{\mathbf n}(p)\cong
[{\mathbf x}(p)+\sqrt{-1}{\mathbf n}(p)]\in
\widetilde{Gr}_{2}({\mathbf R}^{n+2})
\cong Q_{n}({\mathbf C})
\end{equation*}
is always a Lagrangian immersion in the complex hyperquadric $Q_n({\mathbf C})$.
Here the complex hyperquadric $Q_n({\mathbf C})$ is identified with
the real Grassmann manifold $\widetilde{Gr}_{2}({\mathbf R}^{n+2})$
of oriented $2$-dimensional vector subspaces of ${\mathbf R}^{n+2}$,
which has a symmetric space expression $SO(n+2)/(SO(2)\times SO(n))$.

Let $g^{std}_{Q_{n}({\mathbf C})}$ be the standard K\"ahler metric
of $Q_{n}({\mathbf C})$ induced from the standard inner product
of ${\mathbf R}^{n+2}$.  Note that the Einstein constant of
$g^{std}_{Q_{n}({\mathbf C})}$ is equal to $n$.
Let
$\kappa_{i}{\ }(i=1,\cdots,n)$ denote
the principal curvatures of $N^{n}\subset S^{n+1}(1)$
and $\mathrm H$ denote the mean curvature vector field of the Gauss map
${\mathcal G}$.
Palmer showed the following mean curvature form formula (\cite{Palmer97}):
\begin{equation}\label{MeanCurvFormFormula}
\alpha_{\mathrm{H}}
=
-d\left(\sum^{n}_{i=1}\mathrm{arc}\cot\kappa_{i}\right)
=
d\left(\mathrm{Im}
\left(
\log{\prod^{n}_{i=1}(1+\sqrt{-1}\kappa_{i})}
\right)\right).
\end{equation}
Hence, if $N^n$ is an oriented austere hypersurface in $S^{n+1}(1)$,
introduced by Harvey-Lawson (\cite{HL}),
then its Gauss map ${\mathcal G}: N^n \rightarrow Q_n({\mathbf C})$ is
a minimal Lagrangian immersion.
In particular, since any minimal surface in $S^3(1)$ is austere,
its Gauss map is a minimal Lagrangian immersion
in $Q_2({\mathbf C}) \cong S^2 \times S^2$ (\cite{Castro-Urbano}).
Note that more minimal Lagrangian submanifolds of complex hyperquadrics can
be obtained from Gauss maps of certain oriented hypersurfaces in
spheres through Palmer's formula (\cite{LMW}).

\subsection{Gauss maps of isoparametric hypersurfaces in spheres}

Now suppose that $N^n$ is a compact oriented hypersurface in $S^{n+1}(1)$ with constant principal curvatures,
i.e., \emph{isoparametric hypersurface}.
By M\"{u}zner's result (\cite{Muenzner1, Muenzner2}),
the number $g$ of distinct principal curvatures must be $1,2,3,4$ or $6$,
and the distinct principal curvatures have the multiplicities $m_1=m_3=\cdots$,
$m_2=m_4=\cdots$.  We may assume that $m_1 \leq m_2$.
It follows from \eqref{MeanCurvFormFormula}
that its Gauss map ${\mathcal G}: N^n \rightarrow Q_n({\mathbf C})$
is a minimal Lagrangian immersion.
Moreover, the \lq\lq{Gauss image}\rq\rq of $\mathcal G$
is a compact minimal Lagrangian submanifold
$L^n={\mathcal G}(N^n) \cong N^n / {\mathbf Z}_g$ \emph{embedded}
in $Q_n({\mathbf C})$
so that ${\mathcal G}:N^n\rightarrow {\mathcal G}(N^n)=L^n$
is a covering map with the Deck transformation group ${\mathbf Z}_g$
(\cite{Ma-Ohnita1}, \cite{Ma-OhnitaCONM2010}).

Here we mention the following symplectic topological properties of
the Gauss images of isoparametric hypersurfaces.

\begin{thm}\label{MinMaslovGaussImage}
The Gauss image $L=\mathcal{G}(N^n)$ is a compact monotone and
cyclic Lagrangian submanifold embedded in $Q_n(\mathbf C)$ and
its minimal Maslov number $\Sigma_L$ is given by
\begin{equation*}
\Sigma_L=\frac{2n}{g}=\left\{
                        \begin{array}{ll}
                          m_1+m_2, & \hbox{if } g \text{ is even;} \\
                          2m_1, & \hbox{if } g \text{ is odd.}
                        \end{array}
                      \right.
\end{equation*}
\end{thm}

We need to use the following H.~Ono's result
which generalizes Oh's work \cite{Oh94}.
\begin{lem}[\cite{Ono_JMSJ}]\label{Lem:Ono}
Let $M$ be a simply connected
K\"{a}hler-Einstein manifold with positive scalar curvature
with a prequantization complex line bundle $E$.
Then any compact minimal Lagrangian submanifold $L$ in $M$
is monotone and cyclic.
Moreover the minimal Maslov number $\Sigma_L$ of $L$
satisfies the following relation:
\begin{equation}\label{eq:min Maslov no}
n_L\, \Sigma_L\, =\, 2\, \gamma_{c_1},
\end{equation}
where
\begin{equation*}
\gamma_{c_1}:=\min\{c_1(M)(A)\mid A\in H_2(M;\mathbb{Z}), c_1(M)(A)>0\}
\in \mathbf{Z}
\end{equation*}
is called the index of a K\"{a}hler manifold $M$ and
\begin{equation*}
n_L:=\min\{k\in \mathbf{Z} \mid k\geq 1,
\otimes^k(E, \nabla)_{|L} \text{ is trivial}\}.
\end{equation*}
\end{lem}

Using this lemma and the properties of isoparametric hypersurfaces in a sphere,
we shall prove Theorem \ref{MinMaslovGaussImage}.

\begin{proof}

It follows from Lemma \ref{Lem:Ono} and the minimality of the Gauss image $L=\mathcal{G}(N^n)$ that
$L$ is a monotone and cyclic Lagrangian submanifold in
$Q_n(\mathbf C)$.
Remark that the index of $Q_n(\mathbf{C})$ is known as follows
(\cite{Borel-Hirzebruch}): $\gamma_{c_1}=n$
if $n\geq 2$ and $\gamma_{c_1}=2$ if $n=1$.
So in order to find the minimal Maslov number $\Sigma_L$ of $L$,
we only need to compute $n_L$.
Let $\tilde{N}^n$ be the Legendrian lift of $N^n$ to the unit tangent sphere bundle
$UTS^{n+1}(1)=V_2(\mathbf{R}^{n+2})$.
Then $\pi: V_2(\mathbf{R}^{n+2})|_L \rightarrow L=\mathcal{G}(N^n)$
is a flat principal fiber bundle with structure group $SO(2)$
and the covering map $\pi: \tilde{N}^n \rightarrow \mathcal{G}(N^n)$
with Deck transformation group $\mathbf{Z}_g$
coincides with its holonomy subbundle with the holonomy group $\mathbf{Z}_g$.
Let $E$ be a complex line bundle over $Q_{n}({\mathbf C})$
associated with the principal fiber bundle
$\pi: V_2(\mathbf{R}^{n+2}) \rightarrow \widetilde{Gr}_{2}({\mathbf R}^{n+2})
\cong Q_{n}({\mathbf C})$
by the standard action of $SO(2)\cong U(1)$ on $\mathbf{C}$.
Then $E\vert_L$ is a flat complex line bundle over $\mathcal{G}(N^n)$
associated with the principal fiber bundle
$\pi: V_2(\mathbf{R}^{n+2})|_L \rightarrow \mathcal{G}(N^n)$
by the standard action of $SO(2)\cong U(1)$ on $\mathbf{C}$.
The tautological complex line bundle $\mathcal{W}$ over
$Q_n(\mathbf{C})\subset {\mathbf C}P^{n+1}$
is defined by $\mathcal{W}_x:=\mathbf{C}(\mathbf{a}+\sqrt{-1}\mathbf{b})$
for each $[\mathbf{a}+\sqrt{-1}\mathbf{b}]\in Q_n(\mathbf{C})$.
Then $E=\mathcal{W}$ if $n\geq 2$ and $\otimes^2 E=\mathcal{W}$ if $n=1$.
Indeed, $c_1(\mathcal{W})(\mathbf{C}P^1)=1$ if $n\geq 2$.
Here, $\mathbf{C}P^1$ denotes the set of one-dimensional complex vector subspaces in a $2$-dimensional isotropic vector
subspace of $\mathbf{C}^{n+2}$.
For $k=1,\cdots, g$, the generator $e^{\sqrt{-1}\frac{2\pi}{g}}$ of the holonomy group $\mathbf{Z}_g$
on $E|_L$ induces the multiplication by $e^{\sqrt{-1}\frac{2\pi k}{g}}$ on $\otimes^k E|_L$.
Thus the holonomy group of $\otimes^k E|_L$ is generated by $e^{\sqrt{-1}\frac{2\pi k}{g}}$ of $\mathbf{Z}_g$.
Hence, $\otimes^k E|_L$ has non-trivial holonomy for $k=1,\cdots, g-1$ and $\otimes^g E|_L$ has trivial holonomy.
Therefore, $n_L=g$ if $n\geq 2$ and $n_L=2$ if $n=1$.
Thus the conclusion follows from \eqref{eq:min Maslov no}.
\end{proof}

A hypersurface $N^n$ in $S^{n+1}(1)$ is \emph{homogeneous} if it is obtained
as an orbit of a compact connected subgroup $G$ of $SO(n+2)$.
Obviously any homogeneous hypersurface in $S^{n+1}(1)$ is an isoparametric
hypersurface.
It turns out that $N^n$ is homogeneous if and only if its Gauss
image ${\mathcal G}(N^n)$ is homogeneous (\cite{Ma-Ohnita1}).

Consider
\begin{equation*}
{\mathcal G}: N^{n}\ni{p}\longmapsto{\mathbf x}(p)\wedge{\mathbf n}(p)
\in\widetilde{Gr}_{2}({\mathbf R}^{n+2})\subset\bigwedge^{2}{\mathbf R}^{n+2}.
\end{equation*}
Here $\bigwedge^{2}{\mathbf R}^{n+2}\cong{
\mathfrak o}(n+2)$
can be identified with
the Lie algebra of all (holomorphic) Killing vector fields on
$S^{n+1}(1)$ or $\widetilde{Gr}_{2}({\mathbf R}^{n+2})$.
Let
$\tilde{\mathfrak k}$ be the Lie subalgebra of ${\mathfrak o}(n+2)$
consisting of all Killing vector fields tangent to
$N^{n}$ or ${\mathcal G}(N^{n})$ and $\tilde{K}$ be
a compact connected Lie subgroup of $SO(n+2)$
generated by $\tilde{\mathfrak k}$.
Take the orthogonal direct sum
\begin{equation*}
\bigwedge^{2}{\mathbf R}^{n+2}=\tilde{\mathfrak k}+{\mathcal V},
\end{equation*}
where ${\mathcal V}$ is a vector subspace of ${\mathfrak o}(n+2)$.
The linear map
\begin{equation*}
{\mathcal V}\ni{X}\longmapsto\alpha_{X}\vert_{{\mathcal G}(N^{n})}
\in{E_{0}({\mathcal G})}\subset B^{1}({\mathcal G}(N^{n}))
\end{equation*}
is injective and $n_{hk}({\mathcal G})=\dim{\mathcal V}$.
Then
${\mathcal G}(N^{n})\subset{\mathcal V}$
and thus
\begin{equation*}
{\mathcal G}(N^{n})\subset
\widetilde{Gr}_{2}({\mathbf R}^{n+2})\cap{\mathcal V}.
\end{equation*}
Indeed, for each $X\in{\tilde{\mathfrak k}}$ and each $p\in{N^{n}}$,
$\langle{X,{\mathbf x}(p)\wedge{\mathbf n}(p)}\rangle
=\langle{X{\mathbf x}(p),{\mathbf n}(p)}\rangle
-\langle{{\mathbf x}(p),X{\mathbf n}(p)}\rangle
=2\langle{X{\mathbf x}(p),{\mathbf n}(p)}\rangle=0$.

Note that ${\mathcal G}(N^{n})$ is a compact minimal submanifold
embedded in the unit hypersphere of ${\mathcal V}$ and
by the theorem of Tsunero Takahashi
each coordinate function of ${\mathcal V}$ restricted to
${\mathcal G}(N^{n})$ is an eigenfunction of the Laplace operator with
eigenvalue $n$.
Then we observe

\begin{lem}
$n$ is just the first (positive) eigenvalue of
${\mathcal G}(N^{n})$
if and only if
${\mathcal G}(N^{n})\subset{Q_{n}({\mathbf C})}$ is Hamiltonian stable.
Moreover the dimension of  the vector space
${\mathcal V}$ is equal to the multiplicity of the (resp.~first) eigenvalue $n$
if and only if
${\mathcal G}(N^{n})\subset{Q_{n}({\mathbf C})}$ is Hamiltonian rigid
(resp.\ strictly Hamiltonian stable).
\end{lem}

Next we mention a relationship between the Gauss images
${\mathcal G}(N^{n})$ of isoparametric hypersurfaces
and the intersection
$\widetilde{Gr}_{2}({\mathbf R}^{n+2})\cap{\mathcal V}$.
In \cite{Ma-OhnitaCONM2010} we showed that
if $N^{n}$ is homogeneous, then
${\mathcal G}(N^{n})= \widetilde{Gr}_{2}({\mathbf R}^{n+2})\cap{\mathcal V}$.

Define a map $\mu: \widetilde{Gr}_{2}({\mathbf R}^{n+2}) \rightarrow \bigwedge^{2}{\mathbf R}^{n+2}$
by
\begin{equation*}
\mu:\widetilde{Gr}_{2}({\mathbf R}^{n+2})
\ni [W]\longmapsto
{\mathbf a}\wedge{\mathbf b}\in
\bigwedge^{2}{\mathbf R}^{n+2}\cong\mathfrak{o}(n+2)
=\tilde{\mathfrak k}+{\mathcal V}.
\end{equation*}
The moment map of the action $\tilde{K}$ on
$\widetilde{Gr}_{2}({\mathbf R}^{n+2})$
is given by
$\mu_{\tilde{\mathfrak k}}:=
\pi_{\tilde{\mathfrak k}}\circ \mu :\widetilde{Gr}_{2}({\mathbf R}^{n+2})
\rightarrow \tilde{\mathfrak k}$,
where $\pi_{\tilde{\mathfrak k}}:\mathfrak{o}(n+2)\rightarrow\tilde{\mathfrak k}$
denotes the orthogonal projection onto $\tilde{\mathfrak k}$.
For any $p\in{N^{n}}$, we have
\begin{equation*}
\tilde{K}({\mathbf x}(p)\wedge{\mathbf n}(p))
\subset
{\mathcal G}(N^{n})
\subset
\widetilde{Gr}_{2}({\mathbf R}^{n+2})\cap{\mathcal V}
=\mu_{\tilde{\mathfrak k}}^{-1}(0).
\end{equation*}
It is obvious that $N^{n}$ is homogeneous if and only if
$\tilde{K}({\mathbf x}(p)\wedge{\mathbf n}(p))={\mathcal G}(N^{n})$.
In \cite{Ma-OhnitaCONM2010} we proved its inverse as follows:
Assume that
${\mathcal G}(N^{n})
=\widetilde{Gr}_{2}({\mathbf R}^{n+2})\cap{\mathcal V}$.
Then
$\tilde{K}({\mathbf x}(p)\wedge{\mathbf n}(p))={\mathcal G}(N^{n})$,
that is, $N^{n}$ is homogeneous.
Therefore we obtain (\cite{Ma-OhnitaCONM2010}) that
$N ^{n}$ is not homogeneous if and only if
\begin{equation*}
\tilde{K}({\mathbf x}(p)\wedge{\mathbf n}(p))
\subsetneqq
{\mathcal G}(N^{n})
\subsetneqq
\widetilde{Gr}_{2}({\mathbf R}^{n+2})\cap{\mathcal V}
=\mu_{\tilde{\mathfrak k}}^{-1}(0).
\end{equation*}

All isoparametric hypersurfaces in spheres are classified into
homogeneous one and non-homogeneous one. Due to Hsiang-Lawson
(\cite{HL}) and Takagi-Takahashi (\cite{Takagi-Takahashi1972}), any
homogeneous isoparametric hypersurface in a sphere can be obtained
as a principal orbit of the isotropy representation of a compact
Riemannian symmetric pair $(U,K)$ of rank $2$ (see Table
\ref{List:homogIsoparametric}).

Compact homogeneous minimal Lagrangian submanifolds
obtained as the Gauss images of
homogeneous isoparametric hypersurfaces
are constructed in the following way (cf.\ \cite{Ma-Ohnita1}).
Let ${\mathfrak u}={\mathfrak k}+{\mathfrak p}$ be the canonical decomposition of $\mathfrak u$
as a symmetric Lie algebra of a symmetric pair $(U,K)$ of rank $2$
and $\mathfrak a$ be a maximal abelian subspace of $\mathfrak p$.
Define an $\mathrm{Ad}U$-invariant inner product
$\langle{\  ,\ }\rangle_{\mathfrak u}$
of $\mathfrak u$ from the Killing-Cartan form of $\mathfrak u$.
Then the vector space $\mathfrak p$ equipped with the inner product
$\langle{\  ,\ }\rangle_{\mathfrak u}$
can be identified with
the Euclidean space ${\mathbf R}^{n+2}$
and $S^{n+1}(1)$ denotes
the $(n+1)$-dimensional unit standard sphere in $\mathfrak p$.
The linear isotropy action $\mathrm{Ad}_{\mathfrak p}$ of $K$
on $\mathfrak p$ and thus on $S^{n+1}(1)$
induces the group action of $K$ on
$\widetilde{\mathrm{Gr}}_2({\mathfrak p})\cong Q_n({\mathbf C})$.
For each \emph{regular} element $H$ of ${\mathfrak a} \cap S^{n+1}(1)$,
we get a homogeneous isoparametric hypersurface in the unit sphere
\begin{equation*}
N^n=(\mathrm{Ad}_{\mathfrak p}K)H \subset S^{n+1}(1) \subset {\mathfrak p}\cong {\mathbf R}^{n+2}.
\end{equation*}
Its Gauss image is
\begin{equation*}
L^n=
{\mathcal G}(N^n)=K\cdot [\mathfrak a]=
[(\mathrm{Ad}_{\mathfrak p}K){\mathfrak a}]
\subset \widetilde{Gr}_2({\mathfrak p})\cong Q_n(\mathbf C).
\end{equation*}
Here $N$ and ${\mathcal G}(N^{n})$ have homogeneous space
expressions $N\cong K/K_{0}$ and ${\mathcal G}(N^{n})\cong
K/K_{[{\mathfrak a}]}$, where we define
\begin{equation*}
\begin{split}
&K_{0}:=\{k\in{K} \mid
\mathrm{Ad}_{\mathfrak p}(k)(H)=H\}\\
&\quad{\ }=\{k\in{K} \mid
\mathrm{Ad}_{\mathfrak p}(k)(H)=H
\text{ for each }H\in{\mathfrak a}\},\\
&K_{\mathfrak a}:=\{k\in{K} \mid
\mathrm{Ad}_{\mathfrak p}(k)({\mathfrak a})
={\mathfrak a}\},\\
&K_{[{\mathfrak a}]}:=\{k\in{K_{\mathfrak a}} \mid
\mathrm{Ad}_{\mathfrak p}(k):{\mathfrak a}\longrightarrow{\mathfrak a}
\text{ preserves the orientation of }{\mathfrak a}\}.
\end{split}
\end{equation*}
The deck transformation group of the covering map
${\mathcal G}:N\rightarrow {\mathcal G}(N^{n})$ is equal to
$K_{[{\mathfrak a}]}/K_{0}=W(U,K)/{\mathbf Z}_{2}\cong{\mathbf Z}_{g}$,
where
$W(U,K)=K_{\mathfrak a}/K_{0}$ is the Weyl group of $(U,K)$.

Since we know that $\mathrm{Ad}_{\mathfrak p}K$ is
the maximal compact subgroup of $SO(n+2)$ preserving $N$ and/or
${\mathcal G}(N^{n})$
(\cite{Hsiang-Lawson1971}, \cite{Ma-Ohnita1}),
in this case its nullity is given as
\begin{equation*}
n_{hk}({\mathcal G})=n_{hk}({\mathcal G}(N^{n}))
=\dim{SO(n+2)}-\dim{K}.
\end{equation*}

\begin{small}
\begin{table}[ht]
\caption{Homogeneous isoparametric hypersurfaces in spheres}
\label{List:homogIsoparametric}
\renewcommand\arraystretch{1}
\noindent\[
\begin{array}{|c|c|c|c|c|c|}
\hline
$g$
&Type
&(U,K)
& \rm{dim}{N}
& m_{1},m_{2}
&K/K_{0}
\\
\hline
1
& S^{1}\times
&(S^{1}\times SO(n+2),SO(n+1))
&{n}
&{n}
&S^{n}
\\
{\ }
&  \mathrm{BDII}
&(n\geq{1}) \, [{\mathbf R}\oplus A_1]
&{\ }
&{\ }
&{\ }
\\
\hline
2
&\mathrm{BDII}\times
&(SO(p+2)\times SO(n+2-p),
&n
& p,n-p
&S^{p}\times S^{n-p}
\\
{\ }
& \mathrm{BDII}
&SO(p+1)\times SO(n+1-p))
&{\ }
&{\ }
&{\ }
\\
{\ }
&{\ }
&(1\leq{p}\leq{n-1})\, [A_1\oplus A_1]
&{\ }
&{\ }
&{\ }
\\
\hline
3
&\mathrm{AI}_{2}
&(SU(3),SO(3)) \, [A_2]
&3
&1,1
&\frac{SO(3)}{{\mathbf Z}_{2}+{\mathbf Z}_{2}}
\\
\hline
3
&{\mathfrak a}_{2}
&(SU(3)\times SU(3),SU(3))\, [A_2]
&6
&2,2
&\frac{SU(3)}{T^{2}}
\\
\hline
3
&\mathrm{AII}_{2}
&(SU(6),Sp(3)) \, [A_2]
&12
&4,4
&\frac{Sp(3)}{Sp(1)^{3}}
\\
\hline
3
&\mathrm{EIV}
&(E_{6},F_{4}) \, [A_2]
&24
&8,8
&\frac{F_{4}}{Spin(8)}
\\
\hline
4
&{\mathfrak b}_{2}
&(SO(5)\times SO(5),SO(5)) \, [B_2]
&8
&2,2
&\frac{SO(5)}{T^{2}}
\\
\hline
4
&\mathrm{AIII}_{2}
&(SU(m+2),S(U(2)\times U(m)))
&4m-2
&2,
&\frac{S(U(2)\times U(m))}{S(U(1)\times U(1) \times U(m-2))}
\\
{\ }
&{\ }
&(m\geq{2}) \, [BC_2] (m\geq 3), [B_2] (m=2)
&{\ }
&2m-3
&{\ }
\\
\hline
4
&\mathrm{BDI}_{2}
&(SO(m+2),SO(2)\times SO(m))
&2m-2
&1,
&\frac{SO(2)\times SO(m)}{{\mathbf Z}_{2}\times SO(m-2)}
\\
{\ }
&{\ }
&(m\geq{3}) \, [B_2]
&{\ }
&m-2
&{\ }
\\
\hline
4
&\mathrm{CII}_{2}
&(Sp(m+2),Sp(2)\times Sp(m))
&8m-2
&4,
&\frac{Sp(2)\times Sp(m)}{Sp(1)\times Sp(1) \times Sp(m-2)}
\\
{\ }
&{\ }
&(m\geq{2}) \, [BC_2](m\geq 3),[B_2](m=2)
&{\ }
&4m-5
&{\ }
\\
\hline
4
&\mathrm{DIII}_{2}
&(SO(10),U(5)) \, [BC_2]
&18
&4,5
&\frac{U(5)}{{SU(2)}\times{SU(2)}\times U(1)}
\\
\hline
4
&\mathrm{EIII}
&(E_{6}, U(1) \cdot Spin(10)) \, [BC_2]
&30
&6,9
&\frac{U(1) \cdot Spin(10)}{S^1 \cdot Spin(6)}
\\
\hline
6
&{\mathfrak g}_{2}
&(G_{2}\times G_{2},G_{2}) \, [G_2]
&12
&2,2
&\frac{G_{2}}{T^{2}}
\\
\hline
6
&\mathrm{G}
&(G_{2},SO(4)) \, [G_2]
&6
&1,1
&\frac{SO(4)}{{\mathbf Z}_{2}+{\mathbf Z}_{2}}
\\
\hline
\end{array}
\]
\end{table}
\end{small}

\section{The method of eigenvalue computations for our compact homogeneous spaces}
\label{Sec_Method}

\subsection{Basic results from harmonic analysis on compact homogeneous spaces}

First we recall the basic theory of harmonic analysis on general compact homogeneous spaces (cf.\ \cite{Takeuchi}).
Let ${\mathcal D}(G)$ be the complete set of all inequivalent irreducible
unitary representations of a compact connected Lie group $G$.
For a maximal abelian subalgebra ${\mathfrak t}$ of ${\mathfrak g}$,
let $\Sigma(G)$ be the set of all roots of  ${\mathfrak k}$ and
$\Sigma^{+}(G)$ be its subset of all positive root $\alpha\in\Sigma(G)$
relative to a linear order fixed on ${\mathfrak t}$.
Set
\begin{equation*}
\begin{split}
&\Gamma(G):=\{\xi\in{\mathfrak t}\mid \exp(\xi)=e\},\\
& Z(G):=\{\Lambda\in{\mathfrak t}^{\ast}\mid \Lambda(\xi)
\in 2\pi{\mathbf Z}
\,  \text{ for each }\xi\in\Gamma(G)\},\\
&D(G):=\{\Lambda\in Z(G)\mid
\langle{\Lambda,\alpha}\rangle\geq 0
\,  \text{ for each }\alpha\in\Sigma^{+}(G)\}.
\end{split}
\end{equation*}
Then there is a bijective correspondence between
$D(G)\ni \Lambda\longmapsto (V_{\Lambda},\rho_{\Lambda})\in{\mathcal D}(G)$,
where
$(V_{\Lambda},\rho_{\Lambda})$
denotes an irreducible unitary representation of $G$
with the highest weight $\Lambda$
equipped with a $\rho_{\Lambda}(K)$-invariant Hermitian inner product
$\langle{\ ,\ }\rangle_{V_{\Lambda}}$.
Let $\langle{\ ,\ }\rangle_{\mathfrak g}$ be an
$\mathrm{Ad}G$-invariant inner product of ${\mathfrak g}$.
For a compact Lie subgroup $H$ of $G$ with Lie subalgebra ${\mathfrak h}$,
we take the orthogonal direct sum decomposition
${\mathfrak g}={\mathfrak h}+{\mathfrak m}$ relative to
$\langle{\ ,\ }\rangle_{\mathfrak g}$.
Set
\begin{equation}
D(G,H):=\{\Lambda\in D(G)\mid (V_{\Lambda})_{H}\not=\{0\}\},
\end{equation}
where
\begin{equation}
(V_{\Lambda})_{H}:=\{w\in V_{\Lambda}\mid \rho_{\Lambda}(a)w=w \
(\forall a\in H)\}.
\end{equation}
Let $\Lambda\in{D(G,H)}$.
For each
$\bar{w}\otimes v\in (V_{\Lambda})^{\ast}_{H}\otimes V_{\Lambda}$,
we define a real analytic function $f_{\bar{w}\otimes v}$
on $G/H$ by
\begin{equation}
(f_{\bar{w}\otimes v})(a H):=
\langle{v,\rho_{\Lambda}(a)w}\rangle_{V_{\Lambda}}
\end{equation}
for all $a H\in G/H$.
By virtue of the Peter-Weyl's theorem and the Frobenius reciprocity law,
we have a linear injection
\begin{equation}
(V_{\Lambda})^{\ast}_{H}\otimes V_{\Lambda}
\ni
\bar{w}\otimes v
\longmapsto
f_{\bar{w}\otimes v}\in C^{\infty}(G/H,{\mathbf C})
\end{equation}
and the decomposition
\begin{equation}
C^{\infty}(G/H,{\mathbf C})
=
\bigoplus_{\Lambda\in D(G,H)}
(V_{\Lambda})^{\ast}_{H}\otimes V_{\Lambda}.
\end{equation}
in the sense of $C^{\infty}$-topology.
Via the natural homogeneous projection
$\pi: G\rightarrow G/H$,
the vector space $C^{\infty}(G/H,{\mathbf C})$
of all complex valued smooth functions on $G/H$
can be identified with
the vector space $C^{\infty}(G,{\mathbf C})_{H}$
of all complex valued smooth functions on $G$
invariant under the right action of $H$.
Let
${\mathrm U}({\mathfrak g})$
be the universal enveloping algebra of Lie algebra ${\mathfrak g}$,
which is identified to the algebra of all left-invariant linear differential operators
on $C^{\infty}(G,{\mathbf C})$.
Let
\begin{equation*}
{\mathrm U}({\mathfrak g})_{H}:=
\{D\in {\mathrm U}({\mathfrak g})\mid \mathrm{Ad}(h)D
=R_{h}\circ D\circ R_{h^{-1}}
=D\text{ for each }h\in H\}
\end{equation*}
be a subalgebra of ${\mathrm U}({\mathfrak g})$ consisting of
elements fixed by the adjoint action of $H$.
Here define $(R_{h}\tilde{f})(u):=\tilde{f}(uh)$ for
$\tilde{f}\in C^{\infty}(G,{\mathbf C})$.
For each $D\in {\mathrm U}({\mathfrak g})_{H}$,
we have
$D(C^{\infty}(G,{\mathbf C})_{H})\subset C^{\infty}(G,{\mathbf C})_{H}$.
The \emph{Casimir operator}
${\mathcal C}_{G/H,\langle{\ ,\ }\rangle_{\mathfrak g}}$
of $(G,H)$
relative to $\langle{\ ,\ }\rangle_{\mathfrak g}$ is defined by
${\mathcal C}
={\mathcal C}_{G/H,\langle{\ ,\ }\rangle_{\mathfrak g}}
:=\sum^{n}_{i=1}(X_{i})^{2}$,
where $\{X_{i} \mid i=1,\cdots,n\}$
is an orthonormal basis of ${\mathfrak m}$ with respect to
$\langle{\ ,\ }\rangle_{\mathfrak g}$.
Then ${\mathcal C}_{G/H,\langle{\ ,\ }\rangle_{\mathfrak g}}
\in{\mathrm U}({\mathfrak g})_{H}$
and by the $\mathrm{Ad}G$-invariance of
$\langle{\ ,\ }\rangle_{\mathfrak g}$
and Schur's Lemma
there is a non-positive real constant
$c(\Lambda,\langle{\ ,\ }\rangle_{\mathfrak g})$
such that
\begin{equation}
{\mathcal C}_{G/H,\langle{\ ,\ }\rangle_{\mathfrak g}}
(f_{\bar{w}\otimes v})=
c(\Lambda,\langle{\ ,\ }\rangle_{\mathfrak g}) f_{\bar{w}\otimes v}
\end{equation}
for each
$\bar{w}\otimes v\in (V_{\Lambda})^{\ast}_{H}\otimes V_{\Lambda}$.
The eigenvalue $c(\Lambda,\langle{\ ,\ }\rangle_{\mathfrak g})$ is given
by the Freudenthal's formula
\begin{equation}
c(\Lambda,\langle{\ ,\ }\rangle_{\mathfrak g})=
-\langle{\Lambda,\Lambda+2\delta}\rangle_{\mathfrak g},
\end{equation}
where $2\delta=\sum_{\alpha\in\Sigma^{+}(G)}\alpha$.

Now we shall consider our compact homogeneous spaces $N^n=K/K_0$ and
$L^n=\mathcal{G}(N^n)=K/K_{[\mathfrak{a}]}$ (\cite{Ma-Ohnita1}).
Let $\Sigma(U,K)$ be the set of (restricted) roots of $(\mathfrak{u}, \mathfrak{k})$
and $\Sigma^{+}(U,K)$ be its subset of positive roots.
We have the following root decomposition of $\mathfrak{k}$:
\begin{equation*}
\mathfrak{k}=\mathfrak{k}_0 +\sum_{\gamma\in \Sigma^{+}(U,K)}
\mathfrak{k}_{\gamma},
\end{equation*}
where
\begin{equation*}
\begin{split}
\mathfrak{k}_0:
=&\{X\in{\mathfrak k}{\ }\vert{\ }[X,{\mathfrak a}]\subset{\mathfrak a}\}\\
=&\{X\in{\mathfrak k}{\ }\vert{\ }[X,H]=0\quad
\text{ for each }H\in{\mathfrak a}\},\\
{\mathfrak k}_{\gamma} :=&\{X\in{\mathfrak k}{\ }\vert{\
}(\mathrm{ad}H)^{2}X=(\gamma(H))^{2}X
\text{ for each }H\in{\mathfrak a}\}.
\end{split}
\end{equation*}
For each $\gamma\in{\Sigma^{+}(U,K)}$,
set $m(\gamma):=\dim{{\mathfrak k}_{\gamma}}$.
Define
\begin{equation}\label{M&Aperp}
{\mathfrak m}:= \sum_{\gamma\in{\Sigma^{+}(U,K)}}{\mathfrak k}_{\gamma}.
\end{equation}
Then the tangent vector spaces $T_{eK_0}(K/K_0)$ and
$T_{eK_{[\mathfrak a]}}(K/K_{[\mathfrak a]})$
can be identified with the vector subspace $\mathfrak{m}$ of $\mathfrak{k}$.
We can choose an orthonormal basis of ${\mathfrak m}$ with respect to
$\langle{\ ,\ }\rangle_{\mathfrak u}$
\begin{equation*}
\{X_{\gamma,i}{\ }\vert{\
}\gamma\in{\Sigma^{+}(U,K)},i=1,2,\cdots,m(\gamma)\}.
\end{equation*}
Let $\langle{\ },{\ }\rangle$ denote the ${\rm Ad}_{\mathfrak m}(K_0)$-invariant
inner product of ${\mathfrak m}$ corresponding to the induced metric
${\mathcal G}^{\ast}g^{\rm std}_{Q_{n}({\mathbf C})}$ on $K/K_{0}$.
Thus we know (\cite{Ma-Ohnita1}) that
\begin{equation*}
\left\{\, \frac{1}{\Vert{\gamma}\Vert_{\frak u}}X_{\gamma,i} {\ }\vert{\
}\gamma\in{\Sigma^{+}(U,K)},i=1,2,\cdots,m(\gamma)\, \right\}
\end{equation*}
is an orthonormal basis of ${\mathfrak m}$ relative to
$\langle{\ ,\ }\rangle$.

The Laplace operator $\Delta^{0}_{L^{n}}=\delta d$ acting on $C^{\infty}(K/K_{0},{\mathbf C})$
with respect to the induced metric
$\mathcal{G}^{*}g_{Q_n(\mathbf C)}^{\mathrm{std}}$
corresponds to the linear differential operator $-{\mathcal C}_{L^{n}}$
on $C^{\infty}(K,{\mathbf C})_{K_{0}}$,
where ${\mathcal C}_{L^{n}}\in{\mathrm U}({\mathfrak k})$ is the Casimir operator
relative to the $\mathrm{Ad}_{\mathfrak m}(K_{0})$-invariant inner product
$\langle{\ ,\ }\rangle$ of ${\mathfrak m}$ defined by
\begin{equation}
{\mathcal C}_{L^{n}}:=
\sum_{\gamma\in \Sigma^{+}(U,K)}
\sum^{m(\gamma)}_{i=1}
\frac{1}{||\gamma ||_{\mathfrak{u}}^{2}}(X_{\gamma, i})^{2}.
\end{equation}
Note that ${\mathcal C}_{L^{n}}\in{\mathrm U}({\mathfrak k})_{K_{0}}$
because of the
$\mathrm{Ad}_{\mathfrak m}(K_{0})$-invariance of $\langle{\ ,\ }\rangle$.

Suppose that $\Sigma(U,K)$ is irreducible.
Let $\gamma_{0}$ denote the highest root of $\Sigma(U,K)$.
For $g=3,4$, or $6$, the restricted root system $\Sigma(U,K)$
is
of type $A_2$,
$B_2$,
$BC_2$
or
$G_2$.
Then we know that for each $\gamma\in\Sigma^{+}(U,K)$,
\begin{equation*}
\frac{\Vert{\gamma}\Vert_{\mathfrak u}^{2}}
{\Vert{\gamma_{0}}\Vert_{\mathfrak u}^{2}}
=
\begin{cases}
\, 1
&
\text{ if } \Sigma(U,K) \text{ is of type }A_2, \\
\, 1 \text{ or } 1/3
&
\text{ if } \Sigma(U,K) \text{ is of type }G_2, \\
\, 1 \text{ or } 1/2
&
\text{ if } \Sigma(U,K) \text{ is of type }B_2, \\
\, 1, 1/2 \text{ or } 1/4
&
\text{ if } \Sigma(U,K) \text{ is of type }BC_2.
\end{cases}
\end{equation*}
Set
\begin{equation}
\Sigma^{+}_{1}(U,K)
:=\{\gamma\in\Sigma^{+}(U,K)\mid
\Vert{\gamma}\Vert_{\mathfrak u}^{2}=\Vert{\gamma_{0}}\Vert_{\mathfrak u}^{2}
\}.
\end{equation}
Define a symmetric Lie subalgebra $(\mathfrak{u}_{1},\mathfrak{k}_{1})$
by
\begin{equation*}
\begin{split}
&\mathfrak{k}_{1}
:=\mathfrak{k}_0+\sum_{\gamma\in \Sigma^{+}_{1}(U,K)} \mathfrak{k}_\gamma,\quad
\mathfrak{p}_{1}
:=\mathfrak{a}+\sum_{\gamma\in \Sigma^{+}_{1}(U,K)} \mathfrak{p}_\gamma, \\
&\mathfrak{u}_{1}
:=
\mathfrak{k}_{1}+\mathfrak{p}_{1}.
\end{split}
\end{equation*}
Let $K_{1}$ and $U_{1}$ denote
connected compact Lie subgroups of $K$ and $U$
generated by $\mathfrak{k}_{1}$ and $\mathfrak{u}_{1}$.

Suppose that $\Sigma^{+}(U,K)$ is of type $BC_{2}$.
Define
\begin{equation}
\Sigma^{+}_{2}(U,K)
:=\{\gamma\in\Sigma^{+}(U,K)\mid
\Vert{\gamma}\Vert_{\mathfrak u}^{2}=\Vert{\gamma_{0}}\Vert_{\mathfrak u}^{2}
\text{ or }
\Vert{\gamma_{0}}\Vert_{\mathfrak u}^{2}/2
\}.
\end{equation}
Define a symmetric Lie subalgebra $(\mathfrak{u}_{2},\mathfrak{k}_{2})$ by
\begin{equation*}
\begin{split}
&\mathfrak{k}_{2}
:=\mathfrak{k}_0+\sum_{\gamma\in \Sigma^{+}_{2}(U,K)} \mathfrak{k}_\gamma,\quad
\mathfrak{p}_{2}
:=\mathfrak{a}+\sum_{\gamma\in \Sigma^{+}_{2}(U,K)} \mathfrak{p}_\gamma, \\
&\mathfrak{u}_{2}
:=
\mathfrak{k}_{2}+\mathfrak{p}_{2}.
\end{split}
\end{equation*}
Let $K_{2}$ and $U_{2}$ denote
connected compact Lie subgroups of $K$ and $U$
generated by $\mathfrak{k}_{2}$ and $\mathfrak{u}_{2}$.
We have the following subgroups of $K$
in each case:
\begin{equation*}
  \begin{array}{ll}
    K_0\subset K, & \hbox{ if } \Sigma(U,K) \hbox{ is of type } A_2; \\
     K_0\subset K_1\subset K, & \hbox{ if } \Sigma(U,K) \hbox{ is of type } B_2 \hbox{ or } G_2 ;\\
     K_0\subset K_1 \subset K_2 \subset K, & \hbox{ if } \Sigma(U,K) \hbox{ is of type } BC_2.
  \end{array}
\end{equation*}

Set
\begin{equation}
\begin{split}
&{\mathcal C}_{K/K_0,\langle{\ ,\ }\rangle_{\mathfrak u}}:=
\sum_{\gamma\in \Sigma^{+}(U,K)}
\sum^{m(\gamma)}_{i=1}(X_{\gamma, i})^{2},\\
&\\
&{\mathcal C}_{K_{1}/K_0,\langle{\ ,\ }\rangle_{\mathfrak u}}:=
\sum_{\gamma\in \Sigma^{+}_{1}(U,K)}
\sum^{m(\gamma)}_{i=1}(X_{\gamma, i})^{2},\\
&\\
&{\mathcal C}_{K_{2}/K_0,\langle{\ ,\ }\rangle_{\mathfrak u}}:=
\sum_{\gamma\in \Sigma^{+}_{2}(U,K)}
\sum^{m(\gamma)}_{i=1}(X_{\gamma, i})^{2}.
\end{split}
\end{equation}
Then
${\mathcal C}_{K/K_0}, {\mathcal C}_{K_{1}/K_0}, {\mathcal C}_{K_{2}/K_0}\in
{\mathrm U}({\mathfrak k})_{K_{0}}$
and the Casimir operator $\mathcal{C}_{L^{n}}$ can be decomposed as follows:
\begin{lem}\label{LaplaceOperator}
\begin{equation*}
\begin{split}
\mathcal{C}_{L^{n}}=
\begin{cases}
\,
\displaystyle
\frac{1}{\Vert{\gamma_{0}}\Vert^2_{\mathfrak u}}
{\mathcal C}_{K/K_0,\langle{\ ,\ }\rangle_{\mathfrak u}}
\
\hbox{ if } \Sigma(U,K) \hbox{ is of type } A_2;
\\
&
\\
\displaystyle
\frac{3}{\Vert{\gamma_{0}}\Vert^2_{\mathfrak u}}
{\mathcal C}_{K/K_0,\langle{\ ,\ }\rangle_{\mathfrak u}}
-
\frac{2}{\Vert{\gamma_{0}}\Vert^2_{\mathfrak u}}
{\mathcal C}_{K_{1}/K_0,\langle{\ ,\ }\rangle_{\mathfrak u}}
\
\hbox{ if } \Sigma(U,K) \hbox{ is of type } G_2;
\\
&
\\
\displaystyle
\frac{2}{\Vert{\gamma_{0}}\Vert^2_{\mathfrak u}}
{\mathcal C}_{K/K_0,\langle{\ ,\ }\rangle_{\mathfrak u}}
-
\frac{1}{\Vert{\gamma_{0}}\Vert^2_{\mathfrak u}}
{\mathcal C}_{K_{1}/K_0,\langle{\ ,\ }\rangle_{\mathfrak u}}
\
\hbox{ if } \Sigma(U,K) \hbox{ is of type } B_2;
\\
&
\\
\displaystyle
\frac{4}{\Vert{\gamma_{0}}\Vert^2_{\mathfrak u}}
{\mathcal C}_{K/K_0,\langle{\ ,\ }\rangle_{\mathfrak u}}
-\frac{1}{\Vert{\gamma_{0}}\Vert^2_{\mathfrak u}}
{\mathcal C}_{K_{1}/K_0,\langle{\ ,\ }\rangle_{\mathfrak u}}
-\frac{2}{\Vert{\gamma_{0}}\Vert^2_{\mathfrak u}}
{\mathcal C}_{K_{2}/K_0,\langle{\ ,\ }\rangle_{\mathfrak u}}
\\
\quad\quad\quad\quad\quad\quad\quad\quad\quad\quad
\quad\quad\quad\quad\quad\quad\quad
\
\hbox{ if } \Sigma(U,K) \hbox{ is of type } BC_2.
\end{cases}
\end{split}
\end{equation*}
\end{lem}

\subsection{Fibrations on homogeneous isoparametric hypersurfaces
by homogeneous isoparametric hypersurfaces}

For $g=4 \text{ or } 6$, $(U,K)$ is of type $G_2$, $B_2$ or $BC_2$
as indicated at the 3rd column of Table \ref{List:homogIsoparametric}.

In the case when $(U,K)$ is of type
$B_2$ or $G_2$,
we have one fibration as follows:
\begin{center}
\begin{picture}(100,120)
\put(15,65){$K_{1}/K_{0}$}
\put(-30,90){$N^{n}=K/K_{0}$}
\linethickness{0.5pt}
\put(10,85){\vector(0,-2){45}}
\put(0,30){$K/K_{1}$}
\end{picture}
\end{center}

In the case when $(U,K)$ is of type
$BC_2$,
we have the following two fibrations:
\begin{center}
\begin{picture}(150,120)
\put(90,90){$K/K_{0}$}
\put(15,65){$K_{1}/K_{0}$}
\put(105,65){$K_{2}/K_{0}$}
\put(90,30){$K/K_{2}$}
\linethickness{0.5pt}
\put(35,95){\vector(1,0){45}}
\put(-30,90){$N^{n}=K/K_{0}$}
\put(50,100){$=$}
\linethickness{0.5pt}
\put(10,85){\vector(0,-2){45}}
\linethickness{0.5pt}
\put(100,85){\vector(0,-2){45}}
\put(0,30){$K/K_{1}$}
\put(45,40){$K_{2}/K_{1}$}
\linethickness{0.5pt}
\put(35,34){\vector(1,0){45}}
\end{picture}
\end{center}

\subsubsection{
In case $g=6$ and $(U,K)=(G_{2}, SO(4))$, $(m_{1},m_{2})=(1,1)$}
\begin{center}
\begin{picture}(100,100)
\put(15,65){$K_{1}/K_{0}=SO(3)/({\mathbf Z}_{2}+{\mathbf Z}_{2})$}
\put(-30,90){$N^{6}=K/K_{0}=SO(4)/({\mathbf Z}_{2}+{\mathbf Z}_{2})$}
\linethickness{0.5pt}
\put(10,85){\vector(0,-2){45}}
\put(0,30){$K/K_{1}=SO(4)/SO(3)\cong S^{3}$}
\end{picture}
\end{center}
Here $U_{1}/K_{1}=SU(3)/SO(3)$ is a maximal totally geodesic submanifold of
$U/K=G_{2}/SO(4)$. $K/K_{0}=SO(4)/({\mathbf Z}_{2}+{\mathbf Z}_{2})$ is a homogeneous isoparametric
hypersurface with $g=6$, $m_{1}=m_{2}=1$
and $K_{1}/K_{0}=SO(3)/({\mathbf Z}_{2}+{\mathbf Z}_{2})$ is a homogenous isoparametric hypersurface with
$g=3$, $m_{1}=m_{2}=1$.

\begin{rem0}[\cite{Kimura-Tanaka}]
Maximal totally geodesic submanifolds embedded in
$G_2/SO(4)$ are classified as
$SU(3)/SO(3)$,
$\mathbf{C}P^2$,
$S^{2}\cdot S^{2}$.
\end{rem0}

\subsubsection{In case $g=6$ and $(U,K)=(G_{2}\times G_{2}, G_{2})$,
$(m_{1},m_{2})=(2,2)$}

\begin{center}
\begin{picture}(100,110)
\put(15,65){$K_{1}/K_{0}=SU(3)/T^{2}$}
\put(-30,90){$N^{12}=K/K_{0}=G_{2}/T^{2}$}
\linethickness{0.5pt}
\put(10,85){\vector(0,-2){45}}
\put(0,30){$K/K_{1}=G_{2}/SU(3)\cong S^{6}$}
\end{picture}
\end{center}
Here $U_{1}/K_{1}=(SU(3)\times SU(3))/SU(3)$ is a maximal totally geodesic submanifold of
$U/K=(G_{2}\times G_{2})/G_{2}$.
$K/K_{0}=G_{2}/T^{2}$ is a homogenous isoparametric hypersurface with $g=6$, $m_{1}=m_{2}=2$
and $K_{1}/K_{0}=SU(3)/T^{2}$ is a homogenous isoparametric hypersurface with $g=3$, $m_{1}=m_{2}=2$.
\begin{rem0}[\cite{Kimura-Tanaka}]
Maximal totally geodesic submanifolds embedded in
$G_2$ are classified as
$G_2/SO(4)$,
$SU(3)$,
$S^{3}\cdot S^{3}$.
\end{rem0}

\subsubsection{In case $g=4$ and $(U,K)=(SO(5)\times SO(5), SO(5))$,
$(m_{1},m_{2})=(2,2)$}

\begin{center}
\begin{picture}(100,110)
\put(15,60){$K_{1}/K_{0}=SO(4)/T^{2}$}
\put(-30,90){$N^{8}=K/K_{0}=SO(5)/T^{2}$}
\linethickness{0.5pt}
\put(10,85){\vector(0,-2){45}}
\put(0,30){$K/K_{1}=SO(5)/SO(4)\cong S^{4}$}
\end{picture}
\end{center}
Here $U_{1}/K_{1}=(SO(4)\times SO(4))/SO(4) \cong SO(4)\cong S^{3}\cdot S^{3}$
is a maximal totally geodesic submanifold of $U/K=(SO(5)\times SO(5))/SO(5)\cong SO(5)$.
$K/K_0=SO(5)/T^2$ is a homogeneous isoparametric hypersurface with $g=4$, $m_1=m_2=2$
and $K_{1}/K_{0}=SO(4)/T^{2}\cong S^{2}\times S^{2}$ is a homogeneous isoparametric hypersurface with
$g=2$, $m_{1}=m_{2}=2$.

\begin{rem0}[\cite{Kimura-Tanaka}]
Maximal totally geodesic submanifolds embedded in $Sp(2)\cong Spin(5)$ are classified as
$\widetilde{Gr}_{2}({\mathbf R}^{5})$,
$S^{1}\cdot S^{3}$,
$S^{3}\times S^{3}$,
$S^{4}$.
\end{rem0}

\subsubsection{
In case $g=4$ and $(U,K)=(SO(10), U(5))$, $(m_{1},m_{2})=(4,5)$}

\begin{center}
\begin{picture}(150,120)
\put(100,90)
{$K/K_{0}=\displaystyle\frac{U(5)}{SU(2)\times SU(2)\times U(1)}$}
\put(-95,60)
{$K_{1}/K_{0}=\frac{U(2)\times U(2)\times U(1)}{SU(2)\times SU(2)\times U(1)}$}
\put(155,60)
{$K_{2}/K_{0}=\frac{U(4)\times U(1)}{SU(2)\times SU(2)\times U(1)}$}
\put(110,15)
{$K/K_{2}=\displaystyle\frac{U(5)}{U(4)\times U(1)}$}
\linethickness{0.5pt}
\put(70,95){\vector(1,0){20}}
\put(-130,90)
{$N^{18}=K/K_{0}=\displaystyle\frac{U(5)}{SU(2)\times SU(2)\times U(1)}$}
\put(75,100){$=$}
\linethickness{0.5pt}
\put(-100,80){\vector(0,-2){45}}
\linethickness{1pt}
\put(150,80){\vector(0,-2){45}}
\put(-120,15){$K/K_{1}=\displaystyle\frac{U(5)}{U(2)\times U(2)\times U(1)}$}
\put(20,30){$K_{2}/K_{1}=\frac{U(4)\times U(1)}{U(2)\times U(2)\times U(1)}$}
\linethickness{0.5pt}
\put(50,20){\vector(1,0){50}}
\end{picture}
\end{center}
Here $U_{2}/K_{2}=\frac{SO(8)\times SO(2)}{U(4)\times U(1)}\cong \frac{SO(8)}{U(4)}
\cong \frac{SO(8)}{SO(2)\times SO(6)}
\cong \widetilde{Gr}_{2}({\mathbf R}^{8})$
is a maximal totally geodesic submanifold of $U/K=SO(10)/U(5)$,
but $U_{1}/K_{1} =\frac{SO(4)\times SO(4)\times SO(2)}{U(2)\times U(2)\times U(1)}
\cong \widetilde{Gr}_{2}({\mathbf R}^{4})$ is not a maximal totally geodesic submanifold
of $U_2/K_2$.
Notice that
$K/K_{0} = \frac{U(5)}{SU(2)\times SU(2)\times U(1)}$ is a homogeneous isoparametric hypersurface
with $g=4$, $(m_{1},m_{2})=(4,5)$,
$K_{2}/K_{0}=\frac{U(4)\times U(1)}{SU(2)\times SU(2)\times U(1)}
\cong \frac{SO(2)\times SO(6)}{{\mathbf Z}_{2}\times SO(4)}$ is a homogeneous isoparametric hypersurface
with $g=4$, $(m_{1},m_{2})=(1,4)$ and
$K_{1}/K_{0}=\frac{U(2)\times U(2)\times U(1)}{SU(2)\times SU(2)\times U(1)}
\cong \frac{U(2)}{SU(2)}\times\frac{U(2)}{SU(2)}
\cong S^{1}\times S^{1}$ is a homogeneous isoparametric hypersurface with
$g=2$, $(m_{1},m_{2})=(1,1)$.

\begin{rem0}[\cite{Kimura-Tanaka}]
Maximal totally geodesic submanifolds embedded in
$\displaystyle\frac{SO(10)}{U(5)}$ 
are classified as
$\widetilde{Gr}_2(\mathbf{R}^8)$, $Gr_2(\mathbf{C}^5)$,
$SO(5)$, $S^2\times \mathbf{C}P^3$, $\mathbf{C}P^4$.
\end{rem0}

\begin{rem0}[\cite{Kimura-Tanaka}]
Maximal totally geodesic submanifolds embedded in $\widetilde{Gr}_2(\mathbf{R}^{8})$ are classifed as
$\widetilde{Gr}_2(\mathbf{R}^7)$, $S^p\cdot S^q$ $(p+q=6)$, $\mathbf{C}P^3$.
\end{rem0}

\subsubsection{In case $g=4$ and $(U,K)=(SO(m+2), SO(2)\times SO(m))\, (m\geq 3)$,
$(m_{1},m_{2})=(1,m-2)$}
\begin{small}
\begin{center}
\begin{picture}(100,110)
\put(-55,60){$K_{1}/K_{0}
=
\frac{SO(2)\times SO(2)\times SO(m-2)}{{\mathbf Z}_{2}\times SO(m-2)}
\cong \frac{SO(2)\times SO(2)}{{\mathbf Z}_{2}}
\cong S^{1}\times S^{1}$}
\put(-100,90){$N^{2m-2}=K/K_{0}
=\frac{SO(2)\times SO(m)}{{\mathbf Z}_{2}\times SO(m-2)}$}
\linethickness{0.5pt}
\put(-60,85){\vector(0,-2){45}}
\put(-70,30){$K/K_{1}=
\frac{SO(2)\times SO(m)}{SO(2)\times SO(2)\times SO(m-2)}
\cong \frac{SO(m)}{SO(2)\times SO(m-2)}
\cong \widetilde{Gr}_{2}({\mathbf R}^{m})$}
\end{picture}
\end{center}
\end{small}
Here
$U_{1}/K_{1}=\frac{SO(4)\times SO(m-2)}{SO(2)\times SO(2)\times SO(m-2)}\cong \widetilde{Gr}_{2}({\mathbf R}^{4})
\cong S^{2}\times S^{2}$
is not maximal totally geodesic submanifold of
$U/K=\frac{SO(m+2)}{SO(2)\times SO(m)}\cong \widetilde{Gr}_{2}({\mathbf R}^{m+2})$.
Notice that $K/K_{0} =\frac{SO(2)\times SO(m)}{{\mathbf Z}_{2}\times SO(m-2)}$ is a homogeneous isoparametric hypersurface with
$g=4$, $(m_{1},m_{2})=(1,m-2)$ and
$K_{1}/K_{0}=\frac{SO(2)\times SO(2)\times SO(m-2)}{{\mathbf Z}_{2}\times SO(m-2)}
\cong \frac{SO(2)\times SO(2)}{{\mathbf Z}_{2}}\cong S^{1}\times S^{1}$
is a homogeneous isoparametric hypersurface with
$g=2$, $(m_{1},m_{2})=(1,1)$.

\begin{rem0}[\cite{Kimura-Tanaka}]
Maximal totally geodesic submanifolds embedded in $\widetilde{Gr}_2(\mathbf{R}^{m+2})$ $(m\geq 3)$ are
classified as
$\widetilde{Gr}_2(\mathbf{R}^{m+1})$,
$S^p\cdot S^q (p+q=m)$,
$\mathbf{C}P^{[\frac{m}{2}]}$.
\end{rem0}

\subsubsection{In case $g=4$ and $(U,K)=(SU(m+2), S(U(2)\times U(m))\, (m\geq 2)$,
$(m_{1},m_{2})=(2,2m-3)$}

\begin{itemize}
\item[(i)] $m=2$,
$(U,K)=(SU(4), S(U(2)\times U(2))$,
$(m_{1},m_{2})=(2,1)$

\begin{small}
\begin{center}
\begin{picture}(100,110)
\put(-55,60){
$K_{1}/K_{0}=
\frac{S(U(1)\times U(1)\times U(1)\times U(1))}{S(U(1)\times U(1))}
\cong S^{1}\times S^{1}$
}
\put(-100,90){
$N^{6}=K/K_{0}
=\frac{S(U(2)\times U(2))}{S(U(1)\times U(1))}$
}
\linethickness{0.5pt}
\put(-60,85){\vector(0,-2){45}}
\put(-70,30){$K/K_{1}=
\frac{S(U(2)\times U(2))}{S(U(1)\times U(1)\times U(1)\times U(1))}
\cong S^{2}\times S^{2}$}
\end{picture}
\end{center}
\end{small}

Here $U_{1}/K_{1}= \frac{S(U(2)\times U(2))}{S(U(1)\times U(1)\times U(1)\times U(1))}\cong
S^{2}\times S^{2}$ is not a maximal totally geodesic submanifold in
$U/K=\frac{SU(4)}{S(U(2)\times U(2))}\cong Gr_{2}({\mathbf C}^{4})\cong \widetilde{Gr}_2(\mathbf{R}^6)$.
Notice that $K/K_{0}=\frac{S(U(2)\times U(2))}{S(U(1)\times U(1))}$ is a homogeneous isoparametric hypersurface
 with $g=4$, $(m_{1},m_{2})=(2,1)$ and
$K_{1}/K_{0}
\cong S^{1}\times S^{1}$ is a homogeneous isoparametric hypersurface with $g=2$, $(m_{1},m_{2})=(1,1)$.

\item[(ii)] $m\geq 3$

\begin{small}
\begin{center}
\begin{picture}(150,120)
\put(100,90)
{$\frac{K}{K_{0}}=
\frac{S(U(2)\times U(m))}{S(U(1)\times U(1)\times U(m-2))}$}
\put(-95,60)
{$\frac{K_{1}}{K_{0}}
=\frac{S(U(1)\times U(1)\times U(1)\times U(1)\times U(m-2))}
{S(U(1)\times U(1)\times U(m-2))}$}
\put(155,60)
{$\frac{K_{2}}{K_{0}}=
\frac
{S(U(2)\times U(2)\times U(m-2))}
{S(U(1)\times U(1)\times U(m-2))}$}
\put(110,15)
{$\frac{K}{K_{2}}=
\frac{S(U(2)\times U(m))}
{S(U(2)\times U(2)\times U(m-2))}$}
\linethickness{0.5pt}
\put(70,95){\vector(1,0){20}}
\put(-130,90)
{$N^{4m-2}=\frac{K}{K_{0}}=
\frac{S(U(2)\times U(m))}{S(U(1)\times U(1)\times U(m-2))}$}
\put(75,100){$=$}
\linethickness{0.5pt}
\put(-100,80){\vector(0,-2){45}}
\linethickness{1pt}
\put(150,80){\vector(0,-2){45}}
\put(-120,15){$\frac{K}{K_{1}}=
\frac{S(U(2)\times U(m))}{S(U(1)\times U(1)\times U(1)\times U(1)\times U(m-2))}$}
\put(20,30){$\frac{K_{2}}{K_{1}}
\cong {\mathbf C}P^{1}\times {\mathbf C}P^{1}$}
\linethickness{0.5pt}
\put(50,20){\vector(1,0){50}}
\end{picture}
\end{center}
\end{small}
Here $U_{2}/K_{2}\cong Gr_{2}({\mathbf C}^{4})$
is not a maximal totally geodesic submanifold of
$U/K=\frac{SU(m+2)}{S(U(2)\times U(m))}\cong Gr_{2}({\mathbf C}^{m+2})$
and
$U_{1}/K_{1}= \frac{S(U(2)\times U(2)\times U(m-2))}
{S(U(1)\times U(1)\times U(1)\times U(1)\times U(m-2))}
\cong {\mathbf C}P^{1}\times {\mathbf C}P^{1}$ is not a maximal totally geodesic submanifold
of $U_2/K_2$.
Notice that $K/K_{0}=\frac{S(U(2)\times U(m))}{S(U(1)\times U(1)\times U(m-2))}$ is a homogeneous isoparametric
hypersurface with $g=4$, $(m_{1},m_{2})=(2,2m-3)$,
$K_{2}/K_{0}= \frac{S(U(2)\times U(2)\times U(m-2))}{S(U(1)\times U(1)\times U(m-2))}$
is a homogeneous isoparametric hypersurface with $g=4$, $(m_{1},m_{2})=(2,1)$
and $K_{1}/K_{0}\cong S^{1}\times S^{1}$
is a homogeneous isoparametric hypersurface with
$g=2$, $(m_{1},m_{2})=(1,1)$.
\end{itemize}

\begin{rem0}(\cite{Kimura-Tanaka})
Maximal totally geodesic submanifolds embedded in ${Gr}_2(\mathbf{C}^{m+2})$ 
$(m\geq 3)$ are classified as
${Gr}_2(\mathbf{C}^{m+1})$,
$Gr_2(\mathbf{R}^{m+2})$,
${\mathbf C}P^p\times {\mathbf C}P^q$ $(p+q=m)$,
$\mathbf{H}P^{[\frac{m}{2}]}$.
\end{rem0}

\subsubsection{In case $g=4$ and $(U,K)=(Sp(m+2), Sp(2)\times Sp(m))\, (m\geq 2)$,
$(m_{1},m_{2})=(4,4m-5)$}

\begin{itemize}
\item[(i)] In case $g=4$ and $(U,K)=(Sp(4), Sp(2)\times Sp(2))\, (m=2)$,
$(m_{1},m_{2})=(4,3)$

\begin{small}
\begin{center}
\begin{picture}(100,110)
\put(-55,60){
$K_{1}/K_{0}=
\frac{Sp(1)\times Sp(1)\times Sp(1)\times Sp(1)}{Sp(1)\times Sp(1)}
\cong S^{3}\times S^{3}$
}
\put(-100,90){
$N^{14}=K/K_{0}
=\frac{Sp(2)\times Sp(2)}{Sp(1)\times Sp(1)}$
}
\linethickness{0.5pt}
\put(-60,85){\vector(0,-2){45}}
\put(-70,30){$K/K_{1}=
\frac{Sp(2)\times Sp(2)}{Sp(1)\times Sp(1)\times Sp(1)\times Sp(1)}
\cong {\mathbf H}P^{1}\times {\mathbf H}P^{1}
\cong S^{4}\times S^{4}$}
\end{picture}
\end{center}
\end{small}
Here $U_{1}/K_{1}= \frac{Sp(2)\times Sp(2)}{Sp(1)\times Sp(1)\times Sp(1)\times Sp(1)}
\cong {\mathbf H}P^{1}\times {\mathbf H}P^{1}$ is a maximal totally geodesic submanifold of
$U/K=\frac{Sp(4)}{Sp(2)\times Sp(2)}\cong Gr_{2}({\mathbf H}^{4})$.
Notice that $K/K_{0}=\frac{Sp(2)\times Sp(2)}{Sp(1)\times Sp(1)}$ is a homogeneous isoparametric hypersurface with
$g=4$, $(m_{1},m_{2})=(4,3)$
and $K_{1}/K_{0}=\frac{Sp(1)\times Sp(1)\times Sp(1)\times Sp(1)}{Sp(1)\times Sp(1)}
\cong S^{3}\times S^{3}$ is a homogeneous isoparametric hypersurface with
$g=2$, $(m_{1},m_{2})=(3,3)$.

\item[(ii)] {$m\geq 3$}

\begin{small}
\begin{center}
\begin{picture}(150,120)
\put(100,90)
{$\frac{K}{K_{0}}=
\frac{Sp(2)\times Sp(m)}{Sp(1)\times Sp(1)\times Sp(m-2)}$}
\put(-95,60)
{$\frac{K_{1}}{K_{0}}
=\frac{Sp(1)\times Sp(1)\times Sp(1)\times Sp(1)\times Sp(m-2)}
{Sp(1)\times Sp(1)\times Sp(m-2)}$}
\put(155,60)
{$\frac{K_{2}}{K_{0}}=
\frac
{Sp(2)\times Sp(2)\times Sp(m-2)}
{Sp(1)\times Sp(1)\times Sp(m-2)}$}
\put(110,15)
{$\frac{K}{K_{2}}=
\frac{Sp(2)\times Sp(m))}
{Sp(2)\times Sp(2)\times Sp(m-2)}
$}
\linethickness{0.5pt}
\put(50,95){\vector(1,0){30}}
\put(-130,90)
{$N^{8m-2}=\frac{K}{K_{0}}=
\frac{Sp(2)\times Sp(m)}{Sp(1)\times Sp(1)\times Sp(m-2)}$}
\put(60,100){$=$}
\linethickness{0.5pt}
\put(-100,80){\vector(0,-2){45}}
\linethickness{1pt}
\put(150,80){\vector(0,-2){45}}
\put(-120,15){$\frac{K}{K_{1}}=
\frac{Sp(2)\times Sp(m)}{Sp(1)\times Sp(1)\times Sp(1)\times Sp(1)\times Sp(m-2)}$}
\put(30,30){$\frac{K_{2}}{K_{1}}
\cong {\mathbf H}P^{1}\times {\mathbf H}P^{1}
$}
\linethickness{0.5pt}
\put(60,20){\vector(1,0){40}}
\end{picture}
\end{center}
\end{small}
Here $U_{2}/K_{2}=
\frac{Sp(4)\times Sp(m-2)}{Sp(2)\times Sp(2)\times Sp(m-2)}
\cong Gr_{2}({\mathbf H}^{4})$ is not a maximal totally geodesic submanifold of
$U/K=\frac{Sp(m+2)}{Sp(2)\times Sp(m)}\cong Gr_{2}({\mathbf H}^{m+2})$
but $U_{1}/K_{1}=\frac{Sp(2)\times Sp(2)\times Sp(m-2)}
{Sp(1)\times Sp(1)\times Sp(1)\times Sp(1)\times Sp(m-2)}
\cong {\mathbf H}P^{1}\times {\mathbf H}P^{1}$
is a maximal totally geodesic submanifold of $U_2/K_2$.
Notice that $K/K_{0}=\frac{Sp(2)\times Sp(m)}{Sp(1)\times Sp(1)\times Sp(m-2)}$
is a homogeneous isoparametric hypersurface with $g=4$, $(m_{1},m_{2})=(4,4m-5)$,
$K_{2}/K_{0}
\cong \frac{Sp(2)\times Sp(2)}{Sp(1)\times Sp(1)}$ is a homogeneous isoparametric hypersurface with
$g=4$, $(m_{1},m_{2})=(4,3)$ and
$K_{1}/K_{0}\cong S^{3}\times S^{3}$
is a homogeneous isoparametric hypersurface  with
$g=2$, $(m_{1},m_{2})=(3,3)$.
\end{itemize}

\begin{rem0}(\cite{Kimura-Tanaka})
Maximal totally geodesic submanifolds embedded in ${Gr}_2(\mathbf{H}^{4})$ are classified as
$Sp(2)$, $\mathbf{H}P^2$, $S^1\cdot S^5$, $S^4\times S^4$, $Gr_2(\mathbf{C}^4)$.

Maximal totally geodesic submanifolds embedded in ${Gr}_2(\mathbf{H}^{m+2})$ $(m\geq 3)$ are classified as
${Gr}_2(\mathbf{H}^{m+1})$,
${Gr}_2(\mathbf{C}^{m+2})$,
${\mathbf H}P^p\times {\mathbf H}P^q$ $(p+q=m)$.
\end{rem0}

\subsubsection{In case $g=4$ and $(U,K)=(E_{6}, U(1)\cdot Spin(10))$,
$(m_{1},m_{2})=(6,9)$}

\begin{small}
\begin{center}
\begin{picture}(150,120)
\put(100,90)
{$\frac{K}{K_{0}}=
\frac{U(1)\cdot Spin(10)}{S^{1}\cdot Spin(6)}
$}
\put(-95,60)
{$\frac{K_{1}}{K_{0}}
=\frac{S^{1}\cdot(Spin(2)\cdot(Spin(2)\cdot Spin(6)))}{S^{1}\cdot Spin(6)}$}
\put(155,60)
{$\frac{K_{2}}{K_{0}}=
\frac{U(1)\cdot(Spin(2)\cdot Spin(8))}{S^{1}\cdot Spin(6)}$}
\put(110,15)
{$\frac{K}{K_{2}}=
\frac{U(1)\cdot Spin(10)}{U(1)\cdot(Spin(2)\cdot Spin(8))}
$}
\linethickness{0.5pt}
\put(50,95){\vector(1,0){30}}
\put(-120,90)
{$N^{30}=\frac{K}{K_{0}}=
\frac{U(1)\cdot Spin(10)}{S^{1}\cdot Spin(6)}
$}
\put(60,100){$=$}
\linethickness{0.5pt}
\put(-100,80){\vector(0,-2){45}}
\linethickness{1pt}
\put(150,80){\vector(0,-2){45}}
\put(-120,15){$\frac{K}{K_{1}}=
\frac{U(1)\cdot Spin(10)}{S^{1}\cdot(Spin(2)\cdot(Spin(2)\cdot Spin(6)))}$}
\put(0,35){$\frac{K_{2}}{K_{1}}
=
\frac{U(1)\cdot(Spin(2)\cdot Spin(8))}{S^{1}\cdot(Spin(2)\cdot(Spin(2)\cdot Spin(6))}
$}
\linethickness{0.5pt}
\put(60,20){\vector(1,0){40}}
\end{picture}
\end{center}
\end{small}
Here $U_{2}/K_{2}=\frac{U(1)\cdot Spin(10)}{U(1)\cdot(Spin(2)\cdot Spin(8))}
\cong \widetilde{Gr}_{2}({\mathbf R}^{10})$ is a maximal totally geodesic submanifold of
$U/K=\frac{E_{6}}{U(1)\cdot Spin(10)}$ but
$U_{1}/K_{1}= \frac{S^{1}\cdot Spin(4)\cdot Spin(6)}
{S^{1}\cdot (Spin(2)\cdot Spin(2)\cdot Spin(6))}
\cong S^{2}\times S^{2}$ is not a maximal totally geodesic submanifold in $U_2/K_2$.
Notice that $K/K_{0}= \frac{U(1)\cdot Spin(10)}{S^{1}\cdot Spin(6)}$
is a homogeneous isoparametric hypersurface with
$g=4$, $(m_{1},m_{2})=(6,9)$,
$K_{2}/K_{0}=\frac{U(1)\cdot (Spin(2)\cdot Spin(8))}{S^{1}\cdot Spin(6)}
\cong \frac{Spin(2)\cdot Spin(8)}{Spin(6)}
\cong \frac{SO(2)\times SO(8)}{{\mathbf Z}_{2}\times SO(6)}$
is a homogeneous isoparametric hypersurface with $g=4, (m_{1},m_{2})=(1,6)$
and $K_{1}/K_{0}=\frac{S^{1}\cdot (Spin(2)\cdot (Spin(2)\cdot Spin(6)))}{S^{1}\cdot Spin(6)}
\cong S^{1}\times S^{1}$ is a homogeneous isoparametric hypersurface with
$g=2$, $(m_{1},m_{2})=(1,1)$.

\begin{rem0}[\cite{Kimura-Tanaka}]
Maximal totally geodesic submanifolds embedded in
$E_{6}/U(1)\cdot Spin(10)$
are classified as
${Gr}_2(\mathbf{H}^{4})/{\mathbf Z}_{2}$,
$\mathbf{O}P^{2}$,
$S^{2}\times {\mathbf C}P^{2}$,
$SO(10)/U(5)$,
$Gr_{2}({\mathbf C}^{6})$,
$\widetilde{Gr}_{2}({\mathbf R}^{10})$.
\end{rem0}

\smallskip
\section{The case $(U,K)=(G_2 \times G_2, G_2)$}
\label{Sec:G2xG2}

Let $U=G_{2}\times G_{2}$, $K=\{(x,x)\in U \mid x\in G_{2}\}$ and $(U,K)$ is of type $G_2$.
Then
$K_0=\{k\in K \mid
\mathrm{Ad}(k)H=H \text{ for each } H\in \mathfrak{a}\}\cong T^2$
is a maximal torus of $G_2$
and $N^{12}=K/K_0\cong G_2/T^2$ is
a maximal flag manifold of dimension $n=12$.
Thus its Gauss image is
$L^{12}={\mathcal G}(N^{12})(\cong N^{12}/{\mathbf Z}_{6})
=K\cdot[{\mathfrak a}]\cong (K/K_{[{\mathfrak a}]})
\subset Q_{12}({\mathbf C})$.

Set
$\langle{\ ,\ }\rangle_{\mathfrak u}=-B_{\mathfrak u}({\ ,\ })$,
where $B_{\mathfrak u}({\ },{\ })$ denotes the Killing-Cartan form of
${\mathfrak u}$.
Let
$\langle{\ ,\ }\rangle$
be the inner product of ${\mathfrak m}$
corresponding to the invariant induced metric on $L^{n}$
from $(Q_{n}({\mathbf C}),g^{\mathrm{std}}_{Q_{n}({\mathbf C})})$.

The restricted root system $\Sigma(U,K)$ of type $G_2$,
can be given as follows (\cite{Bourbaki}):
\begin{equation*}
\begin{split}
\Sigma(U,K)
=&
\{
\pm(\varepsilon_{1}-\varepsilon_{2})=\pm\alpha_{1},
\pm(\varepsilon_{3}-\varepsilon_{1})=\pm(\alpha_{1}+\alpha_{2}),\\
&\, \pm(\varepsilon_{3}-\varepsilon_{2})=\pm(2\alpha_{1}+\alpha_{2}),
\pm(-2\varepsilon_{1}+\varepsilon_{2}+\varepsilon_{3})=\pm\alpha_{2},\\
&\,
\pm(\varepsilon_{1}-2\varepsilon_{2}+\varepsilon_{3})
=\pm(3\alpha_{1}+\alpha_{2}),\\
&\,
\pm(2\varepsilon_{3}-\varepsilon_{1}-\varepsilon_{2})
=\pm(3\alpha_{1}+2\alpha_{2})=\tilde{\alpha}
\},
\end{split}
\end{equation*}
where
$
\Pi(U,K)=
\{
\alpha_{1}=\varepsilon_{1}-\varepsilon_{2},
\alpha_{2}=-2\varepsilon_{1}+\varepsilon_{2}+\varepsilon_{3}
\}
$
is its fundamental root system.
Here
\begin{equation*}
\begin{split}
\Vert{\gamma}\Vert_{\mathfrak u}^{2}
=
\begin{cases}
\displaystyle\frac{1}{24}&\text{ if }\gamma\text{ is short}, \\
&\\
\displaystyle\frac{1}{8}&\text{ if }\gamma\text{ is long}.
\end{cases}
\end{split}
\end{equation*}
Now $K_1=SU(3)$ and $K_0=T^2\subset K_1=SU(3)\subset K=G_2$.

In Lemma \ref{LaplaceOperator}
the Casimir operator
\begin{equation*}
\mathcal{C}_{L}=
\frac{3}{\Vert{\gamma_0}\Vert_{\mathfrak u}^{2}}
\mathcal{C}_{K/K_0, \langle\,,\,\rangle_{\mathfrak u}}
-\frac{2}{\Vert{\gamma_0}\Vert_{\mathfrak u}^{2}}
\mathcal{C}_{K_1/K_0, \langle\,,\,\rangle_{\mathfrak u}},
\end{equation*}
of $L^n$ with respect to $\langle{\ ,\ }\rangle$
corresponding to $-\Delta_{L^{12}}$
becomes
\begin{equation*}
\begin{split}
\mathcal{C}_{L}
=&
24\,
\mathcal{C}_{K/K_0, \langle\, , \,\rangle_{\mathfrak u}}
-
16\,
\mathcal{C}_{K_1/K_0, \langle\,,\,\rangle_{\mathfrak u}}\\
=&
12\, \mathcal{C}^{\mathfrak k}_{K/K_0}
-
8\, \mathcal{C}^{\mathfrak k}_{K_1/K_0}
\\
=&
12\, \mathcal{C}^{\mathfrak k}_{K/K_0}
-
6\,
\mathcal{C}^{{\mathfrak k}_1}_{K_1/K_0},
\end{split}
\end{equation*}
where
$\mathcal{C}^{\mathfrak k}_{K/K_0}$ and
$\mathcal{C}^{\mathfrak k}_{K_1/K_0}$
denote
the Casimir operators of $K/K_0$ and $K_1/K_0$
relative to the $K_0$-invariant inner product induced from
the Killing-Cartan form of ${\mathfrak k}$, respectively,
and
$\mathcal{C}^{{\mathfrak k}_1}_{K_1/K_0}$ denotes
the Casimir operator of $K_1/K_0$
relative to the $K_0$-invariant inner product induced from
the Killing-Cartan form of ${\mathfrak k}_1$.

Let $\{\alpha_1, \alpha_2\}$ be the fundamental root system of $G_2$
and
$\{\Lambda_1, \Lambda_2\}$ be the fundamental weight system of $G_2$.
In our work we frequently use Satoru Yamaguchi's results (\cite{SYamaguchi79})
on the spectra of maximal flag manifolds.
\begin{lem}[Branching law of $(G_{2},T^{2})$ \cite{SYamaguchi79}]
\begin{equation}
\begin{split}
D(K,K_0)=&\, D(G_{2},T^{2})=\, D(G_{2}) \\
=&
\{
\Lambda=m_{1}\Lambda_{1}+m_{2}\Lambda_{2}
\mid
m_{1},m_{2}\in{\mathbf Z}, m_{1}\geq{0}, m_{2}\geq{0}
\}
\\
=&
\{
\Lambda=p_{1}\alpha_{1}+p_{1}\alpha_{2}
\mid
p_{1},p_{2}\in{\mathbf Z}, p_{1}\geq{1}, p_{2}\geq{1}
\}
\end{split}
\end{equation}
The eigenvalue formula of the Casimir operator $\mathcal{C}_{K/K_0}$
relative to the inner product induced from
the Killing-Cartan form of $G_2$ is
\begin{equation}\label{eigenvalueG2T2}
-c(\Lambda,\langle{\ ,\ }\rangle_{{\mathfrak g}_{2}})
=\frac{1}{24}(m_{1}p_{1}+3m_{2}p_{2}+2p_{1}+6p_{2})
\end{equation}
for each $\Lambda\in{D(G_{2},T^{2})}=D(G_{2})$.
\end{lem}
Since
\begin{equation*}
-\mathcal{C}_{L}
=-\left(
4\,
\mathcal{C}^{{\mathfrak g}_{2}}_{K/K_0}
+
\sum_{\gamma\text{:short}}
16\, (X_{\gamma,i})^{2}\right)
\geq
-4\,
\mathcal{C}^{{\mathfrak g}_{2}}_{K/K_0}\, ,
\end{equation*}
if
the eigenvalue $-c_L$ of $-{\mathcal C}_L$
satisfies $-c_L \leq n=12$,  then $-c_\Lambda \leq 3$.

By using the formula \eqref{eigenvalueG2T2},
we get
\begin{equation*}
\begin{split}
&\{\Lambda\in{D(G_2,T^{2})} \mid
-c(\Lambda,\langle{\ ,\ }\rangle_{\mathfrak{g}_2})\leq{3}\} \\
=&\{0,
\Lambda_1((p_1,p_2)=(2,1)),
2\Lambda_1((p_1,p_2)=(4,2)),
3\Lambda_1((p_1,p_2)=(6,3)), \\
\quad & \Lambda_2((p_1,p_2)=(3,2)),
           2\Lambda_2((p_1,p_2)=(6,4)),
            \Lambda_1+\Lambda_2((p_1,p_2)=(5,3)), \\
\quad & 2\Lambda_1+\Lambda_2 ((p_1,p_2)=(7,4)) \}.
\end{split}
\end{equation*}

Let $\{\alpha^{\prime}_1, \alpha^{\prime}_2\}$
be the fundamental root system of $SU(3)$
and
$\{\Lambda^{\prime}_1, \Lambda^{\prime}_2\}$
be the fundamental weight system of $SU(3)$.
For each $\Lambda\in D(G_2, T^2)$ with
$-c(\Lambda, \langle\, ,\rangle_{\mathfrak{g}_2})\leq 3$,
by using the branching law of $(G_2, SU(3))$ in \cite{McKay-Patera},
we can determine
all irreducible $SU(3)$-submodule $V_{\Lambda^{\prime}}$
with the highest weight
$\Lambda^{\prime}=m_1^{\prime}\Lambda_1^{\prime}+m_2\Lambda^{\prime}_2$
contained in an irreducible $G_2$-module $V_{\Lambda}$
as in the following table:
\begin{center}
\begin{tabular}
{|c|c|c|c|l|}
\hline
$(m_{1},m_{2})$
&
$(p_{1},p_{2})$
&
$-c$
&
$\dim_{\mathbf C}V_{\Lambda}$
&
irred.\, $SU(3)$-submodules $(m^{\prime}_{1},m^{\prime}_{2})$
\\
\hline $(1,0)$ &$(2,1)$ &$\frac{1}{2}$ & $7$ 
& $(1,0), (0,1), (0,0)$
\\
\hline
$(2,0)$ &$(4,2)$ &$\frac{7}{6}$ & $27$ 
&$\begin{array}{l}
(2,0), (1,1), (0,2),(1,0), (0,1),\\
 (0,0)
\end{array}$
\\
\hline
$(3,0)$ &$(6,3)$ &$2$ & $77$
& $\begin{array}{l}
  (3,0), (2,1), (1,2), (0,3),  (2,0), \\
 (1,1), (0,2), (1,0), (0,1), (0,0)
\end{array}$
\\
\hline
$(0,1)$ &$(3,2)$ &$1$ & $14$
& $(1,1), (1,0), (0,1)$
\\
\hline
$(0,2)$ &$(6,4)$ &$\frac{5}{2}$ & $77$
& $\begin{array}{l}
(2,2), (2,1), (1,2), (2,0), (1,1),\\
(0,2)
\end{array}$
\\
\hline
$(1,1)$ &$(5,3)$ &$\frac{7}{4}$ & $64$
& $\begin{array}{l}
(2,1), (1,2), (2,0), 2(1,1), (0,2), \\
(1,0), (0,1)
\end{array}$
\\
\hline
$(2,1)$ &$(7,4)$ &$\frac{8}{3}$ & $189$ 
& $\begin{array}{l}
(3,1), (2,2), (1,3), (3,0), 2(2,1), \\
2(1,2),(0,3), (2,0), 2(1,1), (0,2), \\
(1,0), (0,1)
\end{array}$
\\
\hline
\end{tabular}
\end{center}
Since
\begin{equation*}
\begin{split}
&{\mathfrak g}_{2}^{\mathbf{C}}
={\mathfrak t}^{\mathbf{C}}+
\sum_{\alpha\in\Sigma(G_{2})}{\mathfrak g}^{\alpha}
={\mathfrak t}^{\mathbf{C}}
+\sum_{\alpha\text{:short}}{\mathfrak g}^{\alpha}
+\sum_{\alpha\text{:long}}{\mathfrak g}^{\alpha}, \\
&\mathfrak{su}(3)^{\mathbf{C}}
={\mathfrak t}^{\mathbf{C}}+\sum_{\alpha\text{:long}}{\mathfrak g}^{\alpha},
\end{split}
\end{equation*}
we know that
\begin{equation*}
\begin{split}
T^{2}\cdot{\mathbf Z}_{6}
=&\{a\in{G_{2}}{\ }\vert{\ }\mathrm{Ad}(a)({\mathfrak t})={\mathfrak t}
\text{ preserving the orientation of }{\mathfrak t}\}\\
\supset&
\{a\in{SU(3)}{\ }\vert{\ }\mathrm{Ad}(a)({\mathfrak t})={\mathfrak t}
\text{ preserving the orientation of }{\mathfrak t}\}\\
=&T^{2}\cdot{\mathbf Z}_{3}.
\end{split}
\end{equation*}
Now we use results in $SU(3)/T^2$, which were already treated
in the case of $g=3$ and $m=2$ (\cite{Ma-Ohnita1}).

\begin{lem}[Branching law of $(SU(3),T^2)$ \cite{SYamaguchi79}]
\label{BranchingLawSU3T2}
\begin{equation}
\begin{split}
D(K_1,K_0)=&D(SU(3),T^{2}) \\
=&
\{
\Lambda^{\prime}
=
m^{\prime}_{1}\Lambda^{\prime}_{1}+m^{\prime}_{2}\Lambda^{\prime}_{2}
\mid
m^{\prime}_{i}\in{\mathbf Z}, m^{\prime}_{i}\geq{0}
\}
\\
=&
\{
\Lambda^{\prime}
=
p^{\prime}_{1}\alpha^{\prime}_{1}+p^{\prime}_{2}\alpha^{\prime}_{2}
\mid
p^{\prime}_{i}\in{\mathbf Z}, p^{\prime}_{i}\geq{1}
\},
\end{split}
\end{equation}
where
\begin{equation*}
m^{\prime}_{1}=2p^{\prime}_{1}-p^{\prime}_{2}\geq{0}, \quad
m^{\prime}_{2}=-p^{\prime}_{1}+2p^{\prime}_{2}\geq{0}.
\end{equation*}
The eigenvalue formula is
\begin{equation}\label{eigenvalueSU3T2}
-c(\Lambda^{\prime},\langle{{\ },{\ }}\rangle_{{\mathfrak su}(3)})
=\frac{1}{6}(m^{\prime}_{1}p^{\prime}_{1}+m^{\prime}_{2}p^{\prime}_{2}
+2p^{\prime}_{1}+2p^{\prime}_{2})
\end{equation}
for each $\Lambda^{\prime}\in{D(SU(3),T^{2})}$.
\end{lem}
Using Lemma \ref{BranchingLawSU3T2},
we get that
$\Lambda^{\prime}=
m_1^{\prime}\Lambda_1^{\prime}+m_2^{\prime}\Lambda_2^{\prime}\in D(SU(3), T^2)$
such that $V_{\Lambda^{\prime}}\subset V_{\Lambda}$ for some
$\Lambda\in D(G_2,T^2)$ with
$-c(\Lambda, \langle{\ ,\ }\rangle_{\mathfrak{g}_2})\leq 3$ satisfies
\begin{equation*}
(m_1^{\prime}, m_2^{\prime})\in \{(1,1), (3,0),(0,3),(2,2)\}.
\end{equation*}
By using  the formula \eqref{eigenvalueSU3T2},
we compute
the corresponding eigenvalues of $\mathcal{C}_{K_1/K_0}$
as follows:
\begin{center}
\begin{tabular}
{|c|c|c|}
\hline
$(p^{\prime}_{1},p^{\prime}_{2})$
&
$(m^{\prime}_{1},m^{\prime}_{2})$
&
$-c^{\prime}=-c(\Lambda^{\prime},\langle{{\ },{\ }}\rangle_{{\mathfrak su}(3)})$
\\
\hline
$(1,1)$ &$(1,1)$ &$1$
\\
\hline
$(2,1)$ &$(3,0)$ &$2$
\\
\hline
$(1,2)$ &$(0,3)$ &$2$
\\
\hline
$(2,2)$ &$(2,2)$ &$\frac{8}{3}$
\\
\hline
\end{tabular}
\end{center}
Therefore,
for all $\Lambda\in D(G_2, T^2)$ and all $\Lambda^{\prime}\in D(SU(3),T^2)$
such that
$V_{\Lambda^{\prime}}\subset V_{\Lambda}$
and $-c(\Lambda, \langle \, , \, \rangle_{\mathfrak{g}_2})\leq 3$,
the corresponding eigenvalues of
$-\mathcal{C}_L=-12\, \mathcal{C}^{\mathfrak k}_{K/K_0}
+6\, \mathcal{C}^{{\mathfrak k}_1}_{K_1/K_0}
=-12c+6c^{\prime}$
are given in the following table:
\begin{center}
\begin{tabular}
{|c|c|c|c|c|c|c|}
\hline
$(m_{1},m_{2})$
&
$(p_{1},p_{2})$
&
$\dim_{\mathbf C}V_{\Lambda}$
&
$-c$
&
$(m^{\prime}_{1},m^{\prime}_{2})$
&
$-c^{\prime}$
&
$-12c+6c^{\prime}$
\\
\hline
$(2,0)$ &$(4,2)$ & $27$ &$\frac{7}{6}$ &
$(1,1)$ & $1$ & $8$
\\
\hline
$(3,0)$ &$(6,3)$ & $77$ &$2$ &
$(1,1)$ & $1$ & $18$
\\
\hline
$(3,0)$ &$(6,3)$ & $77$ &$2$ &
$(3,0)$ & $2$ & $12$
\\
\hline
$(3,0)$ &$(6,3)$ & $77$ &$2$ &
$(0,3)$ & $2$ & $12$
\\
\hline
$(0,1)$ &$(3,2)$ & $14$ &$1$ &
$(1,1)$ & $1$ & $6$
\\
\hline
$(0,2)$ &$(6,4)$ & $77$ &$\frac{5}{2}$ &
$(1,1)$ & $1$ & $24$
\\
\hline
$(0,2)$ &$(6,4)$ & $77$ &$\frac{5}{2}$ &
$(2,2)$ & $\frac{8}{3}$ & $14$
\\
\hline
$(1,1)$ &$(5,3)$ & $64$ &$\frac{7}{4}$ &
$2(1,1)$ & $1$ & $15$
\\
\hline
$(2,1)$ &$(7,4)$ & $189$ &$\frac{8}{3}$ &
$2(1,1)$ & $1$ & $26$
\\
\hline
$(2,1)$ &$(7,4)$ & $189$ &$\frac{8}{3}$ &
$(3,0)$ & $2$ & $20$
\\
\hline
$(2,1)$ &$(7,4)$ & $189$ &$\frac{8}{3}$ &
$(0,3)$ & $2$ & $20$
\\
\hline
$(2,1)$ &$(7,4)$ & $189$ &$\frac{8}{3}$ &
$(2,2)$ & $\frac{8}{3}$ & $16$
\\
\hline
\end{tabular}
\end{center}
Since $\Lambda^{\prime}_{1}+\Lambda^{\prime}_{2}
\,
((m^{\prime}_{1},m^{\prime}_{2})=(1,1))$
corresponds to the complexified adjoint representation of $SU(3)$,
we see that
$(V^{\prime}_{\Lambda_1^{\prime}+\Lambda_2^{\prime}})_{T^2}\cong\mathfrak{t}^2$
and
$(V^{\prime}_{\Lambda_1^{\prime}+
\Lambda_2^{\prime}})_{T^2 \cdot \mathbf{Z}_3}=\{0\}$.
Then
\begin{equation*}
\Lambda^{\prime}_{1}+\Lambda^{\prime}_{2}
\not\in{D(SU(3),T^{2}\cdot{\mathbf Z}_{3})}.
\end{equation*}
and thus
\begin{equation*}
 2\Lambda_{1}, \Lambda_{2}
\not\in{D(G_{2},T^{2}\cdot{\mathbf Z}_{6})}.
\end{equation*}
Now we obtain that
${\mathcal G}(G_{2}/T^{2})\subset Q_{12}({\mathbf C})$ is Hamiltonian stable.

We need to examine whether
$3\Lambda_{1}\in D(K,K_{[\mathfrak a]})={D(G_{2},T^{2}\cdot{\mathbf Z}_{6})}$
or not.
Consider
$$
(V_{3\Lambda_{1}})_{T^{2}}
=
(V^{\prime}_{3\Lambda^{\prime}_{1}})_{T^{2}}
\oplus
(V^{\prime}_{3\Lambda^{\prime}_{2}})_{T^{2}}
\oplus
(V^{\prime}_{\Lambda^{\prime}_{1}+\Lambda^{\prime}_{2}})_{T^{2}}.
$$
Since
$$
V^{\prime}_{3\Lambda_1^{\prime}}\cong \mathrm{Sym}^3(\mathbf{C}^3)
=\mathrm{span}_{\mathbf C}\{e_{i_1}\cdot e_{i_2} \cdot e_{i_3} \mid
1\leq i_1\leq i_2\leq i_3\leq 3\},
$$
where $\{e_1,e_2,e_3\}$ is the standard basis of $\mathbf{C}^3$,
we get
$$
(V^{\prime}_{3\Lambda^{\prime}_{1}})_{T^{2}}
=(V^{\prime}_{3\Lambda^{\prime}_{1}})_{T^{2}
\cdot{\mathbf Z}_{3}}
=\mathrm{span}_{\mathbf C}\{e_1\cdot e_2\cdot e_3\}.
$$
Similarly, we get
$V^{\prime}_{3\Lambda^{\prime}_{2}}\cong\mathrm{Sym}^3(\bar{\mathbf{C}}^3)$
and
$(V^{\prime}_{3\Lambda^{\prime}_{2}})_{T^{2}}
=(V^{\prime}_{3\Lambda^{\prime}_{2}})_{T^{2}\cdot{\mathbf Z}_{3}}$
with dimension $1$.
On the other hand we know that
$(V^{\prime}_{\Lambda^{\prime}_{1}+\Lambda^{\prime}_{2}})_{T^{2}}
\cong{\mathfrak t}$ and
$(V^{\prime}_{\Lambda^{\prime}_{1}+\Lambda^{\prime}_{2}})_{T^{2}\cdot{\mathbf Z}_{3}}
=\{0\}$.
Hence we get
$\dim_{\mathbf C}(V_{3\Lambda_{1}})_{T^{2}}=4$ and
$\dim_{\mathbf C}(V_{3\Lambda_{1}})_{T^{2}\cdot{\mathbf Z}_{3}}=2$.
However
$\dim_{\mathbf C}(V_{3\Lambda_{1}})_{T^{2}\cdot{\mathbf Z}_{6}}=1$.
In fact, $T^{2}\cdot{\mathbf Z}_{6}\subset G_{2}$,
$T^{2}\cdot{\mathbf Z}_{6}\not\subset SU(3)$,
$T^{2}\cdot{\mathbf Z}_{3}\subset SU(3)$
and
$(T^{2}\cdot{\mathbf Z}_{6})/(T^{2}\cdot{\mathbf Z}_{3})
\cong {\mathbf Z}_{2}$.
Thus
there exists an element $u\in{T^{2}\cdot{\mathbf Z}_{6}}\subset G_2$ with
$u\not\in SU(3)$ which satisfies
$\mathrm{Ad}(u)(SU(3))\subset SU(3)$
and provides the generators of both
$(T^{2}\cdot{\mathbf Z}_{6})/T^{2}
\cong {\mathbf Z}_{6}$
and $(T^{2}\cdot{\mathbf Z}_{6})/(T^{2}\cdot{\mathbf Z}_{3})
\cong {\mathbf Z}_{2}$.
Then we observe that
$\rho_{3\Lambda^{\prime}_{1}}\circ\mathrm{Ad}(u)\vert_{SU(3)}
\cong \rho_{3\Lambda^{\prime}_{2}}$ and
$\rho_{3\Lambda_{1}}(u)(V^{\prime}_{3\Lambda^{\prime}_{1}})
=V^{\prime}_{3\Lambda^{\prime}_{2}}$.
Thus
$\rho_{3\Lambda_{1}}(u)
(V^{\prime}_{3\Lambda^{\prime}_{1}})_{T^{2}\cdot{\mathbf Z}_{3}}
=(V^{\prime}_{3\Lambda^{\prime}_{2}})_{T^{2}\cdot{\mathbf Z}_{3}}$
and
$
(\rho_{3\Lambda_{1}}(u))^{2}
\vert_{(V^{\prime}_{3\Lambda^{\prime}_{1}})_{T^{2}\cdot{\mathbf Z}_{3}}}
=
(\rho_{3\Lambda_{1}}(u^{2}))
\vert_{(V^{\prime}_{3\Lambda^{\prime}_{1}})_{T^{2}\cdot{\mathbf Z}_{3}}}
=
\mathrm{Id},
$
because $u^{2}\in T^{2}\cdot{\mathbf Z}_{3}$.
Hence we have
$
(V_{3\Lambda_{1}})_{T^{2}\cdot{\mathbf Z}_{6}}
\subset
(V^{\prime}_{3\Lambda^{\prime}_{1}})_{T^{2}\cdot{\mathbf Z}_{3}}
\oplus
(V^{\prime}_{3\Lambda^{\prime}_{2}})_{T^{2}\cdot{\mathbf Z}_{3}}
$
and
$
\dim (V_{3\Lambda_{1}})_{T^{2}\cdot{\mathbf Z}_{6}}=1
$.
Therefore
$3\Lambda_{1}\in{D(G_{2},T^{2}\cdot{\mathbf Z}_{6})}$
and its multiplicity is equal to $1$.
It follows that
\begin{equation*}
n(L^{12})
=\dim_{\mathbf C}(V_{3\Lambda_{1}})
=77
=91-14
=\dim(SO(14))-\dim(G_{2})
=n_{hk}(L^{12}),
\end{equation*}
that is,
${\mathcal G}(G_{2}/T^{2})\subset Q_{12}({\mathbf C})$ is Hamiltonian rigid.

Let $\bigwedge^{2}{\mathbf R}^{14}
={\mathfrak o}(n+2)
=\mathrm{ad}_{\mathfrak p}({\mathfrak g}_{2})+\mathcal{V}
\cong
{\mathfrak g}_{2}+\mathcal{V}$.
Then
\begin{equation*}
\bigwedge^{2}{\mathbf C}^{14}
=(\bigwedge^{2}{\mathbf R}^{14})^{\mathbf C}
={\mathfrak o}(n+2)^{\mathbf C}
={\mathfrak o}(n+2,{\mathbf C})
=\mathrm{ad}_{\mathfrak p}({\mathfrak g}^{\mathbf{C}}_{2})+\mathcal{V}^{\mathbf C}
\cong
{\mathfrak g}_{2}^{\mathbf C}+\mathcal{V}^{\mathbf C},
\end{equation*}
where
$\dim\mathcal{V}=77$ and
$\dim_{\mathbf C}\mathcal{V}^{\mathbf C}=77$.
More precisely, we observe that
$\mathcal{V}$ is a real $77$-dimensional irreducible $G_{2}$-module
with $(\mathcal{V})_{T^{2}\cdot{\mathbf Z}_{6}}\not=\{0\}$,
and $\mathcal{V}^{\mathbf C}$
is a complex $77$-dimensional $G_{2}$-module
with $(\mathcal{V})^{\mathbf C}_{T^{2}\cdot{\mathbf Z}_{6}}\not=\{0\}$.
Moreover, we have
$\mathcal{V}^{\mathbf C}\cong V_{3\Lambda_{1}}$
with $\dim_{\mathbf C}(\mathcal{V})^{\mathbf C}_{T^{2}\cdot{\mathbf Z}_{6}}=1$.

From these arguments we conclude that
\begin{thm}
The Gauss image
$L^{12}
={\mathcal G}(G_{2}/T^{2})
=\frac{G_{2}}{T^{2}\cdot \mathbf{Z}_6}
\subset Q_{12}({\mathbf C})$
is strictly Hamiltonian stable.
\end{thm}

\section{The case $(U,K)=(G_2, SO(4))$}
\label{Section_G2 SO4}
Let $U=G_2$, $K=SO(4)$ and $(U,K)$ is of type $G_2$.
Let ${\mathfrak u}={\mathfrak k} + {\mathfrak p}$
be the orthogonal symmetric Lie algebra of $(G_2, SO(4))$
and ${\mathfrak a}$ be the maximal abelian subspace of ${\mathfrak p}$.
Here
${\mathfrak u}={\mathfrak g}_{2}$,
${\mathfrak k}={\mathfrak so}(4)\cong\mathfrak{su}(2)\oplus\mathfrak{su}(2)$.
Let
\begin{equation*}
p:\widetilde{K}=Spin(4)=SU(2)\times SU(2)\longrightarrow K=SO(4)
\end{equation*}
be the universal covering Lie group homomorphism with
kernel ${\mathbf Z}_{2}$.

Recall that the complete set of all inequivalent irreducible unitary
representations of $SU(2)$ is given by
\begin{equation*}
{\mathcal D}(SU(2)) =\{(V_{m},\rho_{m}) \mid
m\in{\mathbf Z}, m\geq{0}\},
\end{equation*}
where $V_{m}$ denotes the complex vector space of complex homogeneous
polynomials of degree $m$ with two variables $z_{0},z_{1}$ and the
representation $\rho_{m}$ of $SU(2)$ on $V_{m}$ is defined by
$(\rho_{m}(g)f)(z_{0},z_{1}) =f((z_{0},z_{1})g)$
for each $g\in SU(2)$.
Set
\begin{equation}\label{eq:basis_SU(2)}
v^{(m)}_{k}(z_0,z_1):=\frac{1}{\sqrt{k!(m-k)!}}z_{0}^{m-k}z_{1}^{k}
\in{V_{m}} \quad(k=0,1,\dots,m)
\end{equation}
and define
the standard Hermitian inner product of $V_{m}$ invariant under
$\rho_{m}(SU(2))$
such that $\{v^{(m)}_{0},\dots,v^{(m)}_{m}\}$ is a unitary basis of $V_{m}$.
Let $(V_{l}\otimes V_{m},\rho_{l}\boxtimes \rho_{m})$ denote
an irreducible unitary representation of $SU(2)\times SU(2)$
of complex dimension $(l+1)(m+1)$
obtained by taking the exterior tensor product of $V_{l}$ and $V_{m}$
and then
\begin{equation*}
\{(V_{l}\otimes V_{m},\rho_{l}\boxtimes \rho_{m}) \mid
l,m\in{\mathbf Z},{\ }l,m\geq{0}\}
\end{equation*}
is the complete set of all inequivalent
irreducible unitary representations of $SU(2)\times SU(2)$.

The isotropy representation of $(G_2,SO(4))$ is explicitly described
as follows (cf.\, \cite{Goto-Grosshans1978}):
Suppose that $(l,m)=(3,1)$.
The real $8$-dimensional vector subspace $W$
of $V_{3}\otimes V_{1}$ spanned over ${\mathbf R}$ by
\begin{equation*}
\begin{split}
\{{\ }
&v_{0}^{(3)}\otimes v_{0}^{(1)}+v_{3}^{(3)}\otimes v_{1}^{(1)},{\ }
\sqrt{-1}{\ }(v_{0}^{(3)}\otimes v_{0}^{(1)}-v_{3}^{(3)}\otimes v_{1}^{(1)}),\\
&v_{1}^{(3)}\otimes v_{0}^{(1)}-v_{2}^{(3)}\otimes v_{1}^{(1)},{\ }
\sqrt{-1}{\ }(v_{1}^{(3)}\otimes v_{0}^{(1)}+v_{2}^{(3)}\otimes v_{1}^{(1)}),\\
&v_{0}^{(3)}\otimes v_{1}^{(1)}-v_{3}^{(3)}\otimes v_{0}^{(1)},{\ }
\sqrt{-1}{\ }(v_{0}^{(3)}\otimes v_{1}^{(1)}+v_{3}^{(3)}\otimes v_{0}^{(1)}),\\
&v_{2}^{(3)}\otimes v_{0}^{(1)}+v_{1}^{(3)}\otimes v_{1}^{(1)},{\ }
\sqrt{-1}{\ }(v_{2}^{(3)}\otimes v_{0}^{(1)}-v_{1}^{(3)}\otimes v_{1}^{(1)})
{\ }\}
\end{split}
\end{equation*}
gives an irreducible orthogonal representation of $SU(2)\times SU(2)$
whose complexification is $V_{3}\otimes V_{1}$,
i.e.\, $W$ is a {\it real form} of $V_{3}\otimes V_{1}$.
Then the isotropy representation $\mathrm{Ad}_{\mathfrak p}$ of $(G_2, SO(4))$
is given by $\mathrm{Ad}_{\mathfrak p}^{\mathbf C} \circ p \cong \rho_3\boxtimes \rho_1$
and the vector space ${\mathfrak p}$ is linearly isomorphic to $W$.
Moreover ${\mathfrak a}$ corresponds to a vector subspace
\begin{equation*}
{\mathbf R}
(v_{0}^{(3)}\otimes v_{0}^{(1)}+v_{3}^{(3)}\otimes v_{1}^{(1)})
+{\mathbf R}
(v_{2}^{(3)}\otimes v_{0}^{(1)}+v_{1}^{(3)}\otimes v_{1}^{(1)}){\ }.
\end{equation*}

For each
$X=
\begin{pmatrix}
\sqrt{-1}x&u\\
-\bar{u}&-\sqrt{-1}x
\end{pmatrix}$,
$Y=
\begin{pmatrix}
\sqrt{-1}y&w\\
-\bar{w}&-\sqrt{-1}y
\end{pmatrix}
\in\mathfrak{su}(2)$,
the following useful formula holds:
\begin{lem}
\begin{equation}\label{usefulformula}
\begin{split}
&[d(\rho_{l}\boxtimes\rho_{m})(X,Y)]
\left(
v_{i}^{(l)}\otimes v_{j}^{(m)}{\pm}v_{l-i}^{(l)}\otimes v_{m-j}^{(m)}
\right)\\
=&
\{(2i-l)x+(2j-m)y\}{\ }
\sqrt{-1}
(v_{i}^{(l)}\otimes v_{j}^{(m)}{\mp}v_{l-i}^{(l)}\otimes v_{m-j}^{(m)})\\
&-\sqrt{i(l-i+1)}{\ }\mathrm{Re}(u){\ }
(v_{i-1}^{(l)}\otimes v_{j}^{(m)}{\mp}v_{l-i+1}^{(l)}\otimes v_{m-j}^{(m)})\\
&+\sqrt{i(l-i+1)}{\ }\mathrm{Im}(u){\ }
\sqrt{-1}
(v_{i-1}^{(l)}\otimes v_{j}^{(m)}{\pm}v_{l-i+1}^{(l)}\otimes v_{m-j}^{(m)})\\
&-\sqrt{j(m-j+1)}{\ }\mathrm{Re}(w){\ }
(v_{i}^{(l)}\otimes v_{j-1}^{(m)}{\mp}v_{l-i}^{(l)}\otimes v_{m-j+1}^{(m)})\\
&+\sqrt{j(m-j+1)}{\ }\mathrm{Im}(w){\ }
\sqrt{-1}
(v_{i}^{(l)}\otimes v_{j-1}^{(m)}{\pm}v_{l-i}^{(l)}\otimes v_{m-j+1}^{(m)})\\
&+\sqrt{(l-i)(i+1)}{\ }\mathrm{Re}(u){\ }
(v_{i+1}^{(l)}\otimes v_{j}^{(m)}{\mp}v_{l-i-1}^{(l)}\otimes v_{m-j}^{(m)})\\
&+\sqrt{(l-i)(i+1)}{\ }\mathrm{Im}(u){\ }
\sqrt{-1}
(v_{i+1}^{(l)}\otimes v_{j}^{(m)}{\pm}v_{l-i-1}^{(l)}\otimes v_{m-j}^{(m)})\\
&+\sqrt{(m-j)(j+1)}{\ }\mathrm{Re}(w){\ }
(v_{i}^{(l)}\otimes v_{j+1}^{(m)}{\mp}v_{l-i}^{(l)}\otimes v_{m-j-1}^{(m)})\\
&+\sqrt{(m-j)(j+1)}{\ }\mathrm{Im}(w){\ }
\sqrt{-1}
(v_{i}^{(l)}\otimes v_{j+1}^{(m)}{\pm}v_{l-i}^{(l)}\otimes v_{m-j-1}^{(m)})
{\ }.
\end{split}
\end{equation}
\end{lem}
\begin{rem0}
By using the formula \eqref{usefulformula} we can check
that the real vector subspace $W$ is invariant under the action of
$SU(2)\times SU(2)$ via $\rho_{3}\boxtimes\rho_{1}$.
\end{rem0}

Define an orthonormal basis of the real vector space $W\cong{\mathfrak p}$
by
\begin{equation*}
\begin{split}
H_{1}:&=\frac{1}{\sqrt{2}}
(v_{0}^{(3)}\otimes v_{0}^{(1)}+v_{3}^{(3)}\otimes v_{1}^{(1)}),\\
H_{2}:&=\frac{1}{\sqrt{2}}
(v_{2}^{(3)}\otimes v_{0}^{(1)}+v_{1}^{(3)}\otimes v_{1}^{(1)}),\\
E_{1}:&=\frac{1}{\sqrt{2}}
\sqrt{-1}{\ }(v_{0}^{(3)}\otimes v_{0}^{(1)}-v_{3}^{(3)}\otimes v_{1}^{(1)}),\\
E_{2}:&=\frac{1}{\sqrt{2}}
(v_{1}^{(3)}\otimes v_{0}^{(1)}-v_{2}^{(3)}\otimes v_{1}^{(1)}),\\
E_{3}:&=\frac{1}{\sqrt{2}}
\sqrt{-1}{\ }(v_{1}^{(3)}\otimes v_{0}^{(1)}+v_{2}^{(3)}\otimes v_{1}^{(1)}),\\
E_{4}:&=\frac{1}{\sqrt{2}}
(v_{0}^{(3)}\otimes v_{1}^{(1)}-v_{3}^{(3)}\otimes v_{0}^{(1)}),\\
E_{5}:&=\frac{1}{\sqrt{2}}
\sqrt{-1}{\ }(v_{0}^{(3)}\otimes v_{1}^{(1)}+v_{3}^{(3)}\otimes v_{0}^{(1)}),\\
E_{6}:&=\frac{1}{\sqrt{2}}
\sqrt{-1}{\ }(v_{2}^{(3)}\otimes v_{0}^{(1)}-v_{1}^{(3)}\otimes v_{1}^{(1)})
{\ }.
\end{split}
\end{equation*}
Then we have the matrix expression as follows:
\begin{equation*}
\begin{split}
&[d(\rho_{3}\boxtimes\rho_{1})(X,Y)]
\left(H_{1},H_{2}\right)
\\
=&
\left(
E_{1},E_{2},E_{3},E_{4},E_{5},E_{6},
\right)
\begin{pmatrix}
-(3x+y)&0\\
\sqrt{3}{\ }\mathrm{Re}(u)&-(2{\ }\mathrm{Re}(u)+\mathrm{Re}(w))\\
\sqrt{3}{\ }\mathrm{Im}(u)&2{\ }\mathrm{Im}(u)+\mathrm{Im}(w)\\
\mathrm{Re}(w)&-\sqrt{3}{\ }\mathrm{Re}(u)\\
\mathrm{Im}(w)&\sqrt{3}{\ }\mathrm{Im}(u)\\
0&x-y
\end{pmatrix}.
\end{split}
\end{equation*}

The inner product
$\langle{\ ,\ }\rangle$
corresponding to the metric induced from $g^{\mathrm std}_{Q_{6}({\mathbf C})}$
is given as follows:
For $(X,X^{\prime}),(Y,Y^{\prime})\in{\mathfrak{su}(2)\oplus\mathfrak{su}(2)}$,
\begin{equation*}
\begin{split}
&\langle{(X,X^{\prime}),(Y,Y^{\prime})}\rangle\\
:=&{\ }(3x+x^{\prime})(3y+y^{\prime})\\
&+3{\ }\mathrm{Re}(u)\mathrm{Re}(w)
+(2{\ }\mathrm{Re}(u)+\mathrm{Re}(u^{\prime}))
 (2{\ }\mathrm{Re}(w)+\mathrm{Re}(w^{\prime}))\\
&+3{\ }\mathrm{Im}(u)\mathrm{Im}(w)
+(2{\ }\mathrm{Im}(u)+\mathrm{Im}(u^{\prime}))
 (2{\ }\mathrm{Im}(w)+\mathrm{Im}(w^{\prime}))\\
&+\mathrm{Re}(u^{\prime})\mathrm{Re}(w^{\prime})
+3\mathrm{Re}(u)\mathrm{Re}(w)\\
&+\mathrm{Im}(u^{\prime})\mathrm{Im}(w^{\prime})
+3\mathrm{Im}(u)\mathrm{Im}(w)\\
&+(x-x^{\prime})(y-y^{\prime})\\
=&{\ }
10xy+2x^{\prime}y+2xy^{\prime}+2x^{\prime}y^{\prime}\\
&+10{\ }\mathrm{Re}(u)\mathrm{Re}(w)
+2{\ }\mathrm{Re}(u^{\prime})\mathrm{Re}(w)
+2{\ }\mathrm{Re}(u)\mathrm{Re}(w^{\prime})
+2\mathrm{Re}(u^{\prime})\mathrm{Re}(w^{\prime})\\
&+10{\ }\mathrm{Im}(u)\mathrm{Im}(w) +2{\
}\mathrm{Im}(u^{\prime})\mathrm{Im}(w) +2{\
}\mathrm{Im}(u)\mathrm{Im}(w^{\prime})
+2\mathrm{Im}(u^{\prime})\mathrm{Im}(w^{\prime}).
\end{split}
\end{equation*}
Thus
the Casimir operator of $(\widetilde{K},\widetilde{K}_{[{\mathfrak a}]})$
relative to the inner product
$\langle{\ ,\ }\rangle$
is given as follows:
\begin{equation*}
\begin{split}
\mathcal{C}_{L}
={\ }&
 \frac{1}{2}{\ }(X_{1},0)\cdot(X_{1},0)
+\frac{1}{2}{\ }(X_{2},0)\cdot(X_{2},0)
+\frac{1}{2}{\ }(X_{3},0)\cdot(X_{3},0)\\
+&\frac{5}{2}{\ }(0,X_{1})\cdot(0,X_{1})
+\frac{5}{2}{\ }(0,X_{2})\cdot(0,X_{2})
+\frac{5}{2}{\ }(0,X_{3})\cdot(0,X_{3})\\
-&{\ }(X_{1},0)\cdot(0,X_{1})
-{\ }(X_{2},0)\cdot(0,X_{2})
-{\ }(X_{3},0)\cdot(0,X_{3}),
\end{split}
\end{equation*}
where
\begin{equation*}
X_{1}:=
\frac{1}{2}
\begin{pmatrix}
\sqrt{-1}&0\\
0&-\sqrt{-1}
\end{pmatrix},
{\ }
X_{2}:=
\frac{1}{2}
\begin{pmatrix}
0&1\\
-1&0
\end{pmatrix},
{\ }
X_{3}:=
\frac{1}{2}
\begin{pmatrix}
0&\sqrt{-1}\\
\sqrt{-1}&0
\end{pmatrix}
\end{equation*}
is a basis of $\mathfrak{su}(2)$ and
$\{(X_1,0), (X_2,0), (X_3,0), (0,X_1), (0,X_2), (0,X_3)\}$ is a basis of $\mathfrak{su}(2)\oplus \mathfrak{su}(2)$.
Hence, we have the following formula for the Casimir operator:
\begin{lem}\label{eigenvalue_G_2SO(4)}
\begin{equation*}
\begin{split}
 &[d(\rho_l \boxtimes  \rho_m)({\mathcal C}_L)]
(v^{(l)}_i \otimes v^{(m)}_a)\\
= &
-\bigl\{
\frac{l(l+2)}{8}+\frac{5m(m+2)}{8}-\frac{(2i-l)(4a-m)}{4}
\bigl\}(v^{(l)}_i \otimes v^{(m)}_a)\\
& + \frac{1}{2}\sqrt{(i+1)(l-i)a(m-a+1)}
(v^{(l)}_{i+1} \otimes v^{(m)}_{a-1}) \\
& +\frac{1}{2}\sqrt{i(l-i+1)(a+1)(m-a)}
(v^{(l)}_{i-1} \otimes v^{(m)}_{a+1}).
\end{split}
\end{equation*}
\end{lem}

Set
\begin{equation*}
\widetilde{K}_{0}
:=\{(A,B)\in\widetilde{K}{\ }\vert{\ }
\mathrm{Ad}(p(A,B))H=H\text{ for each }H\in{\mathfrak a}{\ }\}.
\end{equation*}
Then using this description of the isotropy representation we can compute directly
\begin{equation*}
\begin{split}
&\widetilde{K}_{0}\\
=&
\Bigl\{
\left(
\begin{pmatrix}
1&0\\
0&1
\end{pmatrix},
\begin{pmatrix}
1&0\\
0&1
\end{pmatrix}
\right),{\ }
\left(
\begin{pmatrix}
\sqrt{-1}&0\\
0&-\sqrt{-1}
\end{pmatrix},
\begin{pmatrix}
\sqrt{-1}&0\\
0&-\sqrt{-1}
\end{pmatrix}
\right), \\
&
\left(
\begin{pmatrix}
-1&0\\
0&-1
\end{pmatrix},
\begin{pmatrix}
-1&0\\
0&-1
\end{pmatrix}
\right),{\ }
\left(
\begin{pmatrix}
-\sqrt{-1}&0\\
0&\sqrt{-1}
\end{pmatrix},
\begin{pmatrix}
-\sqrt{-1}&0\\
0&\sqrt{-1}
\end{pmatrix}
\right), \\
&
\left(
\begin{pmatrix}
0&-1\\
1&0
\end{pmatrix},
\begin{pmatrix}
0&-1\\
1&0
\end{pmatrix}
\right),{\ }
\left(
\begin{pmatrix}
0&\sqrt{-1}\\
\sqrt{-1}&0
\end{pmatrix},
\begin{pmatrix}
0&\sqrt{-1}\\
\sqrt{-1}&0
\end{pmatrix}
\right),\\
&
\left(
\begin{pmatrix}
0&1\\
-1&0
\end{pmatrix},
\begin{pmatrix}
0&1\\
-1&0
\end{pmatrix}
\right),{\ }
\left(
\begin{pmatrix}
0&-\sqrt{-1}\\
-\sqrt{-1}&0
\end{pmatrix},
\begin{pmatrix}
0&-\sqrt{-1}\\
-\sqrt{-1}&0
\end{pmatrix}
\right)\Bigr\}.
\end{split}
\end{equation*}
In particular, the order of $\widetilde{K}_{0}$ is $8$.
This result is consistent with ones of
\cite[p.611]{Asoh1981}, \cite[p.651]{Asoh1983} and
\cite[p.573]{FUchida1980} in topology
of transformation group theory.
Moreover we obtain
\begin{equation*}
\begin{split}
&K_{0}\\
=
\Bigl\{
&p
\left(
\begin{pmatrix}
1&0\\
0&1
\end{pmatrix},
\begin{pmatrix}
1&0\\
0&1
\end{pmatrix}
\right)
=p
\left(
\begin{pmatrix}
-1&0\\
0&-1
\end{pmatrix},
\begin{pmatrix}
-1&0\\
0&-1
\end{pmatrix}
\right),\\
&p\left(
\begin{pmatrix}
\sqrt{-1}&0\\
0&-\sqrt{-1}
\end{pmatrix},
\begin{pmatrix}
\sqrt{-1}&0\\
0&-\sqrt{-1}
\end{pmatrix}
\right)
=
p\left(
\begin{pmatrix}
-\sqrt{-1}&0\\
0&\sqrt{-1}
\end{pmatrix},
\begin{pmatrix}
-\sqrt{-1}&0\\
0&\sqrt{-1}
\end{pmatrix}
\right),\\
&
p\left(
\begin{pmatrix}
0&-1\\
1&0
\end{pmatrix},
\begin{pmatrix}
0&-1\\
1&0
\end{pmatrix}
\right)
=
p\left(
\begin{pmatrix}
0&1\\
-1&0
\end{pmatrix},
\begin{pmatrix}
0&1\\
-1&0
\end{pmatrix}
\right),\\
&
p\left(
\begin{pmatrix}
0&\sqrt{-1}\\
\sqrt{-1}&0
\end{pmatrix},
\begin{pmatrix}
0&\sqrt{-1}\\
\sqrt{-1}&0
\end{pmatrix}
\right)
=
p\left(
\begin{pmatrix}
0&-\sqrt{-1}\\
-\sqrt{-1}&0
\end{pmatrix},
\begin{pmatrix}
0&-\sqrt{-1}\\
-\sqrt{-1}&0
\end{pmatrix}
\right)
\Bigr\}\\
\cong{\ }&{\mathbf Z}_{2}+{\mathbf Z}_{2}.
\end{split}
\end{equation*}
Hence the order of group $K_{0}$ is equal to $4$ and
\begin{equation*}
\widetilde{K}/\widetilde{K}_{0}
\cong{K}/{K}_{0}=SO(4)/{\mathbf Z}_{2}+{\mathbf Z}_{2}.
\end{equation*}
For each $l,m\in{\mathbf Z}$ with $l,m\geq{0}$,
the vector subspace
of $V_{l}\otimes V_{m}$
\begin{equation*}
\begin{split}
&(V_{l}\otimes V_{m})_{\widetilde{K}_{0}}\\
:=&
\{\xi \in V_{l}\otimes V_{m}{\ }\vert{\ }
[(\rho_{l}\boxtimes\rho_{m})(A,B)](\xi)=\xi
\text{ for any }(A,B)\in\widetilde{K}_{0}{\ }\}
\end{split}
\end{equation*}
can be described explicitly as follows:
\begin{lem}\label{VectorSubspacefixedElementsByK0}
When $(l+m)/2$ is even,
\begin{equation*}
\begin{split}
&(V_{l}\otimes V_{m})_{\widetilde{K}_{0}}\\
=&
\{{\ }\xi=\sum_{i+a:\text{even}}\xi_{i,a}
(v^{(l)}_{i}\otimes v^{(m)}_{a}+
v^{(l)}_{l-i}\otimes v^{(m)}_{m-a}) \mid
\xi_{i,a}\in{\bold C}{\ }\}
\end{split}
\end{equation*}
and when $(l+m)/2$  is odd,
\begin{equation*}
\begin{split}
&(V_{l}\otimes V_{m})_{\widetilde{K}_{0}}\\
=&
\{{\ }\xi=\sum_{i+a:\text{odd}}\xi_{i,a}
(v^{(l)}_{i}\otimes v^{(m)}_{a}-
v^{(l)}_{l-i}\otimes v^{(m)}_{m-a}) \mid
\xi_{i,a}\in{\bold C}{\ }\}.
\end{split}
\end{equation*}
\end{lem}

Next we describe the subgroups of $\widetilde{K}$ defined as
\begin{equation*}
\begin{split}
\widetilde{K}_{{\mathfrak a}}{\ }
:=&\{(A,B)\in\widetilde{K}{\ }\vert{\ }
[(\rho_{3}\boxtimes\rho_{1})(A,B)]({\mathfrak a})={\mathfrak a}
\},\\
\widetilde{K}_{[{\mathfrak a}]}
:=&\{(A,B)\in\widetilde{K}{\ }\vert{\ }
[(\rho_{3}\boxtimes\rho_{1})(A,B)]({\mathfrak a})={\mathfrak a}\\
&\quad\quad\quad\quad\quad\quad\quad
\text{ preserving the orientation of }{\mathfrak a}
\}\subset\widetilde{K}_{{\mathfrak a}}.
\end{split}
\end{equation*}
For $(A,B)\in \widetilde{K}=SU(2)\times SU(2)$,
we compute that
$(A,B)\in{\widetilde{K}_{{\mathfrak a}}}$
if and only if $(A,B)$ is one of the following elements:
\begin{equation*}
\Bigg(
\begin{pmatrix}
e^{\sqrt{-1}\theta_{1}}&0\\
0&e^{-\sqrt{-1}\theta_{1}}
\end{pmatrix},
\begin{pmatrix}
e^{\sqrt{-1}\theta^{\prime}_{1}}&0\\
0&e^{-\sqrt{-1}\theta^{\prime}_{1}}
\end{pmatrix}
\Bigg),
\end{equation*}
where
$\theta_{1}=\frac{\pi}{4}k_{1}$, $\theta^{\prime}_{1}=\frac{\pi}{4}k^{\prime}_{1}$,
$k_{1}, k^{\prime}_{1}\in{\mathbf Z}$,
$k_{1}-k^{\prime}_{1}\in{4{\mathbf Z}}$,
\begin{equation*}
\Bigg(
\begin{pmatrix}
0&-e^{-\sqrt{-1}\theta_{2}}\\
e^{\sqrt{-1}\theta_{2}}&0
\end{pmatrix},
\begin{pmatrix}
0&-e^{-\sqrt{-1}\theta^{\prime}_{2}}\\
e^{\sqrt{-1}\theta^{\prime}_{2}}&0
\end{pmatrix}
\Bigg),
\end{equation*}
where
$\theta_{2}=\frac{\pi}{4}k_{2}$, $\theta^{\prime}_{2}=\frac{\pi}{4}k^{\prime}_{2}$,
$k_{2},k^{\prime}_{2}\in{\mathbf Z}$,
$k_{2}-k^{\prime}_{2}\in{4{\mathbf Z}}$,
\begin{equation*}
\Bigg(
\begin{pmatrix}
\frac{1}{\sqrt{2}}e^{\sqrt{-1}\theta_{1}}&
-\frac{1}{\sqrt{2}}e^{-\sqrt{-1}\theta_{2}}\\
\frac{1}{\sqrt{2}}e^{\sqrt{-1}\theta_{2}}&
\frac{1}{\sqrt{2}}e^{-\sqrt{-1}\theta_{1}}
\end{pmatrix},
\begin{pmatrix}
\frac{1}{\sqrt{2}}e^{\sqrt{-1}\theta^{\prime}_{1}}&
-\frac{1}{\sqrt{2}}e^{-\sqrt{-1}\theta^{\prime}_{2}}\\
\frac{1}{\sqrt{2}}e^{\sqrt{-1}\theta^{\prime}_{2}}&
\frac{1}{\sqrt{2}}e^{-\sqrt{-1}\theta^{\prime}_{1}}
\end{pmatrix}
\Bigg)
\end{equation*}
where
$\theta_{1}=\frac{\pi}{4}k_{1}$,
$\theta_{2}=\frac{\pi}{4}k_{2}$,
$\theta^{\prime}_{1}=\frac{\pi}{4}k^{\prime}_{1}$,
$\theta^{\prime}_{2}=\frac{\pi}{4}k^{\prime}_{2}$
and
$k_{1}$, $k_{2}$, $k^{\prime}_{1}$, $k^{\prime}_{2}\in{\mathbf Z}$,
$k_{1}+k_{2}$, $k_{1}-k_{2}$,
$k^{\prime}_{1}+k^{\prime}_{2}$, $k^{\prime}_{1}-k^{\prime}_{2}
\in{2{\mathbf Z}}$,
$k_{1}-k^{\prime}_{1}$, $k_{2}-k^{\prime}_{2}\in{4{\mathbf Z}}$,
$k_{1}+k_{2}-k^{\prime}_{1}-k^{\prime}_{2}$,
$k_{1}-k_{2}-k^{\prime}_{1}+k^{\prime}_{2}
\in{8{\mathbf Z}}$.

In particular,
the order of $\widetilde{K}_{{\mathfrak a}}$ is equal to
$16+16+32+32=96$.

Moreover,
for $(A,B)\in \widetilde{K}=SU(2)\times SU(2)$,
we have that
$(A,B)\in\widetilde{K}_{[{\mathfrak a}]}$
if and only if $(A,B)$ is one of the following elements:
\begin{equation*}
\Bigg(
\begin{pmatrix}
e^{\sqrt{-1}\theta_{1}}&0\\
0&e^{-\sqrt{-1}\theta_{1}}
\end{pmatrix},
\begin{pmatrix}
e^{\sqrt{-1}\theta^{\prime}_{1}}&0\\
0&e^{-\sqrt{-1}\theta^{\prime}_{1}}
\end{pmatrix}
\Bigg)
\end{equation*}
where
$\theta_{1}=\frac{\pi}{4}k_{1},\theta^{\prime}_{1}=\frac{\pi}{4}k^{\prime}_{1}$,
$k_{1},k^{\prime}_{1}\in{2{\mathbf Z}}$,
$k_{1}-k^{\prime}_{1}\in{4{\mathbf Z}}$,
\begin{equation*}
\Bigg(
\begin{pmatrix}
0&-e^{-\sqrt{-1}\theta_{2}}\\
e^{\sqrt{-1}\theta_{2}}&0
\end{pmatrix},
\begin{pmatrix}
0&-e^{-\sqrt{-1}\theta^{\prime}_{2}}\\
e^{\sqrt{-1}\theta^{\prime}_{2}}&0
\end{pmatrix}
\Bigg)
\end{equation*}
where
$\theta_{2}=\frac{\pi}{4}k_{2},\theta^{\prime}_{2}=\frac{\pi}{4}k^{\prime}_{2}$,
$k_{2},k^{\prime}_{2}\in{2{\mathbf Z}}$,
$k_{2}-k^{\prime}_{2}\in{4{\mathbf Z}}$,
\begin{equation*}
\Bigg(
\begin{pmatrix}
\frac{1}{\sqrt{2}}e^{\sqrt{-1}\theta_{1}}&
-\frac{1}{\sqrt{2}}e^{-\sqrt{-1}\theta_{2}}\\
\frac{1}{\sqrt{2}}e^{\sqrt{-1}\theta_{2}}&
\frac{1}{\sqrt{2}}e^{-\sqrt{-1}\theta_{1}}
\end{pmatrix},
\begin{pmatrix}
\frac{1}{\sqrt{2}}e^{\sqrt{-1}\theta^{\prime}_{1}}&
-\frac{1}{\sqrt{2}}e^{-\sqrt{-1}\theta^{\prime}_{2}}\\
\frac{1}{\sqrt{2}}e^{\sqrt{-1}\theta^{\prime}_{2}}&
\frac{1}{\sqrt{2}}e^{-\sqrt{-1}\theta^{\prime}_{1}}
\end{pmatrix}
\Bigg)
\end{equation*}
where
$\theta_{1}=\frac{\pi}{4}k_{1}$,
$\theta_{2}=\frac{\pi}{4}k_{2}$,
$\theta^{\prime}_{1}=\frac{\pi}{4}k^{\prime}_{1}$,
$\theta^{\prime}_{2}=\frac{\pi}{4}k^{\prime}_{2}$ and
$k_{1},k_{2},k^{\prime}_{1},k^{\prime}_{2}\in{2{\mathbf Z}+1}$,
$k_{1}+k_{2},k_{1}-k_{2}$,
$k^{\prime}_{1}+k^{\prime}_{2},k^{\prime}_{1}-k^{\prime}_{2}
\in{2{\mathbf Z}}$,
$k_{1}-k^{\prime}_{1},k_{2}-k^{\prime}_{2}\in{4{\mathbf Z}}$,
$k_{1}+k_{2}-(k^{\prime}_{1}+k^{\prime}_{2}),
k_{1}-k_{2}-(k^{\prime}_{1}-k^{\prime}_{2})
\in{8{\mathbf Z}}$.
In other words,
$(A,B)\in\widetilde{K}_{[{\mathfrak a}]}$
if and only if $(A,B)$ is one of the following elements:
\begin{equation*}
\Bigg(
\begin{pmatrix}
e^{\sqrt{-1}\theta_{1}}&0\\
0&e^{-\sqrt{-1}\theta_{1}}
\end{pmatrix},
\begin{pmatrix}
e^{\sqrt{-1}\theta^{\prime}_{1}}&0\\
0&e^{-\sqrt{-1}\theta^{\prime}_{1}}
\end{pmatrix}
\Bigg),
\end{equation*}
where
$\theta_{1}=\frac{\pi}{2}l_{1},\theta^{\prime}_{1}=\frac{\pi}{2}l^{\prime}_{1}$,
$l_{1},l^{\prime}_{1}\in{\mathbf Z}$,
$l_{1}-l^{\prime}_{1}\in{2{\mathbf Z}}$,
\begin{equation*}
\Bigg(
\begin{pmatrix}
0&-e^{-\sqrt{-1}\theta_{2}}\\
e^{\sqrt{-1}\theta_{2}}&0
\end{pmatrix},
\begin{pmatrix}
0&-e^{-\sqrt{-1}\theta^{\prime}_{2}}\\
e^{\sqrt{-1}\theta^{\prime}_{2}}&0
\end{pmatrix}
\Bigg),
\end{equation*}
where
$\theta_{2}=\frac{\pi}{2}l_{2}$,
$\theta^{\prime}_{2}=\frac{\pi}{2}l^{\prime}_{2}$,
$l_{2},l^{\prime}_{2}\in{\mathbf Z}$,
$l_{2}-l^{\prime}_{2}\in{2{\mathbf Z}}$,
\begin{equation*}
\Bigg(
\begin{pmatrix}
\frac{1}{\sqrt{2}}e^{\sqrt{-1}\theta_{1}}&
-\frac{1}{\sqrt{2}}e^{-\sqrt{-1}\theta_{2}}\\
\frac{1}{\sqrt{2}}e^{\sqrt{-1}\theta_{2}}&
\frac{1}{\sqrt{2}}e^{-\sqrt{-1}\theta_{1}}
\end{pmatrix},
\begin{pmatrix}
\frac{1}{\sqrt{2}}e^{\sqrt{-1}\theta^{\prime}_{1}}&
-\frac{1}{\sqrt{2}}e^{-\sqrt{-1}\theta^{\prime}_{2}}\\
\frac{1}{\sqrt{2}}e^{\sqrt{-1}\theta^{\prime}_{2}}&
\frac{1}{\sqrt{2}}e^{-\sqrt{-1}\theta^{\prime}_{1}}
\end{pmatrix}
\Bigg)
\end{equation*}
where
$\theta_{1}=\frac{\pi}{2}l_{1}+\frac{\pi}{4}$,
$\theta_{2}=\frac{\pi}{2}l_{2}+\frac{\pi}{4}$,
$\theta^{\prime}_{1}=\frac{\pi}{2}l^{\prime}_{1}+\frac{\pi}{4}$,
$\theta^{\prime}_{2}=\frac{\pi}{2}l^{\prime}_{2}+\frac{\pi}{4}$,
$l_{1},l_{2},l^{\prime}_{1},l^{\prime}_{2}\in{\mathbf Z}$,
$l_{1}-l^{\prime}_{1},l_{2}-l^{\prime}_{2}\in{2{\mathbf Z}}$,
$l_{1}+l_{2}-(l^{\prime}_{1}+l^{\prime}_{2}),
l_{1}-l_{2}-(l^{\prime}_{1}-l^{\prime}_{2})
\in{4{\mathbf Z}}$.
In particular,
the order of $\widetilde{K}_{[{\mathfrak a}]}$ is equal to
$8+8+16+16=48=8\times 6=\sharp{\widetilde{K}_{0}}\times\sharp{\mathbf Z}_{6}$.
Then we obtain
\begin{lem}
$\widetilde{K}_{[{\mathfrak a}]}/\widetilde{K}_{0}
\cong{\mathbf Z}_{6}$.
\end{lem}
\begin{proof}
We compute
\begin{equation*}
\begin{split}
A&=
\begin{pmatrix}
\frac{1}{\sqrt{2}}e^{\sqrt{-1}(\frac{\pi}{2}l_{1}+\frac{\pi}{4})}&
-\frac{1}{\sqrt{2}}e^{-\sqrt{-1}(\frac{\pi}{2}l_{2}+\frac{\pi}{4})}\\
\frac{1}{\sqrt{2}}e^{\sqrt{-1}(\frac{\pi}{2}l_{2}+\frac{\pi}{4})}&
\frac{1}{\sqrt{2}}e^{-\sqrt{-1}(\frac{\pi}{2}l_{1}+\frac{\pi}{4})}
\end{pmatrix},\\
A^{3}&=
\begin{pmatrix}
-\sqrt{2}\cos(\frac{\pi}{2}l_{1}+\frac{\pi}{4})&0\\
0&-\sqrt{2}\cos(\frac{\pi}{2}l_{1}+\frac{\pi}{4})
\end{pmatrix}\\
&=
\begin{cases}
-\mathrm{I}_{2}\quad
&\text{ if }
l_{1}\equiv{0}\text{ or }{3}{\ }(\text{mod }4)\\
{\ }\mathrm{I}_{2}\quad
&\text{ if }
l_{1}\equiv{1}\text{ or }{2}{\ }(\text{mod }4)\\
\end{cases}{\ },\\
A^{6}&=\mathrm{I}_{2}{\ }.
\end{split}
\end{equation*}
The generator of
$\widetilde{K}_{[{\mathfrak a}]}/\widetilde{K}_{0}
\cong{\mathbf Z}_{6}$
is represented by the element
\begin{equation}\label{eq:generator_G_2SO(4)}
\begin{split}
\Bigg(
\begin{pmatrix}
\frac{1+\sqrt{-1}}{2}&
-\frac{1-\sqrt{-1}}{2}\\
\frac{1+\sqrt{-1}}{2}&
\frac{1-\sqrt{-1}}{2}
\end{pmatrix},
\begin{pmatrix}
-\frac{1+\sqrt{-1}}{2}&
\frac{1-\sqrt{-1}}{2}\\
-\frac{1+\sqrt{-1}}{2}&
-\frac{1-\sqrt{-1}}{2}
\end{pmatrix}
\Bigg){\ }.
\end{split}
\end{equation}
\end{proof}

Then using Lemmas \ref{eigenvalue_G_2SO(4)} and
\ref{VectorSubspacefixedElementsByK0}
we can determine directly all eigenvalues of ${\mathcal C}_{L}$
on $\tilde{K}/\tilde{K}_0$ less than or equal to $\dim{L}=6$
and corresponding representations of $\tilde{K}$
as in the following table:
\begin{center}
\begin{tabular}
{|c|c|c|c|} \hline $(l,m)$
&$\dim(V_l \otimes V_m)_{\tilde{K}_0}$
& eigenvalues of ${\mathcal C}_L$
&$-\lambda\leq 6$
\\
\hline $(1,1)$ &$1$ &$-3$ & *
\\
\hline $(2,0)$ &$0$ & &
\\
\hline $(0,2)$ &$0$ & &
\\
\hline $(3,1)$ &$2$ & $-3, -3$ & *
\\
\hline $(1,3)$ &$2$ & $-9,-9$ &
\\
\hline $(4,0)$ &$2$ & $-3,-3$ & *
\\
\hline $(0,4)$ &$2$ & $-15,-15$ & *
\\
\hline $(2,2)$ &3&$-5$, $-5$, $-8$ 
& *\\
\hline $(5,1)$ &3&$-8,-5,-8$& *\\  
\hline $(6,0)$ &1&$-6$& *\\
\hline $(4,2)$ &3&$-6,-9,-9$& *\\
\hline $(3,3)$ &4& $-9,-12,-12,-15$ & \\
\hline $(8,0)$ &2&$-10,-10$& \\
\hline $(7,1)$ &4&$-12,-12,-8,-8$& \\
\hline $(6,2)$ &5&$-15,-12,-8,-8,-12$& \\
\hline
\end{tabular}
\end{center}
Hence we get
\begin{equation*}
\begin{split}
&\{(l,m) \mid -c_L\leq 6 \text{ and } (V_l\otimes V_m)_{{\tilde
K}_0}\neq
\{0\}\}\\
=& \{(1,1), (4,0), (2,2), (3,1),(6,0), (5,1), (4,2)\}.
\end{split}
\end{equation*}

Using the generator \eqref{eq:generator_G_2SO(4)}
of $\widetilde{K}_{[{\mathfrak a}]}/\widetilde{K}_{0}
\cong{\mathbf Z}_{6}$,
we compute that
$(V_l\otimes V_m)_{{\tilde K}_{[\mathfrak a]}}=\{0\}$ for
$(l,m)=(1,1)$, $(4,0)$, $(3,1)$, $(5,1)$
and
$\dim_{\mathbf C}(V_l\otimes V_m)_{{\tilde K}_{[\mathfrak a]}}
=1$
for $(l,m)=(2,2), (6,0), (4,2)$.
But we observe that the fixed vector in
$(V_2\otimes V_2)_{{\tilde K}_{[\mathfrak a]}}\neq \{0\}$
corresponds to the larger
eigenvalue $8>6$.
Hence we obtain that
the Gauss image
$L^6
={\mathcal G}(\frac{SO(4)}{\mathbf{Z}_2+\mathbf{Z}_2} )
=\frac{SO(4)}{(\mathbf{Z}_2+\mathbf{Z}_2)\cdot\mathbf{Z}_6}
\subset Q_6(\mathbf{C})$
is Hamiltonian stable.

Moreover from the above result of dimension computation we have
\begin{equation*}
\begin{split}
n(L^{6})=
&\dim_{\mathbf C}V_{6}\boxtimes V_{0}
+\dim_{\mathbf C}V_{4}\boxtimes V_{2}
=7\times 1+5\times 3
=7+15=22\\
=&\mbox{dim} SO(8)-\mbox{dim} SO(4)=n_{hk}(L).
\end{split}
\end{equation*}
Thus the Gauss image
$L^6
={\mathcal G}(\frac{SO(4)}{\mathbf{Z}_2+\mathbf{Z}_2} )
=\frac{SO(4)}{(\mathbf{Z}_2+\mathbf{Z}_2)\cdot\mathbf{Z}_6}
\subset Q_6(\mathbf{C})$
is Hamiltonian rigid.
From these results we conclude
\begin{thm}
The Gauss image
$L^6
={\mathcal G}\left(\frac{SO(4)}{\mathbf{Z}_2+\mathbf{Z}_2}\right)
=\frac{SO(4)}{(\mathbf{Z}_2+\mathbf{Z}_2)\cdot\mathbf{Z}_6}
\subset Q_6(\mathbf{C})$
is strictly Hamiltonian stable.
\end{thm}

\section{The case $(U,K)=(SO(5) \times SO(5), SO(5))$}
\label{Sec_b2}

Now $(U,K)$ is of type $B_2$ and
$U=SO(5)\times SO(5)$, $K=\{(x,x)\in U \mid x\in SO(5)\}$.
Let ${\mathfrak u}={\mathfrak k}+{\mathfrak p}$ be the canonical decomposition,
where ${\mathfrak u}={\mathfrak o}(5) \oplus {\mathfrak o}(5)$,
${\mathfrak k} = \{(X,X) \mid X\in
{\mathfrak o}(5)\}\cong {\mathfrak o}(5)$
and ${\mathfrak p} =\{(X,-X)\mid X\in {\mathfrak o}(5)\}$.
Let ${\mathfrak a}$ be a maximal abelian subspace of  ${\mathfrak p}$
given by
\begin{equation*}
\begin{split}
{\mathfrak a}=& \left\{ (H,-H) \mid H=H(\xi_{1},\xi_{2})=
\begin{pmatrix}
0&-\xi_1&0&0&0\\
\xi_1&0&0&0&0\\
0&0&0&-\xi_2&0\\
0&0&\xi_2&0&0\\
0&0&0&0&0
\end{pmatrix},
\xi_{1},\xi_{2}\in{\mathbf R}\right\}\\
\cong
&\ {\mathfrak t} =\{H(\xi_1,\xi_2)\mid \xi_1,\xi_2 \in {\mathbf R}\}
 \subset {\mathfrak o}(5).
\end{split}
\end{equation*}
Then the centralizer $K_0$ of $\mathfrak{a}$ in $K$ is given by
\begin{equation*}
K_0=\left\{
   \begin{pmatrix}
      A&0&0\\
      0&B&0\\
      0&0&1
    \end{pmatrix}
     \mid
      A,B \in SO(2)
      \right\}
      \cong T^2,
\end{equation*}
which is a maximal torus of $SO(5)$ and
$N=K/K_0\cong SO(5)/T^2$ is a maximal flag manifold of dimension $n=8$.
Moreover $K_{[{\mathfrak a}]}$ is described as
\begin{equation*}
\begin{split}
K_{[{\mathfrak a}]}
=&
\begin{pmatrix}
\mathrm{I}_2&0 &0 \\
0 & \mathrm{I}_2&0 \\
0 &0 & 1
\end{pmatrix}\cdot T^2
\, \cup\,
\begin{pmatrix}
& & 1 &0 &  \\
 & & 0&1 & \\
 1&0 & & & \\
  0&-1 & & & \\
 & & & &-1
\end{pmatrix}\cdot T^2 \\
& \cup
\begin{pmatrix}
1& 0& & & \\
 0&-1& & & \\
 & & 1&0 & \\
  & & 0& -1& \\
  & & & &1
\end{pmatrix}\cdot T^2
\cup
\begin{pmatrix}
 & &1 &0 & \\
 & &0 &-1 & \\
 1&0 & & & \\
 0 &1 & & &  \\
  & & &  &-1
\end{pmatrix}\cdot T^2.
\end{split}
\end{equation*}
The deck transformation group of the covering map
${\mathcal G}:N^{8} \rightarrow {\mathcal G}(N^{8})$
is equal to
$K_{[{\mathfrak a}]}/K_{0}\cong{\mathbf Z}_{4}$.

\subsection{Description of the Casimir operator}
Choose
$\langle{X,Y}\rangle_{\mathfrak{k}}:=-\mathrm{tr}(XY)$
for each $X,Y\in \mathfrak{k}=\mathfrak{so}(5)$.
The restricted root system $\Sigma(U,K)$ of type $B_2$,
can be described as follows (cf.\, \cite{Bourbaki}):
\begin{equation*}
\begin{split}
\Sigma(U,K)=&\{\pm(\epsilon_1-\epsilon_2)
=\pm\alpha_1, \pm \epsilon_2=\pm\alpha_2,
\pm(\epsilon_1+\epsilon_2)=\pm(\alpha_1+2\alpha_2), \\
&\pm \epsilon_1=\pm(\alpha_1+\alpha_2)\}.
\end{split}
\end{equation*}
Then the square length of each $\gamma\in\Sigma(U,K)$
relative to $\langle\,,\,\rangle_{\mathfrak k}$ is
\begin{equation*}
\begin{split}
\Vert{\gamma}\Vert_{\mathfrak u}^{2}
=
\begin{cases}
\displaystyle\frac{1}{4}&\text{ if }\gamma\text{ is short}, \\
&\\
\displaystyle\frac{1}{2}&\text{ if }\gamma\text{ is long}.
\end{cases}
\end{split}
\end{equation*}
In this case
$K = SO(5) \supset K_1 = SO(4) \supset K_0= T^2$.
The Casimir operator $\mathcal{C}_L$ of $L^n$
relative to the induced metric from $g^{\mathrm{std}}_{Q_n(\mathbf{C})}$
becomes
\begin{equation*}
\begin{split}
{\mathcal C}_{L}&=
\frac{2}{\Vert \gamma_0\Vert^2_{\mathfrak u}}
\mathcal{C}_{K/K_0, \langle\,,\,\rangle_{\mathfrak{u}}}
-\frac{1}{\Vert \gamma_0\Vert^2_{\mathfrak u}}
\mathcal{C}_{K_1/K_0, \langle\,,\,\rangle_{\mathfrak{u}}}\\
&=
4\, \mathcal{C}_{K/K_0, \langle\,,\,\rangle_{\mathfrak{u}}}
-2\, \mathcal{C}_{K_1/K_0, \langle\,,\,\rangle_{\mathfrak{u}}}\\
&=
2\, {\mathcal C}_{K/K_0}-{\mathcal C}_{K_1/K_0} \\
&=
\ {\mathcal C}_{K/K_0}+{\mathcal C}_{K/K_1},
\end{split}
\end{equation*}
where
${\mathcal C}_{K/K_{0}}$ and ${\mathcal C}_{K_{1}/K_{0}}$
denote the Casimir operators
of  $K/K_{0}$ and $K_{1}/K_{0}$
relative to
$\langle{\ ,\ }\rangle_{\mathfrak k}$
and
$\langle{{\ },{\ }}\rangle_{\mathfrak k}\vert_{{\mathfrak k}_{1}}$,
respectively.

\subsection{Descriptions of $D(K)$ and $D(K_1)$}
Since the maximal abelian subalgebra $\mathfrak{t}$ of $\mathfrak{k}$
can be given by
\begin{equation*}
{\mathfrak t}=
\left\{
\begin{pmatrix}
0& -\xi_1& & & \\
\xi_1&0&  & & \\
 & &0 & -\xi_2 & \\
 & &\xi_2 & 0 & \\
 & & & & 0
\end{pmatrix}
\mid
\xi_{1}, \xi_{2}\in{\mathbf R}
\right\}
\subset{\mathfrak k}_1\subset{\mathfrak k},
\end{equation*}
we have
\begin{equation*}
\begin{split}
&\Gamma(K)
=\Gamma(K_{1})\\
=&
\left\{
\xi=
\begin{pmatrix}
0& -\xi_1& & & \\
\xi_1&0&  & & \\
 & &0 & -\xi_2 & \\
 & &\xi_2 & 0 & \\
 & & & & 0
\end{pmatrix}
\mid
\xi_{1},\xi_{2}\in 2\pi{\mathbf Z}
\right\}.
\end{split}
\end{equation*}
Denote by $\varepsilon_i\ (i=1,2)$ a linear function
$\epsilon_i : {\mathfrak t} \ni \xi \mapsto \xi_i \in {\mathbf R}$.
Then
\begin{equation*}
\begin{split}
&
D(K)=D(SO(5))
=\{\Lambda =k_{1}\epsilon_{1}+ k_2 \epsilon_2
\mid
k_{1},k_{2}\in{\mathbf Z},
k_{1}\geq k_2 \geq 0
\},
\\
&
D(K_1)=D(SO(4))
=
\{\Lambda =k_{1}\epsilon_{1}+ k_2 \epsilon_2
{\ }\vert{\ }
k_{1},k_{2}\in{\mathbf Z},
k_{1}\geq  |k_2|
\}.
\end{split}
\end{equation*}

\subsection{Branching law of $(SO(5), SO(4))$}
\begin{lem}[Branching law of $(SO(5), SO(4))$ \cite{Ikeda-Taniguchi78}]
\label{Branching_SO(5)_SO(4)}
Let $\Lambda=k_1 \epsilon_1 + k_2 \epsilon_2 \in D(SO(5))$ be
the highest weight of an irreducible $SO(5)$-module $V_{\Lambda}$,
where $k_1,k_2\in {\mathbf Z}$ and $k_1\geq k_2 \geq 0$.
Then $V_\Lambda$ contains an irreducible
$SO(4)$-module $W_{\Lambda^\prime}$
with the highest weight
$\Lambda^\prime=k_1^{\prime} \epsilon_1 + k_2^{\prime} \epsilon_2 \in D(SO(4))$,
where $k_1^{\prime},k_2^{\prime} \in {\mathbf Z}$, $k_1^{\prime}\geq |k_2^{\prime}|$,
if and only if
\begin{equation}\label{Branching_SO(5)_SO(4)}
k_1\geq k_1^\prime\geq k_2\geq |k_2^\prime|.
\end{equation}
\end{lem}
\subsection{Descriptions of $D(K,K_0)$ and $D(K_1,K_0)$}
Define an $\mathrm{Ad}(K)$-invariant inner product of
${\mathfrak k}$ by
$\langle{X, Y}\rangle_{\mathfrak k}:=-\mathrm{tr}(XY)\
(X,Y\in {\mathfrak k}=\mathfrak{o}(5))$.

Let $\{\alpha_1^\prime= \epsilon_1-\epsilon_2,\alpha_2^\prime=\epsilon_1+ \epsilon_2\}$
be the fundamental root system of $SO(4)$ and
$\{\Lambda_1^\prime=\frac{1}{2}(\epsilon_1-\epsilon_2),
\Lambda_2^\prime=\frac{1}{2}(\epsilon_1+\epsilon_2)\}$ be the
fundamental weight system of $SO(4)$.
Then
\begin{lem}[\cite{SYamaguchi79}]\label{D(SO4,T2)}
\begin{equation}
\begin{split}
&D(K_1,K_0)=\, D(SO(4), T^2) \\
=&
\Bigl\{
\Lambda^\prime
=k_1^\prime \epsilon_1+k^\prime_2 \epsilon_2
=m_1^\prime\Lambda_1^\prime +m_2^\prime\Lambda_2^\prime
=p_1^\prime\alpha_1^\prime + p_2^\prime\alpha_2^\prime
\mid \\
& \quad
k^\prime_i\in \mathbf{Z}, \, k^{\prime}_1\geq  \vert{k^{\prime}_2}\vert,
m^\prime_i\in {\mathbf Z}, m^\prime_i\geq 0,
p^\prime_i\in{\mathbf Z}, p^\prime_i\geq 1,
\\
& \quad
m_1^\prime=k_1^\prime-k_2^\prime=2p_1^\prime \geq 0,
m_2^\prime=k_1^\prime+k_2^\prime=2p_2^\prime \geq 0
\Bigr\}.
\end{split}
\end{equation}
The eigenvalue formula the Casimir operator $\mathcal{C}_{K_1/K_0}$
relative to $\langle X, Y \rangle_{\mathfrak k}\vert_{\mathfrak{k}_1}$ is
\begin{equation*}
-c_{{\Lambda}^{\prime}}=\frac{1}{2} ((k_1^{\prime})^2+ (k^{\prime}_2)^2+ 2k^{\prime}_1).
\end{equation*}
for each
$\Lambda^\prime= k^{\prime}_1 \epsilon_1+ k^{\prime}_2 \epsilon_2
\in D(K_1, K_0)$.
\end{lem}

Let $\{\alpha_1=\epsilon_1- \epsilon_2,\alpha_2=\epsilon_2\}$ be
the fundamental root system of $SO(5)$ and
$\{\Lambda_1=\epsilon_1, \Lambda_2=\frac{1}{2}(\epsilon_1+ \epsilon_2)\}$
be the fundamental weight system of $SO(5)$.
Then
\begin{lem}[\cite{SYamaguchi79}]\label{D(SO5,T2)}
\label{D(SO5,T2)}
\begin{equation}
\begin{split}
&D(K,K_0)=\, D(SO(5),T^2) \\
=&
\Bigl\{
\Lambda
=k_1 \epsilon_1+k_2 \epsilon_2
=m_1 \Lambda_1 +m_2 \Lambda_2
=p_1\alpha_1 + p_2\alpha_2
\mid
\\
&\quad
k_i\in \mathbf{Z}, \, k_1\geq  k_2\geq 0, \,
m_i\in {\mathbf Z},\, m_i\geq 0,\,
p_i\in{\mathbf Z},\, p_i\geq 1,
\\
&\quad
m_1=2p_1-p_2\geq 0,\, m_2=-2p_1+2p_2\geq 0,\,
p_1=k_1,\, p_2=k_1+k_2
\Bigr\}
\end{split}
\end{equation}
The eigenvalue formula of the Casimir operator $\mathcal{C}_{K/K_0}$
with respect to the inner product
$\langle X, Y \rangle_{\mathfrak k}$
is
\begin{equation*}
-c_{\Lambda} = \frac{1}{2} (k_1^2+ k_2^2+3 k_1 +k_2).
\end{equation*}
for each $\Lambda= k_1 \epsilon_1+ k_2 \epsilon_2 \in D(K, K_0)$.
\end{lem}

\subsection{Eigenvalue computation}

By Lemmas \ref{D(SO5,T2)} and \ref{D(SO4,T2)}
we have the following eigenvalue formula for ${\mathcal C}_{L}$.
\begin{equation*}
\begin{split}
-c_L
=&{\ }-2c_{K/K_0}+c_{K_1/K_0}
\\
=&{\ }
(k_1^2+k_2^2+3 k_1+ k_2)- \frac{1}{2} ((k^{\prime}_1)^2+ (k^{\prime}_2)^2
+ 2k^{\prime}_1).
\end{split}
\end{equation*}

Since
\begin{equation*}
-{\mathcal C}_L=-{\mathcal C}_{K/K_0}-{\mathcal C}_{S^4}\geq
-{\mathcal C}_{K/K_0},
\end{equation*}
the eigenvalue of ${\mathcal C}_L$, $-c_L \leq n=8$ implies
$-c_\Lambda \leq 8$.
Using Lemma \ref{D(SO5,T2)} we compute
\begin{equation*}
\begin{split}
&\{\Lambda\in{D(SO(5),T^{2})} \mid
-c(\Lambda,\langle{\ },{\ }\rangle_{\mathfrak{k}})\leq{8}\}\\
=&\{
\epsilon_1\, ((k_1,k_2)=(1,0)),\
\epsilon_1+\epsilon_2\, ((k_1,k_2)=(1,1)),\
2\epsilon_1\, ((k_1,k_2)=(2,0)), \\
&\
2\epsilon_1+\epsilon_2\, ((k_1,k_2)=(2,1)),\
2\epsilon_1+2\epsilon_2\, ((k_1,k_2)=(2,2))
\}.
\end{split}
\end{equation*}

\smallskip
Suppose that $(k_1,k_2)=(1,0)$.
Then $\dim_{\mathbf C} V_{\Lambda} =5$.
It follows from
Lemma \ref{Branching_SO(5)_SO(4)}
that $(k_1^\prime, k_2^\prime)=(0,0)$ or $(1,0)$.
By Lemma \ref{D(SO4,T2)},
we have
$(p_1^\prime, p_2^\prime)=(0,0)$ or $(\frac{1}{2}, \frac{1}{2})$,
but $\Lambda^\prime \vert_{(p_1^\prime, p_2^\prime)=(0,0)}$,
$\Lambda^\prime \vert_{(p_1^\prime, p_2^\prime)
=(\frac{1}{2},\frac{1}{2})}\not\in D(SO(4),T^2)$.
Hence $\Lambda=(1,0)\not\in D(SO(5),T^2)=D(K,K_0)$.

\smallskip
Suppose that $(k_1,k_2)=(1,1)$.
Then ${\rm dim}_{\mathbf C} V_{\Lambda}=10$,
$V_{\Lambda}\cong {\mathfrak o}(5,{\mathbf C})$ and
$K_{[\mathfrak a]}/K_0$ acts on
$(V_{\Lambda})_{K_0}\cong ({\mathfrak t}^2)^{\mathbf C}
\cong {\mathfrak a}^{\mathbf C}$
via the action of Weyl group $W(U,K)$.
Thus it must be
$(V_{\Lambda})_{K_{[\mathfrak a]}}=\{0\}$.
Hence, $\Lambda|_{(k_1,k_2)=(1,1)}\not\in D(K,K_{[\mathfrak a]})$.

Suppose that $(k_1,k_2)=(2,0)$.
Then
$(m_1,m_2)=(2,0)$ and
${\rm dim}_{\mathbf C} V_{2\Lambda_1}=14$.
It follows from
Lemma \ref{Branching_SO(5)_SO(4)}
that $(k_1^\prime, k_2^\prime)= (0,0), (1,0)\text{ or } (2,0)$.
By Lemma \ref{D(SO4,T2)}, we have
$(p_1^\prime, p_2^\prime)=(0,0), (\frac{1}{2}, \frac{1}{2})\text{ or } (1,1)$.
Note that
$\Lambda^\prime |_{(p_1^\prime,p_2^\prime)=(0,0)}$,
$\Lambda^\prime |_{(p_1^\prime,p_2^\prime)=
(\frac{1}{2},\frac{1}{2})}\not\in D(SO(4),T^2)$.
If  $(p_1^\prime, p_2^\prime)=(1,1)$, then $(m_1^{\prime},
m_2^{\prime})=(2,2)$ and $-c_{\Lambda}=5$, $-c_{\Lambda^\prime}=4$,
thus
\begin{equation*}
-c_L=-2c_\Lambda + c_{\Lambda^\prime}=10-4=6<8.
\end{equation*}
On the other hand, we observe that
\begin{equation*}
\begin{split}
V_{2\Lambda_1}
\cong& {\rm Sym}_0 ({\mathbf C}^5)\\
=&
{\mathbf C}\cdot
\begin{pmatrix}
-\frac{1}{4} I_4 &0 \\
0&1
\end{pmatrix}
\oplus
\Bigl\{
\begin{pmatrix}
X&0\\
0&0
\end{pmatrix}
\mid
X\in \mathrm{Sym}_0 ({\mathbf C}^4)
\Bigr\}\\
&\oplus \Bigl\{
\begin{pmatrix} 0& Z\\ ^tZ&0 \end{pmatrix}
\mid
Z\in M(4,1;{\mathbf C})
\Bigr\}\\
=&W_{\vert\Lambda^\prime=0} \oplus W_{2\Lambda^\prime_{1}
+2\Lambda^\prime_2}
\oplus W_{\Lambda^\prime_1+\Lambda^\prime_2},
\end{split}
\end{equation*}
and
\begin{equation*} 
(V_{2\Lambda_1})_{K_0}=\left\{
\begin{pmatrix}
c_1 I_2& & \\
&c_2I_2& \\
 & &c_3
\end{pmatrix}
\mid
c_1,c_2,c_3\in {\mathbf C}, 2c_1+2c_2+c_3=0 \right\}.
\end{equation*}
As
\begin{equation*}
\begin{split}
&\begin{pmatrix}
 0&0 &1 &0 &0\\
0 &0 &0 &-1 &0\\
 1&0 & 0&0 &0 \\
  0&1 &0 & 0&0 \\
  0& 0& 0&0 &-1
\end{pmatrix}
\begin{pmatrix}
c_1 I_2& & \\
&c_2I_2& \\
 & &c_3
\end{pmatrix}
\begin{pmatrix}
0&0&1 &0 &0 \\
 0& 0& 0&1 &0 \\
 1&0 &0 &0 &0 \\
  0&-1 & 0& 0&0 \\
  0&0 &0 &0 &-1
\end{pmatrix}\\
&=\begin{pmatrix}
c_2 I_2& & \\
&c_1I_2& \\
 & &c_3
\end{pmatrix},
\end{split}
\end{equation*}
we get
\begin{equation*}
(V_{2\Lambda_1})_{K_{[\mathfrak a]}}=\left\{
\begin{pmatrix}
-\frac{c}{4} \mathrm{I}_4& \\
&c
\end{pmatrix}
\mid
c\in {\mathbf C} \right\} =
W_{\vert\Lambda^\prime=0}.
\end{equation*}
Thus
\begin{equation*}
W^\prime_{2\Lambda_1^\prime+2\Lambda_2^\prime}
\cap (V_{2\Lambda_1})_{K_{[\mathfrak a]}}=\{0\}.
\end{equation*}

\smallskip
Suppose that $(k_1, k_2)=(2,1)$.
Then
$(m_1,m_2)=(2,1)$ and
$\dim_{\mathbf C} V_{2\Lambda_1+\Lambda_2}=35$.
It follows from
Lemma \ref{Branching_SO(5)_SO(4)}
that
$(k_1^\prime, k_2^\prime)=$
$(1,0)$, $(1,-1)$, $(1,1)$, $(2, 0)$, $(2, -1)$ or  $(2,1)$,
that is,
$(m_1^\prime, m_2^\prime)=$ $(1,1)$, $(2,0)$, $(0,2)$, $(2,2)$, $(3,1)$  or  $(1,3)$,
and thus
\begin{equation*}
V_{2\Lambda_1+\Lambda_1}=W_{\Lambda^\prime_1+\Lambda^\prime_2}
\oplus W_{2\Lambda^\prime_1}\oplus W_{2\Lambda^\prime_2}
\oplus W_{2\Lambda^\prime_1+2\Lambda^\prime_2}
\oplus W_{3\Lambda^\prime_1+\Lambda^\prime_2}
\oplus W_{\Lambda^\prime_1 +3\Lambda^\prime_2}.
\end{equation*}
By Lemma \ref{D(SO5,T2)},
we have
$(p_1^\prime, p_2^\prime)=$
$(\frac{1}{2},\frac{1}{2})$, $(1,0)$, $(0,1)$,
$(1,1)$, $(\frac{3}{2}, \frac{1}{2})$ or
$(\frac{1}{2},\frac{3}{2})$.
Then by Lemma \ref{D(SO4,T2)} we see that
$\Lambda^\prime \vert_{(p_1^\prime,p_2^\prime)=(\frac{1}{2},\frac{1}{2})}$,
$\Lambda^\prime \vert_{(p_1^\prime,p_2^\prime)=(1,0)}$,
$\Lambda^\prime \vert_{(p_1^\prime,p_2^\prime)=(0,1)}$,
$\Lambda^\prime \vert_{(p_1^\prime,p_2^\prime)=(\frac{3}{2},\frac{1}{2})}$,
$\Lambda^\prime \vert_{(p_1^\prime,p_2^\prime)=(\frac{1}{2},\frac{3}{2})}$
$\not\in D(SO(4),T^2)$.
If $(p_1^\prime, p_2^\prime)=(1,1)$, i.e.\ $(m_1^\prime,m_2^\prime)$
$=(2,2)$, then $-c_{\Lambda}=6$,  $-c_{\Lambda^\prime}=4$ and thus
\begin{equation*}
-c_L=-2\, c_\Lambda + c_{\Lambda^\prime}=12-4=8.
\end{equation*}
So we need to determine the dimension of
$(W_{2\Lambda^\prime_1+2\Lambda^\prime_2})_{K_{[\mathfrak a]}}\neq \{0\}$.

Since $W_{2\Lambda^\prime_1+2\Lambda^\prime_2}
\cong\mathfrak{sl}(2,{\mathbf C})\boxtimes\mathfrak{sl}(2,{\mathbf C})$
and
$$
(W_{2\Lambda^\prime_1+2\Lambda^\prime_2})_{K_0}
\cong
(\mathfrak{sl}(2,{\mathbf C})\boxtimes\mathfrak{sl}(2,{\mathbf C}))_{K_0}
={\mathbf C}\boxtimes{\mathbf C},
$$
we have
$\dim_{\mathbf C}(W_{2\Lambda^\prime_1+2\Lambda^\prime_2})_{K_0}=1$.
Let $\wedge^{2}{\mathbf R}^{10}=\mathfrak{so}(10)
=\mathrm{ad}_{\mathfrak{p}}(\mathfrak{so}(5))+{\mathcal V}$.
Then
$\wedge^{2}{\mathbf C}^{10}
=(\wedge^{2}{\mathbf R}^{10})^{\mathbf C}=\mathfrak{so}(10,{\mathbf C})
=\mathrm{ad}(\mathfrak{so}(5))^{\mathbf C}+{\mathcal V}^{\mathbf C}
\cong \mathfrak{so}(5,{\mathbf C})+{\mathcal V}^{\mathbf C}$,
where
$
\{0\}\not={\mathcal V}^{\mathbf C}
\subset
V_{2\Lambda_1+\Lambda_2}
$.
By the irreducibility of $V_{2\Lambda_{1}+\Lambda_{2}}$, we see
${\mathcal V}^{\mathbf C}= V_{2\Lambda_{1}+\Lambda_{2}}$.
Since
$$
\{0\}\not=({\mathcal V}^{\mathbf C})_{K_{[{\mathfrak a}]}}
=(W_{2\Lambda^\prime_1+2\Lambda^\prime_2})_{K_{[{\mathfrak a}]}}
\subset(W_{2\Lambda^\prime_1+2\Lambda^\prime_2})_{K_0}
$$
and
$\dim_{\mathbf C}(W_{2\Lambda^\prime_1+2\Lambda^\prime_2})_{K_0}=1$,
we get
$$\{0\}\not=({\mathcal V}^{\mathbf C})_{K_{[{\mathfrak a}]}}
=(W_{2\Lambda^\prime_1+2\Lambda^\prime_2})_{K_{[{\mathfrak a}]}}
=(W_{2\Lambda^\prime_1+2\Lambda^\prime_2})_{K_0}$$
and
$\dim_{\mathbf C}(W_{2\Lambda^\prime_1+
2\Lambda^\prime_2})_{K_{[{\mathfrak a}]}}=1$.
Hence
$2\Lambda_{1}+\Lambda_2 \in D(K, K_{[\mathfrak{a}]})$
and its multiplicity is equal to $1$.

\smallskip
Suppose that $(k_1, k_2)=(2,2)$.
It follows from Lemma \ref{Branching_SO(5)_SO(4)}
that
$(k_1^\prime, k_2^\prime)=(2,0)$, $(2,1)$, $(2,2)$, $(2, -1)$ or  $(2, -2)$.
By Lemma \ref{D(SO4,T2)},
we have
$(p_1^\prime, p_2^\prime)=(1,1)$,
$(\frac{1}{2}, \frac{3}{2})$, $(0,2)$,
$(\frac{3}{2},\frac{1}{2})$ or $(2,0)$ and thus
$\Lambda^\prime \vert_{(p_1^\prime,p_2^\prime)=(\frac{1}{2},\frac{3}{2})}$,
$\Lambda^\prime \vert_{(p_1^\prime,p_2^\prime)=(0,2)}$,
$\Lambda^\prime \vert_{(p_1^\prime, p_2^\prime)=(\frac{3}{2},\frac{1}{2})}$,
$\Lambda^\prime \vert_{(p_1^\prime, p_2^\prime)=(2,0)}$
$\not\in D(SO(4),T^2)$.
If $(p_1^\prime, p_2^\prime)=(1,1)$, then
$-c_{\Lambda}=8$, $-c_{\Lambda^\prime}=4$ and hence
\begin{equation*}
-c_L=-2c_\Lambda + c_{\Lambda^\prime}=16-4=12>8.
\end{equation*}
Now we obtain that the Gauss image
$L^{8}={\mathcal G}(SO(5)/T^{2})\subset Q_{8}({\mathbf C})$
is Hamiltonian stable.
Moreover it also follows that
\begin{equation*}
n(L^{8})
=\dim_{\mathbf C}(V_{2\Lambda_1+\Lambda_2})
=35
=\dim(SO(10))-\dim(SO(5))
=n_{hk}(L^{8}).
\end{equation*}
Hence the Gauss image
$L^{8}={\mathcal G}(SO(5)/T^{2})\subset Q_{8}({\mathbf C})$
is Hamiltonian rigid.

From theses results we conclude that
\begin{thm}
The Gauss image
$L^{8}={\mathcal G}(SO(5)/T^{2})=\frac{SO(5)}{T^{2}\cdot{\mathbf Z}_2}
\subset Q_{8}({\mathbf C})$
is strictly Hamiltonian stable.
\end{thm}

\section{The case $(U,K)=(SO(10), U(5))$}
\label{Sec_DIII2}

In this case, $(U,K)$ is of $BC_2$ type and
$K=U(5)\subset U=SO(10)$.
Here each $A+\sqrt{-1}B\in U(5)$ can be identified with an element
$\left(
\begin{array}{cc}
A & -B \\
B & A \\
\end{array}
\right)
\in SO(10)$ with $A,B\in {\mathfrak gl}(5,\mathbf{R})$.
The canonical decomposition
${\mathfrak u}={\mathfrak k}+{\mathfrak p}$ of ${\mathfrak u}$
and ${\mathfrak a}$ be a maximal abelian subspace of
${\mathfrak p}$ are given by ${\mathfrak u}=\mathfrak{so}(10)$,
\begin{equation*}
\begin{split}
{\mathfrak k} =&
\left\{
\begin{pmatrix}
X & -Y \\
Y & X
\end{pmatrix}
\in \mathfrak{so}(10)
\mid -X^t = X, Y^t=Y \right\}\\
\cong &\
\mathfrak{u}(5)
=
\left\{
T=X+\sqrt{-1}Y\in\mathfrak{gl}(5,{\mathbf C})\mid T^\ast=-T \right\},
\\
{\mathfrak p}
=&
\left\{
\begin{pmatrix}
X & Y \\
Y & -X \\
\end{pmatrix}
\in \mathfrak{so}(10)
\mid X, Y\in \mathfrak{so}(5)
\right\}
\end{split}
\end{equation*}
and
\begin{equation*}
{\mathfrak a}=
\left\{
\begin{pmatrix}
H_1 & 0 \\
0 & -H_1 \\
\end{pmatrix}
\mid
H_1=
\begin{pmatrix}
    0 & -\xi_1 &  &  &  \\
    \xi_1 & 0 &  &  &  \\
     &  & 0 & -\xi_2 &  \\
     &  & \xi_2 & 0 &  \\
     &  &  &  & 0\\
\end{pmatrix}
\xi_{1},\xi_{2}\in{\mathbf R}\right\}.
\end{equation*}
Then the centralizer $K_0$ of $\mathfrak{a}$ in $K$ is as follows:
\begin{equation*}
\begin{split}
&K_{0}
\\
=&\left\{
\left(
  \begin{array}{ccccc}
    a_{11}+\mathbf{i}b_{11} &  a_{12}+\mathbf{i}b_{12} & 0 & 0 & 0\\
   -a_{12}+\mathbf{i}b_{12} & a_{11}-\mathbf{i}b_{11} & 0 & 0 & 0 \\
    0 & 0 & a_{22}+\mathbf{i}b_{22} & a_{21}+\mathbf{i}b_{21} & 0\\
   0 & 0 & -a_{21}+\mathbf{i}b_{21} & a_{22}-\mathbf{i}b_{22} & 0 \\
   0 & 0 & 0 & 0 &a_{33}+\mathbf{i}b_{33} \\
  \end{array}
\right)\in U(5)
\right\}\\
\cong &\ SU(2)\times SU(2)\times U(1)
\end{split}
\end{equation*}
and $N=K/K_0\cong U(5)/SU(2)\times SU(2)\times U(1)$
is of dimension $18$.
Moreover,
\begin{equation*}
\begin{split}
K_{[{\mathfrak a}]}
= &
K_0 \cup
\begin{pmatrix}
 & &1 &0 & \\
 & & 0&1 & \\
 1&0 &  & &  \\
  0&-1 & & & \\
  & & & & 1
\end{pmatrix}\cdot K_0
\cup
\begin{pmatrix}
1&  & & & \\
 &-1& & &\\
 & & 1& & \\
  & & & -1&\\
  & & & &1
\end{pmatrix}\cdot K_0 \\
& \cup
\begin{pmatrix}
 &&1 &0 & \\
 & & 0&-1 & \\
 1&0 & & & \\
  0&1 & & & \\
  & & & &1
\end{pmatrix}\cdot K_0.
\end{split}
\end{equation*}
It means that the deck transformation group of the covering map
${\mathcal G}:N\rightarrow {\mathcal G}(N^{18})$ is equal to
$K_{[{\mathfrak a}]}/K_{0}\cong{\mathbf Z}_{4}$.

\subsection{Description of the Casimir operator}

Choose
$\langle{X,Y}\rangle_{\mathfrak{u}}:=-\rm{tr}(XY)$
for each $X,Y\in \mathfrak{u}=\mathfrak{so}(10)$.
The restricted root system $\Sigma(U,K)$ of type $BC_2$
can be given as follows (\cite{Bourbaki}):
\begin{equation*}
\begin{split}
&\Sigma(U,K) \\
=&\{
\pm\epsilon_2=\pm\alpha_1,
\pm(\epsilon_1-\epsilon_2)=\pm\alpha_2,
\pm\epsilon_1=\pm(\alpha_1+\alpha_2),
\\
&
\pm(\epsilon_1+\epsilon_2)=\pm(2\alpha_1+\alpha_2),
\pm 2\epsilon_1=\pm(2\alpha_1+2\alpha_2),
\pm 2\epsilon_2=\pm 2\alpha_1
\}.
\end{split}
\end{equation*}
Then the square length of each $\gamma\in\Sigma(U,K)$
relative to $\langle\,,\,\rangle_{\mathfrak u}$ is
\begin{equation*}
\Vert{\gamma}\Vert_{\mathfrak u}^{2}
=\frac{1}{4}, \frac{1}{2} \text{ or } 1.
\end{equation*}
Hence the Casimir operator $\mathcal{C}_L$ of $L^n$
with respect to the induced metric from $g^{\rm std}_{Q_n(\mathbf{C})}$
can be expressed as follows:
\begin{equation*}
\begin{split}
{\mathcal C}_{L}
&=
\frac{4}{\Vert{\gamma_{0}}\Vert^2_{\mathfrak u}}\,
{\mathcal C}_{K/K_0,\langle{\ ,\ }\rangle_{\mathfrak u}}
-\frac{1}{\Vert{\gamma_{0}}\Vert^2_{\mathfrak u}}\,
{\mathcal C}_{K_{1}/K_0,\langle{\ ,\ }\rangle_{\mathfrak u}}
-\frac{2}{\Vert{\gamma_{0}}\Vert^2_{\mathfrak u}}\,
{\mathcal C}_{K_{2}/K_0,\langle{\ ,\ }\rangle_{\mathfrak u}}\\
&=
4\, \mathcal{C}_{K/K_0, \langle\,,\,\rangle_{\mathfrak{u}}}
-\mathcal{C}_{K_1/K_0, \langle\,,\,\rangle_{\mathfrak{u}}}
-2\, \mathcal{C}_{K_2/K_0, \langle\,,\,\rangle_{\mathfrak{u}}}\\
&=
2\, {\mathcal C}_{K/K_0}-{\mathcal C}_{K_2/K_0}
-\frac{1}{2}\, {\mathcal C}_{K_1/K_0}
\\
&=
 {\mathcal C}_{K/K_0}+{\mathcal C}_{K/K_2}+\frac{1}{2}{\mathcal C}_{K_1/K_0},
\end{split}
\end{equation*}
where
${\mathcal C}_{K/K_{0}}$, ${\mathcal C}_{K_{2}/K_{0}}$
and ${\mathcal C}_{K_{1}/K_{0}}$
denote
the Casimir operator of  $K/K_{0}$, $K_{2}/K_{0}$ and
$K_{1}/K_{0}$
relative to
$\langle{{\ },{\ }}\rangle\vert_{\mathfrak k}$,
$\langle{{\ },{\ }}\rangle\vert_{{\mathfrak k}_{2}}$
and
$\langle{{\ },{\ }}\rangle\vert_{{\mathfrak k}_{1}}$,
respectively.
Here, $\langle X, Y \rangle:=-\mathrm{tr}({\mathrm{Re}} (XY))$ for all $X,Y\in \mathfrak{k}=u(5)$.

\subsection{Descriptions of $D(K)$,$D(K_1)$ and $D(K_2)$}
Using a maximal abelian subalgebra $\mathfrak{t}$ of $\mathfrak{k}$
given by
\begin{equation*}
\begin{split}
{\mathfrak t}
=&\left\{
\sqrt{-1}
\begin{pmatrix}
y_{1}&0&0&0&0\\
0&y_{2}&0&0&0\\
0&0&y_{3}&0&0\\
0&0&0&y_{4}&0\\
0&0&0&0&y_{5}
\end{pmatrix}
\mid
y_{1},y_{2}, y_{3},y_{4},y_{5}\in{\mathbf R}
\right\}\subset {\mathfrak k},
\end{split}
\end{equation*}
we have
\begin{equation*}
\begin{split}
\Gamma(K)
&=\Gamma(K_2)
=\Gamma(K_{1})
=\Gamma(K_0)
\\
&=
\left\{
\xi=\sqrt{-1}
\begin{pmatrix}
\xi_{1}&0 &0 &0&0\\
0&\xi_{2}&0 &0 & 0 \\
0&0&\xi_{3}&0&0\\
0&0&0&\xi_{4}&0\\
0&0&0&0&\xi_{5}
\end{pmatrix}
\mid
\xi_{1},\xi_{2},\xi_{3},\xi_{4},\xi_{5}\in 2\pi{\mathbf Z}
\right\},
\\
\Gamma(C(K))
&=
2\pi{\mathbf Z}\mathrm{I}_{5}.
\end{split}
\end{equation*}
Then
$D(K)$, $D(K_1)$ and $D(K_2)$
are given as follows:
\begin{equation*}
\begin{split}
D(K)=&D(U(5))\\
=& \{\Lambda =p_{1}y_{1}+\cdots+ p_5 y_5
\mid
p_{1},\cdots,p_{5}\in{\mathbf Z},
p_{1}\geq p_2\geq p_3\geq p_4\geq p_5
\},
\\
D(K_2)=&D(U(4)\times U(1))\\
=& \{\Lambda =p_{1}y_{1}+\cdots+ p_5 y_5
\mid
p_{1},\cdots,p_{5}\in{\mathbf Z},
p_{1}\geq p_2\geq p_3\geq p_4
\},
\\
D(K_1)=&D(U(2)\times U(2)\times U(1))\\
=& \{\Lambda =p_{1}y_{1}+\cdots+ p_5 y_5
\mid
p_{1},\cdots,p_{5}\in{
\bold Z},
p_{1}\geq p_2, p_3\geq p_4
\}.
\end{split}
\end{equation*}

\subsection{Branching laws of $(U(m+1),U(m)\times U(1))$}
\noindent

The branching laws for $(SU(m+1),S(U(1)\times U(m)))$ was shown
by Ikeda and Taniguchi \cite{Ikeda-Taniguchi78}.
It can be reformulated to
the branching laws for $(U(m+1),U(m)\times U(1))$ as follows:

\begin{lem}[Branching laws for $(U(m+1),U(m)\times U(1))$]
\label{BranchingLaws(U(m+1),U(m)XU(1))}
Let $\Lambda=p_1y_1+\cdots+p_m y_m\in D(U(m))$
be the highest weight of an irreducible $U(m)$-module $V_{\Lambda}$,
where $p_i\in \mathbf{Z}$ $(i=1,\cdots, m)$
and $p_1\geq p_2\geq \cdots \geq p_m$.
Then the irreducible decomposition of $V_{\Lambda}$
as a $U(m)\times U(1)$-module contains an irreducible
$U(m)\times U(1)$-module $V_{\Lambda^\prime}$
with the highest weight
$V_{\Lambda^\prime}=q_1 y_1+\cdots +q_m y_m \in D(U(m)\times U(1))$,
where $q_i\in \mathbf{Z}$ and $q_1\geq q_2\geq \cdots\geq q_m$,
if and only if
\begin{equation*}
\begin{split}
&p_1\geq q_1 \geq p_2 \geq q_2 \geq p_3 \geq q_3 \geq \cdots
\geq p_{m-1} \geq q_{m-1} \geq p_m,\\
& \sum_{i=1}^m p_i= \sum_{i=1}^m q_i.
\end{split}
\end{equation*}
In particluar the multiplicity of $V_{\Lambda^\prime}$ is $1$.
\end{lem}

In the next subsection we use
the branching laws of  $(U(m+1),U(m)\times U(1))$
and $(U(m), U(2)\times U(m-2))$ in the case of $m=4$.
The branching laws of $(U(m), U(2)\times U(m-2))$
are described in Lemma \ref{BranchingLaw(U(m), U(2)XU(m-2))}
of Section \ref{Sec_AIII_2}.

\subsection{Descriptions of $D(K, K_0)$, $D(K_2, K_0)$ and $D(K_1,K_0)$}
\noindent

Each $\Lambda\in{D(K)=D(U(5))}$ is expressed as
\begin{equation*}
\Lambda=p_1 y_1+\cdots p_5 y_5,
\end{equation*}
where $p_i\in \mathbf{Z}$, $p_1\geq p_2\geq p_3\geq p_4 \geq p_5$.
Then by Lemma \ref{BranchingLaws(U(m+1),U(m)XU(1))} in the case of $m=4$,
$V_{\Lambda}$ can be decomposed into irreducible
$U(4)\times U(1)$-modules as
\begin{equation*}
V_{\Lambda}=\bigoplus_{i=1}^s V^{\prime}_{\Lambda^{\prime}_i}
=\bigoplus_{i=1}^s W^{\prime}_{{\Lambda^{\prime}_1}_i} \boxtimes U_{q_5 y_5},
\end{equation*}
where
$\Lambda^{\prime}_i=q_1y_1+q_2y_2+ q_3 y_3 +q_4 y_4+q_5 y_5
\in D(K_2)=D(U(4)\times U(1))$,
$\Lambda^{\prime}_{1i}=q_1y_1+q_2y_2+ q_3 y_3 +q_4 y_4\in D(U(4))$,
$q_5 y_5 \in D(U(1))$ and $q_i\in{\mathbf Z}\ (i=1,2,3,4,5)$ satisfy
\begin{equation*}
\begin{split}
& p_1\geq q_1\geq p_2\geq q_2\geq p_3\geq q_3 \geq p_4 \geq q_4 \geq p_5,\\
& \sum_{i=1}^5 p_i =\sum_{j=1}^5 q_j.
\end{split}
\end{equation*}
By the branching law for $(U(4), U(2)\times U(2))$
in Lemma \ref{BranchingLaw(U(m), U(2)XU(m-2))},
each $W^{\prime}_{{\Lambda^{\prime}_1}_i}$ can be decomposed as
\begin{equation*}
W^{\prime}_{{\Lambda^{\prime}_1}_i}
=\bigoplus
W^{\prime\prime}_{\Lambda^{\prime\prime}}
=\bigoplus W^{\prime\prime}_{\tilde{\Lambda}_\sigma} \boxtimes
W^{\prime\prime}_{\tilde{\Lambda}_\rho},
\end{equation*}
where $\Lambda^{\prime\prime}=k_1 y_1+k_2 y_2+k_3 y_3+k_4 y_4\in D(U(2)\times
U(2))$, $\tilde{\Lambda}_\sigma=k_1 y_1+k_2 y_2 \in D(U(2))$,
$\tilde{\Lambda}_\rho=k_3 y_3 +k_4 y_4 \in D(U(2))$ and
$k_{i}\in{\mathbf Z}\ (i=1,2,3,4)$ satisfy
\begin{itemize}
\item[(i)] $\sum_{i=1}^4 k_i= \sum_{i=1}^4 q_i$;
\item[(ii)] $q_1\geq k_1\geq q_3$, $q_2\geq k_2\geq q_4$;
\item[(iii)] in the finite power series expansion in $X$ of
$\displaystyle\frac{\prod_{i=1}^3 (X^{r_i+1}-X^{-(r_i +1)})}{(X-X^{-1})^2}$, where
$r_i (i=1,2,3)$ are defined by
\begin{equation*}
\begin{split}
r_1:=&q_1-\max (k_1, q_2),\\
r_2:=& \min (k_1, q_2)- \max (k_2, q_3), \\
r_3:=& \min (k_2, q_3) -q_4,\\
\end{split}
\end{equation*}
the coefficient of $X^{k_3-k_4+1}$ does not vanish.
Moreover the value of this coefficient is the multiplicity of the
$U(2)\times U(2)$-module
$W^{\prime\prime}_{\Lambda^{\prime\prime}}$.
\end{itemize}

By the branching law of $(U(2),SU(2))$ (see Section \ref{Sec_AIII_2}),
as $SU(2)$-modules they become
\begin{equation*}
W^{\prime\prime}_{\tilde{\Lambda}_\sigma}=W^{\prime\prime}_{\Lambda_\sigma},
\quad
W^{\prime\prime}_{\tilde{\Lambda}_\rho}=W^{\prime\prime}_{\Lambda_\rho},
\end{equation*}
where
$\Lambda_\sigma=\displaystyle\frac{k_1-k_2}{2}(y_1-y_2)\in D(SU(2))$,
$\Lambda_\rho=\displaystyle\frac{k_3-k_4}{2}(y_3-y_4)\in D(SU(2))$.

Hence, one can decompose a $K$-module $V_{\Lambda}$ into the following irreducible $K_0$-modules
\begin{equation*}
V_{\Lambda}=\bigoplus\bigoplus W^{\prime\prime}_{\Lambda_\sigma}
\boxtimes W^{\prime\prime}_{\Lambda_\rho}
\boxtimes U_{q_5y_5}.
\end{equation*}

\medskip

Now assume that $\Lambda\in D(K,K_0)$.
Then there exists at least one nonzero trivial irreducible
$K_0$-module in the above decomposition for some $\sigma$ and $\rho$.
So in this case, we have
\begin{equation*}
k_1-k_2=0,\quad
k_3-k_4=0, \quad
q_5=0.
\end{equation*}
So we know that
\begin{equation*}
\begin{split}
&2k_1+2k_3=\sum_{i=1}^4 q_i=\sum_{j=1}^5 p_j,\\
& q_2\geq k_1=k_2\geq q_3,\\
& r_1=q_1-q_2,\\
& r_2=k_1 -k_2 =0,\\
& r_3=q_3 - q_4\\
\end{split}
\end{equation*}
and in the finite power series expansion in $X$ of
$$ \frac{(X^{q_1-q_2+1}-X^{-(q_1-q_2+1)})(X^{q_3-q_4+1}-X^{-(q_3-q_4+1)})}{X-X^{-1}},$$
the coefficient of $X$ does not vanish. Moreover, the
value of this coefficient is the multiplicity of the
$U(2)\times U(2)$-module.

\bigskip
Therefore, in the above notations, for each $\Lambda\in{D(K,K_{0})}$
given by
$\Lambda= p_1 y_1 +p_2 y_2 + p_3 y_3 +p_4 y_4 +p_5 y_5,$
where $p_{1},\cdots,p_{5}\in{\mathbf Z}$,
$p_{1}\geq p_2\geq p_3\geq p_4\geq p_5,$
each $\Lambda^\prime\in{D(K_2,K_{0})}$
is given by
$\Lambda^{\prime}= q_1 y_1 + q_2 y_2 + q_3 y_3 + q_4 y_4,$
where
$q_{1},\cdots,q_{4}\in{\bold Z}$,
$q_{1}\geq q_2\geq q_3\geq q_4$,
$\sum_{i=1}^5 p_i=\sum_{j=1}^4 q_j$.
Moreover, each $\Lambda^{\prime\prime}\in{D(K_1,K_{0})}$
is given by
$\Lambda^{\prime\prime}= k_1 y_1 + k_2 y_2 + k_3 y_3 + k_4 y_4,$
where
$k_{1},\cdots,k_{4}\in{\mathbf Z}$,
$k_{1}=k_2$,  $k_3=k_4$,
$2k_1+2k_3=\sum_{j=1}^4 q_j$.

\subsection{Eigenvalue computation}
For each
$\Lambda=p_1y_1+p_2y_2+p_3 y_3 +p_4 y_4+ p_5 y_5 \in D(K, K_0)$,
with $p_i\in \mathbf{Z}$, $p_1\geq p_2\geq p_3 \geq p_4 \geq p_5$,
the eigenvalue formula of the Casimir operator $\mathcal{C}_{K/K_0}$
with respect to the inner product $\langle X, Y\rangle_{\mathfrak k}=-\rm{Tr}(\rm{Re}(XY))$
for any $X,Y\in \mathfrak{k}=\mathfrak{u}(5)$ is given by
$$-c_{\Lambda}=p_1^2+p_2^2+p_3^2+p_4^2+p_5^2+4p_1+2p_2-2p_4-4p_5.$$

For each $\Lambda^\prime=q_1 y_1+ q_2 y_2+ q_3 y_3 + q_4 y_4 \in D(K_2, K_0)$
 with $q_i\in \mathbf{Z}$ and $q_1\geq q_2\geq q_3\geq q_4$,
the eigenvalue formula of the Casimir operator $\mathcal{C}_{K_2/K_0}$
with respect to the inner product
$\langle\, , \, \rangle_{\mathfrak{k}}\vert_{\mathfrak{k}_2}$ is given by
\begin{equation*}
-c_{{\Lambda}^{\prime}}
=q_1^2+q_2^2+q_3^2+q_4^2+3q_1+q_2-q_3-3q_4.
\end{equation*}

For each
$\Lambda^{\prime\prime}=k_1 y_1+ k_2 y_2+ k_3 y_3 + k_4 y_4 \in D(K_1, K_0)$
with $k_1=k_2$ and $k_3=k_4$,
the eigenvalue formula of the Casimir operator $\mathcal{C}_{K_1/K_0}$ with
respect to the inner product
$\langle \, ,\,\rangle_{\mathfrak{k}}\vert_{\mathfrak{k}_1}$ is given by
\begin{equation*}
\begin{split}
-c_{{\Lambda}^{\prime\prime}}
=&k_1^2+k_2^2+k_3^2+k_4^2+k_1-k_2+k_3-k_4\\
=&k_1^2+k_2^2+k_3^2+k_4^2.
\end{split}
\end{equation*}

Hence, we have the following eigenvalue formula
\begin{equation}\label{EigenvalueFormulaDIII_2}
\begin{split}
-c_{L}
=&{\ }
-2c_{\Lambda}+c_{\Lambda^{\prime}}+\frac{1}{2}c_{\Lambda^{\prime\prime}}
\\
=&{\ }
2(p_1^2+p_2^2+p_3^2+p_4^2+p_5^2+4p_1+2p_2-2p_4-4p_5)\\
&-(q_1^2+q_2^2+q_3^2+q_4^2+3q_1+q_2-q_3-3q_4)\\
&-\frac{1}{2}(k_1^2+k_2^2+k_3^2+k_4^2).
\end{split}
\end{equation}
By using and estimating the formula \label{EigenvalueFormulaDIII_2}
from above by $18$, we get that
\begin{lem}
$\Lambda=p_1y_1+p_2y_2+p_3 y_3 +p_4 y_4+ p_5 y_5\in D(K)$
belongs to $\Lambda\in D(K,K_0)$ with eigenvalue $-c_{L}\leq{18}$
if and only if $(p_{1},p_{2},p_{3},p_{4}, p_{5})$ is one of
\begin{equation*}
\begin{split}
\Bigl\{\,
&
(0,-1,-1,-1,-1),\,  (1,1,1,1,0),\,  (1,1,0,0,0),\,  (0,0,0,-1,-1),
\\
&
(1,0,0,0,-1),\, (2,1,1,0,0),\, (0,0,-1,-1,-2),\,  (1,1,0,-1,-1)
\, \Bigr\}.
\end{split}
\end{equation*}
\end{lem}

Denote by $\omega_1, \omega_2, \omega_3, \omega_4$
the fundamental weight system of $SU(5)$.

Suppose that $\Lambda=(1,1,1,1,0)$.
Then $\dim V_{\Lambda}=5$.
By the branching law of $(U(5), U(4)\times U(1))$,
$\Lambda^{\prime}=(1,1,1,1,0)\text{ or } (1,1,1,0,1)$,
where $\Lambda^{\prime}=(1,1,1,1,0)\in D(K_2,K_0)$.
By
the branching law of $(U(4), U(2)\times U(2))$,
$\Lambda^{\prime\prime}=(1,1,1,1)\in D(K_1,K_0)$.
Thus $-c_{\Lambda}=8$, $-c_{\Lambda^{\prime}}=4$,
$-c_{\Lambda^{\prime\prime}}=4$ and $-c_L=-2c_{\Lambda}+c_{\Lambda^\prime}+\frac{1}{2}c_{\Lambda^{\prime\prime}}
=10<18$.

On the other hand,
$\Lambda= \Lambda_0+ \omega_4$, where
$\Lambda_0=\frac{4}{5}\sum_{i=1}^{5} y_i$.
The group $K=U(5)=C(U(5))\cdot SU(5)$ acts on
$\dim V_{\Lambda}=5$ and
$V_{\Lambda}\cong \mathbf{C}\otimes \bar{\mathbf{C}}^5$
by $\rho_{\Lambda_0}\boxtimes \bar{\mu}_5$,
where $\bar{\mu}_5$ denotes the conjugate representation of the standard representation of $SU(5)$ on $\mathbf{C}^5$.
For each element $g_0=\left(
                   \begin{array}{ccc}
                     A &  &  \\
                      & B &  \\
                      &  & e^{\sqrt{-1}\theta} \\
                   \end{array}
                 \right)\in K_0
$ and each element
$u\otimes \mathbf{w}\in \mathbf{C}\otimes \bar{\mathbf C}^5$,
where $A,B\in SU(2)$ and $\theta\in \mathbf{R}$,
\begin{equation*}
\begin{split}
\rho_{\Lambda}(g_0)(u\otimes \mathbf{w})=&\rho_{\Lambda_0}(e^{\frac{\sqrt{-1}}{5}\theta}I_5)(u)
\otimes \rho_{\omega_4}(e^{-\frac{\sqrt{-1}}{5}\theta}g_0)\mathbf{w}\\
          &=e^{\frac{4\sqrt{-1}}{5}\theta}u\otimes \left(
                                                     \begin{array}{c}
                                                       e^{\frac{\sqrt{-1}}{5}\theta}\bar{A}\left(
                                                                                             \begin{array}{c}
                                                                                               w_1 \\
                                                                                               w_2 \\
                                                                                             \end{array}
                                                                                           \right)
                                                        \\
                                                       e^{\frac{\sqrt{-1}}{5}\theta}\bar{B}\left(
                                                                                             \begin{array}{c}
                                                                                               w_3 \\
                                                                                               w_4 \\
                                                                                             \end{array}
                                                                                           \right) \\
                                                       e^{-\frac{4\sqrt{-1}}{5}\theta}w_5 \\
                                                     \end{array}
                                                   \right).
\end{split}
\end{equation*}
Hence $(V_{\Lambda})_{K_0}={\rm span}_{\mathbf C}\{ 1\otimes \left(
                                                               \begin{array}{c}
                                                                 0 \\
                                                                 0 \\
                                                                 0 \\
                                                                 0 \\
                                                                 1 \\
                                                               \end{array}
                                                             \right)
\}$.

For a generator $g=\left(
                     \begin{array}{ccccc}
                        &  & 1 & 0 &  \\
                        &  & 0 & 1 &  \\
                       1 & 0 &  &  &  \\
                       0 & -1 & &  &  \\
                        &  &  &  & 1 \\
                     \end{array}
                   \right)
\in K_{[\mathfrak{a}]}\subset K_2$ in $\mathbf{Z}_4$,
\begin{equation*}
\begin{split}
\rho_{\Lambda}(g)(u\otimes \mathbf{e}_5)
&=\rho_{\Lambda_0}(e^{\sqrt{-1}\frac{\pi}{5}}I_5)(u)\otimes \rho_{\omega_4}(e^{-\sqrt{-1}\frac{\pi}{5}}g)(\mathbf{e}_5)\\
&= e^{\sqrt{-1}\frac{4\pi}{5}}u\otimes e^{\sqrt{-1}\frac{\pi}{5}}\mathbf{e}_5
=-u\otimes \mathbf{e}_5.
\end{split}
\end{equation*}
So $(V_{\Lambda})_{K_{[\mathfrak a]}}=\{0\}$, i.e.,
$\Lambda=(1,1,1,1,0)\not\in D(K,K_{[\mathfrak a]})$.
Similarly, we get $\Lambda=(0,-1,-1,-1,-1)\not\in D(K,K_{[\mathfrak a]})$.

Suppose that $\Lambda=(1,1,0,0,0)$.
Then $\dim V_{\Lambda}=10$.
By the branching law of $(U(5), U(4)\times U(1))$,
$\Lambda^{\prime}=(1,1,0,0,0)\text{ or } (1,0,0,0,1)$,
where $\Lambda^{\prime}=(1,1,0,0,0) \in D(K_2,K_0)$.
By the branching law of $(U(4), U(2)\times U(2))$,
$\Lambda^{\prime\prime}=(1,1,0,0), (0,0,1,1)\text{ or } (1,0,1,0)$,
where $\Lambda^{\prime\prime}=(1,1,0,0)\text{ or } (0,0,1,1)\in D(K_1,K_0)$.
Thus $-c_{\Lambda}=8$, $-c_{\Lambda^{\prime}}=6$,
$-c_{\Lambda^{\prime\prime}}=2$ and $-c_L=-2c_{\Lambda}+c_{\Lambda^\prime}+\frac{1}{2}c_{\Lambda^{\prime\prime}}
=9<18$.

On the other hand, $\Lambda= \Lambda_0+ \omega_2$, where $\Lambda_0=\frac{2}{5}\sum_{i=1}^{5} y_i$.
$V_{\Lambda}\cong \mathbf{C}\oplus \wedge^2\mathbf{C}^5$.
Let $\{\mathbf{e}_1, \mathbf{e}_2, \mathbf{e}_3,\mathbf{e}_4,\mathbf{e}_5\}$ be the standard basis of ${\mathbf C}^5$.
For each element $g_0\in K_0$ expressed as above and
each element $u\otimes \mathbf{e}_i\wedge \mathbf{e}_j \in V_{\Lambda}$
$(1\leq i <j \leq 5 )$,
\begin{equation*}
\begin{split}
\rho_{\Lambda}(g_0)(u\otimes \mathbf{e}_i\wedge \mathbf{e}_j)
&=\rho_{\Lambda_0}(e^{\frac{\sqrt{-1}}{5}\theta}I_5)(u)\otimes
\rho_{\omega_2}(e^{-\frac{\sqrt{-1}}{5}\theta}g_0)(\mathbf{e}_i\wedge \mathbf{e}_j)\\
&= e^{\sqrt{-1}\frac{2}{5}\theta}
u\otimes (e^{-\frac{\sqrt{-1}}{5}\theta}g_0 \mathbf{e}_i \wedge
e^{-\frac{\sqrt{-1}}{5}\theta}g_0 \mathbf{e}_j).
\end{split}
\end{equation*}
It follows from this that
$(V_{\Lambda})_{K_0}={\rm span}_{\mathbf{C}}\{1\otimes (\mathbf{e}_1 \wedge \mathbf{e}_2),
1\otimes (\mathbf{e}_3 \wedge \mathbf{e}_4)\}$.
For the generator $g\in K_{[\mathfrak a]}$ of $\mathbf{Z}_4$ given above,
we have
\begin{eqnarray*}
\rho_{\Lambda}(g)(1\otimes \mathbf{e}_1\wedge\mathbf{e}_2)&=&-1\otimes \mathbf{e}_3\wedge \mathbf{e}_4,\\
\rho_{\Lambda}(g)(1\otimes \mathbf{e}_3\wedge\mathbf{e}_4)&=& 1\otimes \mathbf{e}_1\wedge \mathbf{e}_2.
\end{eqnarray*}
Hence $(V_\Lambda)_{K_{[\mathfrak a]}}=\{0\}$, i.e., $\Lambda=(1,1,0,0,0)\not\in D(K,K_{[\mathfrak a]})$.
Similarly, we get $\Lambda=(0,0,0,-1,-1)\not\in D(K,K_{[\mathfrak a]})$.

Suppose that $\Lambda=(1,0,0,0,-1)$.
Then $\dim V_{\Lambda}=24$.
By the branching law of $(U(5), U(4)\times U(1))$,
$\Lambda^{\prime}=$ $(1,0,0,0,-1)$,
$(1,0,0,-1,0)$,
$(0,0,0,0,0)$ or  $(0,0,0,-1,1)$,
where $\Lambda^{\prime}_1=(1,0,0,-1,0)$,
$\Lambda^{\prime}_2=(0,0,0,0,0) \in D(K_2,K_0)$.
By the branching law of $(U(4), U(2)\times U(2))$,
$\Lambda^{\prime\prime}_1=$ $(1,0,0, -1)$, $(1,-1,0,0)$,
$(0,0,0,0)$, $(0,0,1,-1)$ or $(0,-1,1,0)$,
where $\Lambda^{\prime\prime}_1= (0,0,0,0)\in D(K_1,K_0)$.
Also, $\Lambda^{\prime\prime}_2=(0,0,0,0)\in D(K_1,K_0)$.
Thus $-c_{\Lambda}=10$, $-c_{\Lambda^{\prime}_1}=8$,
$-c_{\Lambda^{\prime\prime}_1}=0$, $-c_L=-2c_{\Lambda}+c_{\Lambda^\prime}+\frac{1}{2}c_{\Lambda^{\prime\prime}}
=12<18$ and $-c_{\Lambda^{\prime}_2}=0$,
$-c_{\Lambda^{\prime\prime}_2}=0$,
$-c_L=20>18$.

On the other hand, $\Lambda=\omega_1+\omega_4$ corresponds to
the adjoint representation of $SU(5)$.
\begin{equation*}
\begin{split}
V_{\Lambda}&={\mathbf C}\otimes (\mathbf{C}\cdot \begin{pmatrix} -\frac{1}{4} I_4 &0 \\ 0&1\end{pmatrix}
\oplus \mathbf{C}\cdot \left(
                         \begin{array}{cc}
                           * & 0 \\
                           0 & 0 \\
                         \end{array}
                       \right)\\
& \oplus \mathbf{C}\cdot \left(
                         \begin{array}{ccc}
                            &  & * \\
                            &  & 0\\
                           * & 0 & 0\\
                         \end{array}
                       \right)
\oplus \mathbf{C}\cdot \left(
                         \begin{array}{ccc}
                            &  & 0 \\
                            &  & *\\
                           0 & * & 0\\
                         \end{array}
                       \right)
)\\
&=V^{\prime}_{(0,0,0,0,0)}\oplus V^{\prime}_{(1,0,0,-1,0)}\oplus V^{\prime}_{(1,0,0,0,-1)}\oplus
V^{\prime}_{(0,0,0,-1,1)}.
\end{split}
\end{equation*}

\begin{equation*}
\begin{split}
(V_{\Lambda})_{K_0}=&\left\{
\begin{pmatrix}
c_1 I_2& & \\
&c_2I_2& \\
 & &c_3
\end{pmatrix}
{\ }\vert{\ } c_1,c_2,c_3\in {\mathbf C}, 2c_1+2c_2+c_3=0 \right\}\\
\subset& V^{\prime}_{(0,0,0,0,0)}\oplus V^{\prime}_{(1,0,0,-1,0)}.
\end{split}
\end{equation*}

By direct calculations, we get that for a generator
$g\in K_{[\mathfrak a]}\subset K_2$ in ${\mathbf Z}_4$ as above,
\begin{equation*}
\mathrm{Ad}(g)
\begin{pmatrix}
c_1 I_2& & \\
&c_2I_2& \\
 & &c_3
\end{pmatrix}
=\begin{pmatrix}
c_2 I_2& & \\
&c_1I_2& \\
 & &c_3
\end{pmatrix}.
\end{equation*}
Hence,
\begin{equation*}
(V_{\Lambda})_{K_{[\mathfrak a]}}= \left\{
\begin{pmatrix}
-\frac{c}{4} \mathrm{I}_4& \\
&c
 \end{pmatrix}
\mid
c\in {\mathbf C} \right\} =V^{\prime}_{(0,0,0,0,0)}.
\end{equation*}
But this $1$-dimensional fixed vector space corresponds to the larger
eigenvalue $20$.

Suppose that $\Lambda=(2,1,1,0,0)$.
Then $\dim V_{\Lambda}=45$.
By the branching law of $(U(5), U(4)\times U(1))$ that
$V_{\Lambda}$ can be decomposed into the following irreducible $K_2=U(4)\times U(1)$-submodules:
$$
V_{\Lambda}=V^{\prime}_{(2,1,1,0,0)}\oplus
V^{\prime}_{(1,1,1,0,1)}
\oplus V^{\prime}_{(2,1,0,0,1)}\oplus V^{\prime}_{(1,1,0,0,2)},
$$
where $\Lambda^{\prime}=(2,1,1,0,0) \in D(K_2,K_0)$.
By the branching law of $(U(4), U(2)\times U(2))$,
$\Lambda^{\prime\prime}=$
$(2,1,1,0)$,
$(2,0,1,1)$,
$(1,1,2,0)$,
$(1,1,1,1)$ or
$(1,0,2,1)$,
where $\Lambda^{\prime\prime}= (1,1,1,1)\in D(K_1,K_0)$.
Thus $-c_{\Lambda}=16$, $-c_{\Lambda^{\prime}}=12$,
$-c_{\Lambda^{\prime\prime}}=4$, $-c_L=-2c_{\Lambda}+c_{\Lambda^\prime}+\frac{1}{2}c_{\Lambda^{\prime\prime}}
=18$.

On the other hand,
since
$V^{\prime}_{(1,1,1,0,1)}\oplus V^{\prime}_{(2,1,0,0,1)}\oplus
V^{\prime}_{(1,1,0,0,2)}$ has no nonzero vectors fixed by $K_0$,
we see that $(V_{\Lambda})_{K_0}\subset V^{\prime}_{(2,1,1,0,0)}$.
Note that $\Lambda^{\prime}=2y_1+y_2+y_3=\sum_{i=1}^4 y_i+y_1-y_4
\in D(K_2,K_0)$ corresponds to
the tensor product of $C(U(4))$ representation with the highest weight $\sum_{i=1}^4 y_i$,
the adjoint representation of $SU(4)$ with the highest weight $y_1-y_4$
and the trivial representation of $U(1)$.
Then for each element $g_0\in K_0$ and each element
$u\otimes X\otimes v\in \mathbf{C}\otimes \mathfrak{su}(4)\otimes
\mathbf{C}\cong V_{\Lambda^{\prime}}$,
\begin{equation*}
\rho_{\Lambda^{\prime}}(g_0)(u\otimes X\otimes v)
=u\otimes \mathrm{Ad}\left(
                                        \begin{array}{cc}
                                          A &  \\
                                           & B \\
                                        \end{array}
                                      \right)(X)
                                      \otimes v.
\end{equation*}
Thus $(V_\Lambda)_{K_0}={\rm span}\{1\otimes \left(
                                               \begin{array}{cc}
                                                 I_2 &  \\
                                                  & -I_2\\
                                               \end{array}
                                             \right)\otimes 1
\}$. For the element $g\in K_{[\mathfrak a]}\subset K_2$,
\begin{equation*}
\rho_{\Lambda^{\prime}}(g)(u\otimes \left(
                                               \begin{array}{cc}
                                                 I_2 &  \\
                                                  & -I_2\\
                                               \end{array}
                                             \right)\otimes v )
=e^{\sqrt{-1}\pi} u \otimes \left(
                                               \begin{array}{cc}
                                                 -I_2 &  \\
                                                  & I_2\\
                                               \end{array}
                                             \right)\otimes v.
\end{equation*}
It follows that $(V_{\Lambda})_{K_{[\mathfrak a]}}=(V_{\Lambda})_{K_0}$,
i.e., $\Lambda=(2,1,1,0,0)\in D(K,K_{[\mathfrak a]})$ with multiplicity $1$.
Similarly, $\Lambda=(0,0,-1,-1,-2) \in D(K,K_{[\mathfrak a]})$ with multiplicity $1$
and it also gives the eigenvalue $18$.

Suppose that $\Lambda=(1,1,0,-1,-1)$.
Then $\dim V_{\Lambda}=75$.
By the branching law of $(U(4), U(2)\times U(2))$,
$V_{\Lambda}$ can be decomposed the following irreducible $K_1=U(4)\times U(1)$-submodules:
$$
V_{\Lambda}=
V^{\prime}_{(1,1,0,-1,-1)}\oplus
V^{\prime}_{(1,1,-1,-1,0)}\oplus
V^{\prime}_{(1,0,0,-1,0)}\oplus
V^{\prime}_{(1,0,-1,-1,1)},
$$
where $\Lambda^{\prime}_1=(1,1,-1,-1,0)$ and $\Lambda^{\prime}_2=(1,0,0,-1,0)\in D(K_2,K_0)$.
For $\Lambda^{\prime}_2$, by the branching law of $(U(4), U(2)\times U(2))$,
$\Lambda^{\prime\prime}_2=$
$(1,0,0,-1)$,
$(1,-1,0,0)$,
$(0,0,-1,-1)$,
$(0,0,0,0)$ or
$(0,-1,1,0)$,
where $\Lambda^{\prime\prime}_2=(0,0,0,0)\in D(K_1,K_0)$.
Therefore, $-c_{\Lambda}=16$, $-c_{\Lambda^{\prime}_2}=8$, $-c_{\Lambda^{\prime\prime}_2}=0$,
and $-c_L=-2c_{\Lambda}+c_{\Lambda^\prime_2}
+\frac{1}{2}c_{\Lambda^{\prime\prime}_2}
=24>18$.
For $\Lambda^{\prime}_1$,
by the branching law of $(U(4), U(2)\times U(2))$,
$\Lambda^{\prime\prime}=$
$(1,1,-1,-1)$,
$(1,0,0,-1)$,
$(1,-1,1,-1)$,
$(0,0,0,0)$,
$(0,-1,1,0)$ or
$(-1,-1,1,1)$,
where
$\Lambda^{\prime\prime}_{11}= (1,1,-1,-1)$,
$\Lambda^{\prime\prime}_{12}= (-1,-1,1,1)$,
$\Lambda^{\prime\prime}_{13}= (0,0,0,0)\in D(K_1,K_0)$.
Thus $-c_{\Lambda}=16$, $-c_{\Lambda^{\prime}}=12$,
$-c_{\Lambda^{\prime\prime}_{11}}=-c_{\Lambda^{\prime\prime}_{12}}=4$,
$-c_{\Lambda^{\prime\prime}_{13}}=0$,
$-c_L=-2c_{\Lambda}+c_{\Lambda^\prime}+\frac{1}{2}c_{\Lambda^{\prime\prime}}
=18, 18$ or $20$.
Moreover, from the above irreducible $K_2$-decomposition of $V_{\Lambda}$
and eigenvalue calculations,
we only need to determine
$\dim(V_{\Lambda})_{K_{[\mathfrak a]}}\cap
(V^{\prime\prime}_{11}\oplus V^{\prime\prime}_{12})$
since the fixed vectors in this subspace by $K_{[\mathfrak a]}$ give
the eigenvalue $18$.
Here we set  $V^{\prime\prime}_{11}:=V^{\prime\prime}_{\Lambda^{\prime\prime}_{11}}$ and
$V^{\prime\prime}_{12}:=V^{\prime\prime}_{\Lambda^{\prime\prime}_{12}}$.

Recall that the irreducible representation of $SU(4)$ with the highest weight
$\Lambda^{\prime}_{1}=y_1+y_2-y_3-y_4=2\omega_2$
can be described as follows (\cite{Fulton-Harris}):
\begin{equation*}
{\rm Sym}^2(\wedge^2 \mathbf{C}^4)=I(Gr_2(\mathbf{C}^4))_2\oplus
V^{\prime}_{\Lambda^{\prime}_1},
\end{equation*}
where
$I(Gr_2(\mathbf{C}^4))_2$,
the ideal of the Grassmannian $Gr_2(\mathbf{C}^4)$,
denotes the space of all homogeneous polynomials of degree $2$
on $\mathbf{P}(\wedge^2 {\mathbf{C}^{4}}^{*})$ that vanish on $Gr_2(\mathbf{C}^4)$.
Here $I(Gr_2(\mathbf{C}^4))_2\cong \wedge^4 \mathbf{C}^4\cong\mathbf{C}$
can be written down explicitly
in terms of a basis $\{\mathbf{e}_1, \mathbf{e}_2, \mathbf{e}_3,\mathbf{e}_4\}$
of $\mathbf{C}^4$:
\begin{equation*}
\begin{split}
I(Gr_2(\mathbf{C}^4))_2&= {\rm span}
\{ (\mathbf{e}_1 \wedge \mathbf{e}_2)\cdot (\mathbf{e}_3 \wedge \mathbf{e}_4)
+ (\mathbf{e}_1 \wedge \mathbf{e}_4)\cdot (\mathbf{e}_2 \wedge \mathbf{e}_3) \\
& {\ } -(\mathbf{e}_1 \wedge \mathbf{e}_3)\cdot (\mathbf{e}_2 \wedge \mathbf{e}_4)
\}.
\end{split}
\end{equation*}
Thus a basis for $V^{\prime}_{\Lambda^{\prime}_1}$ can be given explicitly.
For any element $g_0\in K_0$, denote $g_0^{\prime}=\begin{pmatrix}
A& \\
& B
\end{pmatrix}\in SU(2)\times SU(2)\subset U(4)$.
The representation of $K_0$ on any element $u\otimes X\otimes w
\in \mathbf{C}\otimes V^{\prime}_{\Lambda^{\prime}_1}\otimes \mathbf{C}$ is
\begin{equation*}
\rho_{\Lambda}(g)(u\otimes X \otimes w)
=\rho_0(1)(u)\otimes \rho_{\Lambda^{\prime}_{1}}(g_0^{\prime})(X)\otimes \rho_0(e^{\sqrt{-1}\theta})(w).
\end{equation*}
By direct computations, we obtain
\begin{equation*}
\begin{split}
(V_{\Lambda})_{K_0}\cap V^{\prime}_{\Lambda^{\prime}_1}
={\rm span}_{\mathbf C}\{&{\ } 1\otimes (\mathbf{e}_1\wedge \mathbf{e}_2)\cdot (\mathbf{e}_1\wedge \mathbf{e}_2)\otimes 1,
\\
&{\ }1\otimes (\mathbf{e}_3\wedge \mathbf{e}_4)\cdot (\mathbf{e}_3\wedge \mathbf{e}_4)\otimes 1,\\
&{\ }1\otimes (\mathbf{e}_1\wedge \mathbf{e}_2)\cdot (\mathbf{e}_3\wedge \mathbf{e}_4)\otimes 1
\},
\end{split}
\end{equation*}
where $(\mathbf{e}_1\wedge \mathbf{e}_2)\cdot (\mathbf{e}_1\wedge \mathbf{e}_2)\in V^{\prime\prime}_{11}$,
$(\mathbf{e}_3\wedge \mathbf{e}_4)\cdot (\mathbf{e}_3\wedge \mathbf{e}_4) \in V^{\prime\prime}_{12}$
and $(\mathbf{e}_1\wedge \mathbf{e}_2)\cdot (\mathbf{e}_3\wedge \mathbf{e}_4)\in V^{\prime\prime}_{13}$.
For the generator $g\in K_{[\mathfrak a]}\subset K_2$, denote $g^{\prime}=\left(
                                                                 \begin{array}{cccc}
                                                                    &  & 1 &  \\
                                                                    &  &  & 1 \\
                                                                   1 &  &  &  \\
                                                                    & -1 &  &  \\
                                                                 \end{array}
                                                               \right)
$. The representation of $g$ on $u\otimes X\otimes w$ is
\begin{equation*}
\rho_{\Lambda}(g)(u\otimes X \otimes w)
=\rho_0(e^{\frac{\sqrt{-1}}{4}\pi}I_4)(u)
\otimes \rho_{\Lambda^{\prime}_{1}}(e^{-\frac{\sqrt{-1}}{4}\pi}g^{\prime})(X)
\otimes \rho_0(1)(w).
\end{equation*}
It follows that
\begin{equation*}
\begin{split}
&(V_{\Lambda})_{K_{[\mathfrak a]}}\cap V^{\prime}_{\Lambda^{\prime}_1}
={\rm span}_{\mathbf C}\{1\otimes (\mathbf{e}_1\wedge \mathbf{e}_2)\cdot
(\mathbf{e}_3\wedge \mathbf{e}_4)\otimes 1,\\
& 1\otimes (\mathbf{e}_1\wedge \mathbf{e}_2)\cdot (\mathbf{e}_1\wedge
\mathbf{e}_2)\otimes 1
-1\otimes (\mathbf{e}_3\wedge \mathbf{e}_4)\cdot (\mathbf{e}_3\wedge
\mathbf{e}_4)\otimes 1
\}.
\end{split}
\end{equation*}
In particular, $\Lambda=(1,1,0,-1,-1)\in D(K,K_{[\mathfrak a]})$ and
\begin{equation*}
\begin{split}
&(V_{\Lambda})_{K_{[\mathfrak a]}}\cap
(V^{\prime\prime}_{11}\oplus V^{\prime\prime}_{12})\\
&{\ }={\rm span}_{\mathbf C}\{
 1\otimes (\mathbf{e}_1\wedge \mathbf{e}_2)\cdot
(\mathbf{e}_1\wedge \mathbf{e}_2)\otimes 1
-1\otimes (\mathbf{e}_3\wedge \mathbf{e}_4)\cdot
(\mathbf{e}_3\wedge \mathbf{e}_4)\otimes 1
\}
\end{split}
\end{equation*}
with dimension $1$, which corresponds to the eigenvalue $18$.

Now we obtain that the Gauss image $L^{18}$ is Hamiltonian stable.
Moreover,
\begin{equation*}
\begin{split}
n(L^{18})
=&\dim V_{(0,0,-1,-1,-2)} +\dim V_{(2,1,1,0,0)}+\dim V_{(1,1,0,-1,-1)}\\
=&45+45+75=165
=\dim SO(20)-\dim U(5)=n_{hk}(L^{18}).
\end{split}
\end{equation*}
Hence the Gauss image $L^{18}$ is Hamiltonian rigid.

Therefore, we conclude that the Gauss image $L^{18}$ is Hamiltonian stable.
\begin{thm}
The Gauss image
$L^{18}
={\mathcal G}\left(\frac{U(5)}{(SU(2)\times SU(2)\times U(1))}\right)=\\
\frac{U(5)}{(SU(2)\times SU(2)\times U(1))\cdot \mathbf{Z}_4}
\subset Q_{18}(\mathbf C)$
is strictly Hamiltonian stable.
\end{thm}

\section{The case $(U,K)=(SO(m+2), SO(2)\times SO(m))$ $(m\geq 3)$}
\label{Sec_BDI_2}

In this case $(U,K)$ is of type $B_2$.
The canonical decomposition ${\mathfrak u}={\mathfrak k}+{\mathfrak p}$
of ${\mathfrak u}=\mathfrak{o}(m+2)$
and a maximal abelian subspace  ${\mathfrak a}$ of ${\mathfrak p}$
are given as
\begin{equation*}
\begin{split}
\mathfrak{k}&=\Bigl\{
\begin{pmatrix}
T_{1}&0\\
0&T_{2}
\end{pmatrix}
\mid
T_{1}\in{{\mathfrak o}(2)}, T_{2}\in{{\mathfrak o}(m)}\Bigr\}
={\mathfrak o}(2)+{\mathfrak o}(m),\\
{\mathfrak p}&=
\Bigl\{
\begin{pmatrix}
0&-^{t}X\\
X&0
\end{pmatrix}
\mid
X\in{M(m,2;{\mathbf R})}\Bigr\}, \\
{\mathfrak a}&=
\Bigl\{
H=H(\xi_{1},\xi_{2})=
\begin{pmatrix}
0&-^{t}\xi&0\\
\xi&0&0\\
0&0&0
\end{pmatrix}
\mid
\xi=
\begin{pmatrix}
\xi_{1}&0\\
0&\xi_{2}
\end{pmatrix},
\xi_{1},\xi_{2}\in{\mathbf R}\Bigr\}.
\end{split}
\end{equation*}
Then
\begin{equation*}
\begin{split}
K_0=
&
\,
\left\{
\begin{pmatrix}
\pm\mathrm{I}_4&0\\
0& T
\end{pmatrix}
\mid T\in SO(m-2)
\right\}
\\
\cong &
\,
\mathbf{Z}_2 \times SO(m-2).
\end{split}
\end{equation*}
Moreover
\begin{equation*}
K_{[{\mathfrak a}]}
\cong({\mathbf Z}_2\times SO(m-2))\cdot{\mathbf Z}_{4}
\end{equation*}
consists of all elements
\begin{equation*}
a=
\begin{pmatrix}
A&0&0\\
0&B&0\\
0&0&B^{\prime}
\end{pmatrix} \in K=SO(2)\times SO(m),
\end{equation*}
where
\begin{equation*}
\begin{split}
(A,B)=
&
\left(
\begin{pmatrix}
1&0\\
0&1
\end{pmatrix},
\begin{pmatrix}
1&0\\
0&1
\end{pmatrix}
\right),
\left(
\begin{pmatrix}
-1&0\\
0&-1
\end{pmatrix},
\begin{pmatrix}
1&0\\
0&1
\end{pmatrix}
\right),\\
&
\left(
\begin{pmatrix}
1&0\\
0&1
\end{pmatrix},
\begin{pmatrix}
-1&0\\
0&-1
\end{pmatrix}
\right),
\left(
\begin{pmatrix}
-1&0\\
0&-1
\end{pmatrix},
\begin{pmatrix}
-1&0\\
0&-1
\end{pmatrix}
\right),\\
&
\left(
\begin{pmatrix}
0&-1\\
1&0
\end{pmatrix},
\begin{pmatrix}
0&1\\
1&0
\end{pmatrix}
\right),
\left(
\begin{pmatrix}
0&1\\
-1&0
\end{pmatrix},
\begin{pmatrix}
0&1\\
1&0
\end{pmatrix}
\right),\\
&
\left(
\begin{pmatrix}
0&-1\\
1&0
\end{pmatrix},
\begin{pmatrix}
0&-1\\
-1&0
\end{pmatrix}
\right),
\left(
\begin{pmatrix}
0&1\\
-1&0
\end{pmatrix},
\begin{pmatrix}
0&-1\\
-1&0
\end{pmatrix}
\right).
\end{split}
\end{equation*}
Here note that
$K_{[{\mathfrak a}]}\not\subset K_{1}=SO(2)\times SO(2)\times SO(m-2)$.
Thus the deck transformation group of the covering map
${\mathcal G}:
N^{2m-2} \rightarrow {\mathcal G}(N^{2m-2})$ is equal to
$K_{[{\mathfrak a}]}/K_{0}\cong{\mathbf Z}_{4}$.

\subsection{Description of the Casimir operator}

Denote $\langle X, Y\rangle_{\mathfrak u}:= -\frac{1}{2} {\rm tr}XY$ for each  $X,Y\in{\mathfrak u}=\mathfrak{o}(m+2)$.
The restricted root system $\Sigma(U,K)$ of type $B_2$,
can be given as follows (\cite{Bourbaki}):
\begin{equation*}
\Sigma^{+}(U,K)=\{
\varepsilon_1-\epsilon_2=\alpha_1,\,
\varepsilon_2=\alpha_2,\,
\varepsilon_1+\epsilon_2=\alpha_1+2\alpha_2,\,
\varepsilon_1=\alpha_1+\alpha_2\}.
\end{equation*}
Then, relative to the above inner product $\langle\,,\,\rangle_{\mathfrak u}$,
the square length of any restrict root $\gamma\in\Sigma(U,K)$ is
$\Vert\gamma\Vert_{\mathfrak u}^{2}=1$ or $2$.
Hence the Casimir operator $\mathcal{C}_L$ of $L$ with respect to
the induced metric from $Q_{2m-2}({\mathbf C})$ is given as follows:
\begin{equation}\label{CasimirOperatorBDI_2}
\begin{split}
{\mathcal C}_L
&=\frac{2}{\Vert \gamma_0\Vert^2_{\mathfrak u}}
\mathcal{C}_{K/K_0, \langle\,,\,\rangle_{\mathfrak{u}}}
-\frac{1}{\Vert \gamma_0\Vert^2_{\mathfrak u}}
\mathcal{C}_{K_1/K_0, \langle\,,\,\rangle_{\mathfrak{u}}}\\
&={\ }{\mathcal C}_{K/K_{0}}
-\frac{1}{2}{\ }{\mathcal C}_{K_{1}/K_{0}},
\end{split}
\end{equation}
where
$K=SO(2)\times SO(m)\supset K_1=SO(2)\times SO(2)\times SO(m-2)\supset K_0=\mathbf{Z}_2\times SO(m-2)$
and
${\mathcal C}_{K/K_{0}}$,
${\mathcal C}_{K_{1}/K_{0}}$
denote the Casimir operators of  $K/K_{0}$ and $K_{1}/K_{0}$ relative to
$\langle{{\ },{\ }}\rangle_{\mathfrak u}\vert_{\mathfrak k}$
and
$\langle{{\ },{\ }}\rangle_{\mathfrak u}\vert_{{\mathfrak k}_{1}}$,
respectively.

\subsection{Branching laws for $(SO(n+2), SO(2)\times SO(n))$}
We need the branching laws for $(SO(n+2), SO(2)\times SO(n))$
by Tsukamoto (\cite{Tsukamoto}).

\begin{lem}[Branching laws for $(SO(2p+2), SO(2)\times SO(2p)), p\geq 1$]
\label{BranchingLawSO(2p+2)SO(2)XSO(2p)}
Let $\Lambda=h_0\varepsilon_0+h_1\varepsilon_1+\cdots+h_{p-1}\varepsilon_{p-1}
+\epsilon h_{p}\varepsilon_p \in D(SO(2p+2))$,
where $\epsilon=1$ or $-1$ and $h_0,h_1,\cdots,h_p$ are integers satisfying
\begin{equation}\label{D(SO(2p+2))}
h_0\geq h_1\geq \cdots \geq h_p\geq 0
\end{equation}
and
$\Lambda^{\prime}
=k_0\varepsilon_0+k_1\varepsilon_1+\cdots+k_{p-1}\varepsilon_{p-1}
+\epsilon^{\prime} k_{p}\varepsilon_p \in{D(SO(2)\times SO(2p))}
$,
where $\epsilon^{\prime}=1$ or $-1$ and $k_0,k_1,\cdots,k_p$ are integers satisfying
\begin{equation}\label{D(SO(2)xSO(2p))}
k_1\geq \cdots \geq k_p\geq 0.
\end{equation}
The irreducible decomposition of $V_{\Lambda}$ as a
$SO(2)\times SO(2p)$-module contains an irreducible
$SO(2)\times SO(2p)$-module $V^{\prime}_{\Lambda^{\prime}}$ if and only if
\begin{equation*}
\begin{split}
& h_{i-1}\geq k_i \geq h_{i+1} \quad (1\leq i \leq p-1),\\
& h_{p-1}\geq k_p\geq 0,
\end{split}
\end{equation*}
and the coefficient of $X^{k_0}$ in the finite power series
\begin{equation*}
X^{\epsilon\epsilon^{\prime}l_p}\prod_{i=0}^{p-1}\frac{X^{l_i+1}-X^{-l_i-1}}{X-X^{-1}}
\end{equation*}
does not vanish,
where
\begin{equation}\label{eq:l}
\begin{split}
l_0&:=h_0-\max\{h_1, k_1\},\\
l_i&:=\min\{h_i, k_i\}-\max\{h_{i+1}, k_{i+1}\} \quad (1\leq i\leq p-1),\\
l_p&:=\min\{h_p, k_p\}.
\end{split}
\end{equation}
Moreover, the coefficient of $X^{k_0}$ is equal to
the multiplicity of $V^{\prime}_{\Lambda^{\prime}}$
appearing in the irreducible decomposition.
\end{lem}

\begin{lem}[Branching laws for $(SO(2p+3), SO(2)\times SO(2p+1)), p\geq 1$]
\label{BranchingLawSO(2p+3)SO(2)XSO(2p+1)}
Let
$\Lambda=h_0\varepsilon_0+h_1\varepsilon_1+\cdots+h_{p-1}\varepsilon_{p-1}
+ h_{p}\varepsilon_p \in D(SO(2p+3))$,
where $h_0,h_1,\cdots,h_p$ are integers satisfying \eqref{D(SO(2p+2))}
and
$\Lambda^{\prime}
=k_0\varepsilon_0+k_1\varepsilon_1+\cdots+k_{p-1}\varepsilon_{p-1}
+ k_{p}\varepsilon_p \in{D(SO(2)\times SO(2p+1))}
$,
where
$k_0,k_1,\cdots,k_p$ are integers satisfying \eqref{D(SO(2)xSO(2p))}.
The irreducible decomposition of $V_{\Lambda}$ as a
$SO(2)\times SO(2p+1)$-module contains an irreducible
$SO(2)\times SO(2p+1)$-module $V^{\prime}_{\Lambda^{\prime}}$ if and only if
\begin{equation*}
\begin{split}
& h_{i-1}\geq k_i \geq h_{i+1}, \quad (1\leq i \leq p-1)\\
& h_{p-1}\geq k_p\geq 0,
\end{split}
\end{equation*}
and the coefficient of $X^{k_0}$ in the finite power series
\begin{equation*}
\left(\prod_{i=0}^{p-1}\frac{X^{l_i+1}-X^{- {l_i}-1}}{X-X^{-1}}\right)
\frac{X^{l_p+\frac{1}{2}}-X^{-l_p-\frac{1}{2}}}{X^{\frac{1}{2}}-X^{-\frac{1}{2}}}
\end{equation*}
does not vanish, where
integers $l_0, l_1, \cdots, l_p$ are defined by \eqref{eq:l}.
Moreover, the coefficient of $X^{k_0}$ is equal to the multiplicity of $V^{\prime}_{\Lambda^{\prime}}$
appearing in the irreducible decomposition.
\end{lem}

\subsection{Description of $D(K, K_0)$ and eigenvalue computations}
\noindent

For $m=2p {\ } (p\geq 2)$ or $m=2p+1 {\ }(p\geq 1)$,
each $\tilde{\Lambda}\in D(K)=D(SO(2)\times SO(m))$ can be expressed as
\begin{equation*}
\tilde{\Lambda}= k_0\varepsilon_0+k_1\varepsilon_1+\cdots+k_p\varepsilon_p,
\end{equation*}
where $k_0\varepsilon_0\in D(SO(2))$,
$\Lambda:=k_1\varepsilon_1+\cdots+k_p\varepsilon_p\in D(SO(m))$
and $k_0,k_1,\cdots,k_p\in{\mathbf Z}$ satisfying
\begin{equation*}
\begin{split}
&k_{1}\geq k_{2}\geq\cdots\geq k_{p-1}\geq\vert{k_p}\vert
\quad\text{ if }m=2p, \\
&k_{1}\geq k_{2}\geq\cdots\geq k_{p-1}\geq k_p \geq 0
\quad\text{ if }m=2p+1.
\end{split}
\end{equation*}
Then we have
\begin{equation*}
\tilde{V}_{\tilde{\Lambda}}=
U_{k_0\varepsilon_0}\otimes V_{\Lambda}.
\end{equation*}

Note that
\begin{equation*}
\begin{split}
D(K,K_{0})&=
D(SO(2)\times SO(m),{\mathbf Z}_2\times SO(m-2))\\
&\subset
D(SO(2)\times SO(m),SO(m-2)),\\
D(K_{1},K_{0})&=
D(SO(2)\times SO(2)\times SO(m-2),{\mathbf Z}_2\times SO(m-2))\\
&\subset D(SO(2)\times SO(2)\times SO(m-2), SO(m-2)).
\end{split}
\end{equation*}

By applying Lemmas \ref{BranchingLawSO(2p+2)SO(2)XSO(2p)}
and \ref{BranchingLawSO(2p+3)SO(2)XSO(2p+1)}
to both cases $(SO(2p), SO(2)\times SO(2p-2))$
and $(SO(2p), SO(2)\times SO(2p-1))$,
we can describe $D(K,K_0)$ as follows:
\begin{lem}\label{D(SO(2)XSO(m),Z_2XSO(m-2))}
Assume that $p\geq 2$.
Let $\tilde{\Lambda}\in D(K)$.
Then
an irreducible $K$-module $\tilde{V}_{\tilde{\Lambda}}$ with the highest weight
$\tilde{\Lambda}$ contains an irreducible $K_1$-module
$\tilde{V}^{\prime}_{\tilde{\Lambda}^{\prime}}$ with the highest weight
$\tilde{\Lambda}^{\prime}\in D(K_1)$ satisfying
$(\tilde{V}^{\prime}_{\tilde{\Lambda}^{\prime}})_{K_0}\not=\{0\}$
if and only if
\begin{equation*}
\begin{split}
&\tilde{\Lambda}=k_0\varepsilon_0+k_1\varepsilon_1+k_2\varepsilon_2
\in D(K), \\
&\tilde{\Lambda}^{\prime}=k_0\varepsilon_0+k^{\prime}_1\varepsilon_1
\in D(K_1),
\end{split}
\end{equation*}
where $k_0, k_1, k_2, k^{\prime}_1\in{\mathbf Z}$, $k_1\geq k_2\geq 0$
satisfy the following conditions:
\begin{enumerate}
\item[(i)]
The coefficient of $X^{k^{\prime}_1}$ in the finite Laurent series expansion
$\displaystyle\frac{X^{k_1-k_2+1}-X^{-(k_1-k_2+1)}}{X-X^{-1}}$
of $X$ does not vanish.
\item[(ii)]
$k_0+k^{\prime}_1$ is even.
\end{enumerate}
In particular $-(k_1-k_2)\leq k^{\prime}_1 \leq (k_1-k_2)$.
Here the coefficient is equal to the multiplicity
of $\tilde{V}^{\prime}_{\tilde{\Lambda}^{\prime}}$.
\end{lem}

\subsubsection{The case $m=2p{\ }(p\geq{2})$ }
\noindent

Suppose that $m=2p{\ }(p\geq{2})$.
For each
$$\tilde{\Lambda}=k_0\varepsilon_0+k_1\varepsilon_1+k_2\varepsilon_2
\in{D(K,K_0)=D(SO(2)\times SO(2p), \mathbf{Z}_2\times SO(2p-2))}$$
with
$\tilde{\Lambda}^{\prime}=k_0\varepsilon_0+k_1^{\prime}\varepsilon_1
\in D(K_1, K_0)
=D(SO(2)\times SO(2)\times SO(2p-2),{\mathbf Z}_2\times SO(2p-2))$
as in Lemma \ref{D(SO(2)XSO(m),Z_2XSO(m-2))},
$-{\mathcal C}_{K/K_0}$ and $-{\mathcal C}_{K_1/K_0}$
have eigenvalues
\begin{equation*}
\begin{split}
&-c_{\tilde{\Lambda}}
=k_0^2+k_1^2+k_2^2+2(p-1)k_1+2(p-2)k_2,
\\
&-c_{\tilde{\Lambda}^{\prime}}
=\frac{1}{2}(k_0^2+ {k^{\prime}}_{1}^{2}).
\end{split}
\end{equation*}
Hence by the formula \eqref{CasimirOperatorBDI_2}
the corresponding eigenvalue of $-{\mathcal C}_{L}$ is
\begin{equation}\label{EigenvalueFormulaBDI_2 even}
\begin{split}
-c_{L}
&=\ -c_{\tilde{\Lambda}}+\frac{1}{2}c_{\tilde{\Lambda}^{\prime}}
\\
&=\ k_0^2+k_1^2+k_2^2+2(p-1)k_1+2(p-2)k_2-\frac{1}{2}(k_0^2+ {k^{\prime}}_{1}^{2}).
\end{split}
\end{equation}

Denote
$\tilde{\Lambda}=k_0 \varepsilon_0+k_1 \varepsilon_1 +k_2 \varepsilon_2
\in
D(K,K_0)$
by $\tilde{\Lambda}=(k_0, k_1, k_2)$.

For each
$\tilde{\Lambda}=k_0 \varepsilon_0=(k_0,0,0)\in D(K,K_0)$,
as
$k^\prime_1=0$, $k_0=k_0+k^\prime_1$ is even and
$-c_{L}=\frac{1}{2} k_0^2$,
we see that
\begin{equation}\label{-c_{L}leq 4p-2}
-c_{L}\leq 2m-2=4p-2 \text{ if and only if } k_0^2\leq 4(2p-1).
\end{equation}
As
$\tilde{V}_{\tilde\Lambda}\cong
U_{k_0\varepsilon_0}\otimes \mathbf{C} \cong U_{k_0\varepsilon_0}$,
for a generator $g=\begin{pmatrix}
  0 & 1& &  &  \\
  -1& 0 & & &  \\
   & &0& 1 & \\
   & &1&0& \\
   & & & & T^{\prime}
\end{pmatrix}\in K_{[\mathfrak a]}$ which will be used throughout this section,
we have
\begin{equation*}
\rho_{k_0\varepsilon_0}
(g)
(v\otimes 1)
=
e^{\sqrt{-1}\frac{\pi}{2} k_0}\, (v\otimes 1).
\end{equation*}
Hence
\begin{equation}\label{(k_0, 0, 0)}
(k_0, 0, 0)\in D(K,K_{[\mathfrak a]})
\text{ if and only if } k_0\in 4{\mathbf Z}.
\end{equation}

(i) The case
$\mathcal{G}(N^6)\cong
\frac{SO(2)\times SO(4)}{(\mathbf{Z}_2\times SO(2))\cdot \mathbf{Z}_4}
\rightarrow Q_6(\mathbf{C})$ with $p=2$.

Since
$-c_L=-\frac{1}{2}c_{K/K_0}-\frac{1}{2}c_{K/K_1}\geq -\frac{1}{2}c_{K/K_0}$,
note that
$-c_L\geq 6$ implies $-c_{\tilde{\lambda}}=-c_{K/K_0}\leq 12$.
Using the eigenvalue formula \eqref{EigenvalueFormulaBDI_2 even}
we compute that
\begin{lem}
$\tilde{\Lambda}
=k_0\varepsilon_0+k_1\varepsilon_1+k_2\varepsilon_2\in D(K, K_0)$
has eigenvalue $-c_L \leq{6}$ if and only if
$(k_0, k_1, k_2)$ is one of
\begin{equation*}
\{ 0, (\pm 2,0,0),
(\pm 1 ,1,0), (0,1,1), (\pm 2,1,1), (0,2,0),
(0,1,-1), (\pm 2, 1,-1) \}.
\end{equation*}
\end{lem}

Suppose that $\tilde{\Lambda}=(\pm 2,0,0)$. Then
by \eqref{(k_0, 0, 0)}
$\tilde{\Lambda}=(\pm 2,0,0)\not\in D(K,K_{[\mathfrak a]})$.

Suppose that $\tilde{\Lambda}=(\pm 1,1,0)$.
Then $\dim\tilde{V}_{\tilde{\Lambda}}=4$ and
$\tilde{V}_{\tilde\Lambda}\cong U_{k_0\varepsilon_0}\otimes \mathbf{C}^4$,
where $\Lambda=\varepsilon_1\in D(K)$ corresponds to the matrix multiplication of $SO(4)$ on $\mathbf{C}^4$.
It follows from
the branching law
(Lemma \ref{BranchingLawSO(2p+2)SO(2)XSO(2p)}, p=2)
of $(SO(4), SO(2)\times SO(2))$
that $k_1^\prime=\pm 1$.
Hence $-c_L=\frac{1}{2}k_0^2+\frac{5}{2}$.
Note that $U_{k_0\varepsilon_0}\otimes \mathbf{C}^4$ can be decomposed
into irreducible $SO(2)\times SO(2)\times SO(2)$-modules as
\begin{equation*}
U_{k_0\varepsilon_0}\otimes \mathbf{C}^4
=(U_{k_0\varepsilon_0}\otimes (\mathbf{C}^2\oplus\{0\}))
\oplus (U_{k_0\varepsilon_0}\otimes (\{0\}\oplus \mathbf{C}^2)).
\end{equation*}
There is no nonzero fixed vector by $\mathbf{Z}_2\times SO(2)$ in
$U_{k_0 \varepsilon_0} \otimes (\{0\}\oplus \mathbf{C}^2)$.
Moreover, since
\begin{equation*}
\begin{split}
&
\rho_{k_0\varepsilon_0+\varepsilon_1}
\begin{pmatrix}
  -I_2 &  &  \\
   & -I_2 &  \\
   &  & T \\
\end{pmatrix}
(v\otimes \begin{pmatrix}
            w_1 \\
            w_2 \\
            0 \\
            0 \\
          \end{pmatrix}
)
\\
=
& \
e^{\sqrt{-1}\pi k_0}  v\otimes \begin{pmatrix}
            -w_1 \\
            -w_2 \\
            0 \\
            0 \\
          \end{pmatrix}
=
e^{\sqrt{-1}\pi (k_0+1)} v\otimes \begin{pmatrix}
            w_1 \\
            w_2 \\
            0 \\
            0 \\
          \end{pmatrix},
\end{split}
\end{equation*}
it follows that
$(\tilde{V}_{\tilde{\Lambda}})_{K_0}
=(\tilde{V}_{\tilde{\Lambda}})_{\mathbf{Z}_2\times SO(2)}
\not=\{0\}$
if and only if $k_0$ is odd, and then
$(\tilde{V}_{\tilde{\Lambda}})_{\mathbf{Z}_2\times SO(2)}
=U_{k_0\varepsilon_0}\otimes (\mathbf{C}^2\oplus\{0\})$.
Let $k_{0}$ be odd.
However since
\begin{equation*}
\rho_{k_0\varepsilon_0+\varepsilon_1}
(g)
(v\otimes \begin{pmatrix}
            w_1 \\
            w_2 \\
            0 \\
            0 \\
           \end{pmatrix}
)
=
e^{\sqrt{-1}\frac{\pi}{2} k_0}  v\otimes \begin{pmatrix}
            w_2 \\
            w_1 \\
            0 \\
            0 \\
        \end{pmatrix},
\end{equation*}
$U_{k_0\varepsilon_0}\otimes (\mathbf{C}^2\oplus\{0\})$
has no nonzero fixed vector by
$(\mathbf{Z}_2 \times SO(2))\cdot \mathbf{Z}_4$, and hence
$(k_0, 1, 0)\not\in D(K,K_{[\mathfrak a]})$.
In particular $(\pm 1, 1, 0)\not\in D(K,K_{[\mathfrak a]})$.

Suppose that $\tilde{\Lambda}_1=(k_0,1,1)$ and $\tilde{\Lambda}_2=(k_0,1,-1)$.
Then $\dim\tilde{V}_{\tilde{\Lambda}_1}=\dim\tilde{V}_{\tilde{\Lambda}_2}=3$
and $\tilde{V}_{\tilde{\Lambda}_1}\oplus \tilde{V}_{\tilde{\Lambda}_2}
\cong \mathbf{C}\otimes \wedge^2\mathbf{C}^4$.
It follows from
the branching law
(Lemma \ref{BranchingLawSO(2p+2)SO(2)XSO(2p)}, p=2)
$(SO(4), SO(2)\times SO(2))$
that
\begin{equation*}
\tilde{V}_{\tilde{\Lambda}_1}=\tilde{V}^{\prime}_{(k_0,1,1)}
\oplus \tilde{V}^{\prime}_{(k_0,-1,-1)}
\oplus \tilde{V}^{\prime}_{(k_0,0,0)},
\end{equation*}
where $(k_0,0,0)\in D(K_1,K_0)$.
Thus $-c_L=\frac{1}{2} k_0^2+4$, which
is equals to $4$ when $k_0=0$ and $6$ when $k_0=\pm 2$.

Let $\{e_1,e_2,e_3,e_4\}$ be the standard basis of $\mathbf{C}^4$.
Then we have
\begin{equation*}
\begin{split}
\tilde{V}_{{\tilde \Lambda}_1}&={\rm span}\{ e_1\wedge e_2, e_1\wedge e_3-e_2\wedge e_4,
e_1\wedge e_4 + e_2\wedge e_3\},\\
\tilde{V}_{{\tilde \Lambda}_2}&={\rm span}\{ e_3\wedge e_4, e_1\wedge e_3+e_2\wedge e_4,
e_1\wedge e_4 - e_2\wedge e_3\}.
\end{split}
\end{equation*}
Since $e_1\wedge e_2 \in \wedge^2 \mathbf{C}^4$ is fixed by the representation of $SO(2)\times SO(2)$
with respect to the highest weight ${\tilde \Lambda}_1$,
$$({\tilde{V}}_{{\tilde\Lambda}_1})_{K_0}={\rm span}\{ 1\otimes (e_1\wedge e_2)\}.$$
Moreover,
\begin{equation*}
\rho_{{\tilde\Lambda}_1}
(g)
(v\otimes (e_1\wedge e_2))\\
=
e^{\sqrt{-1}\frac{\pi}{2}k_0} v\otimes(e_2\wedge e_1).
\end{equation*}
Hence, ${\tilde\Lambda}_1=(0,1,1)\not\in D(K,K_{[\mathfrak a]})$
but ${\tilde\Lambda}_1=(\pm 2,1,1)\in D(K,K_{[\mathfrak a]})$
and $(\tilde{V}_{{\tilde\Lambda}_1})_{K_{[\mathfrak a]}}\cong \mathbf{C}\otimes \mathbf{C} \{e_1\wedge e_2\}$
for $k_0=2$ or $-2$, both of which give eigenvalue $6$.
Similarly,
${\tilde\Lambda}_2=(0,1,-1)\not\in D(K,K_{[\mathfrak a]})$
but ${\tilde\Lambda}_2=(\pm 2,1,-1)\in D(K,K_{[\mathfrak a]})$
and $(\tilde{V}_{{\tilde\Lambda}_2})_{K_{[\mathfrak a]}}\cong \mathbf{C}\otimes \mathbf{C} \{e_3\wedge e_4\}$
for $k_0=2$ or $-2$, both of which give eigenvalue $6$.

Suppose that $\tilde{\Lambda}=(0,2,0)$. Then $\dim \tilde{V}_{\tilde\Lambda}=9$
and $\tilde{V}_{\tilde\Lambda}\cong \mathbf{C}\otimes {\rm S}^2_0(\mathbf{C}^4)$,
where the corresponding representation of $SO(4)$ is just the adjoint representation
on ${\rm S}^2_0 (\mathbf{C}^4)$.
It follows from the branching law of $(SO(4), SO(2)\times SO(2))$ that $k_1^\prime=0, \pm 2$.
Thus $-c_L=8-\frac{1}{2}{k_1^\prime}^2$.
When $k_1^\prime=\pm 2$, $-c_L=6$, otherwise $-c_L =8> 6$.
On the other hand,
${\rm S}^2_0 (\mathbf{C}^4)$ can be decomposed into the following $SO(2)\times SO(2)$-modules:
\begin{equation*}
V_{2\varepsilon_1}
\cong  {\rm S}^2_0 (\mathbf{C}^4)
={\rm S}^2_0 (\mathbf{C}^2)\oplus {\rm S}^2_0(\mathbf{C}^2)\oplus
M(2,2;\mathbf{C})
\oplus
\mathbf{C}\begin{pmatrix}
I_{2}&\\
&- I_2
\end{pmatrix}.
\end{equation*}
Thus, ${\rm S}^2_0(\mathbf{C}^2) \oplus
\mathbf{C}\begin{pmatrix}
I_{2}&\\
&- I_2
\end{pmatrix}$
is fixed by $\{-I_2\}\times SO(2)$ and
$\dim(\tilde{V}_{\tilde\Lambda})_{K_0}=3$.
Moreover,
\begin{equation*}
\begin{split}
&\rho_{\tilde\Lambda}
(g)
(v \otimes
\begin{pmatrix}
a& b& \\
b&-a&\\
 & & 0
\end{pmatrix}
)
=v\otimes \begin{pmatrix}
-a& b& \\
b&a&\\
& & 0
\end{pmatrix},
\\
&\rho_{\tilde\Lambda}
(g)
(v \otimes
\begin{pmatrix}
I_2& \\
& -I_2
\end{pmatrix}
)
=v\otimes \begin{pmatrix}
I_2& \\
&-I_2
\end{pmatrix}.
\end{split}
\end{equation*}
Hence,
$$(\tilde{V}_{\tilde\Lambda})_{K_{[\mathfrak a]}}=\mathbf{C}\otimes
\mathbf{C} \begin{pmatrix}
0& 1& \\
1&0&\\
& & 0
\end{pmatrix}
\oplus
\mathbf{C}\otimes
\mathbf{C} \begin{pmatrix}
I_2& \\
&-I_2
\end{pmatrix}.$$
Notice that the first summand lies in the $SO(2)\times SO(2)\times SO(2)$-module
$V^{\prime}_{2\varepsilon_1}\oplus V^{\prime}_{-2\varepsilon_1}$, which gives eigenvalue $6$
and the second summand lies in the $SO(2)\times SO(2)\times SO(2)$-module with respect to
weight $(0,0,0)\in D(K_1,K_0)$, which gives eigenvalue $8>6$.
Therefore, $\tilde{\Lambda}=(0,2,0)\in D(K,K_{[\mathfrak a]})$ and
the multiplicity corresponding to eigenvalue $6$ is $1$.

Now we know that $\mathcal{G}(N^6)\subset Q_6(\mathbf{C})$ is
Hamiltonian stable.
Since $\tilde\Lambda=(2,1,1)$, $(-2,1,1)$, $(2,1,-1)$, $(-2,1,-1)$, $(0,2,0)\in D(K,K_{[\mathfrak a]})$
give the smallest eigenvalue $6$ with multiplicity $1$  and
\begin{equation*}
\begin{split}
&n(L^{6})\\
=
&\dim \tilde{V}_{(2,1,1)}+\dim\tilde{V}_{(-2,1,1)}+ \dim\tilde{V}_{(2,1,-1)}
+\dim \tilde{V}_{(-2,1,-1)}+\dim \tilde{V}_{(0,2,0)} \\
=&3+3+3+3+9=21
=\dim SO(8)-\dim (SO(2)\times SO(4))=n_{hk}(L^{6}).
\end{split}
\end{equation*}
Hence we obtain
that
$\mathcal{G}(N^6)\subset Q_6(\mathbf{C})$ is strictly Hamiltonian stable.

(ii) The case
$\mathcal{G}(N^{4p-2})\cong
\frac{SO(2)\times SO(2p)}{(\mathbf{Z}_2\times SO(2p-2))
\cdot \mathbf{Z}_4}
\rightarrow Q_{4p-2}(\mathbf{C})$ with $p\geq 3$.

Suppose that
$\tilde{\Lambda}=(k_0,0,0)
$ and $k_0\in 4{\mathbf Z}\setminus\{0\}$.
Then $k_1^\prime=0$ and
by \eqref{-c_{L}leq 4p-2}
$\tilde{\Lambda}\in D(K,K_{[\mathfrak{a}]})$.
As $p\geq 3$, we have $16< 20\leq 4(2p-1)$.
Hence by \eqref{(k_0, 0, 0)}
we see that
for every $k_{0}\in 4{\mathbf Z}\setminus\{0\}$ such that
$16\leq k_{0}^2 < 4(2p-1)$
we have eigenvalue $-c_L=\frac{1}{2} k_{0}^2 < 4p-2$.
Therefore,
$\mathcal{G}(N^{4p-2})\cong
\frac{SO(2)\times SO(2p)}{(\mathbf{Z}_2\times SO(2p-2))\cdot \mathbf{Z}_4}
\rightarrow Q_{4p-2}(\mathbf{C})$
is not Hamiltonian stable if $p\geq 3$.

\begin{thm}
\begin{equation*}
L^{4p-2}=
(SO(2)\times SO(2p))/({\mathbf Z}_2\times SO(2p-2)){\mathbf Z}_4 \quad (p\geq 2)
\end{equation*}
is not Hamiltonian stable if and only if $(m-2)-1=2p-3\geq 3$.
If $p=2$, then it is strictly Hamiltonian stable.
\end{thm}
\begin{rem0}
The index $i(L^{4p-2})$ goes to $\infty$ as $p\to \infty$.
\end{rem0}

\subsubsection{The case $m=2p+1{\ }(p\geq{1})$ }
\noindent

Assume that $m=2p+1{\ }(p\geq{2})$.
For each
$$\tilde{\Lambda}=k_0\varepsilon_0+k_1\varepsilon_1+k_2\varepsilon_2
\in{D(K,K_0)=D(SO(2)\times SO(2p+1),{\mathbf Z}_2\times SO(2p-1))}$$
with
$\tilde{\Lambda}^{\prime}=k_0\varepsilon_0+k_1^{\prime}\varepsilon_1
\in
D(K_1,K_0)=
D(SO(2)\times SO(2)\times SO(2p-1), {\mathbf Z}_2\times SO(2p-1))$
as in Lemma \ref{D(SO(2)XSO(m),Z_2XSO(m-2))},
$-{\mathcal C}_{K/K_0}$ and $-{\mathcal C}_{K_1/K_0}$ have eigenvalues
\begin{equation*}
\begin{split}
&-c_{\tilde{\Lambda}}= k_0^2+k_1^2+k_2^2+(2p-1)k_1+(2p-3)k_2,
\\
&-c_{\tilde{\Lambda}^{\prime}}=-\frac{1}{2}(k_0^2+ {k^{\prime}}_{1}^{2}).
\end{split}
\end{equation*}
Hence by the formula \eqref{CasimirOperatorBDI_2}
the corresponding eigenvalue of $-{\mathcal C}_{L}$ is
\begin{equation}\label{EigenvalueFormulaBDI_2 odd}
\begin{split}
-c_{L}
&
={\, }-c_{\tilde{\Lambda}}+\frac{1}{2}\, c_{\tilde{\Lambda}^{\prime}}
\\
&=k_0^2+k_1^2+k_2^2+(2p-1)k_1+(2p-3)k_2-\frac{1}{2}(k_0^2+ {k^{\prime}}_{1}^{2}).
\end{split}
\end{equation}

Denote
$\tilde{\Lambda}=k_0 \varepsilon_0+k_1 \varepsilon_1 +k_2 \varepsilon_2
\in D(K, K_0)$ by
$\tilde{\Lambda}=(k_0,k_1,k_2)$.

For each
$\tilde{\Lambda}=k_0 \varepsilon_0=(k_0,0,0)\in D(K,K_0)$,
as $k^\prime_1=0$, $k_0=k_0+k^\prime_1$ is even and
$-c_{L}=\frac{1}{2} k_0^2$,
we see that
\begin{equation}\label{-c_{L}leq 4p}
-c_{L}\leq 2m-2=4p \text{ if and only if } k_0^2\leq 8p.
\end{equation}
As
$\tilde{V}_{\tilde\Lambda}\cong
U_{k_0\varepsilon_0}\otimes \mathbf{C} \cong U_{k_0\varepsilon_0}$,
we have
\begin{equation*}
\rho_{k_0\varepsilon_0}
(g)
(v\otimes 1)
=
e^{\sqrt{-1}\frac{\pi}{2} k_0}\, (v\otimes 1).
\end{equation*}
Hence
\begin{equation}\label{(k_0, 0, 0) m odd}
(k_0, 0, 0)\in D(K,K_{[\mathfrak a]})
\text{ if and only if } k_0\in 4{\mathbf Z}.
\end{equation}

(i) The case $\mathcal{G}(N^4)\cong
\frac{SO(2)\times SO(3)}{\mathbf{Z}_2\cdot \mathbf{Z}_4}
\rightarrow Q_4(\mathbf{C})$ with $p=1$.

\noindent

In this case $K=SO(2)\times SO(3)$, $K_1=SO(2)\times SO(2)$
and $K_0={\mathbf Z}_2$,
where $\mathbf{Z}_2$ is generated by
$\left(
\begin{array}{cc}
-I_4 & 0 \\
0 & 1 \\
\end{array}
\right)\in U=SO(5)$.
Let  $V_{\tilde{\Lambda}}$ be an irreducible $SO(2)\times SO(3)$-module
with the highest weight
$\tilde{\Lambda}=k_0 \varepsilon_0+k_1 \varepsilon_1 \in
D(K)=D(SO(2)\times SO(3))$,
where $k_0,k_1\in \mathbf{Z}$ and $k_1\geq 0$.
It follows from the branching law of $(SO(3), SO(2))$ that
$V_{\tilde{\Lambda}}$ contains an irreducible
$SO(2)\times SO(2)$-module $V_{\tilde{\Lambda}^{\prime}}$
with the highest weight
$\tilde{\Lambda}^\prime=k_0\varepsilon_0+k_1^\prime \varepsilon_1
\in D(K_1)=D(SO(2)\times SO(2))$, where $k_1^\prime\in \mathbf{Z}$,
if and only if $|k_1^\prime| \leq k_1$.
Then we see that
$\tilde{\Lambda}^\prime\in D(SO(2)\times SO(2), \mathbf{Z}_2)$
if and only if $k_0+k_1^\prime$ is even.
By the formula \eqref{CasimirOperatorBDI_2}
the corresponding eigenvalue of the Casimir operator
$-{\mathcal C}_{L}$ is
\begin{equation}\label{EigenvalueFormulaBDI_2 p=1}
-c_{L}
=k_0^2 +k_1^2+k_1-\frac{1}{2}(k_0^2+{k^{\prime}}_{1}^2)
=\frac{1}{2}k_0^2 +k_1^2+k_1-\frac{1}{2} {k^{\prime}}_{1}^2\, .
\end{equation}
Denote $\tilde{\Lambda}=k_0 \varepsilon_0+k_1 \varepsilon_1 \in
D(SO(2)\times SO(3), \mathbf{Z}_2)$ by
$\tilde{\Lambda}=(k_0,k_1)$.
Using the eigenvalue formula \ref{EigenvalueFormulaBDI_2 p=1}, we compute that
$\tilde{\Lambda}=k_0 \varepsilon_0+k_1 \varepsilon_1 \in D(K, K_{0})$
has eigenvalue $-c_L \leq{4}$ if and only if $(k_0,k_1)$ is one of
\begin{equation*}
\Bigl\{\, (\pm 2, 0), (\pm 2,1), (\pm 1,1), (0,1), (0,2)\, \Bigr\}.
\end{equation*}

Suppose that $\tilde{\Lambda}=(\pm 2, 0)$.
Notice that for any $v\otimes w\in {\tilde V}_{k_0\varepsilon_0}\cong \mathbf{C}\otimes \mathbf{C}$,
\begin{equation*}
\rho_{k_0\varepsilon_0}(g)(v\otimes w)=e^{\sqrt{-1}k_0\frac{\pi}{2}}v\otimes w,
\end{equation*}
$\tilde{\Lambda}=k_0\varepsilon_0\in D(K,K_{[\mathfrak a]})$
if and only if $k_0\in 4\mathbf{Z}$.
Hence $\tilde{\Lambda}=(\pm 2, 0)\not \in D(K, K_{[\mathfrak a]})$.

Suppose that $\tilde{\Lambda}=(k_0, 1)$. Then $\dim {\tilde V}_{\tilde \Lambda}=3$.
The complex representation of $K=SO(2)\times SO(3)$ with the highest weight
$\tilde{\Lambda}$
corresponds to
\begin{equation*}
{\tilde V}_{\tilde\Lambda}
=U_{k_0 \varepsilon_0}\otimes V_{\varepsilon_1}
\cong U_{k_0 \varepsilon_0} \otimes \mathbf{C}^3
=(U_{k_0 \varepsilon_0} \otimes \mathbf{C}^2)
\oplus (U_{k_0 \varepsilon_0} \otimes \mathbf{C}^1).
\end{equation*}
For each $v\otimes w\in U_{k_0\varepsilon_0}\otimes \mathbf{C}^3$
and ${\rm diag}(-I_2,-I_2,1)\in K_0$, where $w=(w_1, w_2,w_2)^t \in \mathbf{C}^3$,
the representation of $K_0$  is given by
\begin{equation*}
\rho_{\tilde{\Lambda}}
({\rm diag}(-I_2,-I_2,1))
(v \otimes w)=e^{\sqrt{-1}k_0\pi} v\otimes (-w_1,-w_2,w_3)^t.
\end{equation*}
Then
$(V_{\tilde\Lambda})_{K_0}= \mathbf{C} \otimes \mathbf{C} (0,0,w_3)^t \cong \mathbf{C}\otimes \mathbf{C}$ if $k_0$ is even
and $(V_{\tilde\Lambda})_{K_0}= \mathbf{C} \otimes \mathbf{C} (w_1, w_2, 0)^t
\cong \mathbf{C}\otimes \mathbf{C}^2$ if $k_0$ is odd.
Moreover,
\begin{equation*}
\rho_{\tilde{\Lambda}}(g)
(v \otimes w)
=e^{\sqrt{-1}k_0\frac{\pi}{2}}  v \otimes \left(
             \begin{array}{c}
               w_2 \\
               w_1 \\
               -w_3 \\
             \end{array}
           \right).
\end{equation*}
Thus $\tilde{\Lambda}\in D(K, K_{[\mathfrak a]})$
if and only if $k_0\equiv 2 \, {\rm mod} \, 4$
and its multiplicity is $1$.
In particular,
$\tilde{\Lambda}=(0,1)$ or $(\pm 1, 1)\not \in D(K, K_{[\mathfrak a]})$ and
$\tilde{\Lambda}=(\pm 2, 1)\in D(K, K_{[\mathfrak a]})$.
For $\tilde{\Lambda}=(\pm 2, 1)$,
it follows from the branching laws of $(SO(3), SO(2))$
that $|k_1^\prime|\leq k_1$ thus $k^{\prime}_1=0$
such that $k_0+k_1^\prime$ is even. 
Hence, $-c_L=4$.

Suppose that $\tilde{\Lambda}=(0,2)$.
Then $\dim_{\mathbf C}\tilde{V}_{\tilde\Lambda}=5$.
It follows from the branching law of $(SO(3), SO(2))$ that
$k_1^\prime=0$  or $\pm 2$.
If $k_1^\prime=\pm 2$, then $-c_L=4$.
If $k_1^\prime=0$, then $-c_L=6 > 4$.
On the other hand,
$\Lambda=2\varepsilon_1\in D(SO(3))$ corresponds to
$V_{\Lambda}\cong {\rm S}^2_0 (\mathbf{C}^3)$
and the representation of $SO(3)$ on ${\rm S}^2_0 (\mathbf{C}^3)$ is just
the complexified isotropy representation of a symmetric pair
$(SU(3),SO(3))$.
Thus ${\rm S}^2_0 (\mathbf{C}^3)$ can be decomposed into irreducible
$SO(2)$-modules as
\begin{equation*}
\begin{split}
V_{2\varepsilon_1}\cong&\ {\rm S}^2_0 (\mathbf{C}^3)\\
=&\ {\rm S}^2_0 (\mathbf{C}^2)\oplus
\Bigl\{
\begin{pmatrix}
0&0 & a\\
0&0& b\\
a&b&0
\end{pmatrix}
\mid
a,b\in{\mathbf C}
\Bigr\}
\oplus
\mathbf{C}
\begin{pmatrix}
I_{2}&\\
& -2
\end{pmatrix}\\
=&\ \mathbf{C}
\begin{pmatrix}
                    1 & \sqrt{-1} \\
                    \sqrt{-1} & -1 \\
                  \end{pmatrix}
\oplus
\mathbf{C}
\begin{pmatrix}
                    1 & -\sqrt{-1} \\
                    -\sqrt{-1} & -1 \\
                  \end{pmatrix}\\
&\oplus
\mathbf{C}
\begin{pmatrix}
                    0&0 & 1 \\
                    0&0&\sqrt{-1}\\
                    1&\sqrt{-1} & 0 \\
                  \end{pmatrix}
\oplus
\mathbf{C}
\begin{pmatrix}
                    0&0 & 1 \\
                    0&0&-\sqrt{-1}\\
                    1&-\sqrt{-1} & 0 \\
                  \end{pmatrix}
\oplus
\mathbf{C}
\begin{pmatrix}
I_{2}&\\
& -2
\end{pmatrix}
\\
=&\ V^\prime_{2\varepsilon_1}\oplus V^\prime_{-2\varepsilon_1}\oplus V^\prime_{\varepsilon_1}
\oplus V^\prime_{-\varepsilon_1}\oplus V^{\prime}_0.
\end{split}
\end{equation*}
Using this expression we can directly show that
$(\tilde{V}_{\tilde\Lambda})_{K_0}\cong (\mathbf{C}\otimes {\rm S}^2_0(\mathbf{C}^2))
\oplus (\mathbf{C} \otimes \mathbf{C} \begin{pmatrix}
I_2& \\
 & -2
\end{pmatrix})$
and
$(\tilde{V}_{\tilde\Lambda})_{K_{[\mathfrak a]}}
\cong
\mathbf{C}\otimes\mathbf{C}
\begin{pmatrix}
0&\sqrt{-1}\\
\sqrt{-1}&0
\end{pmatrix}
\oplus (\mathbf{C} \otimes \mathbf{C} \begin{pmatrix}
I_2& \\
 & -2
\end{pmatrix})$.
Hence $\tilde{\Lambda}=(0,2)\in D(K,K_{[\mathfrak a]})$ with multiplicity $2$.
Note that the first summand of
$(\tilde{V}_{\tilde\Lambda})_{K_{[\mathfrak a]}}$ lies in
$\mathbf{C}\otimes (V^{\prime}_{2\varepsilon_1}\oplus
V^{\prime}_{-2\varepsilon_1})$,
which gives eigenvalue $4$ with multiplicity $1$
and the second summand of
$(\tilde{V}_{\tilde\Lambda})_{K_{[\mathfrak a]}}$ lies in
$\mathbf{C}\otimes V^{\prime}_0$,
which gives eigenvalue $6(>4)$ with multiplicity $1$.

Now we obtain that $\mathcal{G}(N^4)\subset Q_4(\mathbf{C})$ is
Hamiltonian stable.
Moreover since
\begin{equation*}
\begin{split}
n(L^{4})
&=\dim \tilde{V}_{(2,1)}+\dim \tilde{V}_{(-2,1)}+\dim \tilde{V}_{(0,2)}=3+3+5
\\&
=11=\dim SO(6)-\dim (SO(2)\times SO(3))=n_{hk}(L^{4}),
\end{split}
\end{equation*}
$L^{4}=\mathcal{G}(N^4)\subset Q_4(\mathbf{C})$ is Hamiltonian rigid.
Therefore
$\mathcal{G}(N^4)\subset Q_4(\mathbf{C})$ is strictly Hamiltonian stable.

(ii) The case
$\mathcal{G}(N^8)\cong
\frac{SO(2)\times SO(5)}{(\mathbf{Z}_2\times SO(3))\cdot \mathbf{Z}_4}
\rightarrow Q_8(\mathbf{C})$ with $p=2$

Denote
$\tilde{\Lambda}=k_0 \varepsilon_0+k_1 \varepsilon_1 + k_2 \varepsilon_2
\in D(K,K_0)=D(SO(2)\times SO(5), \mathbf{Z}_2\times SO(3))$
by $\tilde{\Lambda}=(k_0,k_1, k_2)$.
Let
$\tilde{\Lambda}^\prime=k_0\varepsilon_0+k^{\prime}_1 \varepsilon_1
\in D(K_1,K_0)=D(SO(2)\times SO(2)\times SO(3), \mathbf{Z}_2 \times SO(3))$
as in Lemma \ref{D(SO(2)XSO(m),Z_2XSO(m-2))}.
Then using the eigenvalue formula \eqref{EigenvalueFormulaBDI_2 odd}
we compute
\begin{lem}
$\tilde{\Lambda}=k_0 \varepsilon_0+k_1 \varepsilon_1 + k_2 \varepsilon_2
\in D(K,K_0)$ has eigenvalue $-c_L \leq{8}$ if and only if
$(k_0,k_1, k_2)$ is one of
\begin{equation*}
\{\,  (\pm 4,0,0), (\pm 1,1,0), (\pm 3, 1, 0), (0,1,1), (\pm 2, 1,1), (0,2,0)\, \}.
\end{equation*}
\end{lem}

Suppose that $\tilde{\Lambda}=(\pm 4, 0, 0)$. Then $\dim\tilde{V}_{\tilde\Lambda}=1$.
It follows from the branching law of $(SO(5), SO(2)\times SO(3))$ that $k_1^\prime=0$.
Thus $-c_L=8$. On the other hand,
it follows from \eqref{(k_0, 0, 0) m odd} that $\tilde{\Lambda}=(\pm 4, 0, 0)\in D(K, K_{[\mathfrak a]})$.

Suppose that $\tilde{\Lambda}=(k_0,1,0)$. Then $\dim\tilde{V}_{\tilde{\Lambda}}=5$ and
$\tilde{V}_{\tilde\Lambda}\cong U_{k_0\varepsilon_0}\otimes \mathbf{C}^5$,
where $\Lambda=\varepsilon_1\in D(K)$ corresponds to the matrix multiplication of $SO(5)$ on $\mathbf{C}^5$.
It follows from the branching law of $(SO(5), SO(2)\times SO(3))$ that $k_1^\prime=\pm 1$.
Hence $-c_L=\frac{1}{2}k_0^2+\frac{7}{2}$.
Notice that $U_{k_0\varepsilon_0}\otimes \mathbf{C}^5$ can be decomposed into the following $SO(2)\times SO(3)$-modules:
\begin{equation*}
U_{k_0\varepsilon_0}\otimes \mathbf{C}^5
=(U_{k_0\varepsilon_0}\otimes (\mathbf{C}^2\oplus\{0\}))
\oplus (U_{k_0\varepsilon_0}\otimes (\{0\}\oplus \mathbf{C}^3)),
\end{equation*}
where $U_{k_0}\varepsilon_0 \otimes (\{0\}\oplus \mathbf{C}^3)$ has no nonzero fixed vector by $\mathbf{Z}_2\times SO(3)$.
If $k_0$ is odd, then
\begin{equation*}
\rho_{k_0\varepsilon_0+\varepsilon_1}
\begin{pmatrix}
  -I_2 &  &  \\
   & -I_2 &  \\
   &  & T \\
\end{pmatrix}
(v\otimes \begin{pmatrix}
            w_1 \\
            w_2 \\
            0 \\
            0 \\
            0 \\
          \end{pmatrix}
)
=
e^{\sqrt{-1}\pi k_0}  v\otimes \begin{pmatrix}
            -w_1 \\
            -w_2 \\
            0 \\
            0 \\
            0 \\
          \end{pmatrix}
=v\otimes \begin{pmatrix}
            w_1 \\
            w_2 \\
            0 \\
            0 \\
            0 \\
          \end{pmatrix},
\end{equation*}
i.e., $(\tilde{V}_{\tilde{\Lambda}})_{\mathbf{Z}_2\otimes SO(3)}=U_{k_0\varepsilon_0}\otimes (\mathbf{C}^2\oplus\{0\})$
if $k_0$ is odd.
But since
\begin{equation*}
\rho_{k_0\varepsilon_0+\varepsilon_1}
(g)
(v\otimes \begin{pmatrix}
            w_1 \\
            w_2 \\
            0 \\
            0 \\
            0 \\
          \end{pmatrix}
)
=
e^{\sqrt{-1}\frac{\pi}{2} k_0}  v\otimes
\begin{pmatrix}
            w_2 \\
            w_1 \\
            0 \\
            0 \\
            0 \\
          \end{pmatrix},
\end{equation*}
$U_{k_0\varepsilon_0}\otimes (\mathbf{C}^2\oplus\{0\})$
has no nonzero fixed vector by
$(\mathbf{Z}_2 \times SO(3))\cdot \mathbf{Z}_4$, i.e.,
neither $(\pm 1, 1, 0)$ and $(\pm 3, 1, 0) $ is in $D(K,K_{[\mathfrak a]})$.

Suppose that $\tilde{\Lambda}=(k_0,1,1)$. Then $\dim\tilde{V}_{\tilde\Lambda}=10$
and $\tilde{V}_{\tilde\Lambda}\cong \mathbf{C}\otimes \wedge^2\mathbf{C}^5$.
It follows from the branching law of $(SO(5), SO(2)\times SO(3))$ that
$k_1^\prime=0$. Thus $-c_L=\frac{1}{2} k_0^2+6$.
On the other hand, since $e_1\wedge e_2 \in \wedge^2 \mathbf{C}^5$ is fixed by $SO(2)\times SO(3)$,
$v\otimes (e_1\wedge e_2)\in \mathbf{C} \otimes \wedge^2 \mathbf{C}^5$
is fixed by $\mathbf{Z}_2\times SO(3)\subset SO(2)\times SO(2)\times SO(3)$.
Moreover,
\begin{equation*}
\rho_{k_0\varepsilon_0+\varepsilon_1+\varepsilon_2}
(g)
(v\otimes (e_1\wedge e_2))
=
e^{\sqrt{-1}\frac{\pi}{2}k_0} v\otimes(e_2\wedge e_1).
\end{equation*}
Hence, $\tilde\Lambda=(0,1,1)\not\in D(K,K_{[\mathfrak a]})$
but $\tilde\Lambda=(\pm 2,1,1)\in D(K,K_{[\mathfrak a]})$
and $(\tilde{V}_{\tilde\Lambda})_{K_{[\mathfrak a]}}\cong \mathbf{C}\otimes \mathbf{C} \{e_1\wedge e_2\}$
for $k_0=2$ or $-2$, both of which give eigenvalue $8$.

Suppose that $\tilde{\Lambda}=(0,2,0)$. Then $\dim \tilde{V}_{\tilde\Lambda}=14$
and $\tilde{V}_{\tilde\Lambda}\cong \mathbf{C}\otimes {\rm S}^2_0(\mathbf{C}^5)$,
where the representation of $SO(5)$ with highest weight $2\varepsilon_1$ is just the adjoint representation
on ${\rm S}^2_0 (\mathbf{C}^5)$.
It follows from the branching law of $(SO(5), SO(2)\times SO(3))$ that $k_1^\prime=0, \pm 2$.
Thus $-c_L=10-\frac{1}{2}{k_1^\prime}^2$.
When $k_1^\prime=\pm 2$, $-c_L=8$, otherwise $-c_L =10> 8$.
On the other hand,
${\rm S}^2_0 (\mathbf{C}^5)$ can be decomposed into the following $SO(2)\times SO(3)$-modules:
\begin{equation*}
\begin{split}
V_{2\varepsilon_1}& \cong  {\rm S}^2_0 (\mathbf{C}^5)\\
&={\mathrm S}^2_0 (\mathbf{C}^2)\oplus {\rm S}^2_0(\mathbf{C}^3)\oplus
M(2,3;\mathbf{C})
\oplus
\Bigl\{
\begin{pmatrix}
zI_{2}&\\
0&w I_3
\end{pmatrix}
\mid
z,w\in{\mathbf C}, 2z+3w=0
\Bigr\}.
\end{split}
\end{equation*}
Thus, ${\rm S}^2_0(\mathbf{C}^2)$ is fixed by $\{-I_2\}\times SO(3)$ and
$$(\tilde{V}_{\tilde\Lambda})_{K_0}
\cong \mathbf{C}\otimes {\rm S}^2_0(\mathbf{C}^2)
\oplus \mathbf{C}\otimes \mathbf{C} \left(
                                      \begin{array}{cc}
                                        3I_2 &  \\
                                         & -2I_3 \\
                                      \end{array}
                                    \right).
$$
Moreover,
\begin{equation*}
\rho_{2\varepsilon_1}
(g)
(v \otimes
\begin{pmatrix}
a& b& \\
b&-a&\\
 & & 0
\end{pmatrix}
)
=v\otimes \begin{pmatrix}
-a& b& \\
b&a&\\
& & 0
\end{pmatrix}.
\end{equation*}
Hence,
$(\tilde{V}_{\tilde\Lambda})_{K_{[\mathfrak a]}}=\mathbf{C}\otimes
\mathbf{C}\cdot \begin{pmatrix}
0& 1& \\
1&0&\\
& & 0
\end{pmatrix}\oplus \mathbf{C}\otimes \mathbf{C} \left(
                                      \begin{array}{cc}
                                        3I_2 &  \\
                                         & -2I_3 \\
                                      \end{array}
                                    \right)$.
Therefore, $\tilde{\Lambda}=(0,2,0)\in D(K,K_{[\mathfrak a]})$.
Notice the first summand lies in $\tilde{V}^{\prime}_{(0,2,0)}\oplus \tilde{V}^{\prime}_{(0,-2,0)}$
which gives eigenvalue $8$ and the second summand lies in $\tilde{V}^{\prime}_{(0,0,0)}$
which gives eigenvalue $10$. Hence the multiplicity corresponding to eigenvalue $8$ is $1$.

Since $\tilde\Lambda=(4,0,0)$, $(-4,0,0)$, $(2,1,1)$, $(-2,1,1)$, $(0,2,0)\in D(K,K_{[\mathfrak a]})$
give the smallest eigenvalue $8$ with multiplicity $1$  and
\begin{equation*}
\begin{split}
n(L^{8})
&=\dim \tilde{V}_{(4,0,0)}+\dim \tilde{V}_{(-4,0,0)}+\dim \tilde{V}_{(2,1,1)}+\dim \tilde{V}_{(-2,1,1)}
+ \dim \tilde{V}_{(0,2,0)}\\
&=1+1+10+10+14=36\\
&>34=\dim SO(10)-\dim SO(2)\times SO(5)=n_{hk}(L^{8}),
\end{split}
\end{equation*}
$\mathcal{G}(N^8)\subset Q_8(\mathbf{C})$ is not Hamiltonian rigid.
Therefore
$\mathcal{G}(N^8)\subset Q_8(\mathbf{C})$ is Hamiltonian stable but not strictly Hamiltonian stable.

(iii)
The case $\mathcal{G}(N^{4p})\cong \frac{SO(2)\times SO(2p+1)}{(\mathbf{Z}_2\times SO(2p-1))\cdot \mathbf{Z}_4}
\rightarrow Q_{4p}(\mathbf{C})$ with $p\geq 3$.

Suppose that
$\tilde{\Lambda}=(k_0,0,0)$ and $k_0\in 4{\mathbf Z}\setminus\{0\}$.
Then $k_1^\prime=0$ and
by \eqref{-c_{L}leq 4p}
$\tilde{\Lambda}\in D(K,K_{[\mathfrak{a}]})$.
As $p\geq 3$, we have $16< 24\leq 8p$.
Hence by \eqref{(k_0, 0, 0) m odd}
we see that
for every $k_{0}\in 4{\mathbf Z}\setminus\{0\}$ such that
$16\leq k_{0}^2 < 8p$
we have eigenvalue $-c_L=\frac{1}{2} k_{0}^2 < 4p$.
Therefore,
$\mathcal{G}(N^{4p})\cong
\frac{SO(2)\times SO(2p+1)}{(\mathbf{Z}_2\times SO(2p-1))\cdot \mathbf{Z}_4}
\rightarrow Q_{4p-2}(\mathbf{C})$
is not Hamiltonian stable if $p\geq 3$.

Therefore, we obtain
\begin{thm} The Gauss image
$L^{4p}=\frac{SO(2)\times SO(2p+1)}{({\mathbf Z}_2\times
SO(2p-1)){\mathbf Z}_4}\rightarrow Q_{4p}(\mathbf{C})$
$(p\geq 1)$
is not Hamiltonian stable if and only if $(m-2)-1=2p-2\geq 3$.
If $p=1$, it is strictly Hamiltonian stable and if $p=2$, it is Hamiltonian stable
but not strictly Hamiltonian stable.
\end{thm}
\begin{rem0}
The index $i(L^{4p})$ goes to $\infty$ as $p\to \infty$.
\end{rem0}

\section{The case $(U,K)=(SU(m+2), S(U(2)\times U(m)))$ $(m\geq 2)$}
\label{Sec_AIII_2}

In this case, $U=SU(m+2)$ and $K=S(U(2)\times U(m))$ with $m\geq 2$.
Then $(U,K)$ is of $B_2$ type for $m=2$ and $BC_2$ type for $m\geq 3$.

In this case we use the formulation by the unitary group $U(m)$
rather than one by the special unitary groups $SU(m)$.
It seems to work more successfully
in our argument of applying the branching laws.
Here we will also indicate the relations between both formulations.
Let $\tilde{U}:=U(m+2)$, $\tilde{K}:=U(2)\times U(m)$,
$\tilde{K}_{2}:=U(2)\times U(2) \times U(m-2)$,
$\tilde{K}_{1}:= U(1)\times U(1) \times U(1) \times U(1) \times U(m-2)$
and $\tilde{K}_{0}:=U(1)\times U(1)\times U(m-2)$.
Then
$\tilde{U}=C(\tilde{U})\cdot{U}$,
$\tilde{K}=C(\tilde{U})\cdot{K}$,
$\tilde{K}_{2}=C(\tilde{U})\cdot{K}_{2}$,
$\tilde{K}_{1}=C(\tilde{U})\cdot{K}_{1}$,
and
$\tilde{K}_{0}=C(\tilde{U})\cdot{K}_{0}$,
where $C(\tilde{U})$ is the center of $\tilde{U}$.

Let
${\mathfrak u}={\mathfrak k}+{\mathfrak p}$
and
$\tilde{\mathfrak u}=\tilde{\mathfrak k}+{\mathfrak p}$
be the canonical decomposition
of ${\mathfrak u}$ and $\tilde{\mathfrak u}$
corresponding to $(U,K)$ and $(\tilde{U},\tilde{K})$, respectively.
Let
${\mathfrak a}$ be a maximal abelian subspace of
${\mathfrak p}$, where
\begin{equation*}
\begin{split}
{\mathfrak a}
=&\left\{
\begin{pmatrix}
0&H_{12} \\
-\bar{H}_{12}^{t}&0 \\
\end{pmatrix}
\mid
H_{12}
=
\begin{pmatrix}
\xi_{1}&0&0&\cdots&0 \\
0&\xi_{2}&0&\cdots&0 \\
\end{pmatrix},
\xi_{1},\xi_{2}\in{\mathbf R}
\right\}.
\end{split}
\end{equation*}
Then the centralizer $\tilde{K}_0$ of $\mathfrak{a}$ in $\tilde{K}$ is given
as follows:
\begin{equation*}
\begin{split}
\tilde{K}_0=&\left\{
   P=\begin{pmatrix}
      e^{is} &  &  &  &  \\
       & e^{it} &  &  &  \\
        &  & e^{is} &  &  \\
       &  &  & e^{it} &  \\
       &  &  &  & T \\
    \end{pmatrix}
     \mid
      T\in U(m-2)
         \right\}\\
\cong &\,
U(1)\times U(1)\times U(m-2).
\end{split}
\end{equation*}
Moreover,
\begin{equation*}
\tilde{K}_{[\mathfrak{a}]}=
\tilde{K}_0 \cup (Q\cdot \tilde{K}_0) \cup (Q^2\cdot \tilde{K}_0)
\cup (Q^3\cdot \tilde{K}_0),
\end{equation*}
where
\begin{equation*}
Q=\begin{pmatrix}
0&1& & & \\
-1&0& & & \\
 & & 0 & -1& \\
 & &-1 & 0 & \\
 & & & & I_{m-2}
\end{pmatrix}
\in
\tilde{K}_2\subset\tilde{K}.
\end{equation*}
Thus the deck transformation group of the covering map
${\mathcal G}:N^{8m-2}\rightarrow {\mathcal G}(N^{4m-2})$ $(m\geq 2)$
is equal to
$K_{[{\mathfrak a}]}/K_{0}\cong \tilde{K}_{[{\mathfrak a}]}/\tilde{K}_{0}
\cong{\mathbf Z}_{4}$.
Remark that we will use $P$ and $Q$ to denote the element in $\tilde{K}_0$
and the generator of $\mathbf{Z}_4$ in $\tilde{K}_{[\mathfrak{a}]}$
throughout this section.

\subsection{Description of the Casimir operator}
\noindent

Define an inner product
$\langle{X,Y}\rangle_{\mathfrak u}:=-\mathrm{tr}XY$
for each $X,Y\in{\mathfrak u}=\mathfrak{su}(m+2)$
or for each $X,Y\in\tilde{\mathfrak u}=\mathfrak{u}(m+2)$ .
The restricted root system $\Sigma(U,K)$ is of type $B_2$ for $m=2$ and
type $BC_2$ for $m\geq 3$.
Then
the square length of each restricted roots
with respect to $\langle{\ ,\ }\rangle_{\mathfrak u}$,
is given by
\begin{equation*}
\Vert{\gamma}\Vert_{\mathfrak u}^{2}
=\left\{
   \begin{array}{ll}
     1 \text{ or } 2, & m=2; \\
     \frac{1}{2}, 1 \text{ or } 2, & m\geq 3.
   \end{array}
 \right.
\end{equation*}
Hence the Casimir operator $\mathcal{C}_L$ of $L$
with respect to the induced metric from
$g^{\rm std}_{Q_{4m-2}(\mathbf{C})}$ can be expressed as follows:
\begin{equation*}
{\mathcal C}_{L}
=\left\{
   \begin{array}{ll}
     {\mathcal C}_{K/K_{0}}
-\frac{1}{2}{\ }{\mathcal C}_{K_{1}/K_{0}}, & m=2; \\
     2 {\mathcal C}_{K/K_0}-{\mathcal C}_{K_2/K_0}-\frac{1}{2}{\mathcal C}_{K_1/K_0}, & m\geq 3,
   \end{array}
 \right.
\end{equation*}
where
${\mathcal C}_{K/K_{0}}$, ${\mathcal C}_{K_{2}/K_{0}}$
and ${\mathcal C}_{K_{1}/K_{0}}$
denote
the Casimir operator of  $K/K_{0}$, $K_{2}/K_{0}$ and
$K_{1}/K_{0}$
relative to
$\langle{{\ },{\ }}\rangle_{\mathfrak u}\vert_{\mathfrak k}$,
$\langle{{\ },{\ }}\rangle_{\mathfrak u}\vert_{{\mathfrak k}_{2}}$
and
$\langle{{\ },{\ }}\rangle_{\mathfrak u}\vert_{{\mathfrak k}_{1}}$,
respectively.

\subsection{Descriptions of $D(\tilde{U})$, $D(U)$ and etc.}
\noindent

$D(\tilde{U})$, $D(C(\tilde{U}))$ and $D(U)$ are described as follows:
\begin{equation*}
\begin{split}
&D(\tilde{U})=D(U(m+2))
=\Bigl\{\tilde{\Lambda}
=\tilde{p}_{1}y_{1}+\cdots+\tilde{p}_{m+2}y_{m+2}
\mid
\tilde{p}_{1},\cdots,\tilde{p}_{m+2}\in{\mathbf Z},
\\&
\quad\quad\quad\quad\quad\quad\quad\quad
\quad\quad\quad\quad\quad\quad\quad\quad
\quad\quad
\tilde{p}_{i}-\tilde{p}_{i+1}\geq{0}{\ }(i=1,\cdots,m+1)
\Bigr\},\\
&D(C(\tilde{U}))=D(C(U(m+2)))
=\Bigl\{\Lambda=p_{0}(y_{1}+\cdots+y_{m+2})
\mid
p_{0}\in\frac{1}{m+2}{\mathbf Z}
\Bigr\},\\
&D(U)=D(SU(m+2))=
\Bigl\{\Lambda=p_{1}y_{1}+\cdots+p_{m+2}y_{m+2}
\mid \sum^{m+2}_{i=1}p_{i}=0,
\\
&
\quad\quad\quad\quad\quad\quad\quad\quad
\quad\quad\quad\quad
p_{i}-p_{m+2}\in{\mathbf Z},
p_{i}-p_{i+1}\geq{0}{\ }(i=1,\cdots,m+1)
\Bigr\}.
\end{split}
\end{equation*}
Each
$\tilde{\Lambda}=\tilde{p}_{1}y_{1}+\cdots+\tilde{p}_{m+2}y_{m+2}\in{D(U(m+2))}$
can be decomposed as $\tilde{\Lambda}=\Lambda^{0}+\Lambda$,
where
\begin{equation*}
\Lambda^{0}=
\left(\frac{1}{m+2}\sum^{m+2}_{i=1}\tilde{p}_{i}\right)
\left(\sum^{m+2}_{i=1}y_{i}\right)
\in{D(C(U(m+2)))}
\end{equation*}
and
\begin{equation*}
\Lambda=
(\tilde{p}_{1}-\frac{1}{m+2}\sum^{m+2}_{i=1}\tilde{p}_{i})y_{1}
+\cdots
+(\tilde{p}_{m+2}-\frac{1}{m+2}\sum^{m+2}_{i=1}\tilde{p}_{i})y_{m+2}
\in{D(SU(m+2))}.
\end{equation*}
Note that this projection
$D(\tilde{U})\rightarrow D(U), \tilde{\Lambda}\mapsto \Lambda$ is surjective.

\begin{equation*}
\begin{split}
&D(\tilde{K}) =D(U(2)\times U(m))
\\
=& \{\tilde{\Lambda}= \tilde{q}_{1}y_{1}+\tilde{q}_{2}y_{2}
+\tilde{q}_{3}y_{3}+\cdots+\tilde{q}_{m+2}y_{m+2}
\mid
\\
&\
\tilde{q}_{i}\in{\mathbf Z}{\ }(i=1, \cdots, m+2),
\tilde{q}_{1}-\tilde{q}_{2} \geq{0},
\tilde{q}_{i}-\tilde{q}_{i+1}\geq{0}{\ }(i=3,\cdots, m+1) \},
\\
\end{split}
\end{equation*}
\begin{equation*}
\begin{split}
&D(K) =D(S(U(2)\times U(m)))
\\
=& \{ \Lambda=
q_{1}y_{1}+q_{2}y_{2}+q_{3}y_{3}+\cdots+q_{m+2}y_{m+2}
\mid
\sum^{m+2}_{i=1}q_{i}=0,
q_{i}-q_{j}\in{\mathbf Z}{\ } \\
& \quad (i,j=1,2,\cdots, m+2),
q_{1}-q_{2}\geq{0},
q_{i}-q_{i+1}\geq{0}{\ }(i=3,4,\cdots, m+1)\},
\end{split}
\end{equation*}
\begin{equation*}
\begin{split}
&D(\tilde{K_{2}}) =D(U(2)\times U(2)\times U(m-2))
\\
=&\{\tilde{\Lambda}=
\tilde{q}_{1}y_{1}+\tilde{q}_{2}y_{2}
+\tilde{q}_{3}y_{3}+\tilde{q}_{4}y_{4}
+\tilde{q}_{5}y_{5}+\cdots+\tilde{q}_{m+2}y_{m+2}
\mid \\
&\quad\quad\quad\quad\quad
\tilde{q}_{i}\in{\mathbf Z}{\ }(i=1, \cdots, m+2), \\
&\quad\quad\quad\quad\quad
\tilde{q}_{1}-\tilde{q}_{2}, \tilde{q}_{3}-\tilde{q}_{4},
\tilde{q}_{i}-\tilde{q}_{i+1}\geq{0}{\ }(i=5,\cdots, m+1)\},
\\
&D(K_{2})
=D(S(U(2)\times U(2)\times U(m-2)))
\\
=&
\{
\Lambda=
q_{1}y_{1}+q_{2}y_{2}+q_{3}y_{3}+q_{4}y_{4}
+q_{5}y_{5}+\cdots+q_{m+2}y_{m+2}
\mid
\sum^{m+2}_{i=1}q_{i}=0,
\\ &
q_{i}-q_{j}\in{\mathbf Z}{\ }(i,j=1,2,\cdots, m+2),
q_{1}-q_{2}, q_{3}-q_{4},
q_{i}-q_{i+1}\geq{0}{\ }(i=5,
\cdots, m+1)\},
\end{split}
\end{equation*}
\begin{equation*}
\begin{split}
&D(\tilde{K_{1}})
=D(U(1)\times U(1)\times U(1)\times U(1)\times U(m-2))
\\
=&
\{\tilde{\Lambda}=
\tilde{q}_{1}y_{1}+\tilde{q}_{2}y_{2}
+\tilde{q}_{3}y_{3}+\tilde{q}_{4}y_{4}
+\tilde{q}_{5}y_{5}+\cdots+\tilde{q}_{m+2}y_{m+2}
\mid \\
&\quad\quad\quad\quad
\tilde{q}_{i}\in{\mathbf Z}{\ }(i=1,\cdots, m+2),
\tilde{q}_{i}-\tilde{q}_{i+1}\geq{0}{\ }(i=5,\cdots, m+1){\ }
\},
\\
\end{split}
\end{equation*}
\begin{equation*}
\begin{split}
&D(K_{1})
=D(S(U(1)\times U(1)\times U(1)\times U(1)\times U(m-2)))
\\
=&
\Bigl\{
\Lambda=
q_{1}y_{1}+q_{2}y_{2}+q_{3}y_{3}+q_{4}y_{4}
+q_{5}y_{5}+\cdots+q_{m+2}y_{m+2}
\mid
\sum^{m+2}_{i=1}q_{i}=0,
\\ &\quad
q_{i}-q_{j}\in{\mathbf Z}{\ }(i,j=1, \cdots, m+2),
q_{i}-q_{i+1}\geq{0}{\ }(i=5,\cdots, m+1) \Bigr\},
\end{split}
\end{equation*}
\begin{equation*}
\begin{split}
&D(\tilde{K_{0}})
=D(U(1)\times U(1)\times U(m-2))
\\
=&
\{\tilde{\Lambda}=
\tilde{q}_{1}y_{1}+\tilde{q}_{2}y_{2}
+\tilde{q}_{3}y_{3}+\tilde{q}_{4}y_{4}
+\tilde{q}_{5}y_{5}+\cdots+\tilde{q}_{m+2}y_{m+2}
\mid \\
&\quad\quad\quad\quad\quad
\tilde{q}_{3}=\tilde{q}_{1}\in \frac{1}{2}\mathbf{Z}, \tilde{q}_{4}=\tilde{q}_{2}\in
\frac{1}{2}\mathbf{Z},
\tilde{q}_{i}\in{\mathbf Z}{\ }(i=5, \cdots, m+2), \\
&\quad\quad\quad\quad\quad
\tilde{q}_{i}-\tilde{q}_{i+1}\geq{0}{\ }(i=5,6,\cdots, m+1)\},
\\
\end{split}
\end{equation*}
\begin{equation*}
\begin{split}
&D(K_{0})
=D(S(U(1)\times U(1)\times U(m-2)))
\\
=&
\{
\Lambda=
q_{1}y_{1}+q_{2}y_{2}+q_{3}y_{3}+q_{4}y_{4}
+q_{5}y_{5}+\cdots+q_{m+2}y_{m+2}
\mid \\
&\quad
\sum^{m+2}_{i=1}q_{i}
=0,
q_{i}-q_{j}\in{\bold Z}{\ }(i,j=1,
\cdots, m+2),
\\
&\quad
q_{3}=q_{1}, q_{4}=q_{2},
q_{i}-q_{i+1}\geq{0}{\ }(i=5,
\cdots, m+1)\}.
\end{split}
\end{equation*}
The natural maps
$D(\tilde{K})\longrightarrow D(K)$,
$D(\tilde{K}_{2})\longrightarrow D(K_{2})$,
$D(\tilde{K}_{1})\longrightarrow D(K_{1})$ and
$D(\tilde{K}_{0})\longrightarrow D(K_{0})$
are also surjective.

\subsection{Branching laws of $(U(m), U(2)\times U(m-2))$}
The branching laws for $(SU(m), S(U(m) \times U(2)))$
given in \cite{JLMilhorat1998} can be reformulated to
the branching laws for $(U(m), U(2)\times U(m-2))$ as follows:

\begin{lem}[Branching law of  $(U(m), U(2)\times U(m-2))$]
\label{BranchingLaw(U(m), U(2)XU(m-2))}
For each
$\tilde{\Lambda}=\tilde{p}_{1}y_{1}+\cdots+ \tilde{p}_{m} y_{m}\in{D(U(m))}$,
an irreducible $U(m)$-module $V_{\tilde\Lambda}$ with
the highest weight $\tilde{\Lambda}$
can be decomposed into the direct sum of
irreducible $U(2)\times U(m-2)$-modules as follows:
\begin{equation*}
V_{\tilde\Lambda}
=\bigoplus_{\tilde{\Lambda}^\prime \in D(U(2)\times U(m-2))}
V_{\tilde{\Lambda}^\prime}^{\prime}.
\end{equation*}
Here
$V_{\tilde\Lambda}$ contains an irreducible $U(2)\times U(m-2)$-module
$V_{\tilde{\Lambda}^\prime}^{\prime}$ with the highest weight
$\tilde{\Lambda}^{\prime}= \tilde{q}_{1}y_{1}+\cdots+ \tilde{q}_{m}y_{m}
\in{D(U(2)\times U(m-2))}$
if and only if the following conditions are satisfied:
\begin{itemize}
\item[(i)] $\tilde{q}_1-\tilde{p}_1 \in \mathbf{Z}$;
\item[(ii)] $\tilde{p}_{i-2} \geq \tilde{q}_i \geq \tilde{p}_i $\, $(i=3,\cdots, m)$;
\item[(iii)] In the finite power series expansion in $X$ of
$\displaystyle\frac{\prod_{i=2}^m (X^{r_i +1}- X^{-(r_i +1)})}{(X-X^{-1})^{m-2}} $,
where
\begin{eqnarray*}
& r_2&:=\tilde{p}_1 -\max (\tilde{q}_3, \tilde{p}_2)\\
& r_i&:=\min(\tilde{q}_{i}, \tilde{p}_{i-1})
-\max(\tilde{q}_{i+1}, \tilde{p}_{i}), \quad (3\leq i\leq m-1)\\
& r_m&:=\min(\tilde{q}_m, \tilde{p}_{m-1})-\tilde{p}_m,
\end{eqnarray*}
the coefficient of  $X^{\tilde{q}_1-\tilde{q}_2+1}$
does not vanish.
Moreover, the value of this coefficient is equal to the multiplicity of the
irreducible $U(2)\times U(m-2)$-module
$V_{\tilde{\Lambda}^\prime}^{\prime}$.
\end{itemize}
\end{lem}

\subsection{Branching law of $(U(3), U(2)\times U(1))$}

Now following Lemma \ref{BranchingLaws(U(m+1),U(m)XU(1))}
the branching law of $(U(3), U(2)\times U(1))$ is described as
\begin{lem}\label{BranchingLaws(U(3),U(2)XU(1))}
Let $\tilde{V}_{\tilde{\Lambda}}$ be an irreducible $U(3)$-module
with the highest weight
$\tilde{\Lambda}=\tilde{p}_1 y_1+ \tilde{p}_2 y_2+\tilde{p}_3 y_3 \in D(U(3))$,
where $\tilde{p}_i \in \mathbf{Z}\ (i=1,2,3)$ and
$\tilde{p}_1\geq \tilde{p}_2\geq \tilde{p}_3$.
Then $\tilde{V}_{\tilde{\Lambda}}$ can be decomposed into
irreducible $U(2)\times U(1)$-modules as
\begin{equation*}
\tilde{V}_{\tilde{p}_1 y_1+\tilde{p}_2 y_2 + \tilde{p}_3 y_3}
=\bigoplus_{\alpha=0}^{\tilde{p}_1-\tilde{p}_2}\bigoplus_{\beta=0}^{\tilde{p}_2
-\tilde{p}_3}
\tilde{V}^{\prime}_{(\tilde{p}_1-\alpha)y_1+(\tilde{p}_2-\beta)y_2+
(\tilde{p}_3+\alpha+\beta)y_3}.
\end{equation*}
\end{lem}

\subsection{Descriptions of $D(\tilde{K}, \tilde{K}_{0})$,
$D(\tilde{K}_{2},\tilde{K}_{0})$, $D(\tilde{K}_{1},\tilde{K}_{0})$}
\label{D(K,K_0)AIII}

Let
\begin{equation*}
\tilde{\Lambda}
=\tilde{p}_{1}y_{1}+\tilde{p}_{2}y_{2}
+\tilde{p}_{3}y_{3}+\cdots+\tilde{p}_{m+2}y_{m+2}
\in{D(\tilde{K})}=D(U(2)\times U(m)),
\end{equation*}
where $\tilde{p}_{1},\cdots,\tilde{p}_{m+2}\in{\mathbf Z}$,
$\tilde{p}_{1}\geq\tilde{p}_{2}$,
$\tilde{p}_{3}\geq\cdots\geq\tilde{p}_{m+2}$.
Thus
$\Lambda_{\sigma}=\tilde{p}_{1}y_{1}+\tilde{p}_{2}y_{2}\in{D(U(2))}$,
$\Lambda_{\tau}=\tilde{p}_{3}y_{3}+\cdots+\tilde{p}_{m+2}y_{m+2}\in{D(U(m))}${\ }
and
\begin{equation*}
\tilde{\rho}_{\tilde{\Lambda}}
=\sigma\boxtimes\tau
\in{{\mathcal D}(\tilde{K})}={\mathcal D}(U(2)\times U(m)),
\end{equation*}
where $\sigma\in{{\mathcal D}(U(2))}$,
$\tau\in{{\mathcal D}(U(m))}$.

By Lemma \ref{BranchingLaw(U(m), U(2)XU(m-2))},
an irreducible $U(m)$-module
$V_{\tau}$ with the highest weight $\Lambda_{\tau}$
can be decomposed into the direct sum of
irreducible $U(2) \times U(m-2)$-modules as
\begin{equation*}
V_{\tau}=\bigoplus V_{\tilde{\Lambda}_\tau^\prime}^{\prime},
\end{equation*}
where
$\tilde{\Lambda}_\tau^\prime=\sum_{i=3}^{m+2}\tilde{q}_i y_i
\in D(U(2)\times U(m-2))$
with $\tilde{q}_3, \cdots, \tilde{q}_{m+2}\in \mathbf{Z}$,
$\tilde{q}_i- \tilde{q}_{i+1}\geq 0$ $(i=3, 5,\cdots, m+1)$.
Note that setting
$\Lambda_{\varsigma}:=\tilde{q}_3 y_3 +\tilde{q}_{4} y_4 \in D(U(2))$
and
$\Lambda_{\gamma}:=\tilde{q}_5 y_5 +\cdots + \tilde{q}_{m+2} y_{m+2}
\in D(U(m-2))$, we get a decomposition into
the direct sum of irreducible $\tilde{K}_2$-modules as
\begin{equation*}
V_{\tilde{\Lambda}}=\bigoplus_{\varsigma, \gamma}
(V_{\sigma}\boxtimes V_{\varsigma}\boxtimes V_{\gamma}).
\end{equation*}
By the branching law of $(U(2),U(1)\times U(1))$
(see Lemma \ref{BranchingLaws(U(m+1),U(m)XU(1))}),
\begin{equation*}
\begin{split}
&V_{\sigma}=V_{\tilde{p}_1 y_1+\tilde{p}_2 y_2}
=\bigoplus_{\alpha=0}^{\tilde{p}_1-\tilde{p}_2}V^{\prime}_{(\tilde{p}_1-\alpha) y_1 +(\tilde{p}_2 +\alpha) y_2},\\
&V_{\varsigma}=V_{\tilde{q}_3 y_3+ \tilde{q}_4 y_4}
=\bigoplus_{\beta=0}^{\tilde{q}_3-\tilde{q}_4} V^{\prime}_{(\tilde{q}_3-\beta) y_3 +(\tilde{q}_4 +\beta) y_4}.
\end{split}
\end{equation*}
Thus we have a decomposition into the direct sum of irreducible
$\tilde{K}_1$-modules:
\begin{equation*}
\begin{split}
& V_{\tilde{\Lambda}}\\
=&\bigoplus
\bigoplus^{\tilde{p}_{1}-\tilde{p}_{2}}_{\alpha=0}
\bigoplus^{\tilde{q}_{3}-\tilde{q}_{4}}_{\beta=0}
(V^{\prime}_{(\tilde{p}_{1}-\alpha)y_{1}+(\tilde{p}_{2}+\alpha)y_{2}}
\boxtimes
V^{\prime}_{(\tilde{q}_{3}-\beta)y_{3}+(\tilde{q}_{4}+\beta)y_{4}}
\boxtimes
V_{\tilde{q}_{5}y_{5}+\cdots+\tilde{q}_{m+2}y_{m+2}}).
\end{split}
\end{equation*}
Since as a $U(1)\times U(1)$-module
\begin{equation*}
V^{\prime}_{(\tilde{p}_{1}-\alpha)y_{1}+(\tilde{p}_{2}+\alpha)y_{2}}
\boxtimes
V^{\prime}_{(\tilde{q}_{3}-\beta)y_{3}+(\tilde{q}_{4}+\beta)y_{4}}
=
V^{\prime\prime}_{\frac{1}{2}(\tilde{p}_{1}+\tilde{q}_{3}-\alpha-\beta)(y_{1}+y_{3})+
\frac{1}{2}(\tilde{p}_{2}+\tilde{q}_{4}+\alpha+\beta)(y_{2}+y_{4})},
\end{equation*}
we have a decomposition into the direct sum of irreducible
$\tilde{K}_0$
-modules:
\begin{equation*}
\begin{split}
&V_{\tilde{\Lambda}}\\
=&\bigoplus_{\varsigma, \gamma}
(V_{\tilde{p}_{1}y_{1}+\tilde{p}_{2}y_{2}}
\boxtimes
V_{\tilde{q}_{3}y_{3}+\tilde{q}_{4}y_{4}}
\boxtimes
V_{\tilde{q}_{5}y_{5}+\cdots+\tilde{q}_{m+2}y_{m+2}})\\
=&\bigoplus
\bigoplus^{\tilde{p}_{1}-\tilde{p}_{2}}_{\alpha=0}
\bigoplus^{\tilde{q}_{3}-\tilde{q}_{4}}_{\beta=0}
(V^{\prime}_{(\tilde{p}_{1}-\alpha)y_{1}+(\tilde{p}_{2}+\alpha)y_{2}}
\boxtimes
V^{\prime}_{(\tilde{q}_{3}-\beta)y_{3}+(\tilde{q}_{4}+\beta)y_{4}}
\boxtimes
V_{\tilde{q}_{5}y_{5}+\cdots+\tilde{q}_{m+2}y_{m+2}})\\
=& \bigoplus
\bigoplus^{\tilde{p}_{1}-\tilde{p}_{2}}_{\alpha=0}
\bigoplus^{\tilde{q}_{3}-\tilde{q}_{4}}_{\beta=0}
V^{\prime\prime}_{\frac{1}{2}(\tilde{p}_{1}+\tilde{q}_{3}-\alpha-\beta)(y_{1}+y_{3})+
\frac{1}{2}(\tilde{p}_{2}+\tilde{q}_{4}+\alpha+\beta)(y_{2}+y_{4})}
\boxtimes
V_{\tilde{q}_{5}y_{5}+\cdots+\tilde{q}_{m+2}y_{m+2}}{\ }.
\end{split}
\end{equation*}
Thus we obtain that
$\tilde{\Lambda}\in{D(\tilde{K},\tilde{K}_{0})}$ if and only if
there exist
$\alpha, \beta\in{\mathbf Z}$ with
$0\leq\alpha\leq{\tilde{p}_{1}-\tilde{p}_{2}}$ and
$0\leq\beta\leq{\tilde{q}_{3}-\tilde{q}_{4}}$ such that
$$V^{\prime\prime}_{\frac{1}{2}(\tilde{p}_{1}+\tilde{q}_{3}-\alpha-\beta)(y_{1}+y_{3})+
\frac{1}{2}(\tilde{p}_{2}+\tilde{q}_{4}+\alpha+\beta)(y_{2}+y_{4})}
\boxtimes
V_{\tilde{q}_{5}y_{5}+\cdots+\tilde{q}_{m+2}y_{m+2}}$$
is a trivial
$\tilde{K}_0$
-module,
that is,
\begin{equation*}
\begin{cases}
&\tilde{p}_{1}+\tilde{q}_{3}-\alpha-\beta=0,\\
&\tilde{p}_{2}+\tilde{q}_{4}+\alpha+\beta=0,\\
&\tilde{q}_{5}=\cdots=\tilde{q}_{m+2}=0.
\end{cases}
\end{equation*}
Hence $\tilde{\Lambda}\in{D(\tilde{K},\tilde{K}_{0})}$ must satisfy
\begin{equation*}
\begin{split}
&\tilde{p}_5=\tilde{p}_6=\cdots=\tilde{p}_{m}=0,\\
&\tilde{p}_3 \geq \tilde{p}_4 \geq 0, \
\tilde{p}_{m+2}\leq \tilde{p}_{m+1}\leq 0,\\
&\tilde{p}_1+\tilde{p}_2+\tilde{p}_3+\tilde{p}_4+\tilde{p}_{m+1}+\tilde{p}_{m+2}=0.
\end{split}
\end{equation*}
If $m\geq{4}$, then
each
$\tilde{\Lambda}\in{D(\tilde{K},\tilde{K}_{0})}$
is expressed as
\begin{equation*}
\tilde{\Lambda}=
\tilde{p}_{1}y_{1}+\tilde{p}_{2}y_{2}+
\tilde{p}_{3}y_{3}+\tilde{p}_{4}y_{4}+
\tilde{p}_{m+1}y_{m+1}+\tilde{p}_{m+2}y_{m+2},
\end{equation*}
where $\tilde{p}_{i}\in{\mathbf Z}$,
$\tilde{p}_{1}\geq\tilde{p}_{2}$,
$\tilde{p}_{3}\geq\tilde{p}_{4}\geq0 \geq \tilde{p}_{m+1}\geq\tilde{p}_{m+2}$,
\begin{equation*}
\tilde{p}_{1}+\tilde{p}_{2}+
\tilde{p}_{3}+\tilde{p}_{4}+
\tilde{p}_{m+1}+\tilde{p}_{m+2}=0.
\end{equation*}
If $m=3$, then
each
$\tilde{\Lambda}\in{D(\tilde{K},\tilde{K}_{0})}$
is expressed as
\begin{equation*}
\tilde{\Lambda}=
\tilde{p}_{1}y_{1}+\tilde{p}_{2}y_{2}+
\tilde{p}_{3}y_{3}+\tilde{p}_{4}y_{4}+
\tilde{p}_{5}y_{5},
\end{equation*}
where $\tilde{p}_{i}\in{\mathbf Z}$,
$\tilde{p}_{1}\geq\tilde{p}_{2}$,
$\tilde{p}_{3}\geq\tilde{p}_{4}\geq\tilde{p}_{5}$,
$\tilde{p}_3\geq 0$, $\tilde{p}_5 \leq 0$,
\begin{equation*}
\tilde{p}_{1}+\tilde{p}_{2}+
\tilde{p}_{3}+\tilde{p}_{4}+
\tilde{p}_{5}=0.
\end{equation*}
If $m=2$, then
each
$\tilde{\Lambda}\in{D(\tilde{K},\tilde{K}_{0})}$
is expressed as
\begin{equation*}
\tilde{\Lambda}=
\tilde{p}_{1}y_{1}+\tilde{p}_{2}y_{2}+
\tilde{p}_{3}y_{3}+\tilde{p}_{4}y_{4},
\end{equation*}
where $\tilde{p}_{i}\in{\mathbf Z}$,
$\tilde{p}_{1}\geq\tilde{p}_{2}$,
$\tilde{p}_{3}\geq\tilde{p}_{4}$,
$\tilde{p}_{1}+\tilde{p}_{2}+
\tilde{p}_{3}+\tilde{p}_{4}=0$.

Correspondingly, each
$\tilde{\Lambda}^{\prime}\in{D(\tilde{K}_{2},\tilde{K}_{0})}$
is expressed as
$\tilde{\Lambda}^{\prime}=
\tilde{p}_{1}y_{1}+\tilde{p}_{2}y_{2}+
\tilde{q}_{3}y_{3}+\tilde{q}_{4}y_{4}$,
where
$\tilde{p}_{1},\tilde{p}_{2},\tilde{q}_{3},\tilde{q}_{4}\in{\mathbf Z}$,
$\tilde{p}_{1}\geq\tilde{p}_{2}$,
$\tilde{q}_{3}\geq\tilde{q}_{4}$,
$\tilde{p}_{1}+\tilde{p}_{2}+\tilde{q}_{3}+\tilde{q}_{4}=0$,
in other words,
$\tilde{p}_{1}+\tilde{p}_{2}+\tilde{p}_{3}+\tilde{p}_{4}+\tilde{p}_{m+1}+\tilde{p}_{m+2}=0$
if $m\geq{4}$,
$\tilde{p}_{1}+\tilde{p}_{2}+\tilde{p}_{3}+\tilde{p}_{4}+\tilde{p}_{5}=0$ if $m=3$.
Each
$\tilde{\Lambda}^{\prime\prime}\in{D(\tilde{K}_{1},\tilde{K}_{0})}$
is expressed as
$\tilde{\Lambda}^{\prime\prime}=
\tilde{q}^{\prime}_{1}y_{1}+\tilde{q}^{\prime}_{2}y_{2}+
\tilde{q}^{\prime}_{3}y_{3}+\tilde{q}^{\prime}_{4}y_{4}$,
where
$\tilde{q}^{\prime}_{1},\tilde{q}^{\prime}_{2},
\tilde{q}^{\prime}_{3}, \tilde{q}^{\prime}_{4}\in{\mathbf Z}$,
$\tilde{q}^{\prime}_{1}+\tilde{q}^{\prime}_{3}=0$,
$\tilde{q}^{\prime}_{2}+\tilde{q}^{\prime}_{4}=0$,
$\tilde{q}^{\prime}_{1}=-\alpha+\tilde{p}_{1}$,
$\tilde{q}^{\prime}_{2}=\alpha+\tilde{p}_{2}$
for some $\alpha=0,\cdots,\tilde{p}_{1}-\tilde{p}_{2}$,
and
$\tilde{q}^{\prime}_{3}=-\beta+\tilde{q}_{3}$,
$\tilde{q}^{\prime}_{4}=\beta+\tilde{q}_{4}$
for some $\beta=0,\cdots,\tilde{q}_{3}-\tilde{q}_{4}$.

Moreover the coefficient of $X^{\tilde{q}_3 - \tilde{q}_4 +1}$ in
\begin{equation*}
\begin{split}
&
\frac{1}{X-X^{-1}}
(X^{\tilde{p}_{3}-\tilde{p}_{4}+1}-X^{-(\tilde{p}_{3}-\tilde{p}_{4}+1)})
(X^{\tilde{p}_{m+1}-\tilde{p}_{m+2}+1}-X^{-(\tilde{p}_{m+1}-\tilde{p}_{m+2}+1)})\\
=&
\sum^{\tilde{p}_{3}-\tilde{p}_{4}}_{i=0}
(X^{(\tilde{p}_{3}-\tilde{p}_{4})+(\tilde{p}_{m+1}-\tilde{p}_{m+2})-2i+1}
-
X^{(\tilde{p}_{3}-\tilde{p}_{4})-(\tilde{p}_{m+1}-\tilde{p}_{m+2})-2i-1})
\end{split}
\end{equation*}
is equal to the multiplicity of
the $\tilde{K}_2$-module with the highest weight
$\tilde{\Lambda}^{\prime}=
\Lambda_{\sigma}+\tilde{\Lambda}^{\prime}_{\tau}
=
\tilde{p}_{1}y_{1}+\tilde{p}_{2}y_{2}+
\tilde{q}_{3}y_{3}+\tilde{q}_{4}y_{4}\in D(\tilde{K}_{2},\tilde{K}_{0})$.

\subsection{Eigenvalue computation when $m=2$}

For each
$\tilde{\Lambda}=\tilde{p}_{1}y_{1}+\tilde{p}_{2}y_{2}+\tilde{p}_{3}y_{3}+\tilde{p}_{4}y_{4}
\in D(\tilde{K}, \tilde{K_0})$ and
$\tilde{\Lambda}^{\prime\prime}=\tilde{q}^{\prime}_{1}y_{1}+\tilde{q}^{\prime}_{2}y_{2}+
\tilde{q}^{\prime}_{3}y_{3}+\tilde{q}^{\prime}_{4}y_{4}\in D(\tilde{K}_{1},\tilde{K}_{0})$
defined as above,
the corresponding eigenvalue of $-{\mathcal C}_{L}$ is
\begin{equation}\label{EigenvalueFormulaAIII_2 m=2}
\begin{split}
-c_L=&-c_{\tilde{\Lambda}}+\frac{1}{2}c_{\tilde{\Lambda}^{\prime\prime}}\\
=&{\ }
\tilde{p}_{1}^{2}+\tilde{p}_{2}^{2}
+\tilde{p}_{3}^{2}+\tilde{p}_{4}^{2}
+(\tilde{p}_{1}-\tilde{p}_{2})
+(\tilde{p}_{3}-\tilde{p}_{4}) \\
&
-\frac{1}{2}(
(\tilde{q}^{\prime}_{1})^{2}+(\tilde{q}^{\prime}_{2})^{2}
+(\tilde{q}^{\prime}_{3})^{2}+(\tilde{q}^{\prime}_{4})^{2}).
\end{split}
\end{equation}

Since
\begin{equation*}
-{\mathcal C}_L=-\frac{1}{2}{\mathcal C}_{K/K_0}-{\mathcal C}_{K/K_1}
\geq -\frac{1}{2}{\mathcal C}_{K/K_0},
\end{equation*}
the first eigenvalue of $-{\mathcal C}_L$,
$-c_L \leq n=6$ implies $-c_{\tilde{\Lambda}} \leq 12$.
By estimating the eigenvalue formula \eqref{EigenvalueFormulaAIII_2 m=2}
from above by $6$, we compute
\begin{lem}
$\tilde{\Lambda}=\tilde{p}_1 y_1+ \tilde{p}_2 y_2 +\tilde{p}_3 y_3+ \tilde{p}_4 y_4
\in D(\tilde{K}, \tilde{K}_0)$
has eigenvalue
$-c_L\leq 6$ if and only if
$(\tilde{p}_1, \tilde{p}_2, \tilde{p}_3, \tilde{p}_4)$ is one of
\begin{equation*}
\begin{split}
\Bigl\{
& (0,0,0,0), (1,1,-1,-1), (1,0,0,-1), (1,-1,0,0), (1,-1,1,-1), \\
& (1,1,0,-2), (2,0,-1,-1), (0,-1,1,0), (0,0,1,-1), \\
& (0,-2,1,1), (-1,-1, 2,0), (-1,-1,1,1)\Bigr\}.
\end{split}
\end{equation*}
\end{lem}
Denote
$\tilde{\Lambda}=\tilde{p}_1 y_1+ \tilde{p}_2 y_2 +\tilde{p}_3 y_3+ \tilde{p}_4 y_4 \in D(\tilde{K}, \tilde{K}_0)$
by
$\tilde{\Lambda}=(\tilde{p}_1, \tilde{p}_2, \tilde{p}_3, \tilde{p}_4)$.

Suppose that $\tilde{\Lambda}=(1,1,-1,-1)$.
Then $\dim_{\mathbf C} V_{\tilde{\Lambda}}= 1$.
By the branching law of $(U(2),U(1)\times U(1))$,
$(\tilde{q}^{\prime}_1, \tilde{q}^{\prime}_2, \tilde{q}^{\prime}_3, \tilde{q}^{\prime}_4)=(1,1,-1,-1)\in
D(\tilde{K}_1, \tilde{K}_0)$.
Then
$-c_{\tilde{\Lambda}}=4$,
$-c_{\tilde{\Lambda}^{\prime\prime}}=4$,
$-c_L=-c_{\tilde{\Lambda}}+ \frac{1}{2} c_{\tilde{\Lambda}^{\prime\prime}}= 2 <6$.
On the other hand, $V_{\tilde{\Lambda}}=\mathbf{C}\boxtimes \mathbf{C}$,
which is fixed by the $\rho_{\tilde{\Lambda}}|_{\tilde{K}_0}$-action.
But for a generator $Q$
of $\mathbf{Z}_4$ in $\tilde{K}_{[\mathfrak{a}]}$,
$\rho_{\tilde{\Lambda}}(Q)=-\rm{Id}$ on $V_{\tilde{\Lambda}}$.
Hence, $\tilde{\Lambda}=(1,1,-1,-1)\not\in D(\tilde{K}, \tilde{K}_{[\mathfrak{a}]})$.
Similarly, $\tilde{\Lambda}=(-1,-1,1,1)\not\in D(\tilde{K}, \tilde{K}_{[\mathfrak{a}]})$.

\smallskip

Suppose that $\tilde{\Lambda}=(1,0,0,-1)$.
Then $\dim_{\mathbf C} V_{\tilde{\Lambda}}= 4$.
It follows from the branching law of $(U(2), U(1)\times U(1))$
that $(\tilde{q}^{\prime}_1, \tilde{q}^{\prime}_2)=(1,0)\oplus (0,1)$
and $(\tilde{q}^{\prime}_3, \tilde{q}^{\prime}_4)=(0,-1)$ or $(-1,0)$.
Then $(\tilde{q}^{\prime}_1, \tilde{q}^{\prime}_2, \tilde{q}^{\prime}_3, \tilde{q}^{\prime}_4)=(1,0,-1,0)$
or $(0,1,0,-1)\in D(\tilde{K}_1, \tilde{K}_0)$.
Hence,
$-c_{\tilde{\Lambda}}=4$,
$-c_{\tilde{\Lambda}^{\prime\prime}}=2$,
$-c_L=-c_{\tilde{\Lambda}}+ \frac{1}{2} c_{\tilde{\Lambda}^{\prime\prime}}= 3 <6$.

Let
$\mathcal{D}(SU(2))=\{(V_{\ell}, \rho_{\ell})\mid \ell \in \mathbf{Z}, \ell >0\}$
be a complete set of inequivalent irreducible unitary representations of
$SU(2)$ described in Section \ref{Section_G2 SO4}.
Let $\{v_0^{(\ell)}, v_1^{(\ell)}, \cdots, v_{\ell}^{(\ell)}\}$ be a unitary basis of $V_{\ell}$ defined by \eqref{eq:basis_SU(2)}.
Then
\begin{equation*}
V_{\tilde{\Lambda}}=(W^{\prime}_{\frac{1}{2}(y_1+y_2)}\otimes V_1)
\boxtimes(W^{\prime}_{-\frac{1}{2}(y_1+y_2)}\otimes V_1).
\end{equation*}
The representation of $\tilde{K}_0$ on
$v^{(1)}_{i}\otimes v^{(1)}_j\in V_{\tilde{\Lambda}}\ (i,j=0,1)$ is given by
\begin{equation*}
\begin{split}
&\rho_{\tilde{\Lambda}}(P)
(v^{(1)}_{i} \otimes v^{(1)}_{j})\\
=&
\left[\rho_{1}
\begin{pmatrix}
e^{\frac{\sqrt{-1}(s-t)}{2}}&  \\
 & e^{-\frac{\sqrt{-1}(s-t)}{2}}
\end{pmatrix}
\right](v^{(1)}_{i})
\otimes \left[\rho_{1}
\begin{pmatrix}
e^{\frac{\sqrt{-1}(s-t)}{2}}&  \\
 & e^{-\frac{\sqrt{-1}(s-t)}{2}}
\end{pmatrix}
\right](v^{(1)}_{j})\\
=& e^{\sqrt{-1}(s-t)[1-(i+j)]} v^{(1)}_{i}\otimes v^{(1)}_j.
\end{split}
\end{equation*}
Then
$(V_{\tilde{\Lambda}})_{\tilde{K}_0}=
\mathrm{span}_{\mathbf C}
\{ v^{(1)}_1\otimes v^{(1)}_0, v^{(1)}_0\otimes v^{(1)}_1\}$.
But for $\mathrm{diag}(1,1,-1,-1)\in \tilde{K}_{[\mathfrak a]}$ and $i,j=0,1$,
$\rho_{\tilde{\Lambda}}(\mathrm{diag}(1,1,-1,-1))(v^{(1)}_i\otimes v^{(1)}_j)
= - v^{(1)}_i \otimes v^{(1)}_j$.
So $(V_{\tilde{\Lambda}})_{\tilde{K}_{[\mathfrak a]}}=\{0\}$
and
$\tilde{\Lambda}=(1,0,0,-1) \notin D(\tilde{K}, \tilde{K}_{[\mathfrak a]})$.
Similarly, $\tilde{\Lambda}=(0,-1, 1,0) \notin
D(\tilde{K}, \tilde{K}_{[\mathfrak a]})$.

\smallskip

Suppose that $\tilde{\Lambda}=(1,-1,0,0)$.
Then $\dim_{\mathbf C} V_{\tilde{\Lambda}}=3$.
It follows from the branching law of $(U(2), U(1)\times U(1))$ that
$(\tilde{q}^{\prime}_1, \tilde{q}^{\prime}_2)=
(1,-1), (0,0)\text{ or } (-1,1)$
and $(\tilde{q}^{\prime}_3, \tilde{q}^{\prime}_4)=(0,0)$.
Then
$(\tilde{q}^{\prime}_1, \tilde{q}^{\prime}_2, \tilde{q}^{\prime}_3,
\tilde{q}^{\prime}_4)=(0,0,0,0) \in D(\tilde{K}, \tilde{K}_0)$.
Hence,
$-c_{\tilde{\Lambda}}=4$,
$-c_{\tilde{\Lambda}^{\prime\prime}}=0$,
$-c_L=-c_{\tilde{\Lambda}}+ \frac{1}{2} c_{\tilde{\Lambda}^{\prime\prime}}= 4 < 6$.
On the other hand,
$V_{\tilde{\Lambda}}\cong V_2 \boxtimes \mathbf{C}$.
The representation of $\tilde{K}_0$ on
$v^{(2)}_{i}\otimes w \in V_{\tilde{\Lambda}}$
is given by
\begin{equation*}
\rho_{\tilde{\Lambda}}(P)
(v^{(2)}_{i} \otimes w)
=e^{\sqrt{-1}(s-t)(1-i)} v^{(2)}_{i}\otimes w.
\end{equation*}
Then
$(V_{\tilde{\Lambda}})_{\tilde{K}_0}
={\rm span}_{\mathbf C}\{ v^{(2)}_1\otimes w\}$.
But for
the generator $Q \in \tilde{K}_{[\mathfrak a]}$,
$$\rho_{\tilde{\Lambda}}(Q)
(v^{(2)}_1\otimes w)= - v^{(2)}_1 \otimes w.$$
So $(V_{\tilde{\Lambda}})_{\tilde{K}_{[\mathfrak a]}}=\{0\}$
and
$\tilde{\Lambda}=(1,-1,0,0) \notin D(\tilde{K}, \tilde{K}_{[\mathfrak a]})$.
Similarly, $\tilde{\Lambda}=(0,0,1,-1) \notin D(\tilde{K}, \tilde{K}_{[\mathfrak a]})$.

\smallskip

Suppose that $\tilde{\Lambda}=(1,-1,1,-1)$.
Then $\dim_{\mathbf C} V_{\tilde{\Lambda}}=9$.
It follows from the branching laws of $(U(2), U(1)\times U(1))$
that $(\tilde{q}^{\prime}_1, \tilde{q}^{\prime}_2)=(1,-1)\text{ or } (0,0)$
and $(\tilde{q}^{\prime}_3, \tilde{q}^{\prime}_4)=(1,-1)\text{ or } (0,0)$.
Then $(\tilde{q}^{\prime}_1, \tilde{q}^{\prime}_2, \tilde{q}^{\prime}_3, \tilde{q}^{\prime}_4)=(1,-1,-1,1)$,
$(-1,1,1,-1)$ or $(0,0,0,0)$  $\in D(\tilde{K}, \tilde{K}_0)$.
When $(\tilde{q}^{\prime}_1, \tilde{q}^{\prime}_2, \tilde{q}^{\prime}_3, \tilde{q}^{\prime}_4)=(0,0,0,0)$,
$-c_{\tilde{\Lambda}}=8$,
$-c_{\tilde{\Lambda}^{\prime\prime}}=0$,
$-c_L=-c_{\tilde{\Lambda}}+\frac{1}{2} c_{\tilde{\Lambda}^{\prime\prime}}= 8 >6$.
When $(\tilde{q}^{\prime}_1, \tilde{q}^{\prime}_2, \tilde{q}^{\prime}_3, \tilde{q}^{\prime}_4)=(1,-1,-1,1)$ or
$(-1,1,1,-1)$,
$-c_{\tilde{\Lambda}}=8$,
$-c_{\tilde{\Lambda}^{\prime\prime}}=4$,
$-c_L=-c_{\tilde{\Lambda}}+ \frac{1}{2} c_{\tilde{\Lambda}^{\prime\prime}}= 6$.
On the other hand,
$V_{\tilde{\Lambda}}\cong V_2 \boxtimes V_2$.
The representation of $\tilde{K}_0$ on
$v^{(2)}_{i}\otimes v^{(2)}_j\in V_{\tilde{\Lambda}}$
$(i,j=0,1,2)$  is given by
\begin{equation*}
\begin{split}
&\rho_{\tilde{\Lambda}}(P)
(v^{(2)}_{i} \otimes v^{(2)}_{j})\\
=&
\left[\rho_{2}
\begin{pmatrix}
e^{\frac{\sqrt{-1}(s-t)}{2}}&  \\
 & e^{-\frac{\sqrt{-1}(s-t)}{2}}
\end{pmatrix}
\right](v^{(2)}_{i})
\otimes \left[\rho_{2}
\begin{pmatrix}
e^{\frac{\sqrt{-1}(s-t)}{2}}&  \\
 & e^{-\frac{\sqrt{-1}(s-t)}{2}}
\end{pmatrix}
\right](v_{j}^{(2)})\\
=& e^{\sqrt{-1}(s-t)[2-(i+j)]} v^{(2)}_{i}\otimes v^{(2)}_j.
\end{split}
\end{equation*}
Hence
$(V_{\tilde{\Lambda}})_{\tilde{K}_0}=
{\rm span}_{\mathbf C}\{ v^{(2)}_0\otimes v^{(2)}_2, v^{(2)}_1\otimes v^{(2)}_1,
v^{(2)}_2\otimes v^{(2)}_0\}$.
Moreover, the action of the generator
$Q$
of $\mathbf{Z}_4$ in $\tilde{K}_{[\mathfrak a]}$
on $v^{(2)}_i\otimes v^{(2)}_j$ is given by
\begin{equation*}
\rho_{\tilde{\Lambda}}(Q)
(v^{(2)}_i\otimes v^{(2)}_j)= (-1)^{3-i} v^{(2)}_{2-i} \otimes v^{(2)}_{2-j}.
\end{equation*}
Therefore,
$(V_{\tilde{\Lambda}})_{\tilde{K}_{[\mathfrak a]}}
={\rm span}\{v^{(2)}_0\otimes v^{(2)}_2 - v^{(2)}_2 \otimes v^{(2)}_0,
v^{(2)}_1\otimes v^{(2)}_1\}$
and $\tilde{\Lambda}=(1,-1,1,-1)\in D(\tilde{K}, \tilde{K}_{[\mathfrak a]})$.
Note that the $\tilde{K}_{[\mathfrak a]}$-fixed vector
$v^{(2)}_1\otimes v^{(2)}_1\in V^{\prime}_{0}$,
which corresponds eigenvalue $8$
and the $\tilde{K}_{[\mathfrak a]}$-fixed vector
$v^{(2)}_0\otimes v^{(2)}_2-v^{(2)}_2\otimes v^{(2)}_0
\in V^{\prime}_{y_1-y_2-y_3+y_4} \oplus V^{\prime}_{-y_1+y_2+y_3-y_4}$,
which gives eigenvalue $6$.

\smallskip

Suppose that $\tilde{\Lambda}=(2, 0,-1,-1)$.
Then $\dim_{\mathbf C} V_{\tilde{\Lambda}}=3$.
It follows from the branching law of $(U(2), U(1)\times U(1))$
that $(\tilde{q}^{\prime}_1, \tilde{q}^{\prime}_2)=(2,0), (1,1)\text{ or } (0,2)$
and $(\tilde{q}^{\prime}_3, \tilde{q}^{\prime}_4)=(-1,-1)$.
Then $(\tilde{q}^{\prime}_1, \tilde{q}^{\prime}_2, \tilde{q}^{\prime}_3,
\tilde{q}^{\prime}_4)=(1, 1,-1,-1)
\in D(\tilde{K}, \tilde{K}_0)$.
Hence,
$-c_{\tilde{\Lambda}}=8$,
$-c_{\tilde{\Lambda}^{\prime\prime}}=4$,
$-c_L=-c_{\tilde{\Lambda}}+\frac{1}{2} c_{\tilde{\Lambda}^{\prime\prime}}= 6$.
On the other hand,
\begin{equation*}
V_{\tilde{\Lambda}}\cong (V_2 \otimes \mathbf{C})\boxtimes \mathbf{C}.
\end{equation*}
The representation of $\tilde{K}_0$ on
$v^{(2)}_{i}\otimes w \in V_{\tilde{\Lambda}}$
$(i=0,1,2)$  is given by
\begin{equation*}
\begin{split}
&\rho_{\tilde{\Lambda}}(P)
(v^{(2)}_{i} \otimes w)\\
=&\,
e^{\sqrt{-1}(s+t)}
\left[\rho_{2}
\begin{pmatrix}
e^{\frac{\sqrt{-1}(s-t)}{2}}&  \\
 & e^{-\frac{\sqrt{-1}(s-t)}{2}}
\end{pmatrix}
\right](v^{(2)}_{i})
\otimes e^{-\sqrt{-1}(s+t)}w \\
=&\, e^{\sqrt{-1}(s-t)(1-i)} v^{(2)}_{i}\otimes w.
\end{split}
\end{equation*}
Hence,
$(V_{\tilde{\Lambda}})_{\tilde{K}_0}=
{\rm span}_{\mathbf C}\{ v^{(2)}_1 \otimes 1\}.$
Moreover, the action of the generator $Q$
of $\mathbf{Z}_4$ in $\tilde{K}_{[\mathfrak a]}$
on $v^{(2)}_i\otimes w$ is given by
$\rho_{\tilde{\Lambda}}(Q)
(v^{(2)}_i\otimes 1)= (-1)^{1-i} v^{(2)}_{2-i} \otimes 1.
$
Therefore,
$(V_{\tilde{\Lambda}})_{\tilde{K}_{[\mathfrak a]}}=
{\rm span}\{v^{(2)}_1 \otimes 1\}$
and $\tilde{\Lambda}=(2,0, -1,-1)\in D(\tilde{K}, \tilde{K}_{[\mathfrak a]})$,
which gives eigenvalue $6$.
Similarly, $\tilde{\Lambda}=(-1,-1,2,0)$, $(1,1,0,-2)$, $(0,-2,1,1) \in D(\tilde{K}, \tilde{K}_{[\mathfrak a]})$,
which give eigenvalue $6$ and with multiplicity $1$, respectively.

Moreover we observe that
\begin{equation*}
\begin{split}
n(L^6)
=&\dim_{\mathbf C}V_{(2,0,-1,-1)}+\dim_{\mathbf C}V_{(-1,-1,2,0)}
+\dim_{\mathbf C}V_{(1,1,0,-2)}\\
&+\dim_{\mathbf C}V_{(0,-2,1,1)}+\dim_{\mathbf{C}}V_{(1,-1,1,-1)}=3+3+3+3+9\\
=&21=\dim SO(8)-\dim S(U(2)\times U(2))=n_{hk}(L^6).
\end{split}
\end{equation*}
Therefore we obtain that
$L^6=\mathcal{G}(\frac{S(U(2)\times U(2))}{S(U(1)\times U(1))})
\subset Q_6(\mathbf{C})$ is strictly Hamiltonian stable.

\subsection{Eigenvalue computation when $m=3$}
For each $\tilde{\Lambda}=
\tilde{p}_{1}y_{1}+\tilde{p}_{2}y_{2}+
\tilde{p}_{3}y_{3}+\tilde{p}_{4}y_{4}+
\tilde{p}_{5}y_{5}\in D(\tilde{K},\tilde{K}_0)$,
$\tilde{\Lambda}^{\prime}=\tilde{p}_{1}y_{1}+\tilde{p}_{2}y_{2}+
\tilde{q}_{3}y_{3}+\tilde{q}_{4}y_{4}\in{D(\tilde{K}_{2},\tilde{K}_{0})}$
and
$\tilde{\Lambda}^{\prime\prime}=\tilde{q}^{\prime}_{1}y_{1}+\tilde{q}^{\prime}_{2}y_{2}+
\tilde{q}^{\prime}_{3}y_{3}+\tilde{q}^{\prime}_{4}y_{4}\in D(\tilde{K}_{1},\tilde{K}_{0})$
given as in Subsection \ref{D(K,K_0)AIII},
the corresponding eigenvalue of $-{\mathcal C}_{L}$ is
\begin{equation}
\begin{split}
-c_L=&-2c_{\tilde{\Lambda}}+c_{\tilde{\Lambda}^{\prime}}
+\frac{1}{2}c_{\tilde{\Lambda}^{\prime\prime}}\\
=&{\ }
\tilde{p}_{1}^{2}+\tilde{p}_{2}^{2}
+2(\tilde{p}_{3}^{2}+\tilde{p}_{4}^{2}+
\tilde{p}_{5}^{2})
+(\tilde{p}_{1}-\tilde{p}_{2})
+4(\tilde{p}_{3}-\tilde{p}_{5})
 \\
&
-(\tilde{q}_{3}^{2}+\tilde{q}_{4}^{2})
-(\tilde{q}_{3}-\tilde{q}_{4})
-\frac{1}{2}(
(\tilde{q}^{\prime}_{1})^{2}+(\tilde{q}^{\prime}_{2})^{2}
+(\tilde{q}^{\prime}_{3})^{2}+(\tilde{q}^{\prime}_{4})^{2}).
\end{split}
\end{equation}

$\tilde{\Lambda}=\tilde{p}_1 y_1+ \tilde{p}_2 y_2 +\tilde{p}_3 y_3+ \tilde{p}_4 y_4
+\tilde{p}_5 y_5 \in D(\tilde{K}, \tilde{K}_0)$
is denoted by
$\tilde{\Lambda}=(\tilde{p}_1, \tilde{p}_2, \tilde{p}_3, \tilde{p}_4, \tilde{p}_5)$.
Since $-{\mathcal C}_L \geq -\frac{1}{2}{\mathcal C}_{K/K_0}$,
the eigenvalue of $-{\mathcal C}_L$, $-c_L \leq n=10$
implies $-c_{\tilde{\Lambda}} \leq 20$.
It then follows that
\begin{lem}
$\tilde{\Lambda}=\tilde{p}_1 y_1+ \tilde{p}_2 y_2 +\tilde{p}_3 y_3+ \tilde{p}_4 y_4
+\tilde{p}_5 y_5 \in D(\tilde{K}, \tilde{K}_0)$
has eigenvalue $ -c_L\leq 10$ if and only if
$(\tilde{p}_1, \tilde{p}_2, \tilde{p}_3, \tilde{p}_4, \tilde{p}_5)$ is one of
\begin{equation*}
\begin{split}
\bigl
\{
& (0,0,0,0,0), (1,-1,1,0,-1), (2,0,0,-1,-1), (0,-2,1,1,0),\\
& (1,1,0,0,-2), (-1,-1,2,0,0),(1,-1,0,0,0), (1,0,0,0,-1), \\
& (0,-1,1,0,0),(1,1,0,-1,-1),  (-1,-1,1,1,0), (0,0,1,0,-1)
\bigr
\}.
\end{split}
\end{equation*}
\end{lem}

\smallskip

Suppose that $\tilde{\Lambda}=(1, -1, 1,0,-1)$.
Then $\dim_{\mathbf C} V_{\tilde{\Lambda}}=24$.
It follows from Lemma \ref{BranchingLaws(U(3),U(2)XU(1))}
that $(\tilde{q}_3, \tilde{q}_4, \tilde{q}_5)=(1,-1,0)\text{ or } (0,0,0)$.
When $(\tilde{q}_3, \tilde{q}_4, \tilde{q}_5)=(0,0,0)$,
by the branching law of $(U(2), U(1)\times U(1))$,
$(\tilde{q}^{\prime}_1, \tilde{q}^{\prime}_2,\tilde{q}^{\prime}_3, \tilde{q}^{\prime}_4, \tilde{q}^{\prime}_5)
=(0,0,0,0,0)$.
Hence,
$-c_L=-2c_{\tilde{\Lambda}}+c_{\tilde{\Lambda}^{\prime}}
+\frac{1}{2}c_{\tilde{\Lambda}^{\prime\prime}}
=16>10$.
When $(\tilde{q}_3, \tilde{q}_4, \tilde{q}_5)=(1,-1,0)$,
by the branching law of $(U(2), U(1)\times U(1))$,
$(\tilde{q}^{\prime}_1, \tilde{q}^{\prime}_2,\tilde{q}^{\prime}_3, \tilde{q}^{\prime}_4, \tilde{q}^{\prime}_5)
=(1,-1,-1,1,0)$, $(0,0,0,0,0)$ or $(-1,1,1,-1,0)\in D(\tilde{K}, \tilde{K}_0)$, respectively.
Hence,
$-c_L=-2c_{\tilde{\Lambda}}+c_{\tilde{\Lambda}^{\prime}}
+\frac{1}{2}c_{\tilde{\Lambda}^{\prime\prime}}
=10$, $12$ or
 $10$,
respectively.
On the other hand, now
\begin{equation*}
\begin{split}
(\tilde{V}_{\tilde{\Lambda}})_{\tilde{K}_0}
& \subset (W_{y_1-y_2} \boxtimes W_{y_3-y_4}\boxtimes W_0 )
\oplus (W_{y_1-y_2} \boxtimes W_0 \boxtimes W_0)\\
& \cong (V_2 \boxtimes V_2 \boxtimes \mathbf{C}) \oplus
(V_2 \boxtimes \mathbf{C} \boxtimes \mathbf{C}),
\end{split}
\end{equation*}
where the latter is a $\tilde{K}_2$-module.
The representation $\rho_{\tilde{\Lambda}}$ of $\tilde{K}_0$ on
$u_i\otimes v_j\otimes w \in  V_2 \boxtimes V_2 \boxtimes \mathbf{C}$
$(i,j=0,1,2)$ is given by
\begin{equation*}
\begin{split}
&\, \rho_{\tilde{\Lambda}}(P)(v^{(2)}_i\otimes v^{(2)}_j \otimes w)\\
=&\, \rho_{y_1-y_2} \left(
                    \begin{array}{cc}
                      e^{\sqrt{-1} s} &  \\
                       & e^{\sqrt{-1}t} \\
                    \end{array}
                  \right)
                  (v^{(2)}_i)
\otimes \rho_{y_3-y_4} \left(
                    \begin{array}{cc}
                      e^{\sqrt{-1} s} &  \\
                       & e^{\sqrt{-1}t} \\
                    \end{array}
                  \right)
                  (v^{(2)}_j) \otimes w \\
=&\, e^{\sqrt{-1}(s-t)(2-i-j)} v^{(2)}_i\otimes v^{(2)}_j \otimes w.
\end{split}
\end{equation*}
The representation $\rho_{\tilde{\Lambda}}$ of $\tilde{K}_0$ on
$v^{(2)}_i\otimes v\otimes w \in  V_2 \boxtimes \mathbf{C} \boxtimes \mathbf{C}$
$(i=0,1,2)$ is given by
\begin{equation*}
\begin{split}
\rho_{\tilde{\Lambda}}(P)(v^{(2)}_i\otimes v \otimes w)
=& \rho_{y_1-y_2} \left(
                    \begin{array}{cc}
                      e^{\sqrt{-1} s} &  \\
                       & e^{\sqrt{-1}t} \\
                    \end{array}
                  \right)
                  (v^{(2)}_i)
\otimes v \otimes w\\
=& e^{\sqrt{-1}(s-t)(1-i)}v^{(2)}_i \otimes v\otimes w.
\end{split}
\end{equation*}
Thus,
$(\tilde{V}_{\tilde{\Lambda}})_{\tilde{K}_0}={\rm span}_{\mathbf{C}}
\{v^{(2)}_2\otimes v^{(2)}_0\otimes w, v^{(2)}_0\otimes v^{(2)}_2\otimes w,
v^{(2)}_1\otimes v^{(2)}_1\otimes w,
v^{(2)}_1\otimes v\otimes w\}.$
Moreover, the action of the generator
$Q$
of $\mathbf{Z}_4$ in $\tilde{K}_{[\mathfrak a]}$
on $v^{(2)}_i\otimes v^{(2)}_{2-i}\otimes w$ is given by
\begin{equation*}
\begin{split}
 \rho_{\tilde{\Lambda}}(Q)
(v^{(2)}_i\otimes v^{(2)}_{2-i}\otimes w )
&=  \rho_2 \left(
             \begin{array}{cc}
               0 & 1 \\
               -1 & 0 \\
             \end{array}
           \right)(v^{(2)}_i)
\otimes \rho_2 \left(
             \begin{array}{cc}
               0 & \sqrt{-1} \\
               \sqrt{-1} & 0 \\
             \end{array}
           \right) (v^{(2)}_{2-i}) \otimes w \\
&=(-1)^{1-i}u_{2-i}\otimes v^{(2)}_i \otimes w
\end{split}
\end{equation*}
and the action
on $v^{(2)}_i\otimes v \otimes w$ is given by
\begin{equation*}
\begin{split}
\rho_{\tilde{\Lambda}}(Q)
(v^{(2)}_i\otimes v \otimes w )
&= \rho_2 \left(
             \begin{array}{cc}
               0 & 1 \\
               -1 & 0 \\
             \end{array}
           \right)(v^{(2)}_i)
\otimes v \otimes w\\
&=(-1)^{2-i}v^{(2)}_{2-i}\otimes v \otimes w.
\end{split}
\end{equation*}
Therefore,
$(\tilde{V}_{\tilde{\Lambda}})_{\tilde{K}_{[\mathfrak a]}}={\rm span}_{\mathbf{C}}
\{v^{(2)}_2\otimes v^{(2)}_0\otimes w -v^{(2)}_0\otimes v^{(2)}_2\otimes w, v^{(2)}_1\otimes v^{(2)}_1\otimes w \}$
and $\tilde{\Lambda}=(1, -1, 1, 0,-1)\in D(\tilde{K}, \tilde{K}_{[\mathfrak a]})$.
Notice that the $\tilde{K}_{[\mathfrak a]}$-fixed vector
$v^{(2)}_1\otimes v^{(2)}_1\otimes w \in V_{\tilde{\Lambda}^{\prime\prime}}$,
which corresponds eigenvalue $12$, where $\tilde{\Lambda}^{\prime\prime}=0$.
And the $\tilde{K}_{[\mathfrak a]}$-fixed vector
$v^{(2)}_2\otimes v^{(2)}_0\otimes w -v^{(2)}_0\otimes v^{(2)}_2\otimes w
\in V_{\tilde{\Lambda}_1^{\prime\prime}} \oplus V_{\tilde{\Lambda}_2^{\prime\prime}}$,
which gives eigenvalue $10$, where $\tilde{\Lambda}_1^{\prime\prime}=(1,-1,-1,1,0)$
and $\tilde{\Lambda}_2^{\prime\prime}=(-1,1,1,-1,0)$.

\smallskip

Suppose that $\tilde{\Lambda}=(2, 0, 0,-1,-1)$.
Then $\dim_{\mathbf C} V_{\tilde{\Lambda}}=9$.
It follows from the branching law of $(U(3), U(2)\times U(1))$
that $(\tilde{q}_3, \tilde{q}_4, \tilde{q}_5)=(0,-1,-1)\text{ or } (-1,-1,0)$.
When $(\tilde{q}_3, \tilde{q}_4, \tilde{q}_5)=(-1,-1,0)$,
by the branching law of $(U(2), U(1)\times U(1))$,
$(\tilde{q}^{\prime}_1, \tilde{q}^{\prime}_2,\tilde{q}^{\prime}_3, \tilde{q}^{\prime}_4, \tilde{q}^{\prime}_5)
=(1,1,-1,-1,0)$.
Hence,
$-c_L=-2c_{\tilde{\Lambda}}+c_{\tilde{\Lambda}^{\prime}}
+\frac{1}{2}c_{\tilde{\Lambda}^{\prime\prime}}
=10$.
On the other hand,
\begin{equation*}
\begin{split}
(\tilde{V}_{\tilde{\Lambda}})_{\tilde{K}_0}
& \subset (W_{2y_1} \boxtimes W_{-(y_3+y_4)}\boxtimes W_0 )
 \cong V_2 \boxtimes \mathbf{C} \boxtimes \mathbf{C},
\end{split}
\end{equation*}
and the representation $\rho_{\tilde{\Lambda}}$ of $\tilde{K}_0$ on
$v^{(2)}_i\otimes v\otimes w \in  V_2 \boxtimes \mathbf{C} \boxtimes \mathbf{C}$
$(i=0,1,2)$ is given by
\begin{equation*}
\begin{split}
& \rho_{\tilde{\Lambda}}(P)(v^{(2)}_i\otimes v \otimes w)\\
=& \rho_{2y_1} \left(
                    \begin{array}{cc}
                      e^{\sqrt{-1} s} &  \\
                       & e^{\sqrt{-1}t} \\
                    \end{array}
                  \right)
                  (v^{(2)}_i)
\otimes \rho_{-y_3-y_4}\left(
                    \begin{array}{cc}
                      e^{\sqrt{-1} s} &  \\
                       & e^{\sqrt{-1}t} \\
                    \end{array}
                  \right) (v) \otimes w\\
=& e^{\sqrt{-1}(s-t)(1-i)}v^{(2)}_i \otimes v\otimes w.
\end{split}
\end{equation*}
Thus,
$(\tilde{V}_{\tilde{\Lambda}})_{\tilde{K}_0}={\rm span}_{\mathbf{C}}
\{v^{(2)}_1\otimes v\otimes w\}$.
Moreover, the action of the generator
$Q$
of $\mathbf{Z}_4$ in $\tilde{K}_{[\mathfrak a]}$
on $v^{(2)}_i\otimes v\otimes w$ is given by
\begin{equation*}
\begin{split}
\rho_{\tilde{\Lambda}}(Q)
(v^{(2)}_i\otimes v \otimes w )
&= \rho_{2y_1} \left(
             \begin{array}{cc}
               0 & 1 \\
               -1 & 0 \\
             \end{array}
           \right)(v^{(2)}_i)
\otimes \rho_{-(y_3+y_4)}\left(
             \begin{array}{cc}
               0 & -1 \\
               -1 & 0 \\
             \end{array}
           \right)(v) \otimes w\\
&=(-1)^{1+i}\, v^{(2)}_{2-i}\otimes v \otimes w.
\end{split}
\end{equation*}
Therefore,
$(\tilde{V}_{\tilde{\Lambda}})_{\tilde{K}_{[\mathfrak a]}}={\rm span}_{\mathbf{C}}
\{v^{(2)}_1\otimes v \otimes w \}$,
where $\dim_{\mathbf{C}}(\tilde{V}_{\tilde{\Lambda}})_{\tilde{K}_{[\mathfrak a]}}=1$
 and $\tilde{\Lambda}=(2, 0, 0, -1,-1)\in D(\tilde{K}, \tilde{K}_{[\mathfrak a]})$,
which gives eigenvalue $10$.
Similarly, $\tilde{\Lambda}=(0,-2,1, 1, 0)\in D(\tilde{K}, \tilde{K}_{[\mathfrak a]})$,
which gives eigenvalue $10$ and with multiplicity $1$ and dimension $9$.

\smallskip

Suppose that $\tilde{\Lambda}=(1, 1, 0, 0,-2)$.
Then $\dim_{\mathbf C} V_{\tilde{\Lambda}}=6$.
It follows from Lemma \ref{BranchingLaws(U(3),U(2)XU(1))}
that $(\tilde{q}_3, \tilde{q}_4, \tilde{q}_5)=(0,0,-2), (0,-1,-1)$
or $(0,-2,0)$.
When $(\tilde{q}_3, \tilde{q}_4, \tilde{q}_5)=(0,-2,0)$,
by the branching law of $(U(2), U(1)\times U(1))$,
$(\tilde{q}^{\prime}_1, \tilde{q}^{\prime}_2,\tilde{q}^{\prime}_3, \tilde{q}^{\prime}_4,
\tilde{q}^{\prime}_5)=(1,1,-1,-1,0)$.
Hence,
$-c_L=-2c_{\tilde{\Lambda}}+c_{\tilde{\Lambda}^{\prime}}
+\frac{1}{2}c_{\tilde{\Lambda}^{\prime\prime}}
=10$.
On the other hand,
\begin{equation*}
\begin{split}
(\tilde{V}_{\tilde{\Lambda}})_{\tilde{K}_0}
& \subset W_{0} \boxtimes W_{-2y_4}\boxtimes W_0
 \cong \mathbf{C}  \boxtimes V_2 \boxtimes \mathbf{C}
\end{split}
\end{equation*}
and the representation $\rho_{\tilde{\Lambda}}$ of $\tilde{K}_0$ on
$u\otimes v^{(2)}_i\otimes w \in\mathbf{C} \boxtimes V_2 \boxtimes \mathbf{C}$
$(i=0,1,2)$ is given by
\begin{equation*}
\begin{split}
& \rho_{\tilde{\Lambda}}(P)(u \otimes v^{(2)}_i \otimes w)\\
=& \rho_{y_1+y_2} \left(
                    \begin{array}{cc}
                      e^{\sqrt{-1} s} &  \\
                       & e^{\sqrt{-1}t} \\
                    \end{array}
                  \right)
                  (u)
\otimes \rho_{-2y_4}\left(
                    \begin{array}{cc}
                      e^{\sqrt{-1} s} &  \\
                       & e^{\sqrt{-1}t} \\
                    \end{array}
                  \right) (v^{(2)}_i) \otimes w\\
=& e^{\sqrt{-1}(s-t)(1-i)}u \otimes v^{(2)}_i \otimes w.
\end{split}
\end{equation*}
Thus,
$(\tilde{V}_{\tilde{\Lambda}})_{\tilde{K}_0}={\rm span}_{\mathbf{C}}
\{u \otimes v^{(2)}_1\otimes w\}.$
Moreover, the action of the generator $Q$
of $\mathbf{Z}_4$ in $\tilde{K}_{[\mathfrak a]}$
on $u\otimes v^{(2)}_i\otimes w$ is given by
\begin{equation*}
\begin{split}
& \rho_{\tilde{\Lambda}}(Q)
(u\otimes v^{(2)}_i \otimes w )\\
= & \rho_{y_1+y_2} \left(
             \begin{array}{cc}
               0 & 1 \\
               -1 & 0 \\
             \end{array}
           \right)(u)
\otimes \rho_{-2y_4}\left(
             \begin{array}{cc}
               0 & -1 \\
               -1 & 0 \\
             \end{array}
           \right)(v^{(2)}_i) \otimes w
=u \otimes v^{(2)}_{2-i} \otimes w.
\end{split}
\end{equation*}
Therefore,
$(\tilde{V}_{\tilde{\Lambda}})_{\tilde{K}_{[\mathfrak a]}}={\rm span}_{\mathbf{C}}
\{u \otimes v^{(2)}_1 \otimes w \}$,
where $\dim_{\mathbf{C}}(\tilde{V}_{\tilde{\Lambda}})_{\tilde{K}_{[\mathfrak a]}}=1 $
 and $\tilde{\Lambda}=(1, 1, 0, 0,-2)\in D(\tilde{K}, \tilde{K}_{[\mathfrak a]})$,
which gives eigenvalue $10$.
Similarly, $\tilde{\Lambda}=(-1,-1,2,0, 0)\in D(\tilde{K}, \tilde{K}_{[\mathfrak a]})$,
gives eigenvalue $10$ and has multiplicity $1$ and dimension $6$.

\smallskip

Suppose that $\tilde{\Lambda}=(1, -1, 0, 0, 0)$.
Then $(\tilde{q}_1, \tilde{q}_2, \tilde{q}_3, \tilde{q}_4, \tilde{q}_5)=(1,-1,0,0,0)$.
It follows from the branching law of $(U(2), U(1)\times U(1))$,
$(\tilde{q}^{\prime}_1, \tilde{q}^{\prime}_2,\tilde{q}^{\prime}_3, \tilde{q}^{\prime}_4, \tilde{q}^{\prime}_5)
=(0,0,0,0,0) \in D(\tilde{K}_1, \tilde{K}_0)$.
Hence,
$-c_L=-2c_{\tilde{\Lambda}}+c_{\tilde{\Lambda}^{\prime}}
+\frac{1}{2}c_{\tilde{\Lambda}^{\prime\prime}}
=4<10$.
On the other hand,
$\tilde{V}_{\tilde{\Lambda}}=
W_{y_1-y_2} \boxtimes W_{0}
 \cong  V_2 \boxtimes \mathbf{C}$
and the representation $\rho_{\tilde{\Lambda}}$ of $\tilde{K}_0$ on
$u_i \otimes w \in   V_2 \boxtimes \mathbf{C}$
$(i=0,1,2)$ is given by
\begin{equation*}
\begin{split}
 \rho_{\tilde{\Lambda}}(P)(u_i \otimes w)
&= \rho_{2} \left(
                    \begin{array}{cc}
                      e^{\sqrt{-1} \frac{s-t}{2}} &  \\
                       & e^{-\sqrt{-1}\frac{s-t}{2}} \\
                    \end{array}
                  \right)
                  (u_i)
\otimes  w\\
&= e^{\sqrt{-1}(s-t)(1-i)}u_i \otimes w.
\end{split}
\end{equation*}
Thus,
$(\tilde{V}_{\tilde{\Lambda}})_{\tilde{K}_0}={\rm Span}_{\mathbf{C}}
\{u_1\otimes w\}$.
Moreover, the action of the generator $Q$
of $\mathbf{Z}_4$ in $\tilde{K}_{[\mathfrak a]}$
on $u_1\otimes w$ is given by
$\rho_{\tilde{\Lambda}}(Q)
(u_1 \otimes w )=-u_1 \otimes w.$
Therefore,
$(\tilde{V}_{\tilde{\Lambda}})_{\tilde{K}_{[\mathfrak a]}}=\{0\}$
 and $\tilde{\Lambda}=(1, -1, 0, 0, 0)\not\in D(\tilde{K}, \tilde{K}_{[\mathfrak a]})$.

\smallskip

Suppose that $\tilde{\Lambda}=(1, 0, 0, 0, -1)$.
It follows from Lemma \ref{BranchingLaws(U(3),U(2)XU(1))}
that $(\tilde{q}_3, \tilde{q}_4, \tilde{q}_5)=(0,0,-1)\text{ or } (0,-1,0)$.
When $(\tilde{q}_3, \tilde{q}_4, \tilde{q}_5)=(0,-1,0)$,
by the branching law of $(U(2), U(1)\times U(1))$,
$(\tilde{q}^{\prime}_1, \tilde{q}^{\prime}_2,\tilde{q}^{\prime}_3, \tilde{q}^{\prime}_4, \tilde{q}^{\prime}_5)
=(1,0,-1,0,0)\text{ or } (0,1,0,-1,0) \in D(\tilde{K}_1, \tilde{K}_0)$.
Hence,
$-c_L=-2c_{\tilde{\Lambda}}+c_{\tilde{\Lambda}^{\prime}}
+\frac{1}{2}c_{\tilde{\Lambda}^{\prime\prime}}
=5, 5 <10$.
On the other hand,
$
(\tilde{V}_{\tilde{\Lambda}})_{\tilde{K}_0}
\subset W_{y_1} \boxtimes W_{-y_4}\boxtimes W_0
\cong V_1 \boxtimes V_1 \boxtimes \mathbf{C},
$
where the latter is the $\tilde{K}_2=U(2)\times U(2)\times U(1)$-module.
The representation $\rho_{\tilde{\Lambda}}$ of $\tilde{K}_0$ on
$v^{(1)}_i \otimes v^{(1)}_j \otimes w \in   V_1\boxtimes V_1 \boxtimes \mathbf{C}$
$(i, j=0,1)$ is given by
\begin{equation*}
 \rho_{\tilde{\Lambda}}(P)(v^{(1)}_i \otimes v^{(1)}_j \otimes w)
=e^{\sqrt{-1}(s-t)(1-i-j)}\, v^{(1)}_i \otimes v^{(1)}_j \otimes w.
\end{equation*}
Thus,
$(\tilde{V}_{\tilde{\Lambda}})_{\tilde{K}_0}={\rm span}_{\mathbf{C}}
\{v^{(1)}_1\otimes v^{(1)}_0 \otimes w, u_0 \otimes v^{(1)}_1 \otimes w\}.$
Moreover, the action of the generator
$Q$
of $\mathbf{Z}_4$ in $\tilde{K}_{[\mathfrak a]}$
on $v^{(1)}_i \otimes v^{(1)}_{1-i}\otimes w$ $(i=0,1)$ is given by
\begin{equation*}
\rho_{\tilde{\Lambda}}(Q)
(v^{(1)}_i \otimes v^{(1)}_{1-i} \otimes w )
=(-1)^{1-i}\, v^{(1)}_{1-i}\otimes v^{(1)}_{i} \otimes w.
\end{equation*}
Therefore,
$(\tilde{V}_{\tilde{\Lambda}})_{\tilde{K}_{[\mathfrak a]}}=\{0\}$
 and $\tilde{\Lambda}=(1, 0, 0, 0, -1)\not\in D(\tilde{K}, \tilde{K}_{[\mathfrak a]})$.
Similarly,
$\tilde{\Lambda}=(0,-1,1,0,0)\not\in D(\tilde{K}, \tilde{K}_{[\mathfrak a]})$.

\smallskip

Suppose that $\tilde{\Lambda}=(1, 1, 0,-1, -1)$.
It follows from Lemma \ref{BranchingLaws(U(3),U(2)XU(1))}
that $(\tilde{q}_3, \tilde{q}_4, \tilde{q}_5)=(0,-1,-1)\text{ or } (-1,-1,0)$.
For the element $(\tilde{p}_1, \tilde{p}_2, \tilde{q}_3, \tilde{q}_4, \tilde{q}_5)=(1,1,-1,-1,0)$
in $D(\tilde{K}_2, \tilde{K}_0)$,
by the branching law of $(U(2), U(1)\times U(1))$,
$(\tilde{q}^{\prime}_1, \tilde{q}^{\prime}_2,\tilde{q}^{\prime}_3, \tilde{q}^{\prime}_4)
=(1,1,-1,-1)\in D(\tilde{K}_1, \tilde{K}_0)$.
Hence, $-c_L=-2c_{\tilde{\Lambda}}+c_{\tilde{\Lambda}^{\prime}}
+\frac{1}{2}c_{\tilde{\Lambda}^{\prime\prime}}
=6<10$.
On the other hand,
$
(\tilde{V}_{\tilde{\Lambda}})_{\tilde{K}_0}
\subset W_{y_1+y_2} \boxtimes W_{-y_3-y_4}\boxtimes W_0
\cong \mathbf{C} \boxtimes \mathbf{C} \boxtimes \mathbf{C} ,
$
where the latter is the $\tilde{K}_2=U(2)\times U(2)\times U(1)$-module.
The representation $\rho_{\tilde{\Lambda}}$ of $\tilde{K}_0$ on
$u \otimes v \otimes w \in   \mathbf{C} \boxtimes \mathbf{C} \boxtimes \mathbf{C}$
is given by
\begin{equation*}
 \rho_{\tilde{\Lambda}}(P)(u \otimes v \otimes w)
= e^{\sqrt{-1}(s+t)}u \otimes e^{-\sqrt{-1}(s+t)}v \otimes w
= u\otimes v\otimes w.
\end{equation*}
It follows that $(\tilde{V}_{\tilde{\Lambda}})_{\tilde{K}_0}=\textrm{span}_{\mathbf{C}}\{1\otimes 1\otimes 1\}$.
Moreover, the action of the generator
$Q$
of $\mathbf{Z}_4$ in $\tilde{K}_{[\mathfrak a]}$
on $u \otimes v\otimes w$
 is given by
\begin{equation*}
\rho_{\tilde{\Lambda}}(Q)
(u \otimes v \otimes w )=-u\otimes v \otimes w.
\end{equation*}
Therefore $(\tilde{V}_{\tilde{\Lambda}})_{\tilde{K}_{[\mathfrak{a}]}}=\{0\}$
and $\tilde{\Lambda}=(1,1,0,-1,-1)\not\in D(\tilde{K}, \tilde{K}_0)$.
Similarly, $\tilde{\Lambda}=(-1,-1,1,1,0)\not\in D(\tilde{K}, \tilde{K}_0)$.

\smallskip

Suppose that $\tilde{\Lambda}=(0, 0, 1, 0, -1)$.
It follows from the branching law of $(U(3), U(2)\times U(1))$
that $(\tilde{q}_3, \tilde{q}_4, \tilde{q}_5)=(1,0,-1), (0,0,0), (1,-1,0)\text{ or } (0,-1,1)$.
For the element $(\tilde{p}_1, \tilde{p}_2, \tilde{q}_3, \tilde{q}_4, \tilde{q}_5)=(0,0,0,0,0)$
in $D(\tilde{K}_2, \tilde{K}_0)$,
by the branching law of $(U(2), U(1)\times U(1))$,
$(\tilde{q}^{\prime}_1, \tilde{q}^{\prime}_2,\tilde{q}^{\prime}_3, \tilde{q}^{\prime}_4)
=(0,0,0,0)\in D(\tilde{K}_1, \tilde{K}_0)$.
Hence,
$-c_L=-2c_{\tilde{\Lambda}}+c_{\tilde{\Lambda}^{\prime}}
+\frac{1}{2}c_{\tilde{\Lambda}^{\prime\prime}}
=12>10$.
For the element $(\tilde{p}_1, \tilde{p}_2, \tilde{q}_3, \tilde{q}_4, \tilde{q}_5)=(0,0,1,-1,0)$
in $D(\tilde{K}_2, \tilde{K}_0)$,
by the branching laws of $(U(2), U(1)\times U(1))$,
$(\tilde{q}^{\prime}_1, \tilde{q}^{\prime}_2,\tilde{q}^{\prime}_3, \tilde{q}^{\prime}_4)
=(0,0,0,0)\in D(\tilde{K}_1, \tilde{K}_0)$.
Hence,
$-c_L=-2c_{\tilde{\Lambda}}+c_{\tilde{\Lambda}^{\prime}}
+\frac{1}{2}c_{\tilde{\Lambda}^{\prime\prime}}
=8<10$.
On the other hand,
$
(\tilde{V}_{\tilde{\Lambda}})_{\tilde{K}_0}
\subset \tilde{V}^{\prime}_{(0,0,0,0,0)} \oplus \tilde{V}^{\prime}_{(0,0,1,-1,0)}
$.
We are concerned with only $\tilde{V}^{\prime}_{(0,0,1,-1,0)}$
since it corresponds to the smaller eigenvalue $8$.
Note that $\tilde{V}^{\prime}_{(0,0,1,-1,0)}=W_{0} \boxtimes
W_{y_3-y_4}\boxtimes W_0
\cong \mathbf{C}\boxtimes V_2 \boxtimes \mathbf{C}$,
which is a $\tilde{K}_2$-module.
The representation $\rho_{\tilde{\Lambda}}$ of $\tilde{K}_0$ on
$u \otimes v^{(2)}_i \otimes w \in   \tilde{V}^{\prime}_{(0,0,1,-1,0)}$ $(i=0,1,2)$
is given by
\begin{equation*}
\rho_{\tilde{\Lambda}}(P)(u \otimes v^{(2)}_i \otimes w)
=
e^{\sqrt{-1}(s-t)(1-i)} u \otimes v^{(2)}_i \otimes w.
\end{equation*}
Thus
$(\tilde{V}_{\tilde{\Lambda}})_{\tilde{K}_0}=\textrm{Span}_{\mathbf{C}}\{1\otimes v_1\otimes 1\}
\oplus \tilde{V}^{\prime}_{(0,0,0,0,0)}.$
Moreover, the action of the generator
$Q$
of $\mathbf{Z}_4$ in $\tilde{K}_{[\mathfrak a]}$
on $u \otimes v^{(2)}_1\otimes w$
is given by
\begin{equation*}
\begin{split}
& \rho_{\tilde{\Lambda}}(Q)
(u \otimes v^{(2)}_1 \otimes w )\\
=& u\otimes \rho_2(\left(
                                             \begin{array}{cc}
                                               0 & \sqrt{-1} \\
                                               \sqrt{-1} & 0 \\
                                             \end{array}
                                           \right)
)v_1\otimes w=-u\otimes v^{(2)}_1 \otimes w.
\end{split}
\end{equation*}
Therefore,
$1\otimes v^{(2)}_1\otimes 1\not\in
(\tilde{V}_{\tilde{\Lambda}})_{\tilde{K}_{[\mathfrak{a}]}}$
and $(\tilde{V}_{\tilde{\Lambda}})_{\tilde{K}_{[\mathfrak{a}]}}
=\tilde{V}^{\prime}_{(0,0,0,0,0)}$,
which gives a larger eigenvalue $10$.

\smallskip

Moreover,
\begin{equation*}
\begin{split}
n(L^{10})=&\dim_{\mathbf{C}}V_{(1,-1,1,0,-1)}+\dim_{\mathbf C}V_{(2,0,0,-1,-1)}+ \dim_{\mathbf{C}}V_{(0,-2,1,1,0)}\\
&\, +\dim_{\mathbf C}V_{(1,1,0,0,-2)}+ \dim_{\mathbf{C}}V_{(-1,-1,2,0,0)}\\
=&24+9+9+6+6 = 54\\
=&\dim SO(12)-\dim S(U(2)\times U(3))=n_{hk}(L^{10}).
\end{split}
\end{equation*}
Therefore we obtain that
$L^{10}=\mathcal{G}(\frac{S(U(2)\times U(3))}{S(U(1)\times U(1)\times U(1))})
\subset Q_{10}(\mathbf{C})$ is strictly Hamiltonian stable.

\subsection{Eigenvalue computation when $m\geq 4$}

For each $\tilde{\Lambda}=
\tilde{p}_{1}y_{1}+\tilde{p}_{2}y_{2}+
\tilde{p}_{3}y_{3}+\tilde{p}_{4}y_{4}+
\tilde{p}_{m+1}y_{m+1}+\tilde{p}_{m+2}y_{m+2}\in{D(\tilde{K},\tilde{K}_{0})}$,
$\tilde{\Lambda}^{\prime}=
\tilde{p}_{1}y_{1}+\tilde{p}_{2}y_{2}+
\tilde{q}_{3}y_{3}+\tilde{q}_{4}y_{4}\in{D(\tilde{K}_{2},\tilde{K}_{0})}$
and
$\tilde{\Lambda}^{\prime\prime}=
\tilde{q}^{\prime}_{1}y_{1}+\tilde{q}^{\prime}_{2}y_{2}+
\tilde{q}^{\prime}_{3}y_{3}+\tilde{q}^{\prime}_{4}y_{4}\in{D(\tilde{K}_{1},\tilde{K}_{0})}$,
the eigenvalue formula is
\begin{equation*}
\begin{split}
-c_L=&-2c_{\tilde{\Lambda}}+c_{\tilde{\Lambda}^{\prime}}
+\frac{1}{2}c_{\tilde{\Lambda}^{\prime\prime}}\\
=&{\ }
\tilde{p}_{1}^{2}+\tilde{p}_{2}^{2}
+2(\tilde{p}_{3}^{2}+\tilde{p}_{4}^{2}+
\tilde{p}_{m+1}^{2}+\tilde{p}_{m+2}^{2})\\
&
+(\tilde{p}_{1}-\tilde{p}_{2})
+2(m-1)(\tilde{p}_{3}-\tilde{p}_{m+2})
+2(m-3)(\tilde{p}_{4}-\tilde{p}_{m+1}) \\
&
-(\tilde{q}_{3}^{2}+\tilde{q}_{4}^{2})
-(\tilde{q}_{3}-\tilde{q}_{4})
-\frac{1}{2}(
(\tilde{q}^{\prime}_{1})^{2}+(\tilde{q}^{\prime}_{2})^{2}
+(\tilde{q}^{\prime}_{3})^{2}+(\tilde{q}^{\prime}_{4})^{2}).
\end{split}
\end{equation*}

\noindent
In case
$\tilde{\Lambda}=(\tilde{p}_{1},\tilde{p}_{2},\tilde{p}_{3},\tilde{p}_{4},
\tilde{p}_{m+1},\tilde{p}_{m+2})
=(\tilde{p}_{1},\tilde{p}_{2},0,0,0,0)\in D(\tilde{K}, \tilde{K}_0)$,
since
$
\tilde{p}_{3}=\tilde{p}_{4}=\tilde{p}_{m+1}=\tilde{p}_{m+2}=0,
$
we have
$
\tilde{q}_{3}=\tilde{q}_{4}=\tilde{q}_{5}=\cdots=\tilde{q}_{m+2}=0
$
and thus
$
\tilde{q}^{\prime}_{3}=\tilde{q}^{\prime}_{4}=0.
$
Since
$\tilde{p}_{1}+\tilde{p}_{2}=0$,
by the branching law of $(U(2), U(1)\times U(1))$ we have
$\tilde{q}^{\prime}_{1}=-\alpha+\tilde{p}_{1}$,
$\tilde{q}^{\prime}_{2}=\alpha+\tilde{p}_{2}=\alpha-\tilde{p}_{1}=-\tilde{q}^{\prime}_{1}$
for some $\alpha=0,1\cdots,\tilde{p}_{1}-\tilde{p}_{2}=2\tilde{p}_{1}$.
$\tilde{\Lambda}\in D(\tilde{K}, \tilde{K}_0)$ implies that
$\tilde{q}^{\prime}_1=\tilde{q}^{\prime}_2=0$
since  $\tilde{q}^{\prime}_1+\tilde{q}^{\prime}_3=0$
and $\tilde{q}^{\prime}_2 + \tilde{q}^{\prime}_4=0$.
Then $-c_L=2\tilde{p}_{1}(\tilde{p}_{1}+1)$.

Now
$\tilde{\Lambda}
=\tilde{p}_{1}y_{1}+\tilde{p}_{2}y_{2}
=2\tilde{p}_{1}\frac{1}{2}(y_{1}-y_{2}).
$
Set
$\ell:=2\tilde{p}_{1}$.
Then $\tilde{V}_{\tilde{\Lambda}}\cong V_{\ell} \boxtimes \mathbf{C}$.
The representation $\rho_{\tilde{\Lambda}}$ of $\tilde{K}_0$ on
$v^{(\ell)}_i\otimes w \in \tilde{V}_{\tilde{\Lambda}}$
is given by
\begin{equation*}
\begin{split}
\rho_{\tilde{\Lambda}}(P)
(v^{(\ell)}_i \otimes w)
=&
\left[\rho_{\ell}
\begin{pmatrix}
e^{\sqrt{-1}(s-t)/2}&0 \\
0&e^{-\sqrt{-1}(s-t)/2}
\end{pmatrix}
\right](v^{(\ell)}_i)\otimes w\\
=&e^{\frac{\sqrt{-1}(s-t)}{2}(\ell-2i)}\, v^{(\ell)}_i\otimes w.
\end{split}
\end{equation*}
Hence,
$(\tilde{V}_{\tilde{\Lambda}})_{\tilde{K}_0}=
{\rm span}_{\mathbf C}\{v_{\tilde{p}_1}^{(\ell)}\otimes w\}$.
On the other hand, the action of the generator $Q$ of $\mathbf{Z}_4$ in $\tilde{K}_{[\mathfrak a]}$
is given by
\begin{equation*}
\begin{split}
\rho_{\tilde{\Lambda}}(Q)
(v^{(\ell)}_{\tilde{p}_1} \otimes w)
=&
\left[\rho_{\ell}
\begin{pmatrix}
0& 1 \\
-1 & 0
\end{pmatrix}
\right](v^{(\ell)}_{\tilde{p}_1})\otimes w\\
=& (-1)^{\tilde{p}_1}\, v^{(\ell)}_{\tilde{p}_1}\otimes w.
\end{split}
\end{equation*}
Therefore,
$(\tilde{V}_{\tilde{\Lambda}})_{\tilde{K}_{[\mathfrak a]}}=
{\rm span}_{\mathbf C}\{ v^{(\ell)}_{\tilde{p}_1} \otimes w\}$
for $\tilde{p}_1$ is even.
As $m\geq 4$,
for every even number $\tilde{p}_1\geq 2$ such that
$12\leq 2\tilde{p}_1(\tilde{p}_1 +1 )<4m-2$,
$\tilde{\Lambda}=\tilde{p}_1(y_1-y_2) \in
D(\tilde{K}, \tilde{K}_{[\mathfrak a]})$
has eigenvalue
$12\leq -c_L=2\tilde{p}_1(\tilde{p}_1 +1 )< 4m-2$.
It means that
$
L^{4m-2}
\subset Q_{4m-2}(\mathbf{C})
$
is NOT Hamiltonian stable for $m\geq 4$.

From these results we conclude
\begin{thm}
The Gauss image
$L^{4m-2}=\frac{S(U(2)\times U(m))}
{S(U(1)\times U(1)\times U(m-2)) \cdot \mathbf{Z}_4}
\subset Q_{4m-2}(\mathbf{C})$ $(m\geq 2)$
is not Hamiltonian stable if and only if $m\geq 4$.
If $m=2$ or $3$, it is strictly Hamiltonian stable.
\end{thm}
\begin{rem0}
The index $i(L^{4m-2})$ goes to $\infty$ as $m\to \infty$.
\end{rem0}

\section{The case $(U,K)=(Sp(m+2), Sp(2)\times Sp(m))$ $(m\geq 2)$}
\label{Sec:CII2}

In this case,
$K=Sp(2)\times Sp(m)\subset U=Sp(m+2)$,
$(U,K)$ is of type $B_2$ for $m=2$ and type $BC_2$ for $m\geq 3$.
Let ${\mathfrak u}={\mathfrak k}+{\mathfrak p}$ be the canonical decomposition
of ${\mathfrak u}$ and ${\mathfrak a}$ be a maximal abelian subspace of
${\mathfrak p}$,
where
\begin{equation*}
\begin{split}
{\mathfrak u}=&\mathfrak{sp}(m+2) \\
=&
\Bigl\{ \begin{pmatrix}
A&B\\
-\bar{B}&\bar{A}
\end{pmatrix} \mid A\in{\mathfrak u}(m+2), B\in{M(m+2,{\mathbf C})},
B^{t}=B
\Bigr\}
\subset{\mathfrak u}(2m+4),
\end{split}
\end{equation*}

\begin{equation*}
\begin{split}
&{\mathfrak k}
=\mathfrak{sp}(2)+\mathfrak{sp}(m)\\
=&
\Bigl\{
\begin{pmatrix}
A_{11}&0&B_{11}&0 \\
0&A_{22}&0&B_{22}& \\
-\bar{B}_{11}&0&\bar{A}_{11}&0 \\
0&-\bar{B}_{22}&0&\bar{A}_{22} \\
\end{pmatrix}
{\ }\vert{\ }
A_{11}\in{\mathfrak u}(2), B_{11}\in{M(2,{\mathbf C})}, B_{11}^{t}=B_{11}, \\
&\quad\quad\quad\quad\quad\quad\quad\quad\quad\quad\quad
\quad\quad\quad\quad
A_{22}\in{\mathfrak u}(m), B_{22}\in{M(m,{\mathbf C})}, B_{22}^{t}=B_{22}
\Bigr\},
\end{split}
\end{equation*}
\begin{equation*}
\begin{split}
{\mathfrak p}
=\Bigl\{
\begin{pmatrix}
0&A_{12}&0&B_{12} \\
-\bar{A}_{12}^{t}&0&B_{12}^{t}&0 \\
0&-\bar{B}_{12}&0&\bar{A}_{12} \\
-\bar{B}_{11}^{t}&0&-A_{12}^{t}&0 \\
\end{pmatrix}
\mid &
A_{12}\in{M(2,m;{\mathbf C})}, B_{12}\in{M(2,m;{\mathbf C})}
\Bigr\},
\end{split}
\end{equation*}
\begin{equation*}
{\mathfrak a}
=\Bigl\{
\begin{pmatrix}
0&H_{12}&0&0 \\
-\bar{H}_{12}^{t}&0&0&0 \\
0&0&0&\bar{H}_{12} \\
0&0&-H_{12}^{t}&0 \\
\end{pmatrix}
\mid
H_{12}
=
\begin{pmatrix}
\xi_{1}&0&0&\cdots&0 \\
0&\xi_{2}&0&\cdots&0 \\
\end{pmatrix},
\xi_{1},\xi_{2}\in{\mathbf R}
\Bigr\}.
\end{equation*}
Then the centralizer $K_0$ of $\mathfrak{a}$ in $K$ is given as follows:
\begin{equation*}
\begin{split}
K_{0}
=&Sp(1)\times Sp(1)\times Sp(m-2)\\
=&
\Bigl\{
\begin{pmatrix}
a_{1}&0& & & &b_{1}&0& & &  \\
0&a_{2}& & & &0&b_{2}& & &  \\
 & &a_{1}&0&  & & &b_{1}&0&  \\
 & &0&a_{2}& & & &0&b_{2}&  \\
 & & & &A_{11}& & & & &A_{12} \\
-\bar{b}_{1}&0& & & &\bar{a}_{1}&0& & &  \\
0&-\bar{b}_{2}& & & &0&\bar{a}_{2}& & &  \\
 & &-\bar{b}_{1}&0& & & &\bar{a}_{1}&0&  \\
 & &0&-\bar{b}_{2}& & & &0&\bar{a}_{2}&  \\
 & & & &A_{21}&         & & & &A_{22}
\end{pmatrix}
\mid
\\& \
\begin{pmatrix}
a_{1}&b_{1} \\
-\bar{b}_{1}&\bar{a}_1
\end{pmatrix},
\begin{pmatrix}
a_{2}&b_{2} \\
-\bar{b}_{2}&\bar{a}_2
\end{pmatrix}
\in{Sp(1)=SU(2)},
\begin{pmatrix}
A_{11}&A_{12} \\
A_{21}&A_{22}
\end{pmatrix}
\in{Sp(m-2)}
\Bigr\}.
\end{split}
\end{equation*}
Moreover,
\begin{equation*}
\begin{split}
K_{[{\mathfrak a}]}
&=
K_0 \cup
(Q \cdot K_0)
\cup
(Q^2 \cdot K_0)
 \cup
(Q^3 \cdot K_0),
\end{split}
\end{equation*}
where
\begin{equation*}
D=\begin{pmatrix}
0& 1 & & & \\
1 & 0& & &\\
 & & 0 & 1 & \\
  & & -1& 0&\\
  & & & & I_{m-2}
\end{pmatrix}
\mbox{ and } Q:=\begin{pmatrix}
 D& 0 \\
 0& D
\end{pmatrix}.
\end{equation*}
Thus the deck transformation group of the covering map ${\mathcal G}:
N^{8m-2}\rightarrow {\mathcal G}(N^{8m-2})$ $(m\geq 2)$ is equal to $K_{[{\mathfrak
a}]}/K_{0}\cong{\mathbf Z}_{4}$.

\subsection{Description of the Casimir operator}
\noindent

Denote $\langle{X, Y}\rangle_{\mathfrak u}:= -\frac{1}{2}\mathrm{tr}XY$
for each $X,Y\in \mathfrak{sp}(m+2)
\subset \mathfrak{u}(2m+4)$.
Then
the square length of each restricted root
relative to the above inner product $\langle{\ ,\ }\rangle_{\mathfrak u}$,
is given by
\begin{equation*}
\Vert{\gamma}\Vert_{\mathfrak u}^{2}
=\left\{
   \begin{array}{ll}
     1 \text{ or } 2, & m=2; \\
     \frac{1}{2}, 1 \text{ or } 2, & m\geq 3.
   \end{array}
 \right.
\end{equation*}
Hence the Casimir operator $\mathcal{C}_L$ of $L$, with respect to the induced metric from
$g^{\rm std}_{Q_{8m-2}(\mathbf{C})}$ can be expressed as follows:
\begin{equation*}
{\mathcal C}_{L}
=\left\{
   \begin{array}{ll}
     {\mathcal C}_{K/K_{0}}
-\frac{1}{2}\, {\mathcal C}_{K_{1}/K_{0}}, & m=2; \\
     2\, {\mathcal C}_{K/K_0}-{\mathcal C}_{K_2/K_0}
-\frac{1}{2}\, {\mathcal C}_{K_1/K_0}, & m\geq 3,
   \end{array}
 \right.
\end{equation*}
where
${\mathcal C}_{K/K_{0}}$, ${\mathcal C}_{K_{2}/K_{0}}$
and ${\mathcal C}_{K_{1}/K_{0}}$
denote the Casimir operator of  $K/K_{0}$, $K_{2}/K_{0}$ and
$K_{1}/K_{0}$ relative to
$\langle{\ ,\ }\rangle_{\mathfrak u}\vert_{\mathfrak k}$,
$\langle{\ ,\ }\rangle_{\mathfrak u}\vert_{{\mathfrak k}_{2}}$
and
$\langle{{\ },{\ }}\rangle_{\mathfrak u}\vert_{{\mathfrak k}_{1}}$,
respectively.

\subsection{Descriptions of $D(Sp(m))$ and $D(Sp(2)\times Sp(m))$}
\noindent

Let $G=Sp(m)$ and $K=Sp(2)\times Sp(m-2)$ in this subsection.
Their Lie algebras are $\mathfrak g$ and $\mathfrak k$, respectively.
\begin{equation*}
{\mathfrak t}=\{ \xi=
\sqrt{-1}\mathrm{diag}(\xi_1,\cdots,\xi_m,-\xi_1, \cdots, -\xi_m)
\mid \xi_1, \cdots, \xi_m \in \mathbf{R}\}.
\end{equation*}
is a maximal abelian subalgebra $\mathfrak t$ in both $\mathfrak g$ and
$\mathfrak k$.
Let $y_i: \xi\mapsto \xi_i$ be a linear form on $\mathfrak{t}$.
Then the fundamental root system of $\mathfrak{g}$ relative to $\mathfrak{t}$
is given by
$\{\alpha_1=y_1-y_2, \cdots, \alpha_{m-1}=y_{m-1}-y_m, \alpha_m=2y_m\}$
and the fundamental root system of $\mathfrak{k}$
relative to $\mathfrak{t}$ can be given by
$\{\alpha^\prime=y_1-y_2, \alpha^\prime=2y_2,
\alpha^\prime_3=y_3-y_4, \cdots, \alpha^\prime_{m-1}=y_{m-1}-y_m,
\alpha^\prime_m=2y_m\}$.
Thus each $\Lambda\in D(G)$ for $G=Sp(m)$ relative to $\mathfrak{t}$
is uniquely expressed as
$\Lambda =p_{1}y_{1}+\cdots+ p_{m} y_{m}$ with
$p_{1},\cdots,p_{m}\in{\mathbf Z}$ and
$p_{1}\geq p_2\geq \cdots \geq p_m \geq 0$.
And also each $\Lambda\in D(K)$ for $K=Sp(2)\times Sp(m-2)$
relative to $\mathfrak{t}$ is uniquely expressed as
$\Lambda^\prime =q_{1}y_{1}+\cdots+ q_{m} y_{m}$ with
$q_{1},\cdots,q_{m}\in{\mathbf Z}$ and
$q_{1}\geq q_2\geq 0, q_3\geq \cdots \geq q_m \geq 0$.

\smallskip
\subsection{Branching law of $(Sp(2), Sp(1)\times Sp(1))$}
\begin{lem}[Branching law of $(Sp(2), Sp(1)\times Sp(1))$ \cite{Lepowsky},
\cite{Tsukamoto}]
\label{BranchingLawSp(2)/Sp(1)XSp(1)}
Let $V_{\Lambda}$ be an irreducible $Sp(2)$-module with the highest weight
$\Lambda=p_1y_1+p_2 y_2\in D(Sp(2))$,
where $p_1, p_2\in \mathbf{Z}$ and $p_1\geq p_2\geq 0$.
Then
$V_{\Lambda}$ contains an irreducible
$Sp(1)\times Sp(1)$-module $V_{\Lambda^\prime}$
with the highest weight
$\Lambda^\prime=q_1 y_1+q_2 y_2 \in D(Sp(1)\times Sp(1))$,
where $q_1, q_2\in \mathbf{Z}$ and $q_1\geq0,  q_2\geq 0$,
if and only if
\begin{itemize}
\item[(i)] $p_1\geq q_2 \geq 0$;
\item[(ii)] in the finite power series expansion in $X$ of
$\frac{\prod_{i=0}^1  (X^{r_i+1}-X^{-(r_i +1)})}{X-X^{-1}}$, where
$r_i (i=0,1)$ are defined as
\begin{equation*}
r_0:=
p_1-\max (p_2, q_2),\quad r_1:=
\min (p_2, q_2),
\end{equation*}
the coefficient of $X^{q_1+1}$ does not vanish.
\end{itemize}
Here
that coefficient is equal to the multiplicity of a
$Sp(1)\times Sp(1)$-module $V_{\Lambda^\prime}$ in $V_{\Lambda}$.
\end{lem}
\subsection{Descriptions of
$D(K,K_{0})$ and $D(K_{1},K_{0})$ when $m=2$}
\noindent

For each $\Lambda= p_1 y_1 +p_2 y_2 + p_3 y_3 +p_4 y_4\in
D(K)=D(Sp(2)\times Sp(2))$
with $p_{1},\cdots,p_{4}\in{\mathbf Z}$ and
$p_{1}\geq p_2\geq 0$, $p_3\geq p_4\geq 0$,
we know that $p_1 y_1+p_2 y_2 \in D(Sp(2))$, $p_3 y_3 +p_4 y_4 \in D(Sp(2))$
and $V_{\Lambda}= W_{p_1 y_1 + p_2 y_2} \boxtimes W_{p_3 y_3 +p_4 y_4}$.
By Lemma \ref{BranchingLawSp(2)/Sp(1)XSp(1)},
$W_{p_1 y_1 + p_2 y_2}$ and $W_{p_3 y_3 +p_4 y_4}$
can be decomposed into irreducible
$Sp(1)\times Sp(1)$-modules as
$$
W_{p_1 y_1 + p_2 y_2}=\bigoplus_{q_1,q_2} W^\prime_{q_1 y_1 +q_2 y_2},
\quad
W_{p_3 y_3 + p_4 y_4}=\bigoplus_{q_3,q_4} W^\prime_{q_3 y_3 +q_4 y_4},
$$
where $q_1, q_2$ and $q_3, q_4$
vary as in Lemma \ref{BranchingLawSp(2)/Sp(1)XSp(1)}.
Thus we have a decomposition of $V_\Lambda$ into the direct sum
of irreducible $Sp(1)\times Sp(1)\times Sp(1)\times Sp(1)$-modules:
$$
V_\Lambda=\bigoplus_{q_1,q_2}\bigoplus_{q_3,q_4}
(W^\prime_{q_1 y_1 +q_2 y_2}\boxtimes W^\prime_{q_3 y_3 +q_4 y_4}).
$$
Further
by the Clebsch-Gordan formula it can be decomposed into into the sum of
irreducible $Sp(1)\times Sp(1)$-modules as
$$
V_\Lambda=
\bigoplus_{q_1,q_2}\bigoplus_{q_3,q_4}
\left(\bigoplus_{i=1}^{q_3} U_{q_1+q_3-2i}\right) \boxtimes
\left(\bigoplus_{j=0}^{q_4} U_{q_2+q_4-2j}\right).
$$
Here we assume that $q_1 \geq q_3\geq 0$ and $q_2 \geq q_4 \geq 0$.
Hence
\begin{lem}\label{D(K,K_0)m=2DIII}
$\Lambda \in D(K,K_0)$ if and only if there exist
$i, j\in \mathbf{Z}$ with $0\leq i\leq q_3$ and $0\leq j\leq q_4$
such that $U_{q_1+q_3-2i} \boxtimes U_{q_2+q_4-2j}$ is a trivial
$Sp(1)\times Sp(1)$-module.
Then it must be that $(q_1,q_2)=(q_3,q_4)$.
\end{lem}

\subsection{Eigenvalue computation when $m=2$}

For $\Lambda=p_1y_1+p_2y_2+p_3 y_3 +p_4 y_4 \in D(K, K_0)$ and
$\Lambda^\prime=q_1 y_1+ q_2 y_2+ q_3 y_3 + q_4 y_4 \in D(K_1, K_0)$
with $q_1=q_3$, $q_2=q_4$ as in Lemma \ref{D(K,K_0)m=2DIII},
the corresponding eigenvalue of $-{\mathcal C}_{L}$ is
\begin{equation}\label{EigenvalueFormulaDIIIm=2}
\begin{split}
-c_L
=&\, -c_{\Lambda}+\frac{1}{2}\, c_{\Lambda^\prime}\\
=&\,
\Bigl(\sum_{i=1}^4 p_i^2+ 4p_1+2p_2+4p_3+2p_4\Bigr)
-\Bigl(q_1^2+q_2^2+2q_1+2q_2\Bigr).
\end{split}
\end{equation}
Denote $\Lambda=p_1y_1+p_2y_2+p_3y_3+p_4 y_4\in D(K,K_0)$ by
$\Lambda=(p_1,p_2,p_3,p_4)$.
Then using the eigenvalue formula \eqref{EigenvalueFormulaDIIIm=2}
we compute
\begin{lem}
$\Lambda\in D(K, K_{0})$ has eigenvalue $-c_L \leq{14}$
if and only if  $(p_1,p_2,p_3,p_4)$ is one of
\begin{equation*}
\begin{split}
\{\, (0,0,0,0), (1,1,0,0), (0,0,1,1), (1,0,1,0),
(1,1,1,1), (1,1,2,0), (2,0,1,1)\, \}.
\end{split}
\end{equation*}
\end{lem}

Suppose that $\Lambda=(1,1,0,0)$. Then ${\rm dim}_{\mathbf C} V_{\Lambda}=5$.
It follows from
Lemma \ref{BranchingLawSp(2)/Sp(1)XSp(1)}
that
$(q_1, q_2)=(0,0)\text{ or } (1,1)$ and $(q_3, q_4)=(0,0)$.
Then $(q_1, q_2,q_3,q_4)=(0,0,0,0)\in D(K_1,K_0)$.
Hence,
$-c_{\Lambda}=8$,
$-c_{\Lambda^\prime}=0$,
$-c_L=-c_{\Lambda}+\frac{1}{2}c_{\Lambda^\prime}=8<14$.
On the other hand, there is a double covering $\pi : Sp(2) \rightarrow SO(5)$, and $\pi (Sp(1)\times Sp(1))=SO(4)$.
Let $\lambda_5$ denote the standard representation of $SO(5)$ and $1$ the trivial representation of $SO(5)$.
Then the complex representation of $K=Sp(2)\times Sp(2)$ with the highest weight $(1,1,0,0)$ is
$(\lambda_5 \otimes 1)\otimes {\mathbf C}$ and $V_{\Lambda}={\mathbf C}^5$.
It is easy to see that $(V_{\Lambda})_{K_0}={\mathbf C} \mathbf{e}_1$,
where $\mathbf{e}_1=(1,0,0,0,0)^t \in {\mathbf C}^5$.
However for
$$a=\begin{pmatrix}
0&1& & & & & & \\
1&0& & & & & & \\
 & &0 &1 & & & \\
 & &-1&0 & & & \\
 & & & & 0 & 1& & \\
 & & & & 1 & 0& & \\
  & & & &  & &0 &1 \\
   & & & &  & & -1& 0
\end{pmatrix}
\in K_{[\mathfrak a]}\subset K, \quad a \not\in K_0,$$
$\pi(a)={\rm diag}(-1,1,-1,-1,-1)\not\in SO(4)$ and $\pi(a)\mathbf{e}_1 =- \mathbf{e}_1\neq \mathbf{e}_1$.
Therefore  $(V_{\Lambda})_{K_{[\mathfrak a]}} =\{0\}$
and $\Lambda=(1,1,0,0)\not\in D(K,K_{[\mathfrak a]})$.
Similarly, $\Lambda=(0,0,1,1)\not\in D(K, K_{[\mathfrak a]})$.

\smallskip
Suppose that $\Lambda=(1,0,1,0)$. Then ${\rm dim}_{\mathbf C} V_{\Lambda}=16$.
The corresponding representation with the highest weight $\Lambda$ is just
the complexified isotropy representation ${\rm Ad}_{\mathfrak p}(K)^{\mathbf C}$.
Hence $\Lambda\not\in D(K,K_{[\mathfrak{a}]})$.

\smallskip
Suppose that $\Lambda=(1,1,1,1)$.
Then ${\rm dim}_{\mathbf C} V_{\Lambda}=25$.
By Lemma \ref{BranchingLawSp(2)/Sp(1)XSp(1)},
$(q_1, q_2)=(1,1)\text{ or } (0,0)$ and $(q_3,q_4)=(1,1)\text{ or } (0,0)$.
Then $(q_1, q_2,q_3,q_4)=(1,1,1,1)\text{ or }  (0,0,0,0)\in D(K_1,K_0)$.
If $(q_1,q_2,q_3,q_4)=(1,1,1,1)$, then $-c_L=10<14$.
If $(q_1,q_2,q_3,q_4)=(0,0,0,0)$, then $-c_L=16>14$.
On the other hand, $V_{(1,1,1,1)}$ is explicitly given as
\begin{equation*}
V_{(1,1,1,1)} ={\mathbf C}^5 \boxtimes {\mathbf C}^5 \cong M(5, {\mathbf C}).
\end{equation*}
There are doubly covering homomorphisms
\begin{eqnarray*}
\pi : K=Sp(2)\times Sp(2) & \longrightarrow & SO(5)\times SO(5),\\
\pi|_{K_1}: K_1=Sp(1)\times Sp(1)\times Sp(1)\times Sp(1)& \longrightarrow
& SO(4)\times SO(4),\\
\pi|_{K_0}:K_0=Sp(1)\times Sp(1) &\longrightarrow & SO(4).
\end{eqnarray*}

The representation of $K$ on $V_\Lambda$ is realized as the action of $\pi(K)=SO(5)\times SO(5)$
on $M(5, {\mathbf C})$ in the following way:
For each $(A,B) \in SO(5)\times SO(5)$, $X\in M(5,{\mathbf C})$
is mapped to $AXB^{-1}\in M(5,\mathbf{C})$.
Then as a $K_1$-module,
\begin{equation*}
\begin{split}
M(5,\mathbf{C})&
=
\left\{\begin{pmatrix} 0&0\\ * &0 \end{pmatrix}\right\}
\oplus
\left\{\begin{pmatrix} 0& *\\0&0 \end{pmatrix}\right\}
\oplus
\left\{\begin{pmatrix} *&0\\0&0 \end{pmatrix}\right\}
\oplus
\left\{\begin{pmatrix} 0&0\\0& * \end{pmatrix}\right\}
\\
&= W_{(1,1,0,0)} \oplus W_{(0,0,1,1)}\oplus W_{(0,0,0,0)} \oplus W_{(1,1,1,1)}.
\end{split}
\end{equation*}
$K_0$ acts on $M(5, \mathbf{C})$ by the adjoint action as a diagonal subgroup of $K_1$.
Hence,
\begin{equation*}
\begin{split}
(M(5,\mathbf{C}))_{K_0}
&=\Bigl\{
\begin{pmatrix} x&0\\ 0& yI_4 \end{pmatrix}
\mid
x,y\in {\mathbf C}
\Bigr\},\\
(M(5,\mathbf{C}))_{K_{[\mathfrak a]}}&
=\mathbf{C}\begin{pmatrix}1&0\\0&0\end{pmatrix}
=W{(0,0,0,0)}.
\end{split}
\end{equation*}
Though $\Lambda=(1,1,1,1)\in D(K,K_{[\mathfrak a]})$,
by the preceding computation (in case $(q_1,q_2,q_3,q_4)=(0,0,0,0)$)
we see that a nonzero element in
$(M(5,\mathbf{C}))_{K_{[\mathfrak a]}}=W{(0,0,0,0)}$
gives eigenvalue $-c_L=16>14$.

\smallskip
Suppose that $\Lambda=(1,1,2,0)$. Then ${\rm dim}_{\mathbf C} V_{\Lambda}=50$.
It follows from Lemma \ref{BranchingLawSp(2)/Sp(1)XSp(1)}
that
$(q_1,q_2)=(1,1)\text{ or } (0,0)$ and
$(q_3,q_4)=(0,2), (1,1) \text{ or } (2,0)$.
Thus
\begin{equation*}
\begin{split}
V_{\Lambda}=
&\ (W_{(1,1)} \boxtimes U_{(0,2)}) \oplus (W_{(1,1)} \boxtimes U_{(1,1)})
\oplus (W_{(1,1)}\boxtimes U_{(2,0)})
\\ &
\oplus (W_{(0,0)} \boxtimes U_{(0,2)})\oplus (W_{(0,0)} \boxtimes U_{(1,1)})
\oplus (W_{(0,0)} \boxtimes U_{(2,0)}).
\end{split}
\end{equation*}
Here only $(q_1,q_2,q_3,q_4) = (1,1,1,1)\ (W_{(1,1)} \boxtimes U_{(1,1)})$
belongs to $D(K_1,K_0)$.
and the corresponding eigenvalue is
$-c_L=14$.
On the other hand, the representation of $K$ with highest weight $\Lambda=(1,1,2,0)$
is $\lambda_5\boxtimes\mathrm{Ad}_{\mathfrak{sp}(2)}^{\mathbf C}$.
Set
$\Lambda_1=(p_1,p_2)=(1,1)\in D(Sp(2))$.
Then
$$
V_{\Lambda_1} \cong \mathbf{C}^5 =\mathbf{C}\mathbf{e}_1\oplus
{\rm span}_{\mathbf C}\{\mathbf{e}_2, \mathbf{e}_3, \mathbf{e}_4, \mathbf{e}_5\}
=W_{(0,0)}\oplus W_{(1,1)}.
$$
Using the quaternionic representation
\begin{equation*}
\mathfrak{sp}(2)=\{X\in M(2, \mathbf{H}) \mid X^* +X=0\},
\end{equation*}
we chose the following basis of $\mathfrak{sp}(2)$:
\begin{equation*}
\begin{split}
& E_1:=\begin{pmatrix} 0& 1\\ -1 & 0 \end{pmatrix},\,
  E_2:=\begin{pmatrix} 0& i\\ i  & 0 \end{pmatrix}, \,
  E_3:=\begin{pmatrix} 0& j\\ j & 0 \end{pmatrix},\,
  E_4:=\begin{pmatrix} 0& k\\ k & 0 \end{pmatrix},\\
& E_5:=\begin{pmatrix} i& 0\\ 0 & 0 \end{pmatrix}, \,
  E_6:=\begin{pmatrix} j& 0\\ 0 & 0 \end{pmatrix}, \,
  E_7:=\begin{pmatrix} k& 0\\ 0 & 0 \end{pmatrix}, \\
& E_8:=\begin{pmatrix} 0& 0\\ 0 & i \end{pmatrix}, \,
  E_9:=\begin{pmatrix} 0& 0\\ 0 & j \end{pmatrix}, \,
  E_{10}:=\begin{pmatrix} 0& 0\\ 0 & k \end{pmatrix},
\end{split}
\end{equation*}
where $\{i, j, k\}$ denote the unit pure quaternions.

Set $\Lambda_2=(p_3,p_4)=(2,0)\in D(Sp(2))$.
Then
\begin{equation*}
\begin{split}
V_{\Lambda_2}&\cong
\mathrm{span}_{\mathbf C}\{E_1,E_2,E_3,E_4\} \oplus
\mathrm{span}_{\mathbf C}\{E_5,E_6,E_7\}
               \oplus {\rm span}_{\mathbf C}\{E_8,E_9,E_{10}\}\\
             &= W_{(1,1)}\oplus W_{(2,0)} \oplus W_{(0,2)}
\end{split}
\end{equation*}
By a direct computation, we get that
\begin{equation*}
\begin{split}
(V_{\Lambda})_{K_0}=&{\rm span}_{\mathbf C}\{\mathbf{e}_2\otimes E_1+\mathbf{e}_3\otimes E_2
+ \mathbf{e}_4\otimes E_3 +\mathbf{e}_5\otimes E_4 \}\\
=&(V_{\Lambda})_{K_{[\mathfrak a]}} \subset W_{(1,1)}\otimes U_{(1,1)}.
\end{split}
\end{equation*}
Therefore, $\Lambda=(1,1,2,0)\in D(K,K_{[\mathfrak a]})$, which
gives eigenvalue $14$ with multiplicity $1$.
Similarly, we can show that $\Lambda=(2,0,1,1) \in D(K,K_{[\mathfrak a]})$ which
gives eigenvalue $14$ with multiplicity $1$.

Moreover, we observe that
\begin{equation*}
\begin{split}
n(L^{14})&=\dim_{\mathbf C} V_{(1,1,2,0)} + \dim_{\mathbf C} V_{(2,0,1,1)}=100\\
&=\dim SO(16)-\dim Sp(2)\times Sp(2)=n_{hk}(L^{14}).
\end{split}
\end{equation*}
 From these results we obtain that
$L^{14}=\mathcal{G}(\frac{Sp(2)\times Sp(2)}{Sp(1)\times Sp(1)})
\subset Q_{14}(\mathbf{C})$ is strictly Hamiltonian stable.

\subsection{Eigenvalue computation when $m\geq 3$}
For each
$$\Lambda=p_1y_1+p_2y_2+p_3 y_3 + \cdots +p_{m+2}y_{m+2} \in D(K, K_0)$$
with $p_i\in \mathbf{Z}$, $p_1\geq p_2$, $p_3 \geq p_4\geq \cdots \geq p_{m+2} \geq 0$,
$$\Lambda^\prime=q_1 y_1+ q_2 y_2+ q_3 y_3 + q_4 y_4 +q_5 y_5 +\cdots +q_{m+2}y_{m+2} \in D(K_2, K_0),$$
with $q_i\in \mathbf{Z}$, $q_1 \geq q_2\geq 0$, $q_3\geq q_4\geq 0$, $q_5\geq \cdots\geq q_{m+2} \geq 0$,
$q_1=p_1$, $q_2=p_2$,
and
$$\Lambda^{\prime\prime}=k_1 y_1+ k_2 y_2+ k_3 y_3 + k_4 y_4 + k_5 y_5 +\cdots + k_{m+2}y_{m+2} \in D(K_1, K_0)$$
with $k_i\in \mathbf{Z}$, $k_i \geq 0$ for $1\leq i\leq 4$, $k_5\geq k_6\geq \cdots\geq k_{m+2} \geq 0$,
$k_j=q_j$  for $5\leq j \leq m+2$,
the corresponding eigenvalue of $-{\mathcal C}_{L}$ is expressed as follows:
\begin{equation}\label{EigenvalueFormulaCIII}
\begin{split}
-c_L
=&{\ }-2 c_{\Lambda}+c_{\Lambda^\prime}+\frac{1}{2}c_{\Lambda^{\prime\prime}}\\
=&{\ }
2\Bigl(
\sum_{i=1}^{m+2} p_i^2+4p_1+2p_2+2mp_3+(2m-2)p_4+ \cdots +2p_{m+2}
\Bigr)\\
& {\ }
-\Bigl(
\sum_{i=1}^{m+2} q_i^2+4q_1+2q_2+4q_3+2q_4+(2m-4)q_5+\cdots +2q_{m+2}
\Bigr)\\
& {\ } -\frac{1}{2}
\Bigl(\sum_{i=1}^{m+2} k_i^2+ 2 k_1+2k_2+2k_3+2k_4+(2m-4)k_5+\cdots +2k_{m+2}
\Bigr),
\end{split}
\end{equation}
where $q_i=k_i$ for $5\leq i \leq m+2$, $p_1=q_1$, $p_2=q_2$
and $k_1=k_3$, $k_2=k_4$.

Suppose that $\Lambda=(p_1,p_2, \cdots, p_{m+2})=(2,2,0,\cdots,0)\in D(K)$.
Then by using the branching law of $(Sp(2),Sp(1)\times Sp(1))$
we see that
$\Lambda\in D(K,K_0)$,
$\Lambda^\prime=(q_1,q_2,\cdots,q_{m+2})=(2,2,0,\cdots,0)\in D(K_2,K_0)$
and
$\Lambda^{\prime\prime}=(k_1,k_2,\cdots,k_{m+2})
=(0,0,0,\cdots,0)\in D(K_1,K_0)$.
Hence by \eqref{EigenvalueFormulaCIII}
the corresponding eigenvalue is $-c_L=20 < 8m-2$ for $m\geq 3$.
On the other hand, the representation of $K$ with highest weight $\Lambda=(2,2,0,\cdots,0)$
is a $14$-dimensional irreducible representation
$\rho_{\mathrm{Sym}^2_0 (\mathbf{C}^5)}\boxtimes \mathrm{I}$
of $Sp(2)\times Sp(m)$,
where $\rho_{\mathrm{Sym}_0^2(\mathbf{C}^5)}$ is the composition of
the natural surjective homomorphism $Sp(2)\rightarrow SO(5)$
and the traceless symmetric product representation of $SO(5)$ on
$\mathrm{Sym}^2_0 (\mathbf{C}^5):=
\{X\in M(5;\mathbf{C})\mid X^t=X, \mathrm{tr}X=0\}$.
Here each $A \in SO(5)$ acts on ${\rm Sym}_0^2(\mathbf{C}^5)$
by
$\mathrm{Sym}_0^2(\mathbf{C}^5)
\ni X\mapsto AXA^t \in\mathrm{Sym}_0^2(\mathbf{C}^5)$.
So
\begin{equation*}
\begin{split}
\mathrm{Sym}_0 ({\mathbf C}^5)
=&\, {\mathbf C}\cdot
\begin{pmatrix}
1&0 \\
0&-\frac{1}{4} I_4
\end{pmatrix}
\oplus \Bigl\{
\begin{pmatrix} 0&0\\ 0& X^{\prime} \end{pmatrix}
\mid
X^\prime \in \mathrm{Sym}_0 ({\mathbf C}^4)
\Bigr\}\\
&\, \oplus \Bigl\{
\begin{pmatrix} 0& Z\\ Z^t&0 \end{pmatrix}
\mid
Z\in M(1,4;{\mathbf C})
\Bigr\}\\
=&\, \mathbf{C} \oplus {\rm Sym}_0 ({\mathbf C}^4) \oplus \mathbf{C}^4
\end{split}
\end{equation*}
and
$$
(\mathrm{Sym}_0 ({\mathbf C}^5))_{SO(4)}= {\mathbf C}\cdot
\begin{pmatrix}
1&0 \\
0&-\frac{1}{4} I_4
\end{pmatrix}
\cong \mathbf{C}.
$$
Under the natural surjective homomorphism
$Sp(2) (\subset SU(4))\rightarrow SO(5)$,
the element
$\begin{pmatrix}
0 & 1& & \\
1&0& & \\
 & &0& 1\\
 & & 1 & 0
\end{pmatrix} \in Sp(2)$
corresponds to ${\rm diag}(-1,1,-1,-1,-1)\in SO(5)$, denoted by $Q^\prime$.
By a direct computation, we know that
$({\rm Sym}_0 ({\mathbf C}^5))_{Q^{\prime}\cdot SO(4)}\cap ({\rm Sym}_0 ({\mathbf C}^5))_{SO(4)}
=({\rm Sym}_0 ({\mathbf C}^5))_{SO(4)}$.
Thus,
\begin{equation*}
(V_{\Lambda=(2,2,0,\cdots, 0)})_{K_0}={\mathbf C}\cdot \begin{pmatrix}  1&0 \\
0&-\frac{1}{4} I_4\end{pmatrix} \boxtimes \mathbf{C}
\end{equation*}
and moreover,
\begin{equation*}
(V_{\Lambda=(2,2,0,\cdots, 0)})_{K_{[\mathfrak a]}}={\mathbf C}\cdot \begin{pmatrix}  1&0 \\
0&-\frac{1}{4} I_4\end{pmatrix} \boxtimes \mathbf{C}.
\end{equation*}
This means that
$\Lambda=(2,2,0,\cdots,0)\in D(K,K_{[\mathfrak{a}]})$ has multiplicity $1$,
which corresponds to eigenvalue $20 <8m-2$.
Therefore, $L^{8m-2}\subset Q_{8m-2}(\mathbf{C})$
is not Hamiltonian stable.

From our results of this section we conclude
\begin{thm}
The Gauss image
$L=\frac{Sp(2)\times Sp(m)}{(Sp(1)\times Sp(1) \times Sp(m-2)) \cdot \mathbf{Z}_4}
\subset Q_{8m-2}(\mathbf{C})$
$(m\geq 2)$
is not Hamiltonian stable if and only if $m\geq 3$.
If $m=2$, it is strictly Hamiltonian stable.
\end{thm}

\section{The case $(U,K)=(E_6, U(1)\cdot Spin(10))$}
\label{Sec_EIII}

In this case, $U=E_6$ and $K=U(1) \cdot Spin(10)$. Then $(U,K)$ is
of $BC_2$ type. For the sake of completeness, we first settle our
notations following \cite{Ozeki-TakeuchiII}, \cite{IchiroYokota09},
\cite{MIse76} and the references therein.

\subsection{Cayley algebra}

Let $\mathbf K$ be the real Cayley algebra and $\{ c_0=1, c_1,
\cdots, c_7\}$ the standard units of ${\mathbf K}$. They satisfy the
following relations (\cite{Ozeki-TakeuchiII}):

\begin{equation*}
\begin{split}
c_i c_{i+1}=-c_{i+1}c_i =c_{i+3}, & \quad  c_{i+1} c_{i+3}=-c_{i+3}c_{i+1}=c_i,\\
c_{i+3} c_{i}=-c_{i}c_{i+3}=c_{i+1}, & \quad c_{i}^2=-1 \,\,
\text{for}\,\, i\in {\mathbf Z}_7.
\end{split}
\end{equation*}
$\mathbf{K}$ is a noncommutative and nonassociative normed division
algebra with the conjugation $x\mapsto {\bar x}$ and the canonical
inner product $(\, , \, )$ defined respectively by
$$\overline{x_0+\sum_{i=1}^7 x_i c_i}=x_0-\sum_{i=1}^7 x_i c_i,
\quad\quad (\sum_{i=0}^7 x_i c_i, \sum_{i=0}^7 y_i c_i)=\sum_{i=0}^7
x_i y_i.$$ We extend the conjugation and the inner product $\mathbf
C$-linearly to the complexified algebra ${\mathbf K}^{\mathbf C}$ of
$\mathbf K$ and denote them by the same notions $x\mapsto {\bar x}$
and $(\, , \,)$ respectively.

\subsection{Exceptional Jordan algebra}

The exceptional Jordan algebra $H_3({\mathbf K})$ is defined as the
set
$$H_3({\mathbf K})=\{ u\in M_3({\mathbf K}) |  \bar{u}^{t} =u \},$$
with the Jordan product
$$u\circ v=\frac{1}{2} (uv+vu), \quad\quad \text{for}\quad u,v \in H_3({\mathbf K}).$$
The real dimension of $H_3({\mathbf K})$ is $27$ and a typical
element
\begin{equation}\label{eq:elementH_3(K)}
u=\left(
    \begin{array}{ccc}
      \xi_1 & x_3 & {\bar x}_2 \\
      {\bar x}_3 & \xi_2 & {x}_1 \\
      x_2 & {\bar x}_1 & \xi_3 \\
    \end{array}
  \right),
  \quad\quad \xi_i\in {\mathbf R}, x_i\in {\mathbf K}
\end{equation}
of $H_3({\mathbf K})$ will be denoted by
$$
u=\xi_1 e_1 +\xi_2 e_2 +\xi_3 e_3 +x_1 u_1 +x_2 u_2 +x_3 u_3.
$$
In $H_3(\mathbf{K})$, we define the trace $\mathrm{tr}(u)$ and an
inner product $(u,v)$ respectively by
$$\mathrm{tr}(u)=\xi_1+\xi_2+\xi_3, \quad\quad (u,v):=\mathrm{tr}(u\circ v).$$
for each $u,v\in H_3(\mathbf{K})$. Moreover, the Freudenthal product
$u\times v$ is defined by
\begin{equation*}
u\times v:= \frac{1}{2}(2u\circ v-\mathrm{tr}(u)v-\mathrm{tr}(v)u
+(\mathrm{tr}(u)\mathrm{tr}(v)-(u,v))\mathrm{I}_{3}),
\end{equation*}
where $\mathrm{I}_3$ is the $3$-order identity matrix, and
a trilinear form $(u,v,w)$ and the determinant $\det u$ are defined
respectively  by
$$(u,v,w)=(u,v\times w), \quad \det u=\frac{1}{3}(u,u,u).$$

Put
$$
SH_3({\mathbf K})=\{ u \in M_3({\mathbf K})|  \bar{u}^{t} =-u,
\mathrm{tr}(u)=0\}.
$$
An element $u\in SH_3({\mathbf K})$ of the form
\begin{equation}\label{eq:SH}
u=\left(
    \begin{array}{ccc}
      z_1 & x_3 & -{\bar x}_2 \\
      -{\bar x}_3 & z_2 & x_1 \\
      x_2 & -{\bar x}_1 & z_3 \\
    \end{array}
  \right),
\quad z_i, x_i\in {\mathbf K}, {\bar z}_i=-z_i, \Sigma z_i=0
\end{equation}
is denoted by
$$
u=z_1 e_1+z_2 e_2 + z_3 e_3 +x_1 {\bar u}_1 +x_2 {\bar u}_2 +x_3
{\bar u}_3.
$$
Now we define two injective linear maps $R: H_3({\mathbf
K})\rightarrow \mathfrak{gl}(H_3({\mathbf K}))$ and $D:
SH_3({\mathbf K}) \rightarrow \mathfrak{gl}(H_3({\mathbf K}))$
respectively by
\begin{equation}\label{eq:defR&D}
\begin{split}
   R(u)v&=u\circ v = \frac{1}{2} (uv+vu), \quad \hbox{ for }\, u,v\in H_3({\mathbf K}), \\
    D(u)v&=\frac{1}{2}[u,v]=\frac{1}{2}(uv-vu), \quad \hbox{ for }\, u\in SH_3({\mathbf K}),
v\in H_3({\mathbf K}).
\end{split}
\end{equation}
Denote by $\mathfrak{D}$ and $\mathfrak{R}$ the images of $D$ and
$R$ in $\mathfrak{gl}(H_3(\mathbf{K}))$. Introduce some subspaces of
$\mathfrak{D}$ and $\mathfrak{R}$ in the following:
\begin{equation*}
\begin{split}
{\mathfrak D}_0&=\{ \delta \in \mathfrak{D}| \delta (e_i)=0 \, (i=1,2,3) \},\\
{\mathfrak D}_i&=\{ D(x{\bar u}_i)|x\in {\mathbf K}\} \quad \mbox{for}\, i=1,2,3,\\
{\mathfrak R}_0&=\{ R(\sum \xi_i e_i)| \xi_i \in {\mathbf R}, \sum \xi_i=0 \},\\
{\mathfrak R}_i&=\{ R(xu_i)|x\in {\mathbf K}\} \quad \mbox{for}\,
i=1,2,3.
\end{split}
\end{equation*}
Remark that $\dim \mathfrak{D}_0=28$, $\dim {\mathfrak D}_1=\dim
{\mathfrak D}_2=\dim {\mathfrak D}_3=8$, $\dim {\mathfrak R}_0=2$
and $\dim {\mathfrak R}_1=\dim {\mathfrak R}_2=\dim {\mathfrak
R}_3=8$. Moreover, it is easy to know that ${\mathfrak D}_0$ is a
subalgebra of $\mathfrak{gl}(H_3({\mathbf K}))$ generated by the set
$\{D(\Sigma z_i e_i)| z_i\in {\mathbf K}, {\bar z}_i =-z_i, \Sigma
z_i =0\}$. In fact, $\mathfrak{D}_0$ is isomorphic to the Lie
algebra $\mathfrak{o}(8)$ and its basis can be chosen as $\{D_{i,r}
(1\leq r\leq 7), D_{i,pq} (1\leq p <q\leq 7)\}$ for $i=1, 2$ or $3$
(\cite{Chevalley-Schafer}, \cite{MIse76}, \cite{IchiroYokota09}). We
now explain in details by using Ise's notions (\cite{MIse76}, p.82).
Put
$$D_{i,r}=D(c_r(-e_j+e_k)), \quad (1\leq i \leq 3, 1\leq r \leq 7),$$
and
\begin{equation}\label{eq:D_{i,pq}}
D_{i,pq}=[D_{i,p}, D_{i,q}], \quad (1\leq i \leq 3, 1\leq p,q \leq
7),
\end{equation}
where $\{i,j,k\}$ is a cyclic permutation of $\{1,2,3\}$. Then from
\begin{equation*}
D(\sum_{i=1}^3 z_i e_i)(v) =\frac{1}{2}\sum_{\{i,j,k\}} (z_jx_i -x_i
z_k )u_i,
\end{equation*}
we can obtain
\begin{equation*}
\left\{
  \begin{array}{l}
    D_{i,r}(xu_i)= \left\{
                     \begin{array}{ll}
                       -c_r u_i, & \hbox{if}\, x=c_0 \\
                       c_0 u_i, & \hbox{if}\, x=c_r \\
                       0, & \hbox{if}\, x=c_q (q\neq r),
                     \end{array}
                   \right.
 \\
\\
   D_{i,r}(xu_j)=\frac{1}{2} (c_r x)u_j,  \\
\\
    D_{i,r}(xu_k)=\frac{1}{2} (xc_r)u_k,
  \end{array}
\right.
\end{equation*}
and
\begin{equation*}
\left\{
  \begin{array}{l}
    D_{i,pq}(xu_i)= \left\{
                     \begin{array}{ll}
                       c_q u_i, & \hbox{if}\, x=c_p \\
                       -c_p u_i, & \hbox{if}\, x=c_q \\
                       0, & \hbox{if}\, x=c_r (r \leq 0, \neq p,q),
                     \end{array}
                   \right.
 \\
\\
   D_{i,pq}(xu_j)=\frac{1}{2} \{c_p(c_q x)\}u_j,  \\
\\
    D_{i,pq}(xu_k)=\frac{1}{2} \{(xc_q)c_p\}u_k.
  \end{array}
\right.
\end{equation*}
These mean that every $D_{i,r}, D_{i,pq}$ leave ${\mathfrak T}_i=\{x
u_i| x\in {\mathbf K}\}$ invariant $(1 \leq i \leq 3, 1\leq p,q,r
\leq 7 )$ and identifying ${\mathfrak T}_i$ with $\mathbf K$, it
represents a skew-symmetric matrix with respect to the basis $\{c_0,
c_1, \cdots, c_7\}$; namely $D_{i,r}=E_{0r}-E_{r0}$ and
$D_{i,pq}=E_{qp}-E_{pq}$, where $E_{pq}$ denotes the $8\times 8$
matrix with all $0$-components except $(p,q)$-component, $1$.
Moreover,
\begin{eqnarray}
\left[ D_{i,r}, D_{i,pq} \right]&=& D_{i,p}\delta_{qr}-D_{i,q}\delta_{rp},\label{eq:D01}\\
\left[ D_{i,pq}, D_{i,rs}\right]&=&
D_{i,pr}\delta_{sq}+D_{i,qs}\delta_{pr}+D_{i,rq}\delta_{sp}+D_{i,sp}\delta_{rq},
\label{eq:D02}
\end{eqnarray}
where $1\leq i \leq 3$ and $1\leq p,q,r,s \leq 7$. Particulary, we
have
$$
[D_{i,r}, D_{i,pq}]=0, \quad [ D_{i,pq}, D_{i,rs}]=0,
$$
if $p,q,r,s$ are all different each other. Denote the real linear
space spanned by all $D_{i,r}, D_{i,pq}$ $(1\leq p,q,r \leq 7)$ by
${\mathfrak D}_{i,0}$. Then all ${\mathfrak D}_{i,0} (1\leq i \leq
3)$ are isomorphic to each other, and they are isomorphic to the Lie
algebra ${\mathfrak o}(8)$. We shall use ${\mathfrak D}_0={\mathfrak
D}_{1,0}$ in the next.

Let
$$
H_3(\mathbf{K})^{\mathbf{C}}:= H_3(\mathbf{K})+
\sqrt{-1}H_3(\mathbf{K})
$$
be the complexification of $H_3(\mathbf{K})$.
Then there are two complex conjugations on
$H_3(\mathbf{K})^{\mathbf{C}}$, namely,
\begin{equation*}
\overline{u_1+ \sqrt{-1}u_2}=\bar{u}_1 + \sqrt{-1}\bar{u}_2, \quad
\tau(u_1+\sqrt{-1} u_2)=u_1- \sqrt{-1} u_2,
\end{equation*}
where $u_1, u_2 \in H_3(\mathbf{K})$.
Then $H_3({\mathbf K})^{\mathbf C}$ is canonically identified with
$$
H_3({\mathbf K}^{\mathbf C})=\{ u\in M_3({\mathbf K}^{\mathbf C})|
\bar{u}^{t}=u \}.
$$
An element $u\in H_3({\mathbf K}^{\mathbf C})$ of the form
\eqref{eq:elementH_3(K)}, with $\xi_i\in {\mathbf C}$, $x_i\in
{\mathbf K}^{\mathbf C}$, is still denoted by $u=\xi_1 e_1+\xi_2 e_2
+\xi_3 e_3 +x_1 u_1 + x_2 u_2 + x_3 u_3$. The standard Hermitian
inner product $\langle{\ },{\ }\rangle$ of
$H_3(\mathbf{K}^{\mathbf{C}})$ is defined by
\begin{equation*}
\langle{u,v}\rangle:=(\tau u,v).
\end{equation*}
Meanwhile, the complexification $SH_3({\mathbf K})^{\mathbf C}$ of
$SH_3({\mathbf K})$ is identified with
$$SH_3({\mathbf K}^{\mathbf C})=\{u\in M_3(\mathbf{K}^{\mathbf C})| \bar{u}^t=-u, \mathrm{tr}(u)=0\},$$
whose element $u$ of the form \eqref{eq:SH}, with $z_i, x_i \in
\mathbf{K}^{\mathbf C}$, is still denoted by $ u=z_1 e_1+z_2 e_2 +
z_3 e_3 +x_1 {\bar u}_1 +x_2 {\bar u}_2 +x_3 {\bar u}_3 $.
Then  $D(u)\in \mathfrak{gl}(H_3({\mathbf K}^{\mathbf C}))$ for
$u\in SH_3({\mathbf K}^{\mathbf C})$ and $R(u)\in
\mathfrak{gl}(H_3({\mathbf K}^{\mathbf C}))$ for $u\in H_3({\mathbf
K}^{\mathbf C})$ can be defined by the same formula as
\eqref{eq:defR&D}.

\subsection{Basic formulas in ${\mathfrak e}_6$}

\begin{lem}\label{lem:27dimrep}
For $v=\xi_1 e_1 +\xi_2 e_2 +\xi_3 e_3 +x_1 u_1 +x_2 u_2 +x_3 u_3
\in H_3({\mathbf K}^{\mathbf C})$, we have
\begin{equation*}
\begin{split}
R(\sum \eta_l e_l)v&=\eta_1\xi_1 e_1 +\eta_2\xi_2 e_2 +\eta_3\xi_3 e_3+\frac{1}{2}(\eta_2+\eta_3)x_1 u_1\\
& \quad\quad\quad +\frac{1}{2}(\eta_3+\eta_1)x_2 u_2+\frac{1}{2}(\eta_1+\eta_2)x_3 u_3,\\
D(\sum z_l e_l)v&=\frac{1}{2}(z_2x_1 - x_1z_3)u_1 +\frac{1}{2}(z_3
x_2 - x_2 z_1)u_2
+\frac{1}{2}(z_1 x_3 - x_3 z_2)u_3,\\
\end{split}
\end{equation*}
\begin{equation*}
\begin{split}
D(x{\bar u}_i)v&= (x,x_i)(e_j-e_k)+\frac{1}{2}(\xi_k - \xi_j)x u_i
- \frac{1}{2}({\bar x}{\bar x}_k)u_j+\frac{1}{2}({\bar x}_j {\bar x})u_k,\\
R(xu_i)v&=(x,x_i)(e_j+e_k)+\frac{1}{2}(\xi_j + \xi_k) x u_i
+\frac{1}{2}{\bar x}{\bar x}_k u_j +\frac{1}{2}{\bar x}_j {\bar x}
u_k,
\end{split}
\end{equation*}
where $\{i,j,k\}$ is a cyclic permutation of $\{1,2,3\}$.
\end{lem}

The relations \eqref{eq:D_{i,pq}}, \eqref{eq:D01}, \eqref{eq:D02}
and the following list  give commutation rules for ${\mathfrak
e}_6^{\mathbf C}$. Here, $x,y, z_i \in {\mathbf K}^{\mathbf C}$,
${\bar z}_i=-z_i$ for $i=1,2,3$, $\sum_i z_i=0$,
 and $\xi_1,\xi_2,\xi_3\in {\mathbf C}$
with $\sum_l \xi_l =0$.
 In formulae \eqref{eq:RiRj} ~ \eqref{eq:RiRi&DiDi},
$(i,j,k)$ is a cyclic permutation of $(1,2,3)$. In formulae
\eqref{eq:RiRiRi} and \eqref{eq:DiDiDi}, $i=1,2,3$.

\begin{eqnarray}
&&[ R(xu_i), R(yu_j) ] = -(1/2) D({\overline{xy}} \, {\bar u}_k), \label{eq:RiRj}\\
&&\left[ R(xu_i), D(yu_j)\right]= [D(x \bar{u}_i),R(y u_j)]=(1/2) R({\overline{xy}}\, \bar{u}_k),\label{eq:RiDj}\\
&&\left[ D(x \bar{u}_i), D(y \bar{u}_j)\right]=-(1/2) D( {\overline{xy}}{\bar u}_k),\label{eq:DiDj}\\
&&\left[ D(x \bar{u}_i), R(y \bar{u}_i)\right]= (x,y) R( e_j -e_k),\label{eq:DiRi}\\
&&\left[ R(\sum \, \xi_l e_l), R(x {\bar u}_i)\right]= (1/2)(\xi_j -\xi_k) D(x {\bar u}_i),\label{eq:R0Ri}\\
&&\left[ R(\sum \,\xi_l e_l), D(x {\bar u}_i)\right]= (1/2)(\xi_j -\xi_k) R(x {\bar u}_i),\label{eq:R0Di}\\
&&\left[D(\sum z_l e_l),D(x{\bar u}_i)\right]=(1/2)D((z_j x -x z_k){\bar u}_i),\label{eq:D0Di}\\
&&\left[D(\sum z_l e_l),R(x{\bar u}_i)\right]=(1/2)R((z_j x -x z_k) u_i),\label{eq:D0Ri}\\
&&\left[ R(x u_i),R(y u_i)\right]=-[D(x{\bar u}_i),D(y {\bar u}_i)]\label{eq:RiRi&DiDi}\\
&&\quad\quad \quad\quad =D((\frac{y+{\bar y}}{2}\frac{x-{\bar x}}{2}-\frac{x+{\bar x}}{2}\frac{y-{\bar y}}{2})(e_j-e_k))\nonumber\\
&&\quad\quad\quad\quad\quad -[D(\frac{x-{\bar x}}{2}(e_j-e_k)),D(\frac{y-{\bar y}}{2} (e_j-e_k))],\nonumber\\
&&\left[ {\mathfrak R}_0^{\mathbf C}, {\mathfrak R}_0^{\mathbf C}+
{\mathfrak D}_0^{\mathbf C}\right]=\{0\},
\label{eq:R0D0}\\
&&\left[ R(x u_i), [R(xu_i),R(yu_i)]\right]= R(((x,x)y-(x,y)x)u_i),\label{eq:RiRiRi}
\end{eqnarray}
\begin{eqnarray}
&&\left[ D(x {\bar u}_i), [D(x {\bar u}_i), D(y {\bar u}_i)]\right]=
D(((x,y)x-(x,x)y){\bar u}_i).\label{eq:DiDiDi}
\end{eqnarray}

Remark that the Killing-Cartan form $B$ of ${\mathfrak e}_6^{\mathbf
C}$ is given by (\cite{MIse76}, p.88 or \cite{IchiroYokota09}, p.74)
\begin{equation}\label{lem:Killingform}
B(u,v)=4\mathrm{tr}(uv),
\end{equation}
for each $u,v\in {\mathfrak e}_6^{\mathbf C}\subset
\mathfrak{gl}(H_3(\mathbf K)^{\mathbf C})$.

\subsection{Realization of $E_6/(U(1)\cdot Spin(10))$}
We recall and combine the settings given in \cite{IchiroYokota09}
and \cite{MIse76}. It is known  that $F_4$ is defined to be the
automorphism group of the Jordan algebra $H_3(\mathbf{K})$:
\begin{equation*}
\begin{split}
F_{4}&:=\{\alpha\in {GL({H_3(\mathbf{K})})} {\ }\vert{\ } \alpha
(u\circ v) = \alpha u\circ \alpha v
\}\\
&=\{\alpha\in {GL({H_3(\mathbf{K})})} {\ }\vert{\ } \det (\alpha u)=
\det u, (\alpha u, \alpha v)= (u, v) \}.
\end{split}
\end{equation*}
Its Lie algebra $\mathfrak{f}_4$ is thus given by
\begin{equation*}
\mathfrak{f}_4:=\{ \delta\in  {\mathfrak{gl}({H_3(\mathbf{K})})} {\
}\vert{\ } \delta(u\circ v)=\delta u\circ v+u\circ \delta v \},
\end{equation*}
which is isomorphic to
$\mathfrak{D}=\mathfrak{D}_0+\mathfrak{D}_1+\mathfrak{D}_2+\mathfrak{D}_3$.
We are interested in the following two subgroups of $F_4$:
\begin{equation*}
\begin{split}
(F_4)_{e_1}&:=\{\alpha \in{F_{4}}{\ }\vert{\ } \alpha e_{1}=e_{1}\}\cong Spin(9), \\
(F_4)_{e_1, e_2, e_3}&:=\{\alpha \in{F_{4}}{\ }\vert{\ }\alpha
e_{i}=e_{i}{\ }(i=1,2,3)\}\cong Spin(8).
\end{split}
\end{equation*}
Note that the Lie algebra of $(F_4)_{e_1,e_2,e_3}$ is
$\mathfrak{D}_0$.

The groups $E_{6}^{\mathbf C}$ and $E_6$ are defined by
\begin{equation*}
\begin{split}
E_{6}^{\mathbf C}&:= \{\alpha \in{GL_{\mathbf
C}({H_3(\mathbf{K})}^{\mathbf C})} {\ }\vert{\ } \det(\alpha
u)=\det(u)\},\\
E_{6}&:=\{\alpha \in{GL_{\mathbf C}({H_3(\mathbf{K})}^{\mathbf C})}
{\ }\vert{\ } \det (\alpha u)=\det u, \langle{\alpha u, \alpha
v}\rangle=\langle{u,v}\rangle \},
\end{split}
\end{equation*}
respectively. Hence $F_4$ is a subgroup of $E_6$. The Lie algebras
of $E_6^{\mathbf C}$ and $E_6$ are given respectively by
\begin{equation*}
\begin{split}
{\mathfrak e}_{6}^{\mathbf C}&:= \{\phi\in{\mathfrak{gl}_{\mathbf
C}({H_3(\mathbf{K})}^{\mathbf C})} {\ }\vert{\ } (\phi u,u,u)=0
\}, \\
{\mathfrak e}_{6}&:=\{\phi\in{\mathfrak{gl}_{\mathbf C}(
{H_3(\mathbf{K})}^{\mathbf C}}) {\ }\vert{\ } (\phi u, u, u)=0,
\langle{\phi u, v}\rangle+\langle{u,\phi v}\rangle=0 \}.
\end{split}
\end{equation*}
It can be shown (\cite{IchiroYokota09}, p.68) that any element
$\phi\in{\mathfrak e}_{6}^{\mathbf C}$ is uniquely expressed as
\begin{equation*}
\phi=\delta+ \varsigma, \quad
\delta\in {\mathfrak D}^{\mathbf C}, \quad \varsigma
\in \mathfrak{R}^{\mathbf C}, 
\end{equation*}
where ${\mathfrak D}^{\mathbf C}$ and ${\mathfrak R}^{\mathbf C}$
denote the complexifications of $\mathfrak D$ and $\mathfrak R$
respectively. So we get the so-called Chevalley-Schafer model
(\cite{Chevalley-Schafer}) of $\mathfrak{e}_6^{\mathbf C}$:
${\mathfrak e}_6^{\mathbf C}={\mathfrak D}^{\mathbf C}+{\mathfrak
R}^{\mathbf C}$ as a subalgebra of $\mathfrak{gl}(H_3({\mathbf
K})^{\mathbf C})$. The inclusion $\phi: {\mathfrak e}_6^{\mathbf C}
\subset \mathfrak{gl}(H_3({\mathbf K})^{\mathbf C})$ is a
$27$-dimensional irreducible representation of ${\mathfrak
e}_6^{\mathbf C}$.
Furthermore, any element $\phi\in{\mathfrak e}_{6}$ is uniquely
expressed as
\begin{equation*}
\phi =\delta +\sqrt{-1} \varsigma, \quad \delta\in{\mathfrak D},
\quad \varsigma \in \mathfrak{R}.
\end{equation*}
Equivalently, $\mathfrak{e}_6:= \mathfrak{D} +
\sqrt{-1}\mathfrak{R}$.

Consider a $\mathbf{C}$-linear transformation $\sigma$ of
$H_3(\mathbf{K})^{\mathbf{C}}$ defined by
\begin{equation*}
\sigma\left(
        \begin{array}{ccc}
          \xi_1 & x_3 & \bar{x}_2 \\
          \bar{x}_3 & \xi_2 & x_1 \\
          x_2 & \bar{x}_1 & \xi_3 \\
        \end{array}
      \right)
=\left(
        \begin{array}{ccc}
          \xi_1 & -x_3 & -\bar{x}_2 \\
          -\bar{x}_3 & \xi_2 & x_1 \\
          -x_2 & \bar{x}_1 & \xi_3 \\
        \end{array}
      \right).
\end{equation*}
Then $\sigma\in E_6$ and $\sigma^2=1$.
$\sigma$ induces an involutive automorphism of $E_6$ by $\alpha
\mapsto \sigma \alpha \sigma$, which we also still denote it by
$\sigma$.
In order to investigate the subgroup $(E_6)^{\sigma}$ of all fixed
points of $\sigma$,
\begin{equation}
(E_6)^{\sigma}=\{\alpha\in E_6 \mid \sigma\alpha=\alpha \},
\end{equation}
consider two subgroups
$$(E_6)_{e_1}=\{\alpha\in E_6 {\ } \vert {\ } \alpha e_1=e_1\}\cong Spin(10)$$
and
\begin{equation}\label{eq:u(1)}
\begin{split}
U(1)&=\{
\phi(\theta)
\in GL_{\mathbf{C}}(H_3(\mathbf{K})^{\mathbf C})
\mid
\theta=e^{\sqrt{-1}t/2}, t\in\mathbf{R}\},
\end{split}
\end{equation}
where
$\phi(\theta):=\exp(t\sqrt{-1} R(2e_1-e_2-e_3))
\in GL_{\mathbf{C}}(H_3(\mathbf{K})^{\mathbf C})$
and
\begin{equation}\label{defU(1)}
\phi(\theta)\left(
        \begin{array}{ccc}
          \xi_1 & x_3 & \bar{x}_2 \\
          \bar{x}_3 & \xi_2 & x_1 \\
          x_2 & \bar{x}_1 & \xi_3 \\
        \end{array}
      \right)
=\left(
        \begin{array}{ccc}
          \theta^4 \xi_1 & \theta x_3 & \theta\bar{x}_2 \\
          \theta \bar{x}_3 & \theta^{-2}\xi_2 & \theta^{-2}x_1 \\
          \theta x_2 & \theta^{-2}\bar{x}_1 & \theta^{-2}\xi_3 \\
        \end{array}
      \right).
\end{equation}
Here the subgroups $U(1)$ and $Spin(10)$ of $(E_6)^{\sigma}$ are
elementwise commutative.
Define a mapping
$$
p: \tilde{K}=U(1)\times Spin(10) \ni
(\theta, \alpha)\longmapsto \phi(\theta)\alpha\in
K=(E)^{\sigma},
$$
which is a surjective Lie group homomorphism.
Since
\begin{equation*}
U(1)\cap Spin(10)= \{1=\phi(1),\phi(-1),
\phi(\sqrt{-1}),\phi(-\sqrt{-1})\},
\end{equation*}
we have
$
\mathrm{Ker}(p)
 =
\{(1,\phi(1)),(-1,\phi(-1)),(\sqrt{-1},\phi(-\sqrt{-1})),
(-\sqrt{-1},\phi(\sqrt{-1}))\}$,
which is isomorphic to
${\mathbf Z}_{4}.$
Thus
$$
K=(E_6)^\sigma=\tilde{K}/\mathbf{Z}_4 = (U(1)\times Spin(10))/\mathbf{Z}_4,$$
and $U/K=E_6/(U(1)\cdot Spin(10))$.
Correspondingly,
\begin{equation*}
\begin{split}
{\mathfrak k}&=({\mathfrak e}_{6})_{\sigma} =
\{\phi\in{\mathfrak e}_{6}{\ }\vert{\ }\sigma_{\ast}\phi
=\phi\}\\
&=({\mathfrak e}_{6})_{e_{1}}+{\mathbf R}\sqrt{-1}R(2e_1-e_2-e_3).
\end{split}
\end{equation*}
Since for any $\phi\in \mathfrak{e}_6$
there exist $u \in
SH_3(\mathbf{K})$ and $v\in H_3(\mathbf{K})$ such that
\begin{equation*}
\begin{split}
\phi e_{1}&=D(u)(e_{1})+ \sqrt{-1}R(v)(e_{1}), \\
\end{split}
\end{equation*}
it holds that
$\phi e_{1}=0$ if and only if
\begin{equation*}
u=z_1e_1+z_2e_2+z_3e_3+a_1\bar{u}_1 \in SH_3(\mathbf{K}),{\ }
v=\xi_2 e_2+\xi_3 e_3+ x_1 u_1 \in H_3(\mathbf{K}),
\end{equation*}
where $a_1, x_1\in \mathbf{K}$, $\xi_{2},\xi_{3}\in{\mathbf R}$ with
$\xi_{2}+\xi_{3}=0$. Hence
\begin{equation*}
\begin{split}
({\mathfrak e}_{6})_{e_{1}} &:=\{\phi\in{\mathfrak e}_{6}{\ }\vert{\
}
\phi e_{1}=0\}\\
&= \mathfrak{D}_0 +\mathfrak{D}_1 +{\mathbf
R}\sqrt{-1}R(e_{2}-e_{3}) +\sqrt{-1}\mathfrak{R}_1
\\
&\cong
{\mathfrak o}(10)
\end{split}
\end{equation*}
Therefore, we have the Cartan decomposition of
a compact simple Lie algebra $\mathfrak{u}=\mathfrak{e}_6$ of type EIII:
\begin{equation*}
\begin{split}
\mathfrak{u}&={\mathfrak e}_{6}
=\mathfrak{D}+\sqrt{-1}\mathfrak{R},\\
{\mathfrak k}&=(\mathfrak{e}_6)_{\sigma}=\mathfrak{D}_0+\mathfrak{D}_1+\sqrt{-1}\mathfrak{R}_0+\sqrt{-1}\mathfrak{R}_1,\\
{\mathfrak p}&=(\mathfrak
e)_{-\sigma}=\mathfrak{D}_2+\mathfrak{D}_3+\sqrt{-1}\mathfrak{R}_2+\sqrt{-1}\mathfrak{R}_3,
\end{split}
\end{equation*}
where $\mathfrak k$ is isomorphic to
${\mathfrak u}(1)+{\mathfrak o}(10)$,
$$
[{\mathfrak k}, {\mathfrak k}]={\mathfrak D}_0 + {\mathfrak D}_1
+\sqrt{-1} {\mathbf R} R(e_2-e_3)
+\sqrt{-1} {\mathfrak R}_1=(\mathfrak{e}_6)_{e_1}
$$
is isomorphic to
${\mathfrak o}(10)$ and the center of $\mathfrak k$ is spanned by
$$
Z=\sqrt{-1} R(2e_1-e_2-e_3).
$$

On the other hand,
a compact Hermitian symmetric space of type EIII
can be defined by (\cite[p74-75]{Abe-Yokota97})
\begin{equation*}
\mathrm{EIII} =\{u\in{H_3(\mathbf{K})}^{\mathbf C}
\mid
u\times u=0, u\not=0 \} /{\mathbf C}^{\ast}
\subset P({H_3(\mathbf{K})}^{\mathbf C}),
\end{equation*}
which is considered as a compact complex submanifold embedded in
a complex projective space ${\mathbf C}P^{26}$.
Since $E_6$ acts transitively on $\mathrm{EIII}$ and the isotropy subgroup of
$E_6$ at $o=[e_1]$ is $(E_6)^\sigma$, we know that
$\mathrm{EIII}\cong E_6/(E_6)^{\sigma}=E_{6}/(U(1)\cdot Spin(10))$.
The tangent vector space $T_{o}(U/K)$ at $o$
can be identified with a vector subspace
\begin{equation*}
\begin{split}
T_{o}(\mathrm{EIII})&
\cong
\{u\in {H_3(\mathbf{K})}^{\mathbf C} \mid
u\times e_{1}=0,
\langle{u, e_{1}}\rangle=0\}  \\
&
=
\{x_{2}u_2 + x_{3} u_3 \mid
x_{2},x_{3}
\in{\mathbf{K}}^{\mathbf C}\}.
\end{split}
\end{equation*}

The differential of the natural projection $p:U=E_{6}\rightarrow
U/K=\mathrm{EIII}$ induces a linear isomorphism $p_{\ast}:{\mathfrak
p}\rightarrow T_{o}(\mathrm{EIII})$. Then
$p_{\ast}(\phi)=\phi(e_{1})$ and
\begin{equation}\label{eq:p_*}
\begin{split}
&p_{\ast}\Bigl( 2(D(x_{2}\bar{u}_2)- D(x_{3}\bar{u}_3))
+2\sqrt{-1}(R(x^{\prime}_{2}u_2)+ R(x^{\prime}_{3} u_3))\Bigr)\\
=&\, (x_{2}+\sqrt{-1}x^{\prime}_{2})u_2+(x_{3}+\sqrt{-1}x^{\prime}_{3})u_3.
\end{split}
\end{equation}

\subsection{Restricted root systems of EIII}
Define $H_1, H_2 \in \mathfrak{p}$ by
\begin{equation*}
\begin{split}
H_1&= D\bar{u}_2+\sqrt{-1} R(c_4 u_2),\\
H_2&= D\bar{u}_2-\sqrt{-1} R(c_4 u_2).
\end{split}
\end{equation*}
Then by \eqref{eq:DiRi}, $[H_1, H_2]=0$. Hence
\begin{equation}\label{eq:a}
\mathfrak{a}=\{H(\xi_1, \xi_2)=\xi_1 H_1 +\xi_2 H_2 | \xi_1, \xi_2
\in \mathbf{R} \}
\end{equation}
is a maximal abelian subalgebra in $\mathfrak{p}$. Remark that this
maximal abelian subalgebra $\mathfrak a$ is different from the one
given by M.~Ise and used in \cite{Ozeki-TakeuchiII}. Then by direct
computations using \eqref{eq:D_{i,pq}}-\eqref{eq:DiDiDi}, we get the
following restricted root system decomposition of $\mathfrak{k}$ and
$\mathfrak{p}$:
\begin{eqnarray*}
\mathfrak{k}&=&\mathfrak{k}_0 + \mathfrak{k}_{2\xi_1} +
\mathfrak{k}_{2\xi_2}+\mathfrak{k}_{\xi_1+\xi_2}+
\mathfrak{k}_{\xi_1-\xi_2}+\mathfrak{k}_{\xi_1}+\mathfrak{k}_{\xi_2},\\
\mathfrak{p}&=&\mathfrak{a} + \mathfrak{p}_{2\xi_1} +
\mathfrak{p}_{2\xi_2}+\mathfrak{p}_{\xi_1+\xi_2}+
\mathfrak{p}_{\xi_1-\xi_2}+\mathfrak{p}_{\xi_1}+\mathfrak{p}_{\xi_2},
\end{eqnarray*}
where
\begin{equation*}
\begin{split}
\mathfrak{k}_0&=\{X\in{\mathfrak k}{\ }\vert{\ }[X,H]=0\quad
\text{ for each }H\in{\mathfrak a}\},\\
&= \mathrm{span}_{\mathbf R}\{\sqrt{-1} R(e_1-2e_2+e_3)\} +
\mathrm{span}_{\mathbf R}\{-D_{1,4}+D_{1,12}, D_{1,12}+D_{1,36}, \\
&\quad D_{1,36}+D_{1,57}, -D_{1,1}+D_{1,24}, -D_{1,2}-D_{1,14}, -D_{1,3}+D_{1,46}, -D_{1,5}-D_{1,47},\\
&\quad -D_{1,6}+D_{1,34}, -D_{1,7}+D_{1,45}, D_{1,13}-D_{1,26}, D_{1,15}+D_{1,27}, D_{1,16}+D_{1,23},\\
&\quad D_{1,17}-D_{1,25}, D_{1,35}-D_{1,67}, D_{1,37}-D_{1,56}\},\\
\end{split}
\end{equation*}
\begin{equation*}
\begin{split}
{\mathfrak k}_{2\xi_1}
&=\mathrm{span}_{\mathbf R}\{\frac{1}{2}(-D_{1,4}-D_{1,12}+D_{1,36}-D_{1,57})+\sqrt{-1}R(e_3-e_1)\},\\
{\mathfrak k}_{2\xi_2}&=\mathrm{span}_{\mathbf R}
\{\frac{1}{2}(D_{1,4}+D_{1,12}-D_{1,36}+D_{1,57})+\sqrt{-1}R(e_3-e_1)\},
\end{split}
\end{equation*}
\begin{equation*}
\begin{split}
{\mathfrak k}_{\xi_1+\xi_2}&= \mathrm{span}_{\mathbf R}\{
-D_{1,1}-D_{1,24}-D_{1,37}-D_{1,56}=2D_{2,1},\\
\quad\quad \quad &  -D_{1,2}+D_{1,14}-D_{1,35}-D_{1,67}=2D_{2,2},\\
\quad \quad\quad &  -D_{1,3}-D_{1,46}+D_{1,17}+D_{1,25}=2D_{2,3},\\
\quad \quad\quad &  -D_{1,5}+D_{1,47}+D_{1,16}-D_{1,23}=2D_{2,5},\\
\quad \quad\quad &  -D_{1,6}-D_{1,34}-D_{1,15}+D_{1,27}=2D_{2,6},\\
\quad \quad\quad &  -D_{1,7}-D_{1,45}-D_{1,13}-D_{1,26}=2D_{2,7} \},
\end{split}
\end{equation*}
\begin{equation*}
\begin{split}
{\mathfrak k}_{\xi_1-\xi_2}&= \mathrm{span}_{\mathbf R}\{
-D_{1,1}-D_{1,24}+D_{1,37}+D_{1,56}=2D_{2,24},\\
\quad\quad \quad &-D_{1,2}+D_{1,14}+D_{1,35}+D_{1,67}=2D_{2,14},\\
\quad\quad \quad &-D_{1,3}-D_{1,46}-D_{1,25}-D_{1,17}=-2D_{2,46},\\
\quad\quad \quad &-D_{1,5}+D_{1,47}-D_{1,16}+D_{1,23}=2D_{2,47},\\
\quad\quad \quad &-D_{1,6}-D_{1,34}+D_{1,15}-D_{1,27}=-2D_{2,34},\\
\quad\quad \quad &-D_{1,7}-D_{1,45}+D_{1,13}+D_{1,26}=-2D_{2,45}\},
\end{split}
\end{equation*}
\begin{equation*}
\begin{split}
{\mathfrak k}_{\xi_1}&=\mathrm{span}_{\mathbf R}
\{D(x_1 \bar{u}_1)+\sqrt{-1}R(y_1 u_1), \, (x_1, y_1)=(1, c_4), (c_1, -c_2), (c_2, c_1),\\
\quad\quad\quad & (c_3, c_6), (c_4, -1), (c_5, -c_7), (c_6, -c_3), (c_7, c_5)\},\\
{\mathfrak k}_{\xi_2}&=\mathrm{span}_{\mathbf R}
\{D(x_1 \bar{u}_1)+\sqrt{-1}R(y_1 u_1), \, (x_1, y_1)=(1, -c_4), (c_1, c_2), (c_2, -c_1),\\
\quad\quad\quad & (c_3, -c_6), (c_4, 1), (c_5, c_7), (c_6, c_3),
(c_7, -c_5)\},
\end{split}
\end{equation*}
\begin{equation*}
\begin{split}
\mathfrak{p}_{2\xi_1}&=\mathrm{span}_{\mathbf R}\{ D(c_4 \bar{u}_2) -\sqrt{-1}R u_2\},\\
\mathfrak{p}_{2\xi_2}&=\mathrm{span}_{\mathbf R}\{ D(c_4 \bar{u}_2) +\sqrt{-1}R u_2\},\\
\mathfrak{p}_{\xi_1+\xi_2}&=\mathrm{span}_{\mathbf R}\{ D(c_i \bar{u}_2), \, i=1,2,3,5,6,7\},\\
\mathfrak{p}_{\xi_1-\xi_2}&=\mathrm{span}_{\mathbf R}\{ \sqrt{-1} R(c_i u_2), \, i=1,2,3,5,6,7\},\\
\mathfrak{p}_{\xi_1}&=\mathrm{span}_{\mathbf
R}\{D(x_3\bar{u}_3)+\sqrt{-1} R(y_3 u_3),\,
(x_3, y_3)=(1, c_4), (c_1, c_2), (c_2, c_1),\\
\quad\quad\quad & (c_3, c_6), (c_4, -1), (c_5, -c_7), (c_6, -c_3), (c_7, c_5)\}\\
\mathfrak{p}_{\xi_2}&=\mathrm{span}_{\mathbf
R}\{D(x_3\bar{u}_3)+\sqrt{-1} R(y_3 u_3),\,
(x_3, y_3)=(1, -c_4), (c_1, -c_2), (c_2, c_1),\\
\quad\quad\quad & (c_3, c_6), (c_4, 1), (c_5, -c_7), (c_6, -c_3),
(c_7, c_5)\}.
\end{split}
\end{equation*}
Thus we see that
\begin{equation*}
\begin{split}
\mathfrak{k}_0&= \mathfrak{k}_0^\prime+\mathfrak{c}(\mathfrak{k}_0)
=\mathfrak{k}_0^\prime +\mathbf{R}\sqrt{-1}R(e_1-2e_2+e_3) \cong \mathfrak{so}(6)+\mathbf{R},\\
{\mathfrak k}_{1}&:=\mathfrak{k}_0+ \mathfrak{k}_{2\xi_1}+\mathfrak{k}_{2\xi_2}\\
&=\mathfrak{k}_0^\prime+{\mathbf R}\sqrt{-1}R( e_{1}-2e_{2}+e_{3})+ \mathbf{R}(D_{1,4}+D_{1,12}-D_{1,36}+D_{1,57})\\
&\quad\quad\quad +{\mathbf R} \sqrt{-1}R(e_{3}-e_{1})\\
&\cong \mathfrak{so}(6)+\mathbf{R}+\mathbf{R}+\mathbf{R},
\\
\end{split}
\end{equation*}
\begin{equation*}
\begin{split}
\mathfrak{k}_2&:=\mathfrak{k}_1+\mathfrak{k}_{\xi_1+\xi_2}+\mathfrak{k}_{\xi_1-\xi_2}
=\mathfrak{D}_0+\sqrt{-1}\mathfrak{R}_0\\
&=\mathfrak{D}_0+ \mathbf{R}\sqrt{-1}R(e_1-2e_2+e_3)+\mathbf{R} \sqrt{-1} R(e_3-e_1)\\
&=\mathfrak{D}_0+{\mathbf R} \sqrt{-1}R(e_{2}-e_{3})
+{\mathbf R}\sqrt{-1}R(2 e_{1}-e_{2}-e_{3})\\
&\cong \mathfrak{so}(8)+\mathbf{R}+\mathbf{R},\\
\mathfrak{k}&:=\mathfrak{k}_2+\mathfrak{k}_{\xi_1}+\mathfrak{k}_{\xi_2}
=\mathfrak{D}_0+\sqrt{-1}\mathfrak{R}_0+\mathfrak{D}_1+\sqrt{-1}\mathfrak{R}_1\\
&=(\mathfrak{D}_0+\mathfrak{D}_1+\sqrt{-1}\mathfrak{R}_1+\mathbf{R}\sqrt{-1}R(e_2-e_3))
+\mathbf{R}\sqrt{-1}R(2e_1-e_2-e_3)\\
&=\mathfrak{k}^{\prime}+\mathfrak{c}(\mathfrak{k})\cong
\mathfrak{so}(10)+\mathbf{R}.
\end{split}
\end{equation*}

Correspondingly, consider the subgroup
\begin{equation*}
\tilde{K}_2=U(1)\times Spin(2)\times Spin(8) \subset
\tilde{K}=U(1)\times Spin(10),
\end{equation*}
where $U(1)$ is given by \eqref{eq:u(1)}, $Spin(2)\subset
Spin(10)\cong (E_6)_{e_1}$ is generated by
\begin{equation*}
\begin{split}
\alpha_{23}(t)&:=\exp(t\sqrt{-1}R(e_2-e_3)):\\
&
\begin{pmatrix}
\xi_{1}&x_{3}&\bar{x}_{2}\\
\bar{x}_{3}&\xi_{2}&x_{1}\\
x_{2}&\bar{x}_{1}&\xi_{3}
\end{pmatrix}
\mapsto
\begin{pmatrix}
\xi_{1}&e^{\frac{t\sqrt{-1}}{2}}x_{3}
&e^{-\frac{t\sqrt{-1}}{2}}\bar{x}_{2}\\
e^{\frac{t\sqrt{-1}}{2}}\bar{x}_{3}&e^{t\sqrt{-1}}\xi_{2}&x_{1}\\
e^{-\frac{t\sqrt{-1}}{2}}x_{2}&\bar{x}_{1} &e^{-t\sqrt{-1}}\xi_{3}
\end{pmatrix},
\end{split}
\end{equation*}
and $Spin(8)=(E_6)_{e_1,e_2,e_3}$ whose Lie algebra is just
$\mathfrak{D}_0$. Therefore,
\begin{equation*}
Spin(2)\cap Spin(8)=\{\alpha_{23}(t){\ } \vert {\ }
e^{t\sqrt{-1}}=1\}=\{\alpha_{23}(0), \alpha_{23}(2\pi)\}.
\end{equation*}
Then the natural projection
\begin{equation*}
\begin{split}
p_2: \quad  & Spin(2)\times Spin(8)\rightarrow K_2^{\prime}\\
& (\alpha_{23}(t), \beta) \mapsto \alpha_{23}(t)\beta
\end{split}
\end{equation*}
has a kernel
\begin{equation*}
\begin{split}
\ker p_2&=\{(\alpha_{23}(t), \alpha_{23}(t)^{-1}) {\ }\vert {\ } t=2k\pi, k\in \mathbf{Z}\}\\
&=\{(\alpha_{23}(0), \alpha_{23}(0)), (\alpha_{23}(2\pi),
\alpha_{23}(2\pi))\} \cong \mathbf{Z}_2.
\end{split}
\end{equation*}
Hence $K^{\prime}_2\cong (Spin(2)\times Spin(8))/\mathbf{Z}_2$.

On the other hand, we also have
$$\tilde{K}_2 = S^1 \times Spin(2) \times Spin(8),$$
where $S^1$ is generated by
\begin{equation*}
\begin{split}
&\exp(t\sqrt{-1}R(e_1-2e_2+e_3)):\\
&
\begin{pmatrix}
\xi_{1}&x_{3}&\bar{x}_{2}\\
\bar{x}_{3}&\xi_{2}&x_{1}\\
x_{2}&\bar{x}_{1}&\xi_{3}
\end{pmatrix}
\mapsto
\begin{pmatrix}
e^{t\sqrt{-1}}\xi_{1} & e^{-\frac{t\sqrt{-1}}{2}}x_{3} & e^{t\sqrt{-1}}\bar{x}_{2}\\
e^{-\frac{t\sqrt{-1}}{2}}\bar{x}_{3}& e^{-2t\sqrt{-1}}\xi_{2}& e^{-\frac{t\sqrt{-1}}{2}} x_{1}\\
e^{t\sqrt{-1}}x_{2}& e^{-\frac{t\sqrt{-1}}{2}} \bar{x}_{1}
&e^{t\sqrt{-1}}\xi_{3}
\end{pmatrix},
\end{split}
\end{equation*}
$Spin(2)\subset E_6$ is generated by
\begin{equation*}
\begin{split}
\alpha_{31}(t)&:=\exp(t\sqrt{-1}R(e_3-e_1)):\\
&
\begin{pmatrix}
\xi_{1}&x_{3}&\bar{x}_{2}\\
\bar{x}_{3}&\xi_{2}&x_{1}\\
x_{2}&\bar{x}_{1}&\xi_{3}
\end{pmatrix}
\mapsto
\begin{pmatrix}
e^{-t\sqrt{-1}}\xi_{1}&e^{-\frac{t\sqrt{-1}}{2}}x_{3} &\bar{x}_{2}\\
e^{-\frac{t\sqrt{-1}}{2}}\bar{x}_{3}& \xi_{2}& e^{\frac{t\sqrt{-1}}{2}}x_{1}\\
x_{2}&e^{\frac{t\sqrt{-1}}{2}}\bar{x}_{1}& e^{t\sqrt{-1}}\xi_{3}
\end{pmatrix}.
\end{split}
\end{equation*}
and $Spin(8)=(E_6)_{e_1, e_2, e_3}$. Here $Spin(2) \times
Spin(8)\subset (E_6)_{e_2}\cong Spin(10)$. Similarly, here
\begin{equation*}
Spin(2)\cap Spin(8)=\{\alpha_{31}(t){\ } \vert {\ }
e^{t\sqrt{-1}}=1\}=\{\alpha_{31}(0), \alpha_{31}(2\pi)\}.
\end{equation*}
Then the natural projection
\begin{equation*}
\begin{split}
p_2^{\prime}: \quad  & Spin(2)\times Spin(8)\rightarrow K_2^{\prime}\\
& (\alpha_{31}(t), \beta) \mapsto \alpha_{31}(t)\beta
\end{split}
\end{equation*}
has a kernel
\begin{equation*}
\begin{split}
\ker p_2^{\prime}&=\{(\alpha_{31}(t), \alpha_{31}(t)^{-1}) {\ }\vert {\ } t=2k\pi, k\in \mathbf{Z}\}\\
&=\{(\alpha_{31}(0), \alpha_{31}(0)), (\alpha_{31}(2\pi),
\alpha_{31}(2\pi))\} \cong \mathbf{Z}_2.
\end{split}
\end{equation*}
Thus,
\begin{eqnarray*}
K_2&=&(S^1\times (Spin(2)\cdot Spin(8)))/\mathbf{Z}_4,\\
Spin(2)\cdot Spin(8)&=& (Spin(2)\times Spin(8))/\mathbf{Z}_2.
\end{eqnarray*}

Furthermore, we have
$$Spin(8)\supset Spin(2)\cdot Spin(6)\cong (Spin(2)\times Spin(6))/{\mathbf Z}_{2},$$
where
\begin{equation*}
\begin{split}
Spin(8)&= \{ (\alpha_{1},\alpha_{2},\alpha_{3}) \in{SO({\mathbf
K})\times SO({\mathbf K})\times SO({\mathbf K})}
{\ }\vert{\ }\\
&\quad\quad\quad\quad\quad\quad\quad
(\alpha_{1}x)(\alpha_{2}y)=\overline{\alpha_{3}(\overline{xy})}
\text{ for each }x,y\in{\mathbf K} \}
\end{split}
\end{equation*}
acts on $H_3(\mathbf{K})$ by
\begin{equation*}
(\alpha_{1},\alpha_{2},\alpha_{3})
\begin{pmatrix}
\xi_{1}&x_{3}&\bar{x}_{2}\\
\bar{x}_{3}&\xi_{2}&x_{1}\\
x_{2}&\bar{x}_{1}&\xi_{3}
\end{pmatrix}
:=
\begin{pmatrix}
\xi_{1}&\alpha_{3}x_{3}&\overline{\alpha_{2}{x}_{2}}\\
\overline{\alpha_{3}x_{3}}&\xi_{2}&\alpha_{1}x_{1}\\
\alpha_{2}{x}_{2}&\overline{\alpha_{1}{x}_{1}}&\xi_{3}
\end{pmatrix},
\end{equation*}
\begin{equation*}
Spin(2):=\{(\alpha_1, \alpha_2, \alpha_3)\in Spin(8){\ } \vert {\ }
\alpha_2(c_i)=c_i, \text{ if } i\neq 0,4\}
\end{equation*}
is generated by $D_{1,4}+D_{1,12}-D_{1,36}+D_{1,57}$ and
\begin{equation*}
Spin(6):=\{(\alpha_1, \alpha_2,\alpha_3)\in Spin(8) {\ } \vert {\ }
\alpha_2(1)=1, \alpha_2(c_4)=c_4 \}
\end{equation*}
is generated by $\mathfrak{k}_0^{\prime}$. Notice
\begin{equation*}
Spin(2)\cap Spin(6) =\{ (\mathrm{Id},\mathrm{Id},\mathrm{Id}),
(-\mathrm{Id},\mathrm{Id},-\mathrm{Id})\},
\end{equation*}
we see that
${\mathbf Z}_{2} =\{
((\mathrm{Id},\mathrm{Id},\mathrm{Id}),(\mathrm{Id},\mathrm{Id},\mathrm{Id})),
((-\mathrm{Id},\mathrm{Id},-\mathrm{Id}),(-\mathrm{Id},\mathrm{Id},-\mathrm{Id}))
\}. $
Thus, the connected compact Lie subgroup $K_1$ of $K$ generated by
$\mathfrak{k}_1$ is
\begin{equation*}
K_1=(S^1\times (Spin(2)\cdot (Spin(2)\cdot Spin(6))))/\mathbf{Z}_4.
\end{equation*}

Moreover,
\begin{equation*}
S^1 \cap Spin(6) =\{ (\mathrm{Id},\mathrm{Id},\mathrm{Id}),
(-\mathrm{Id},\mathrm{Id},-\mathrm{Id})\},
\end{equation*}
hence the connected compact Lie group $K_0$ of $K$ generated by
$\mathfrak{k}_0$ is
\begin{equation*}
K_0=(S^1\times Spin(6))/\mathbf{Z}_2,
\end{equation*}
where
${\mathbf Z}_{2} =\{
((\mathrm{Id},\mathrm{Id},\mathrm{Id}),(\mathrm{Id},\mathrm{Id},\mathrm{Id})),
((-\mathrm{Id},\mathrm{Id},-\mathrm{Id}),(-\mathrm{Id},\mathrm{Id},-\mathrm{Id}))
\}. $

\subsection{Isotropy representation of $(E_6, U(1)\cdot Spin(10))$}

Via the linear isomorphism $p_{*}: \mathfrak{p} \rightarrow
T_o(\mathrm{EIII})$ given by \eqref{eq:p_*}, we can describe the
isotropy representation of $(E_6, U(1)\cdot Spin(10))$.
\begin{lem}\label{p_*Ad}
\begin{enumerate}
\renewcommand{\labelenumi}{(\arabic{enumi})}
\item
For each $a\in{K}$ and each $\xi\in{\mathfrak p}$,
\begin{equation*}
p_{\ast}(\mathrm{Ad}(a)\xi)=(\mathrm{Ad}(a)\xi)(e_{1})
=(a\circ\xi\circ a^{-1})(e_{1}).
\end{equation*}
\item
For each $T\in{\mathfrak k}$ and each $\xi\in{\mathfrak p}$,
\begin{equation*}
p_{\ast}(\mathrm{ad}(T)\xi)
=
p_{\ast}([T,\xi])
=([T,\xi])(e_{1}).
\end{equation*}
\end{enumerate}
\end{lem}

The restriction $(\rho_K, V=H_3(\mathbf{K}^{\mathbf{C}}))$ of
Chevally-Schafer's representation $(\tilde{\rho},
H_3(\mathbf{K}^{\mathbf{C}}))$ of $E_6$ to $K$
can be decomposed into three irreducible representations
\begin{equation*}
(\rho_{K},V)=(\rho_{1},V_{1})\oplus(\rho_{2},V_{2})\oplus(\rho_{3},V_{3}),
\end{equation*}
where $V_1$, $V_2$ and $V_3$ are given as follows:
\begin{equation*}
\begin{split}
V_{1}&=\{\xi e_{1}{\ }\vert{\ }\xi\in{\mathbf C}\}, \\
V_{2}
&=(H_3(\mathbf{K}^{\mathbf C}))_{-\sigma}
=\{ x_2 u_2+x_3 u_3
\mid
x_{2},x_{3}\in{\mathbf{K}}^{\mathbf C} \}
\cong T_o(\mathrm{EIII}),
\\
V_{3}
&= H_2(\mathbf{K}^{\mathbf C})
=\{ \xi_2 e_2+\xi_3 e_3+ x_1 u_1 \mid
x_{1}\in{\mathbf{K}}^{\mathbf C}, \xi_2,\xi_3\in \mathbf{C} \},
\end{split}
\end{equation*}
and  $V_{1}\oplus V_{3}=(H_3(\mathbf{K}^{\mathbf C}))_{\sigma}$.
$\rho_1$ is a scalar representation, the restriction of $\rho_2$ to
$Spin(10)$ is equivalent to one of the half-spin representations of
$Spin(10,\mathbf{C})$, called $\Delta^{+}_{10}$, and the restriction
of $\rho_3$ to $Spin(10)$ is equivalent to the standard
representation of $Spin(10,\mathbf{C})$.

Now we discuss the linear isotropy action of an element
$\phi(\theta)=\exp(t\sqrt{-1}R(2e_{1}-e_{2}-e_{3})):
H_3(\mathbf{K}^{\mathbf C})\rightarrow H_3(\mathbf{K}^{\mathbf C})$
generating
the center $U(1)$ of $K$
on both ${\mathfrak p}$ and
$V_{2}=(H_3(\mathbf{K}^{\mathbf C}))_{-\sigma}$,
which are identified with $T_{o}(\mathrm{EIII})$.

Using the formula \eqref{defU(1)} and Lemma \ref{p_*Ad}, we compute
\begin{equation*}
\begin{split}
&
p_{\ast}(\mathrm{Ad}(\phi(\theta))D(x_2\bar{u}_2))
=
\theta^{-3} p_{\ast}(D(x_2\bar{u}_2)), \\
&p_{\ast}(\mathrm{Ad}(\phi(\theta))R(x_2 u_2))
= \theta^{-3} p_{\ast}(R(x_2 u_2)),\\
&p_{*}(Ad(\phi(\theta))D(x_3\bar{u}_3))=\theta^{-3}p_{*}(D(x_3\bar{u}_3)),\\
&p_{*}(Ad(\phi(\theta))R(x_3 u_3))=\theta^{-3}p_{*}(R(x_3 u_3)).
\end{split}
\end{equation*}
On the other hand,
the tangent vector space $T_{o}(\mathrm{EIII})$ at
$o=[e_1]\in\mathrm{EIII})\subset P({H_3(\mathbf{K})}^{\mathbf C})$
is identified with
the vector subspace $V_{2}=(H_3(\mathbf{K}^{\mathbf C}))_{-\sigma}$,
which is a horizontal vector subspace at a point $e_1$
under the Hopf fibration
${H_3(\mathbf{K})}^{\mathbf C}\supset S^{53}(1)
\rightarrow P({H_3(\mathbf{K})}^{\mathbf C})$.
By the formula \eqref{defU(1)} we see that
a vector $x_2u_2+x_3u_3\in (H_3(\mathbf{K}^{\mathbf C}))_{-\sigma}$
at a point $e_1$ in a vector space ${H_3(\mathbf{K})}^{\mathbf C}$
representing a tangent vector of $\mathrm{EIII}$ at $o=[e_1]$
is moved by
the linear action of $\phi(\theta)$ to a vector
$\theta x_2u_2+\theta x_3u_3\in (H_3(\mathbf{K}^{\mathbf C}))_{-\sigma}$
at $\theta^{4}e_1$.
Thus its corresponding tangent vector of $\mathrm{EIII}$ at $o=[e_1]$
must be $\theta^{-4}(\theta x_2u_2+\theta x_3u_3)
=\theta^{-3}(x_2u_2+ x_3u_3)\in V_2=(H_3(\mathbf{K}^{\mathbf C}))_{-\sigma}$
at $e_1$.
Hence the linear isotropy action of $\phi(\theta)$
on $V_2=(H_3(\mathbf{K}^{\mathbf C}))_{-\sigma}$
is given by the multiplication by $\theta^{-3}$ on
$V_2=(H_3(\mathbf{K}^{\mathbf C}))_{-\sigma}$.

Therefore the linear isotropy representation of
$(E_6, U(1)\cdot Spin(10))$ is
$(\mu_3\otimes_{\mathbf{C}}\triangle^{+}_{10})_{\mathbf{R}}$.

\subsection{The subgroup $K_{[\mathfrak{a}]}$}

The maximal abelian subspace ${\mathfrak a}$ of
${\mathfrak p}$ is described as follows:
\begin{equation*}
\begin{split}
{\mathfrak a} &=
{\mathbf R}H_{1}\oplus{\mathbf R}H_{2}\\
&= {\mathbf R}(D\bar{u}_{2}+\sqrt{-1}R({\mathbf c}_{4}u_{2})) \oplus
{\mathbf R}(D\bar{u}_{2} -\sqrt{-1}R({\mathbf c}_{4}u_{2})).
\end{split}
\end{equation*}
and
\begin{equation}\label{eq:p_* a}
\begin{split}
p_{\ast}({\mathfrak a}) &={\mathbf R}(1+\sqrt{-1}{\mathbf c}_{4})u_2
\oplus{\mathbf R}(1-\sqrt{-1}{\mathbf c}_{4})u_2.
\end{split}
\end{equation}

We will use the map $\varphi: Sp(4) \rightarrow E_6$ given by Yokota
(\cite{IchiroYokota09}) and the known results for $(Sp(4),
Sp(2)\times Sp(2))$ case to find a generator of $K_{[\mathfrak a]}$
here.

The Cayley algebra $\mathbf{K}$ naturally contains the field
$\mathbf{H}$ of quaternions as
\begin{equation*}
\mathbf{H}=\{x_0+x_2 c_2 +x_3 c_3+x_5 c_5 | x_i\in \mathbf{R}\}.
\end{equation*}
Any element $x\in \mathbf{K}$ can be expressed by
\begin{equation*}
\begin{split}
x&=x_0+x_1 c_1 + x_2 c_2 +x_3 c_3 + x_4 c_4 + x_5 c_5 + x_6 c_6 + x_7 c_7\\
&=(x_0+x_2 c_2 +x_3 c_3+x_5 c_5)+ (x_4+x_1 c_2 +x_6 c_3-x_7 c_5)c_4\\
&=: m + a \mathbf{e}\in \mathbf{H}\oplus
\mathbf{H}\mathbf{e}=\mathbf{K},
\end{split}
\end{equation*}
where $m:=x_0+x_2 c_2 +x_3 c_3+x_5 c_5\in \mathbf{H}$, $a:=x_4+x_1
c_2 +x_6 c_3 -x_7 c_5\in \mathbf{H}$ and $\mathbf{e}:=c_4$. In
$\mathbf{H}\oplus \mathbf{H}\mathbf{e}$, we define a multiplication
by
\begin{equation*}
(m+a\mathbf{e})(n+b\mathbf{e})=(mn-\bar{b}a)+(a\bar{n}+bm)\mathbf{e}.
\end{equation*}
More explicitly,
\begin{equation*}
(a\mathbf{e})n=(a\bar{n})\mathbf{e}, \quad
m(b\mathbf{e})=(bm)\mathbf{e}, \quad
(a\mathbf{e})(b\mathbf{e})=-\bar{b}a.
\end{equation*}
We can also define a conjugation and an $\mathbf{R}$-linear
transformation $\gamma$ on $\mathbf{H}\oplus \mathbf{H}\mathbf{e}$
respectively by
\begin{equation*}
\overline{m+a\mathbf{e}}=\bar{m}-a\mathbf{e},\quad\quad
\gamma(m+a\mathbf{e})=m-a\mathbf{e}.
\end{equation*}
Thus $\gamma\in G_2=\{\alpha\in \mathrm{Iso}_{\mathbf R}(\mathbf{K})
{\ }\vert {\ } \alpha(xy)=\alpha(x)\alpha(y)\}$. Consider an
$\mathbf{R}$-linear transformation of  $H_3(\mathbf{K})$, denoted
still by $\gamma$, defined by
\begin{equation*}
\gamma \left(
         \begin{array}{ccc}
           \xi_1 & x_3 & \bar{x}_2 \\
           \bar{x}_3 & \xi_2 & x_1 \\
           x_2 & \bar{x}_1 & \xi_3 \\
         \end{array}
       \right)
       := \left(
         \begin{array}{ccc}
           \xi_1 & \gamma x_3 & \overline{\gamma x_2} \\
           \overline{\gamma x_3} & \xi_2 & \gamma x_1 \\
           \gamma x_2 & \overline{\gamma x_1} & \xi_3 \\
         \end{array}
       \right),
\end{equation*}
for $x_i\in \mathbf{K}$ $(i=1,2,3)$. Thus $\gamma\in F_4
=\{\alpha\in \mathrm{Iso}_{\mathbf{R}}(H_3(\mathbf{K})){\ }\vert {\
} \alpha(X\circ Y)=\alpha(X)\circ \alpha(Y)\}$. Any element
\begin{equation*}
X=\left(
         \begin{array}{ccc}
           \xi_1 & x_3 & \bar{x}_2 \\
           \bar{x}_3 & \xi_2 & x_1 \\
           x_2 & \bar{x}_1 & \xi_3 \\
         \end{array}
       \right)
       =\left(
         \begin{array}{ccc}
           \xi_1 & m_3 & \bar{m}_2 \\
           \bar{m}_3 & \xi_2 & m_1 \\
           m_2 & \bar{m}_1 & \xi_3 \\
         \end{array}
       \right)
       +
       \left(
         \begin{array}{ccc}
           0 & a_3\mathbf{e} & -a_2\mathbf{e} \\
           -a_3\mathbf{e} & 0 & a_1\mathbf{e} \\
           a_2 \mathbf{e} & -a_1 \mathbf{e} & 0 \\
         \end{array}
       \right),
\end{equation*}
of $H_3(\mathbf{K})$, where $x_i=m_i+a_i\mathbf{e}\in
\mathbf{H}\oplus \mathbf{H}\mathbf{e}=\mathbf{K}$ and $\xi_i\in
\mathbf{R}$, can be identified with the element
\begin{equation*}
\left(
         \begin{array}{ccc}
           \xi_1 & m_3 & \bar{m}_2 \\
           \bar{m}_3 & \xi_2 & m_1 \\
           m_2 & \bar{m}_1 & \xi_3 \\
         \end{array}
       \right) + (a_1, a_2, a_3)
\end{equation*}
in $H_3(\mathbf{H})\oplus \mathbf{H}^3$. Hereafter, there exists an
identification $H_3(\mathbf{K})\cong H_3(\mathbf{H})\oplus
\mathbf{H}^3$.

Let the $\mathbf{C}$-linear mapping $\gamma :
H_3(\mathbf{K}^\mathbf{C}) \rightarrow H_3(\mathbf{K}^{\mathbf{C}})$
be the complexification of $\gamma\in G_2\subset F_4$. Then
$\gamma\in E_6$ and $\gamma^2=1$. Recall that $\tau$ is the complex
conjugation of $H_3(\mathbf{K}^{\mathbf C})$ with respect to
$H_3(\mathbf K)$. Consider an involutive complex conjugate linear
transformation $\tau \gamma$ of $H_3(\mathbf{K}^{\mathbf{C}})$ and
the following subgroup $(E_6)^{\tau\gamma}$ of $E_6$:
\begin{equation*}
(E_6)^{\tau\gamma}=\{\alpha\in E_6 {\ }\vert {\ } \tau\gamma
\alpha=\alpha \tau\gamma\}.
\end{equation*}
Correspondingly, $H_3(\mathbf{K}^{\mathbf{C}})$ can be decomposed
into the following two $\mathbf{R}$-vector subspaces:
\begin{equation*}
H_3(\mathbf{K}^{\mathbf{C}})
=(H_3(\mathbf{K}^{\mathbf{C}}))_{\tau\gamma}\oplus
(H_3(\mathbf{K}^{\mathbf{C}}))_{-\tau\gamma},
\end{equation*}
where
\begin{equation*}
\begin{split}
&(H_3(\mathbf{K}^{\mathbf{C}}))_{\tau\gamma}:=\{X\in
H_3(\mathbf{K}^{\mathbf{C}})
\mid \tau\gamma X=X\}\\
&=\{\left(
      \begin{array}{ccc}
        \xi_1 & m_3 & \bar{m}_2 \\
        \bar{m}_3 & \xi_2 & m_1 \\
        m_2 & \bar{m}_1 & \xi_3 \\
      \end{array}
    \right)
    +\sqrt{-1} \begin{pmatrix}
0&a_{3}e&-a_{2}e\\
-a_{3}e&0&a_{1}e\\
a_{2}e&-a_{1}e&0
\end{pmatrix}
\mid \xi_{i}\in{\mathbf R}, m_{i}, a_{i}\in{\mathbf H}
\}\\
&=H_3(\mathbf{H})\oplus \sqrt{-1}\mathbf{H}^3,
\end{split}
\end{equation*}
\begin{equation*}
\begin{split}
&(H_3(\mathbf{K}^{\mathbf{C}}))_{-\tau\gamma}:=\{X\in
H_3(\mathbf{K}^{\mathbf{C}})
\mid \tau\gamma X=-X\}\\
&=\{\sqrt{-1}\left(
      \begin{array}{ccc}
        \xi_1 & m_3 & \bar{m}_2 \\
        \bar{m}_3 & \xi_2 & m_1 \\
        m_2 & \bar{m}_1 & \xi_3 \\
      \end{array}
    \right)
    +\begin{pmatrix}
0&a_{3}e&-a_{2}e\\
-a_{3}e&0&a_{1}e\\
a_{2}e&-a_{1}e&0
\end{pmatrix}
{\ }\vert{\ } \xi_{i}\in{\mathbf R}, m_{i}, a_{i}\in{\mathbf H}
\}\\
&=\sqrt{-1} H_3(\mathbf{H})\oplus \mathbf{H}^3.
\end{split}
\end{equation*}
In particular,
$H_3(\mathbf{K}^{\mathbf{C}})=((H_3(\mathbf{K}^{\mathbf{C}}))_{\tau\gamma})^{\mathbf{C}}$.

Let $H_4(\mathbf{H})_0:=\{P\in H_4(\mathbf{H}){\ }\vert {\ }
\mathrm{tr} P=0\}$.
Define a $\mathbf{C}$-linear isomorphism 
$g: H_3(\mathbf{K}^{\mathbf{C}})=H_3(\mathbf{H}^{\mathbf{C}})\oplus
(\mathbf{H}^3)^{\mathbf{C}} \rightarrow H_4(\mathbf{H})_0^{\mathbf
C}$ by
\begin{equation*}
g(M+{\mathbf a}) :=
\begin{pmatrix}
\frac{1}{2}\mathrm{tr}(M)&\sqrt{-1}{\mathbf a}\\
\sqrt{-1}{\mathbf a}^{\ast}&M-\frac{1}{2}\mathrm{tr}(M)\mathrm{I}
\end{pmatrix},
\end{equation*}
for $M+{\mathbf a}\in H_3(\mathbf{K}^{\mathbf{C}})$. Then we have
\begin{equation*}
\begin{split}
&g((H_3(\mathbf{K}^{\mathbf{C}}))_{\tau\gamma})=H_4(\mathbf{H})_0,
\\
&g((H_3(\mathbf{K}^{\mathbf{C}}))_{-\tau\gamma})
=\sqrt{-1}H_4(\mathbf{H})_0.
\end{split}
\end{equation*}

The mapping $ \varphi : Sp(4)\longrightarrow (E_{6})^{\tau\gamma}
\subset E_{6} $, defined by $\varphi(A)X:=g^{-1}(A(gX)A^{\ast})$ for
each $X\in H_3(\mathbf{K}^{\mathbf{C}})$, is a surjective Lie group
homomorphism and
$\mathrm{Ker}(\varphi)=\{\mathrm{I},-\mathrm{I}\}\cong{\mathbf
Z}_{2}$. Therefore we obtain
\begin{equation*}
Sp(4)/{\mathbf Z}_{2}\cong (E_{6})^{\tau\gamma}.
\end{equation*}
Consider $\mathbf{R}$-vector subspaces $(H_3(\mathbf{K}^{\mathbf
C}))_{\tau\gamma,\sigma}$, $(H_3(\mathbf{K}^{\mathbf
C}))_{\tau\gamma,-\sigma}$ of $(H_3(\mathbf{K}^{\mathbf
C}))_{\tau\gamma}$ and $(H_3(\mathbf{K}^{\mathbf
C}))_{-\tau\gamma,\sigma}$, $(H_3(\mathbf{K}^{\mathbf
C}))_{-\tau\gamma,-\sigma}$ of $(H_3(\mathbf{K}^{\mathbf
C}))_{-\tau\gamma}$, which are eigenspaces of $\sigma$, respectively
given by
\begin{equation*}
\begin{split}
&(H_3(\mathbf{K}^{\mathbf C}))_{\tau\gamma,\sigma}
=\{X\in H_3(\mathbf{K}^{\mathbf C}){ }\vert {\ } \tau\gamma X=X, \sigma X=X\}\\
=& \{
\begin{pmatrix}
\xi_{1}&0&0\\
0&\xi_{2}&m_{1}\\
0&\bar{m}_{1}&\xi_{3}
\end{pmatrix}
+ \sqrt{-1}
\begin{pmatrix}
0&0&0\\
0&0&a_{1}\mathbf{e}\\
0&-a_{1}\mathbf{e}&0
\end{pmatrix}
\mid \xi_{i}\in{\mathbf R}, m_{1}, a_{1}\in{\mathbf H} \},
\end{split}
\end{equation*}

\begin{equation*}
\begin{split}
&(H_3(\mathbf{K}^{\mathbf C}))_{\tau\gamma,-\sigma}
=\{X\in H_3(\mathbf{K}^{\mathbf C}){ }\vert {\ } \tau\gamma X=X, \sigma X=-X\}\\
=& \{
\begin{pmatrix}
0&m_3&\bar{m}_2\\
\bar{m}_3&0&0\\
m_2&0&0
\end{pmatrix}
+ \sqrt{-1}
\begin{pmatrix}
0&a_3\mathbf{e}&-a_2\mathbf{e}\\
-a_3\mathbf{e}&0& 0\\
a_2 \mathbf{e}&0&0
\end{pmatrix}
\mid m_{2},m_{3}, a_{2}, a_{3}\in{\mathbf H} \},
\end{split}
\end{equation*}

\begin{equation*}
\begin{split}
&(H_3(\mathbf{K}^{\mathbf C}))_{-\tau\gamma,\sigma}
=\{X\in H_3(\mathbf{K}^{\mathbf C}){ }\vert {\ } \tau\gamma X=-X, \sigma X=X\}\\
=& \{\sqrt{-1}
\begin{pmatrix}
\xi_{1}&0&0\\
0&\xi_{2}&m_{1}\\
0&\bar{m}_{1}&\xi_{3}
\end{pmatrix}
+
\begin{pmatrix}
0&0&0\\
0&0&a_{1}\mathbf{e}\\
0&-a_{1}\mathbf{e}&0
\end{pmatrix}
{\ }\vert{\ } \xi_{i}\in{\mathbf R}, m_{1}, a_{1}\in{\mathbf H} \},
\end{split}
\end{equation*}

\begin{equation*}
\begin{split}
&(H_3(\mathbf{K}^{\mathbf C}))_{-\tau\gamma,-\sigma}
=\{X\in H_3(\mathbf{K}^{\mathbf C}){ }\vert {\ } \tau\gamma X=-X, \sigma X=-X\}\\
=& \{\sqrt{-1}
\begin{pmatrix}
0&m_3&\bar{m}_2\\
\bar{m}_3&0&0\\
m_2&0&0
\end{pmatrix}
+
\begin{pmatrix}
0&a_3\mathbf{e}&-a_2\mathbf{e}\\
-a_3\mathbf{e}&0& 0\\
a_2 \mathbf{e}&0&0
\end{pmatrix}
{\ }\vert{\ } m_{2},m_{3}, a_{2}, a_{3}\in{\mathbf H} \}.
\end{split}
\end{equation*}

Thus we have the following decompositions
\begin{eqnarray*}
(H_3(\mathbf{K}^{\mathbf C}))_{\sigma}&=&(H_3(\mathbf{K}^{\mathbf
C}))_{\tau\gamma,\sigma}
\oplus(H_3(\mathbf{K}^{\mathbf C}))_{-\tau\gamma,\sigma},\\
(H_3(\mathbf{K}^{\mathbf C}))_{-\sigma}&=&(H_3(\mathbf{K}^{\mathbf
C}))_{\tau\gamma,-\sigma} \oplus(H_3(\mathbf{K}^{\mathbf
C}))_{-\tau\gamma,-\sigma}.
\end{eqnarray*}

Note that the images of $(H_3(\mathbf{K}^{\mathbf
C}))_{\tau\gamma,\sigma}$ and $(H_3(\mathbf{K}^{\mathbf
C}))_{\tau\gamma,-\sigma}$ of the mapping $g$ defined above can be
expressed explicitly as follows:
\begin{equation*}
\begin{split}
&g((H_3(\mathbf{K}^{\mathbf C}))_{\tau\gamma,\sigma})\\
=& \{
\begin{pmatrix}
\frac{1}{2}(\xi_{1}+\xi_{2}+\xi_{3})&-a_{1}&0&0\\
-\bar{a}_{1}&\frac{1}{2}(\xi_{1}-\xi_{2}-\xi_{3})&0&0\\
0&0&\frac{1}{2}(-\xi_{1}+\xi_{2}-\xi_{3})&m_{1}\\
0&0&\bar{m}_{1}&\frac{1}{2}(-\xi_{1}-\xi_{2}+\xi_{3})
\end{pmatrix}
\\
& \quad {\ }\vert{\ } \xi_{1},\xi_{2},\xi_{3}\in \mathbf{R}, a_{1},
m_1\in \mathbf{H} \},
\end{split}
\end{equation*}
\begin{equation*}
\begin{split}
g((H_3(\mathbf{K}^{\mathbf C}))_{\tau\gamma,-\sigma}) =& \{
\begin{pmatrix}
0&0& -a_{2}&-a_{3}\\
0&0&m_{3}&\bar{m}_{2}\\
-\bar{a}_{2}&\bar{m}_{3}&0&0\\
-\bar{a}_{3}&m_{2}&0&0
\end{pmatrix}
{\ }\vert{\ } a_{2},a_{3}, m_{2},m_{3}\in \mathbf{H} \}.
\end{split}
\end{equation*}

For any element $A\in Sp(2)\times Sp(2)\subset Sp(4)$, we can check
that $\varphi(A)\sigma=\sigma \varphi(A)$, hence $\varphi(A)\in
(E_6)^{\sigma}$ and
we have
\begin{equation*}
\varphi : Sp(2)\times Sp(2)\longrightarrow
(E_{6})^{\tau\gamma,\sigma}\subset (E_{6})^{\sigma}\cong U(1)\cdot
Spin(10).
\end{equation*}
Next, the restriction of $\varphi$ to the subgroup $Sp(1)\times
Sp(1) \times Sp(1) \times Sp(1)$ gives
\begin{equation*}
\begin{split}
\varphi : Sp(1)\times Sp(1)\times Sp(1)\times Sp(1) \longrightarrow
&\{\alpha\in E_{6}{\ }\vert{\ }\alpha(e_{i})=e_{i}{\ }(i=1,2,3)\}\\
\cong &Spin(8).
\end{split}
\end{equation*}
And the group $Sp(1)\times Sp(1)$ can be considered as the diagonal
subgroup of $Sp(1)\times Sp(1)\times Sp(1)\times Sp(1)$, namely,
each $(a,b)\in Sp(1)\times Sp(1)$ corresponds to $(a,b,a,b)\in
Sp(1)\times Sp(1)\times Sp(1)\times Sp(1)$. Thus the restriction of
$\varphi$ to $Sp(1)\times Sp(1)$ is mapped to a subgroup
$K_0=S^1\cdot Spin(6)$ of $K=E^{\sigma}=U(1)\cdot Spin(10)$. In
fact, for a $2$-dimensional $\mathbf{R}$-vector subspace
\begin{equation*}
\tilde{\mathfrak{a}}:= \{
\begin{pmatrix}
0&0&a_{2}&0\\
0&0&0&m_{2}\\
a_{2}&0&0&0\\
0&m_{2}&0&0
\end{pmatrix}
{\ }\vert{\ } a_{2}, m_{2}\in{\mathbf R} \} \subset
g((H_3(\mathbf{K}^{\mathbf C}))_{\tau\gamma,-\sigma}),
\end{equation*}
it follows from
\begin{equation*}
\begin{split}
&
\begin{pmatrix}
a&0&0&0\\
0&b&0&0\\
0&0&a&0\\
0&0&0&b
\end{pmatrix}
\begin{pmatrix}
0&0&a_{2}&0\\
0&0&0&m_{2}\\
a_{2}&0&0&0\\
0&m_{2}&0&0
\end{pmatrix}
\begin{pmatrix}
a^{\ast}&0&0&0\\
0&b^{\ast}&0&0\\
0&0&a^{\ast}&0\\
0&0&0&b^{\ast}
\end{pmatrix}
\\
=&
\begin{pmatrix}
0&0&a_{2}&0\\
0&0&0&m_{2}\\
a_{2}&0&0&0\\
0&m_{2}&0&0
\end{pmatrix}.
\end{split}
\end{equation*}
that $\tilde{\mathfrak{a}}$ corresponds to the subspace
\begin{equation*}
\{
\begin{pmatrix}
0&0&m_{2}-\sqrt{-1}a_{2}\mathbf{e}\\
0&0&0\\
m_{2}+\sqrt{-1}a_{2}\mathbf{e}&0&0
\end{pmatrix}
{\ }\vert{\ } m_{2},a_{2}\in{\mathbf R} \} \subset {\
}(H_3(\mathbf{K})^{\mathbf C})_{\tau\gamma,-\sigma},
\end{equation*}
which corresponds to the image $\mathfrak{p}_{*}(\mathfrak a)$ of
the maximal abelian subspace $\mathfrak{a}$ of $\mathfrak{p}$ under
the linear isomorphism $\mathfrak{p}_{*}$ given by \eqref{eq:p_* a}.
It implies that $\varphi$ maps the subgroup $\check{K}_0=Sp(1)\times
Sp(1)$ for the exceptional symmetric space $(E_6,
Sp(4)/\mathbf{Z}_2)$ of type EI to the subgroup $K_0=S^1\cdot
Spin(6)$ of the exceptional symmetric space $(E_6, U(1)\cdot
Spin(10))$ of type EIII.

Recall that
\begin{equation*}
\check{k}:=
\begin{pmatrix}
0&1&0&0\\
1&0&0&0\\
0&0&0&-1\\
0&0&1&0
\end{pmatrix}
\in \check{K}_{[\check{\mathfrak{a}}]}=(Sp(1)\times Sp(1))\cdot
\mathbf{Z}_4
\end{equation*}
is a generator of $\mathbf{Z}_4$. Its adjoint actions on
$g((H_3(\mathbf{K}^{\mathbf{C}}))_{\tau\gamma,\sigma})$ and
$g((H_3(\mathbf{K}^{\mathbf{C}}))_{\tau\gamma, -\sigma})$ are given
in the following:
\begin{equation*}
\begin{split}
&
\check{k}
\begin{pmatrix}
\frac{1}{2}(\xi_{1}+\xi_{2}+\xi_{3})&-a_{1}&0&0\\
-\bar{a}_{1}&\frac{1}{2}(\xi_{1}-\xi_{2}-\xi_{3})&0&0\\
0&0&\frac{1}{2}(-\xi_{1}+\xi_{2}-\xi_{3})&m_{1}\\
0&0&\bar{m}_{1}&\frac{1}{2}(-\xi_{1}-\xi_{2}+\xi_{3})
\end{pmatrix}
\check{k}^{-1}
\\
=&
\begin{pmatrix}
\frac{1}{2}(\xi_{1}-\xi_{2}-\xi_{3})&-\bar{a}_{1}&0&0\\
-a_{1}&\frac{1}{2}(\xi_{1}+\xi_{2}+\xi_{3})&0&0\\
0&0&\frac{1}{2}(-\xi_{1}-\xi_{2}+\xi_{3})&-\bar{m}_{1}\\
0&0&-m_{1}&-\frac{1}{2}(-\xi_{1}+\xi_{2}-\xi_{3})\\
\end{pmatrix},
\end{split}
\end{equation*}

\begin{equation*}
\begin{split}
&
\check{k}
\begin{pmatrix}
0&0&-a_{2}&-a_{3}\\
0&0&m_{3}&\bar{m}_{2}\\
-\bar{a}_{2}&\bar{m}_{3}&0&0\\
-\bar{a}_{3}&m_{2}&0&0
\end{pmatrix}
\check{k}^{-1}
=
\begin{pmatrix}
0&0&-\bar{m}_{2}&m_{3}\\
0&0&a_{3}&-a_{2}\\
-m_{2}&\bar{a}_{3}&0&0\\
\bar{m}_{3}&-\bar{a}_{2}&0&0\\
\end{pmatrix}.
\end{split}
\end{equation*}

Taking $(H_3(\mathbf{K}^{\mathbf
C}))_{\tau\gamma}=(H_3(\mathbf{K}^{\mathbf C}))_{\tau\gamma, \sigma}
\oplus (H_3(\mathbf{K}^{\mathbf C}))_{\tau\gamma,-\sigma}$ and
$H_3(\mathbf{K}^{\mathbf C})=((H_3(\mathbf{K}^{\mathbf
C}))_{\tau\gamma})^{\mathbf C}$ into account, together with the
above computation, we know that any element
\begin{equation*}
\left(
   \begin{array}{ccc}
     \xi_1 & x_3 & \bar{x}_2 \\
     \bar{x}_3 & \xi_2 & x_1 \\
     x_2 & \bar{x}_1 & \xi_3 \\
   \end{array}
 \right)
 = \left(
   \begin{array}{ccc}
     \xi_1 & m_3+\sqrt{-1}a_3\mathbf{e} & \bar{m}_2-\sqrt{-1}a_2\mathbf{e} \\
     \bar{m}_3-\sqrt{-1}a_3 \mathbf{e}& \xi_3 & m_1+\sqrt{-1}a_1\mathbf{e} \\
     m_2+\sqrt{-1}a_2\mathbf{e} & \bar{m}_1-\sqrt{-1}a_1\mathbf{e} & \xi_3 \\
   \end{array}
 \right)
\end{equation*}
 in $H_3(\mathbf{K}^{\mathbf C})$ is mapped by the adjoint action of $\check{k}$ up to isomorphism
to the element
\begin{eqnarray*}
&&\left(
   \begin{array}{ccc}
     \xi_1 & a_3-\sqrt{-1}m_3\mathbf{e} & -a_2-\sqrt{-1}\bar{m}_2\mathbf{e} \\
     \bar{a}_3+\sqrt{-1}m_3 \mathbf{e}& -\xi_2 & -\bar{m}_1+\sqrt{-1}\bar{a}_1\mathbf{e} \\
     -\bar{a}_2+\sqrt{-1}\bar{m}_2\mathbf{e} & -m_1-\sqrt{-1}\bar{a}_1\mathbf{e} & -\xi_3 \\
   \end{array}
 \right)\\
&=&\left(
   \begin{array}{ccc}
 \xi_1 & \sqrt{-1}(-\sqrt{-1}a_3-m_3\mathbf{e}) & -\sqrt{-1}(-\sqrt{-1}a_2+\bar{m}_2\mathbf{e}) \\
  \sqrt{-1}(-\sqrt{-1}\bar{a}_3 +m_3 \mathbf{e})& -\xi_2 & -(\bar{m}_1+\sqrt{-1}\overline{\bar{a}_1\mathbf{e}}) \\
  -\sqrt{-1}(-\sqrt{-1}\bar{a}_2-\bar{m}_2\mathbf{e})& -(m_1+\sqrt{-1}\bar{a}_1\mathbf{e}) & -\xi_3 \\
   \end{array}
 \right)\\
&=&\alpha_{23}(\pi)\circ (\alpha_1,\alpha_2,\alpha_3)(\left(
   \begin{array}{ccc}
     \xi_1 & x_3 & \bar{x}_2 \\
     \bar{x}_3 & \xi_2 & x_1 \\
     x_2 & \bar{x}_1 & \xi_3 \\
   \end{array}
 \right)),
\end{eqnarray*}
where
$\alpha_{1},\alpha_{2},\alpha_{3}\in{SO(\mathbf{K})\cong SO(8)}$ are
defined by
\begin{equation}\label{eq:generator_EIII}
\begin{split}
&\alpha_{1}(m_{1}+a_{1}\mathbf{e}):=-(\bar{m}_{1}-\bar{a}_{1}\mathbf{e}),\\
&\alpha_{2}(m_{2}+a_{2}\mathbf{e}):=-\bar{a}_{2}-\bar{m}_{2}\mathbf{e},\\
&\alpha_{3}(m_{3}+a_{3}\mathbf{e}):=-a_{3}-m_{3}\mathbf{e}.
\end{split}
\end{equation}
By simple computation,
we know $\alpha_{1}(m_{1}+a_{1}\mathbf{e}){\
}\alpha_{2}(m_{2}+a_{2}\mathbf{e})=\overline{
\alpha_{3}(\overline{(m_{1}+a_{1}\mathbf{e})(m_{2}+a_{2}\mathbf{e})})}$.
Hence, $(\alpha_1, \alpha_2, \alpha_3)\in Spin(8)$.
Notice that
\begin{equation*}
\begin{split}
&\alpha_{23}(\pi)(\alpha_{1},\alpha_{2},\alpha_{3})(u_2)
=\alpha_{23}(\alpha_2(u_2)) =\alpha_{23}(\pi) (-\mathbf{e} u_2) =
\sqrt{-1}\mathbf{e}u_2,
\\
&\alpha_{23}(\pi)(\alpha_{1},\alpha_{2},\alpha_{3})(\sqrt{-1}\mathbf{e}u_2)
=\alpha_{23}(\pi)( \alpha_2(-\sqrt{-1}\mathbf{e}u_2))
=\alpha_{23}(\pi)(\sqrt{-1}u_2) =-u_2.
\end{split}
\end{equation*}
It implies that
\begin{equation*}
\alpha_{23}(\pi)(\alpha_1,\alpha_2,\alpha_3)\in Spin(2)\cdot Spin(8)
\subset (U(1)\times (Spin(2)\cdot Spin(8)))/\mathbf{Z}_4=K_2
\end{equation*}
induces an isometry of the maximal abelian subspace $\mathfrak{a}$
of order $4$ which is a $\pi/2$-rotation of $\mathfrak{a}$, we
obtain
\begin{equation*}
\alpha_{23}(\pi)(\alpha_1,\alpha_2,\alpha_3)\in K_{[\mathfrak{a}]}
\end{equation*}
and it is a generator of $K_{[\mathfrak{a}]}/K_0\cong \mathbf{Z}_4$.


\subsection{Description of the Casimir operator}

Define $\langle u, v\rangle_{\mathfrak u}:=-\mathrm{tr}(uv)$ for
each $u,v\in \mathfrak{e}_6\subset
\mathfrak{gl}(H_3(\mathbf{K})^{\mathbf{C}})$. Now the restricted
root system is $\Sigma^{+}(U,K)=\{2\xi_1, 2\xi_2, \xi_1+\xi_2,
\xi_1-\xi_2, \xi_1,\xi_2\}$ and
\begin{equation*}
H_{\xi_1}=\frac{1}{12}(D(\bar{u}_2)+\sqrt{-1}R(c_4 u_2)), \quad
H_{\xi_2}=\frac{1}{12}(D(\bar{u}_2)-\sqrt{-1}R(c_4 u_2)).
\end{equation*}
With respect to $\langle\,,\,\rangle_{\mathfrak u}$, the lengths of
the restricted roots are given by
\begin{equation*}
\Vert{\gamma}\Vert_{\mathfrak u}^{2} = \frac{1}{3}, \frac{1}{6}
\text{ or } \frac{1}{12}.
\end{equation*}
Then the Casimir operator $\mathcal{C}_L$ with respect to the
induced metric $\mathcal{G}^{*}g^{\rm std}_{Q_{30}(\mathbf{C})}$ can
be expressed as
\begin{equation*}
\mathcal{C}_L=12C_{K/K_0}-6C_{K_2/K_0}-3C_{K_1/K_0},
\end{equation*}
where $C_{K/K_0}$, $C_{K_2/K_0}$ and $C_{K_1/K_0}$ are the Casimir
operators of homogeneous
spaces $K/K_0$, $K_2/K_0$ and $K_1/K_0$ with respect to the $K$
-invariant metric induced from the metric $\langle\, , \,
\rangle_{\mathfrak u}$ of $E_6$.

\subsection{Descriptions of $D(K)$, $D(K_2)$, $D(K_1)$ and $D(K_0)$}

A maximal torus $\tilde{T}^{5}$ of $Spin(10)$ can be given by
\begin{equation*}
\begin{split}
&\tilde{T}^{5}=
\{ \tilde{t} =(\cos{\theta_{1}}-e_{1}e_{2}\sin{\theta_{1}}) \cdot
(\cos{\theta_{2}}-e_{3}e_{4}\sin{\theta_{2}}) \cdot
(\cos{\theta_{3}}-e_{5}e_{6}\sin{\theta_{3}})
\\
&{\ } \cdot (\cos{\theta_{4}}-e_{7}e_{8}\sin{\theta_{4}}) \cdot
(\cos{\theta_{5}}-e_{9}e_{10}\sin{\theta_{5}})
\mid \theta_{i}\in{\mathbf R}{\ }(i=1,2,3,4,5) \}.
\end{split}
\end{equation*}
Under the standard universal ${\mathbf Z}_{2}$-covering map $p:
Spin(10)\rightarrow SO(10)$ defined by
\begin{equation*}
(p(\alpha)){\mathbf x}:=\alpha\cdot{\mathbf x}\cdot{^{t}\alpha}
\in{\mathbf R}^{10}\subset{Cl({\mathbf R}^{10})}
\end{equation*}
for each $\alpha\in{Spin(10)}$ and each ${\mathbf x}\in{\mathbf
R}^{10}$, an element of the maximal torus $\tilde{T}^5$ of
$Spin(10)$ is mapped to an element in the maximal torus $T^5$ of
$SO(10)$, namely,
\begin{equation*}
\begin{split}
& \tilde{T}^5 \ni(\cos{\theta_{1}}-e_{1}e_{2}\sin{\theta_{1}}) \cdot
(\cos{\theta_{2}}-e_{3}e_{4}\sin{\theta_{2}}) \cdot
(\cos{\theta_{3}}-e_{5}e_{6}\sin{\theta_{3}})
\\
& \cdot (\cos{\theta_{4}}-e_{7}e_{8}\sin{\theta_{4}}) \cdot
(\cos{\theta_{5}}-e_{9}e_{10}\sin{\theta_{5}})
\\
&\longmapsto \quad
\begin{pmatrix}
\begin{pmatrix}
\cos{2\theta_{1}}&-\sin{2\theta_{1}}\\
\sin{2\theta_{1}}&\cos{2\theta_{1}}
\end{pmatrix}
&0&0\\
0&\cdots&0\\
0&0&
\begin{pmatrix}
\cos{2\theta_{5}}&-\sin{2\theta_{5}}\\
\sin{2\theta_{5}}&\cos{2\theta_{5}}
\end{pmatrix}
\end{pmatrix} \in T^5.
\end{split}
\end{equation*}
Hence, we have the exponential map as follows:
\begin{equation*}
\begin{split}
\exp: \tilde{\mathfrak t} =&{\mathfrak t}=
\{(\theta_{1},\theta_{2},\theta_{3},\theta_{4},\theta_{5}) {\
}\vert{\ } \theta_{i}\in{\mathbf R}{\ }(i=1,2,3,4,5)
\}\\
\longrightarrow&\\
\tilde{T} =& \{ {\ }
(\cos{(\theta_{1}/2)}-e_{1}e_{2}\sin{(\theta_{1}/2)}) \cdot
(\cos{(\theta_{2}/2)}-e_{3}e_{4}\sin{(\theta_{2}/2)})
\\
&{\ } \cdot (\cos{(\theta_{3}/2)}-e_{5}e_{6}\sin{(\theta_{3}/2)})
\cdot (\cos{(\theta_{4}/2)}-e_{7}e_{8}\sin{(\theta_{4}/2)})
\\
&{\ } \cdot (\cos{(\theta_{5}/2)}-e_{9}e_{10}\sin{(\theta_{5}/2)})
\\
& {\ }\vert{\ } \theta_{i}\in{\mathbf R}{\ }(i=1,2,3,4,5) \}\subset
Spin(10).
\end{split}
\end{equation*}
Thus
\begin{equation*}
\begin{split}
&\Gamma(Spin(10)) =\{
\xi=(\theta_{1},\theta_{2},\theta_{3},\theta_{4},\theta_{5})
\in{\tilde{\mathfrak t}} {\ }\vert{\ }\exp(\xi)=e
\}\\
=&\{ \xi=2\pi{\ }(k_{1},k_{2},k_{3},k_{4},k_{5}) {\ }\vert{\ }
k_{i}\in{\mathbf Z}{\ }(i=1,2,3,4,5), \sum^{5}_{i=1}k_{i}\in
2{\mathbf Z} \}\subset\Gamma(SO(10)).
\end{split}
\end{equation*}
Denote by $y_i$ $(i=1,\cdots, 5)$ a linear functional $y_i:
\tilde{\mathfrak t} \ni \tilde{t} \mapsto \theta_i\in \mathbf{R}$.
Then
\begin{equation*}
\begin{split}
D(Spin(10))
&=\{ \Lambda=p_{1}y_{1}+p_{2}y_{2}+p_{3}y_{3}+p_{4}y_{4}+p_{5}y_{5}
\in{\mathfrak t}^{\ast}\\
& {\ }\vert{\ } (p_1, \cdots , p_5)\in
\mathbf{Z}^5+\varepsilon(1,1,1,1,1),
 \text{ where } \varepsilon=0 \text{ or }\frac{1}{2},{\ }\\
&\quad p_{1}\geq p_{2}\geq p_{3}\geq p_{4}\geq \vert{p_{5}}\vert \}
\supset D(SO(10)).
\end{split}
\end{equation*}

A maximal torus $T_K$ of $K=(U(1)\times Spin(10))/{\mathbf Z}_{4}$
can be given as follows:
\begin{equation*}
\begin{split}
T_K=&\{(e^{\sqrt{-1}\theta_{0}},
(\cos\frac{\theta_{1}}{2}-e_{1}e_{2}\sin\frac{\theta_{1}}{2})
(\cos\frac{\theta_{2}}{2}-e_{3}e_{4}\sin\frac{\theta_{2}}{2})\\
& (\cos\frac{\theta_{3}}{2}-e_{5}e_{6}\sin\frac{\theta_{3}}{2})
(\cos\frac{\theta_{4}}{2}-e_{7}e_{8}\sin\frac{\theta_{4}}{2})
(\cos\frac{\theta_{5}}{2}-e_{9}e_{10}\sin\frac{\theta_{5}}{2}))\\
& {\ }\vert{\ }\theta_0, \cdots, \theta_5 \in{\mathbf R}\} /{\mathbf
Z}_{4},
\end{split}
\end{equation*}
where $t_0=2\theta_0$, $t_1=\theta_1$,
$U(1)=\{\exp(t_{0}\sqrt{-1}R(2e_{1}-e_{2}-e_{3})){\ }\vert{\
}t_0\in{\mathbf R}\}$,
$Spin(2)=\{\exp(t_{1}\sqrt{-1}R(e_{2}-e_{3})){\ }\vert{\
}t_1\in{\mathbf R}\}$ and
\begin{equation*}
{\mathbf Z}_{4}: =\{(1,1),(-1,-1),
(\sqrt{-1},-e_{1}e_{2}\cdots e_{10}),(-\sqrt{-1},e_{1}e_{2}\cdots e_{10})\}. 
\end{equation*}
The corresponding maximal abelian subalgebra $\mathfrak{t}$ of
$\mathfrak{k}$ is
\begin{equation*}
{\mathfrak t}
=\{(\theta_{0},\theta_{1},\theta_{2},\theta_{3},\theta_{4},\theta_{5}){\
}\vert{\ } \theta_{i}\in{\mathbf R}{\ }(i=0,1,2,3,4,5) \}.
\end{equation*}
Then
\begin{equation*}
\begin{split}
\Gamma(K) =& \{ \xi =2\pi
(\frac{k_{0}}{2},k_{1},k_{2},k_{3},k_{4},k_{5})
+\pi\varepsilon(\frac{1}{2},1,1,1,1,1)
\\
&{\ }\vert{\ } k_{0},k_{1},k_{2},k_{3},k_{4},k_{5}\in{\mathbf Z},{\
} \varepsilon=0\text{ or }1, {\ } \sum^{5}_{\alpha=0}k_{\alpha}\in
2{\mathbf Z} \},
\end{split}
\end{equation*}
\begin{equation*}
\begin{split}
D(K)
=&D((U(1)\times Spin(10))/{\mathbf Z}_{4})\\
=&\{ \Lambda
=p_{0}y_{0}+p_{1}y_{1}+p_{2}y_{2}+p_{3}y_{3}+p_{4}y_{4}+p_{5}y_{5}\in{\mathfrak
t}^{\ast}
{\ }\vert{\ }\\
& \frac{1}{2}p_{0}+p_{1}+p_{2}+p_{3}+p_{4}+p_{5}\in 2{\mathbf Z}, {\
}
p_{0}\in{\mathbf Z}, \\
&(p_{1},p_{2},p_{3},p_{4},p_{5}) \in{\mathbf
Z}^{5}+\varepsilon(1,1,1,1,1),
{\ }\varepsilon=0\text{ or }\frac{1}{2},\\
&\quad p_{1}\geq p_{2}\geq p_{3}\geq p_{4}\geq \vert{p_{5}}\vert{\ }
\}.
\end{split}
\end{equation*}

Since $T_{K}$ is also a maximal torus of $K_2=(U(1)\times
(Spin(2)\cdot Spin(8)))/\mathbf{Z}_4\subset K$,
$\Gamma(K_2)=\Gamma(K)$ 
and
\begin{equation*}
\begin{split}
D(K_{2})
=&D((U(1)\times Spin(2)\cdot Spin(8))/{\mathbf Z}_{4}) \\
=&\{ \Lambda
=p_{0}y_{0}+p_{1}y_{1}+p_{2}y_{2}+p_{3}y_{3}+p_{4}y_{4}+p_{5}y_{5}\in{\mathfrak
t}^{\ast}
{\ }\vert{\ }\\
& \frac{1}{2}p_{0}+p_{1}+p_{2}+p_{3}+p_{4}+p_{5}\in 2{\mathbf Z}, {\
}
p_{0}\in{\mathbf Z}, \\
&(p_{1},p_{2},p_{3},p_{4},p_{5}) \in{\mathbf
Z}^{5}+\varepsilon(1,1,1,1,1),
{\ }\varepsilon=0\text{ or }\frac{1}{2},\\
&\quad p_{2}\geq p_{3}\geq p_{4}\geq \vert{p_{5}}\vert{\ } \}.
\end{split}
\end{equation*}

On the other hand, $K_2=(S^1\times (Spin(2)\cdot
Spin(8)))/\mathbf{Z}_4$, where
\begin{equation*}
\begin{split}
S^{1}&=\{\exp(\hat{t}_{0}\sqrt{-1}R(-e_{1}+2e_{2}-e_{3})){\ }\vert{\ }\hat{t}_0\in{\mathbf R}\},\\
Spin(2)&=\{\exp(\hat{t}_{1}\sqrt{-1}R(e_{3}-e_{1})){\ }\vert{\
}\hat{t}_1\in{\mathbf R}\}
\end{split}
\end{equation*}
and here $Spin(2)\cdot Spin(8)\subset (E_6)_{e_2}\cong Spin(10)$.
Since
\begin{equation*}
\begin{split}
&\exp(t_0\sqrt{-1}R(2e_1-e_2-e_3))\cdot \exp(t_1\sqrt{-1}R(e_2-e_3))\\
& \quad =\exp(-\frac{t_0-t_1}{2}\sqrt{-1}R(-e_1+2e_2-e_3))\cdot
\exp(-\frac{3t_0+t_1}{2}\sqrt{-1}R(e_3-e_1)),
\end{split}
\end{equation*}
one can take
$\hat{t}_{0}=-\frac{t_{0}-t_{1}}{2}$, 
$\hat{t}_{1}=-\frac{3t_{0}+t_{1}}{2}$ such that the maximal torus
$T_{K_2}=T_K$ of $K_2$ can also be described as
\begin{equation*}
\begin{split}
\hat{T}_{K_2}&=T_{K_2}= T_K
=\{\hat{t}=(e^{\sqrt{-1}\hat{\theta}_{0}},
(\cos\frac{\hat{\theta}_{1}}{2}-e_{1}e_{2}\sin\frac{\hat{\theta}_{1}}{2})
(\cos\frac{\hat{\theta}_{2}}{2}-e_{3}e_{4}\sin\frac{\hat{\theta}_{2}}{2})\\
&
(\cos\frac{\hat{\theta}_{3}}{2}-e_{5}e_{6}\sin\frac{\hat{\theta}_{3}}{2})
(\cos\frac{\hat{\theta}_{4}}{2}-e_{7}e_{8}\sin\frac{\hat{\theta}_{4}}{2})
(\cos\frac{\hat{\theta}_{5}}{2}-e_{9}e_{10}\sin\frac{\hat{\theta}_{5}}{2}))\\
& {\ }\vert{\ }\hat{\theta}_0, \cdots, \hat{\theta}_5 \in{\mathbf
R}\} /{\mathbf Z}_{4},
\end{split}
\end{equation*}
where $\hat{\theta}_0=\hat{t}_0/2$, $\hat{\theta}_1=\hat{t}_1$.
Taking account of the triality of
$Spin(8)=(E_6)_{e_1,e_2,e_3}\subset (E_6)_{e_1}\cong
(E_6)_{e_2}\cong Spin(10)$, we choose a new basis $\hat{y}_i:
\hat{t} \mapsto \hat{\theta}_i$ for $\mathfrak{t}^*$ satisfying
\begin{equation*}
\begin{split}
\hat{y}_{0}
&=-\frac{1}{2}y_{0}+\frac{1}{4}y_{1},\quad
\hat{y}_{1}
=-3y_{0}-\frac{1}{2}y_{1},\quad
\hat{y}_{2}
:=
\frac{1}{2}(y_{2}+y_{3}+y_{4}+y_{5}),\\
\hat{y}_{3}
&:= \frac{1}{2}(y_{2}+y_{3}-y_{4}-y_{5}),\quad
\hat{y}_{4}
:=\frac{1}{2}(y_{2}-y_{3}+y_{4}-y_{5}),\\
\hat{y}_{5}
&:=\frac{1}{2}(-y_{2}+y_{3}+y_{4}-y_{5}).
\end{split}
\end{equation*}
Thus any weight $\Lambda=p_{0}y_{0}+p_{1}y_{1}+p_{2}y_{2}+
p_{3}y_{3}+p_{4}y_{4}+p_{5}y_{5} \in D(K_{2})$ can also be written
as $\Lambda =\hat{p}_{0}\hat{y}_{0}+\hat{p}_{1}\hat{y}_{1}
+\hat{p}_{2}\hat{y}_{2}+\hat{p}_{3}\hat{y}_{3}
+\hat{p}_{4}\hat{y}_{4}+\hat{p}_{5}\hat{y}_{5}$, where
\begin{equation*}
\begin{split}
\hat{p}_{0}=&-\frac{1}{2}p_{0}+3p_{1},\quad
\hat{p}_{1}=-\frac{1}{4}p_{0}-\frac{1}{2}p_{1},\quad
\hat{p}_{2}=\frac{1}{2}
(p_{2}+p_{3}+p_{4}+p_{5}),\\
\hat{p}_{3}=&\frac{1}{2} (p_2+p_3-p_4-p_5),\quad
\hat{p}_{4}=\frac{1}{2}(p_2-p_3+p_4-p_5),\\
\hat{p}_{5}=&\frac{1}{2}(-p_2+p_3+p_4-p_5).
\end{split}
\end{equation*}
Thus $D(K_2)$ has the following another expression:
\begin{equation*}
\begin{split}
D(K_{2})
=&D((S^1 \times Spin(2)\cdot Spin(8))/{\mathbf Z}_{4}) \\
=&\{ \Lambda
=\hat{p}_{0}\hat{y}_{0}+\hat{p}_{1}\hat{y}_{1}+\hat{p}_{2}\hat{y}_{2}
+\hat{p}_{3}\hat{y}_{3}+\hat{p}_{4}\hat{y}_{4}+\hat{p}_{5}\hat{y}_{5}
\in{\mathfrak t}^{\ast}
{\ }\vert{\ }\\
&
\frac{1}{2}\hat{p}_{0}+\hat{p}_{1}+\hat{p}_{2}+\hat{p}_{3}+\hat{p}_{4}+\hat{p}_{5}\in
2{\mathbf Z}, {\ }
\hat{p}_{0}\in{\mathbf Z}, \\
&(\hat{p}_{1},\hat{p}_{2},\hat{p}_{3},\hat{p}_{4},\hat{p}_{5})
\in{\mathbf Z}^{5}+\varepsilon(1,1,1,1,1),
{\ }\varepsilon=0\text{ or }\frac{1}{2},\\
&\quad \hat{p}_{2}\geq \hat{p}_{3}\geq \hat{p}_{4}\geq
\vert{\hat{p}_{5}}\vert{\ } \}.
\end{split}
\end{equation*}

Notice that the subgroup $K_1=(S^1\times (Spin(2)\cdot (Spin(2)\cdot
Spin(6))))/\mathbf{Z}_4$ also has the same maximal torus
$T_{K_1}=\hat{T}_{K_2}=T_{K_2}=T_K$
and the corresponding maximal abelian subalgebra
$\mathfrak{t}_{\mathfrak{k}_1}$ of $\mathfrak{k}_1$ is
\begin{equation*}
{\mathfrak t}_{\mathfrak{k}_1}=\hat{\mathfrak{t}}_{\mathfrak{k}_2}
=\{(\hat{\theta}_{0},\hat{\theta}_{1},\hat{\theta}_{2},\hat{\theta}_{3},\hat{\theta}_{4},\hat{\theta}_{5})
{\ }\vert{\ } \hat{\theta}_{i}\in{\mathbf R}{\ }(i=0,1,2,3,4,5)
\}
=\mathfrak{t}_{\mathfrak{k}_2}=\mathfrak{t},
\end{equation*}
we get
\begin{equation*}
\begin{split}
D(K_{1})
=&\{ \Lambda
=\hat{p}_{0}\hat{y}_{0}+\hat{p}_{1}\hat{y}_{1}+\hat{p}_{2}\hat{y}_{2}
+\hat{p}_{3}\hat{y}_{3}+\hat{p}_{4}\hat{y}_{4}+\hat{p}_{5}\hat{y}_{5}
\in \mathfrak{t}^{\ast}_{\mathfrak{k}_1}={\mathfrak t}^{\ast}
{\ }\vert{\ }\\
&
\frac{1}{2}\hat{p}_{0}+\hat{p}_{1}+\hat{p}_{2}+\hat{p}_{3}+\hat{p}_{4}+\hat{p}_{5}\in
2{\mathbf Z}, {\ }
\hat{p}_{0}\in{\mathbf Z}, \\
&(\hat{p}_{1}, \hat{p}_{2}, \hat{p}_{3}, \hat{p}_{4}, \hat{p}_{5})
\in{\mathbf Z}^{5}+\varepsilon(1,1,1,1,1),
{\ }\varepsilon=0\text{ or }\frac{1}{2},\\
&\quad \hat{p}_{3}\geq \hat{p}_{4}\geq \vert{\hat{p}_{5}}\vert{\ }
\}.
\end{split}
\end{equation*}

Finally, the maximal torus of $K_{0}=(S^{1}\times Spin(6))/{\mathbf
Z}_{2}$ is given as follows:
\begin{equation*}
\begin{split}
T_{K_{0}} =& \{ ({e^{\sqrt{-1}\hat{\theta}_{0}}},
(\cos\frac{\hat{\theta}_{3}}{2}-e_{5}e_{6}\sin\frac{\hat{\theta}_{3}}{2})
(\cos\frac{\hat{\theta}_{4}}{2}-e_{7}e_{8}\sin\frac{\hat{\theta}_{4}}{2})\\
&\quad
(\cos\frac{\hat{\theta}_{5}}{2}-e_{9}e_{10}\sin\frac{\hat{\theta}_{5}}{2})
) {\ }\vert{\ } \hat{\theta}_{i}\in{\mathbf R}{\ }(i=0,3,4,5)
\}/{\mathbf Z}_{2}{\ } \subset \hat{T}_{K_2}=T_K
\end{split}
\end{equation*}
and the corresponding maximal abelian subalgebra of ${\mathfrak
k}_{0}$ is
\begin{equation*}
{\mathfrak t}_{{\mathfrak k}_{0}} =\{(\hat{\theta}_{0}, 0, 0,
\hat{\theta}_{3}, \hat{\theta}_{4}, \hat{\theta}_{5}){\ }\vert{\ }
\hat{\theta}_{i}\in{\mathbf R}{\ }(i=0,3,4,5) \} \subset
\mathfrak{t}_{\mathfrak{k}_2}=\mathfrak{t}.
\end{equation*}
Then
\begin{equation*}
\begin{split}
D(K_{0})
=& \{ \Lambda
=\hat{q}_{0}\hat{y}_{0}+\hat{q}_{3}\hat{y}_{3}+\hat{q}_{4}\hat{y}_{4}+\hat{q}_{5}\hat{y}_{5}
\in{\mathfrak t}^{\ast}_{{\mathfrak k}_{0}}
{\ }{\ }\vert{\ }\\
& \frac{1}{2}\hat{q}_{0}+\hat{q}_{3}+\hat{q}_{4}+\hat{q}_{5}\in
2{\mathbf Z}, {\ }
\hat{q}_{0}\in{\mathbf Z}, \\
& (\hat{q}_{3}, \hat{q}_{4}, \hat{q}_{5})\in{\mathbf Z}^{3}
+\varepsilon(1,1,1), \varepsilon=0\text{ or }\frac{1}{2},{\ }
\\
&\quad \hat{q}_{3}\geq \hat{q}_{4}\geq\vert{\hat{q}_{5}}\vert{\ }
\}.
\end{split}
\end{equation*}

\subsection{Branching Laws}

Based on the branching laws of $(SO(2n+2), SO(2)\times SO(2n))$
obtained by Tsukamoto (\cite{Tsukamoto}), we formulate the following
branching laws.
\begin{lem}[Branching Law of $(Spin(10), Spin(2)\cdot Spin(8))$]
For each
\begin{equation*}
\Lambda=p_{1}y_{1}+p_{2}y_{2}+p_{3}y_{3}+p_{4}y_{4}+\delta
p_{5}y_{5} \in D(Spin(10)),
\end{equation*}
with $\delta=1\text{ or }-1$ and
\begin{equation*}
\begin{split}
&(p_{1},p_{2},p_{3},p_{4},p_{5})\in {\mathbf
Z}^{5}+\varepsilon(1,1,1,1,1), {\ }\varepsilon=0 \text{ or
}\frac{1}{2},
\\
&p_{1}\geq p_{2}\geq p_{3}\geq p_{4}\geq p_{5}\geq{0},
\end{split}
\end{equation*}
$V_{\Lambda}$ contains an irreducible $Spin(2)\cdot Spin(8)$-module
with the highest weight
\begin{equation*}
\Lambda^{\prime}
=q_{1}y_{1}+q_{2}y_{2}+q_{3}y_{3}+q_{4}y_{4}+\delta^{\prime}
q_{5}y_{5} \in D(Spin(2)\cdot Spin(8))
\end{equation*}
with $\delta^{\prime}=1\text{ or }-1$ and
\begin{equation*}
\begin{split}
&(q_{1},q_{2},q_{3},q_{4},q_{5})\in{\mathbf Z}^{5}
+\varepsilon(1,1,1,1,1), {\ }\varepsilon=0 \text{ or }\frac{1}{2},
\\
&q_{2}\geq q_{3}\geq q_{4}\geq q_{5}\geq{0},
\end{split}
\end{equation*}
if and only if $\Lambda^{\prime}$
satisfies the following conditions:
\begin{enumerate}
\item
\begin{equation*}
\begin{split}
&p_{1}+1> q_{2}> p_{3}-1,\\
&p_{2}+1> q_{3}> p_{4}-1,\\
&p_{3}+1> q_{4}> p_{5}-1,\\
&p_{4}+1> q_{5}\geq 0.
\end{split}
\end{equation*}
\item
The coefficient of $X^{q_{1}}$ in the following power series
expansion in $X$ of
\begin{equation*}
X^{\delta\delta^{\prime}\ell_{5}}
(\prod^{4}_{i=1}\frac{X^{\ell_{i}+1}-X^{-\ell_{i}-1}}{X-X^{-1}})
\end{equation*}
does not vanish. Here
\begin{equation*}
\begin{split}
&\ell_{1}:=p_{1}-\max\{p_{2},q_{2}\},\\
&\ell_{2}:=\min\{p_{2},q_{2}\}-\max\{p_{3},q_{3}\},\\
&\ell_{3}:=\min\{p_{3},q_{3}\}-\max\{p_{4},q_{4}\},\\
&\ell_{4}:=\min\{p_{4},q_{4}\}-\max\{p_{5},q_{5}\},\\
&\ell_{5}:=\min\{p_{5},q_{5}\}.
\end{split}
\end{equation*}
\end{enumerate}
Moreover its multiplicity is equal to the coefficient of
$X^{q_{1}}$.
\end{lem}

\begin{lem}[Branching Law of $(Spin(8),Spin(2)\cdot Spin(6))$]
For each
\begin{equation*}
\Lambda=p_{2}y_{2}+p_{3}y_{3}+p_{4}y_{4}+\delta p_{5}y_{5} \in
D(Spin(8)),
\end{equation*}
with $\delta=1\text{ or }-1$ and
\begin{equation*}
\begin{split}
&(p_{2},p_{3},p_{4},p_{5})\in {\mathbf Z}^{4}+\varepsilon(1,1,1,1),
{\ }\varepsilon= 0\text{ or }\frac{1}{2},
\\
&p_{2}\geq p_{3}\geq p_{4}\geq p_{5}\geq{0},
\end{split}
\end{equation*}
$V_{\Lambda}$ contains an irreducible $Spin(2)\cdot Spin(6)$-module
with the highest weight
\begin{equation*}
\Lambda^{\prime} =q_{2}y_{2}+q_{3}y_{3}+q_{4}y_{4}+\delta^{\prime}
p_{5}y_{5} \in D(Spin(2)\cdot Spin(6))
\end{equation*}
with $\delta^{\prime}=1\text{ or }-1$ and
\begin{equation*}
\begin{split}
&(q_{2},q_{3},q_{4},q_{5})\in{\mathbf Z}^{4} +\varepsilon(1,1,1,1),
{\ }\varepsilon=0\text{ or }\frac{1}{2},
\\
&q_{3}\geq q_{4}\geq q_{5}\geq{0}.
\end{split}
\end{equation*}
if and only if $\Lambda^{\prime}$
satisfies the following conditions:
\begin{enumerate}
\item
\begin{equation*}
\begin{split}
&p_{2}+1> q_{3}> p_{4}-1,\\
&p_{3}+1> q_{4}> p_{5}-1,\\
&p_{4}+1> q_{5}\geq 0.
\end{split}
\end{equation*}
\item
The coefficient of $X^{q_{2}}$
\begin{equation*}
X^{\delta\delta^{\prime}\ell_{5}}
(\prod^{4}_{i=2}\frac{X^{\ell_{i}+1}-X^{-\ell_{i}-1}}{X-X^{-1}})
\end{equation*}
does not vanish. Here
\begin{equation*}
\begin{split}
&\ell_{2}:=p_{2}-\max\{p_{3},q_{3}\},\\
&\ell_{3}:=\min\{p_{3},q_{3}\}-\max\{p_{4},q_{4}\},\\
&\ell_{4}:=\min\{p_{4},q_{4}\}-\max\{p_{5},q_{5}\},\\
&\ell_{5}:=\min\{p_{5},q_{5}\}.
\end{split}
\end{equation*}
\end{enumerate}
Moreover its multiplicity is equal to the coefficient of
$X^{q_{2}}$.
\end{lem}

\subsection{Description of $D(K,K_{0})$}

Let
\begin{equation*}
\begin{split}
\Lambda&= p_{0}y_{0}+p_{1}y_{1}+p_{2}y_{2}+p_{3}y_{3}+p_{4}y_{4}
+{\epsilon}p_{5}y_{5}
\in{D(K)},\\
\Lambda^{\prime}&= p^{\prime}_{0}y_{0}+p^{\prime}_{1}y_{1}
+p^{\prime}_{2}y_{2}+p^{\prime}_{3}y_{3}+p^{\prime}_{4}y_{4}
+{\epsilon}^{\prime}p^{\prime}_{5}y_{5}
\\
&= \hat{p}^{\prime}_{0}\hat{y}_{0}+\hat{p}^{\prime}_{1}\hat{y}_{1}
+\hat{p}^{\prime}_{2}\hat{y}_{2}+\hat{p}^{\prime}_{3}\hat{y}_{3}
+\hat{p}^{\prime}_{4}\hat{y}_{4}
+\hat{\epsilon}^{\prime}\hat{p}^{\prime}_{5}\hat{y}_{5}
\in{D(K_{2})},\\
\Lambda^{\prime\prime}&=
\hat{p}^{\prime\prime}_{0}\hat{y}_{0}+\hat{p}^{\prime\prime}_{1}\hat{y}_{1}
+\hat{p}^{\prime\prime}_{2}\hat{y}_{2}+\hat{p}^{\prime\prime}_{3}\hat{y}_{3}
+\hat{p}^{\prime\prime}_{4}\hat{y}_{4}
+\hat{\epsilon}^{\prime\prime}\hat{p}^{\prime\prime}_{5}\hat{y}_{5}
\in{D(K_{1})},\\
\Lambda^{\prime\prime\prime}&=
\hat{p}^{\prime\prime\prime}_{0}\hat{y}_{0}
+\hat{p}^{\prime\prime\prime}_{3}\hat{y}_{3}
+\hat{p}^{\prime\prime\prime}_{4}\hat{y}_{4}
+\hat{\epsilon}^{\prime\prime\prime}\hat{p}^{\prime\prime\prime}_{5}\hat{y}_{5}
\in{D(K_{0})}.
\end{split}
\end{equation*}

Assume that the corresponding representation spaces satisfy
\begin{equation*}
V_{\Lambda} \supset W_{\Lambda^{\prime}} \supset
U_{\Lambda^{\prime\prime}} =
U_{\Lambda^{\prime\prime\prime}}\not=\{0\}.
\end{equation*}
Suppose that $U_{\Lambda^{\prime\prime\prime}}\not=\{0\}$ is a
trivial representation of $K_{0}$, that is,
$\Lambda^{\prime\prime\prime}=0$.
Then we have
\begin{equation*}
\begin{split}
&\hat{p}^{\prime\prime\prime}_{0}
=\hat{p}^{\prime\prime}_{0}=0,\quad \hat{p}^{\prime\prime\prime}_{3}
=\hat{p}^{\prime\prime}_{3}=0, \quad 
\hat{p}^{\prime\prime\prime}_{4} =\hat{p}^{\prime\prime}_{4}=0,
\quad \hat{p}^{\prime\prime\prime}_{5}
=\hat{p}^{\prime\prime}_{5}=0.
\end{split}
\end{equation*}
Thus
$\Lambda^{\prime\prime}=\hat{p}^{\prime\prime}_{1}\hat{y}_{1}+\hat{p}^{\prime\prime}_{2}\hat{y}_{2}
\in{D(K_{1})}$ with $\hat{p}^{\prime\prime}_{1}$,
$\hat{p}^{\prime\prime}_{2}\in{\mathbf Z}$,
$\hat{p}^{\prime\prime}_{1}+\hat{p}^{\prime\prime}_{2}\in 2{\mathbf
Z}$.

By the branching law of $(Spin(8), Spin(2)\cdot Spin(6))$, we get
\begin{equation*}
\begin{split}
&\hat{p}^{\prime}_{2}\geq \hat{p}^{\prime\prime}_{3}=0 \geq \hat{p}^{\prime}_{4},\\
&\hat{p}^{\prime}_{3}\geq \hat{p}^{\prime\prime}_{4}=0 \geq \hat{p}^{\prime}_{5},\\
&\hat{p}^{\prime}_{4}\geq \hat{p}^{\prime\prime}_{5}=0\geq 0.
\end{split}
\end{equation*}
Thus
$(\hat{p}^{\prime}_{4},\hat{p}^{\prime}_{5})=(0,0)$ and
$\hat{p}^{\prime}_2\geq 0$, $\hat{p}^{\prime}_3\geq 0$. It follows
that
\begin{equation*}
\begin{split}
\ell_{2} &
=\hat{p}^{\prime}_{2}-\max\{\hat{p}^{\prime}_{3},\hat{p}^{\prime\prime}_{3}\}
=\hat{p}^{\prime}_{2}-\max\{\hat{p}^{\prime}_{3},0\}
=\hat{p}^{\prime}_{2}-\hat{p}^{\prime}_{3},
\\
\ell_{3}& =\min\{\hat{p}^{\prime}_{3},\hat{p}^{\prime\prime}_{3}\}
-\max\{\hat{p}^{\prime}_{4},\hat{p}^{\prime\prime}_{4}\}
=\min\{\hat{p}^{\prime}_{3},0\} -\max\{0,0\} =0-0=0
,\\
\ell_{4}& =\min\{\hat{p}^{\prime}_{4},\hat{p}^{\prime\prime}_{4}\}
-\max\{\hat{p}^{\prime}_{5},\hat{p}^{\prime\prime}_{5}\}
=\min\{0,0\} -\max\{0,0\}
=0-0=0,\\
\ell_{5}& =\min\{\hat{p}^{\prime}_{5},\hat{p}^{\prime\prime}_{5}\}
=\min\{0,0\}=0.
\end{split}
\end{equation*}
Then the coefficient of $X^{\hat{p}^{\prime\prime}_{2}}$ in the
(finite) power series expansion in $X$
\begin{equation*}
X^{\hat{\epsilon}^{\prime}\hat{\epsilon}^{\prime\prime}\ell_{5}}
\prod^{4}_{i=2}\frac{X^{\ell_{i}+1}-X^{-\ell_{i}-1}}{X-X^{-1}}\\
=
\frac{X^{\hat{p}^{\prime}_{2}-\hat{p}^{\prime}_{3}+1}
-X^{-(\hat{p}^{\prime}_{2}-\hat{p}^{\prime}_{3})-1}}{X-X^{-1}}
\end{equation*}
is equal to its multiplicity. Hence we have
\begin{equation*}
-(\hat{p}^{\prime}_{2}-\hat{p}^{\prime}_{3}) \leq
\hat{p}^{\prime\prime}_{2}=\hat{p}^{\prime}_{2}-\hat{p}^{\prime}_{3}-2i
\leq \hat{p}^{\prime}_{2}-\hat{p}^{\prime}_{3}
\end{equation*}
for some $i\in{\mathbf Z}$ with $0\leq i\leq
\hat{p}^{\prime}_{2}-\hat{p}^{\prime}_{3}$.
Moreover, $\hat{p}_0^{\prime}=\hat{p}_0^{\prime\prime}=0,
\hat{p}_1^{\prime}=\hat{p}_1^{\prime\prime}$. Thus we get
\begin{equation*}
\Lambda^{\prime}
= \hat{p}^{\prime}_{1}\hat{y}_{1}
+\hat{p}^{\prime}_{2}\hat{y}_{2}+\hat{p}^{\prime}_{3}\hat{y}_{3}
\in{D(K_{2})}
\end{equation*}
with
\begin{equation*}
\begin{split}
& \hat{p}^{\prime}_{1}=\hat{p}^{\prime\prime}_{1},
\hat{p}^{\prime}_{2},\hat{p}^{\prime}_{3}
\in{\mathbf Z}, {\ } \hat{p}^{\prime}_{1}+
\hat{p}^{\prime}_{2}+\hat{p}^{\prime}_{3} \in 2\mathbf{Z}, {\ }
\\
& -(\hat{p}^{\prime}_{2}-\hat{p}^{\prime}_{3}) \leq
\hat{p}^{\prime\prime}_{2}=\hat{p}^{\prime}_{2}-\hat{p}^{\prime}_{3}-2i
\leq \hat{p}^{\prime}_{2}-\hat{p}^{\prime}_{3}
\end{split}
\end{equation*}
for some $i\in{\mathbf Z}$ with $0\leq i\leq
\hat{p}^{\prime}_{2}-\hat{p}^{\prime}_{3}$. Therefore,
\begin{equation*}
\Lambda^{\prime}=
p^{\prime}_{0}y_{0}+p^{\prime}_{1}y_{1}+p^{\prime}_{2}y_{2}+
p^{\prime}_{3}y_{3}+p^{\prime}_{4}y_{4}+\epsilon^{\prime}
p^{\prime}_{5}y_{5} \in{D(K_{2})}
\end{equation*}
with
\begin{equation*}
\begin{split}
p^{\prime}_{0}&=-\frac{1}{2}\hat{p}^{\prime}_0-3\hat{p}^{\prime}_{1}=-3\hat{p}^{\prime}_{1},\\
p^{\prime}_{1}&=\frac{1}{4}\hat{p}^{\prime}_0-\frac{1}{2}\hat{p}^{\prime}_{1}=-\frac{1}{2}\hat{p}^{\prime}_{1},\\
p^{\prime}_{2}&=\frac{1}{2}(\hat{p}^{\prime}_{2}+\hat{p}^{\prime}_{3}
+\hat{p}^{\prime}_{4}-\hat{\epsilon}^{\prime}\hat{p}^{\prime}_{5})
=\frac{1}{2}(\hat{p}^{\prime}_{2}+\hat{p}^{\prime}_{3}),\\
p^{\prime}_{3}
&=\frac{1}{2}(\hat{p}^{\prime}_{2}+\hat{p}^{\prime}_{3}
-\hat{p}^{\prime}_{4}+\hat{\epsilon}^{\prime}\hat{p}^{\prime}_{5})
=\frac{1}{2}(\hat{p}^{\prime}_{2}+\hat{p}^{\prime}_{3}),\\
p^{\prime}_{4}
&=\frac{1}{2}(\hat{p}^{\prime}_{2}-\hat{p}^{\prime}_{3}
+\hat{p}^{\prime}_{4}+\hat{\epsilon}^{\prime}\hat{p}^{\prime}_{5})
=\frac{1}{2}(\hat{p}^{\prime}_{2}-\hat{p}^{\prime}_{3}),\\
\epsilon^{\prime}p^{\prime}_{5}
&=\frac{1}{2}(\hat{p}^{\prime}_{2}-\hat{p}^{\prime}_{3}
-\hat{p}^{\prime}_{4}-\hat{\epsilon}^{\prime}\hat{p}^{\prime}_{5})
=\frac{1}{2}(\hat{p}^{\prime}_{2}-\hat{p}^{\prime}_{3}).
\end{split}
\end{equation*}
In particular, $\epsilon^{\prime}=1$,
$p^{\prime}_{2}=p^{\prime}_{3}=\frac{1}{2}(\hat{p}^{\prime}_{2}+\hat{p}^{\prime}_{3})$,
$p^{\prime}_{4}=p^{\prime}_{5}=\frac{1}{2}(\hat{p}^{\prime}_{2}-\hat{p}^{\prime}_{3})$.

Then $p_0=p_0^{\prime}$ and by the branching laws of $(Spin(10),
Spin(2)\cdot Spin(8))$, we get
\begin{equation*}
\begin{split}
&p_{1} \geq p^{\prime}_{2} \geq p_{3},\quad
p_{2} \geq p^{\prime}_{3}=p^{\prime}_{2} \geq p_{4},\\
&p_{3} \geq p^{\prime}_{4} \geq p_{5},\quad p_{4} \geq
p^{\prime}_{5}=p^{\prime}_{4}\geq 0.
\end{split}
\end{equation*}
Thus $p_{1}\geq p_{2}\geq p^{\prime}_{2}=p^{\prime}_{3} \geq
p_{3}\geq p_{4}\geq p^{\prime}_{4}=p^{\prime}_{5}\geq p_{5}\geq 0$.
It follows that
\begin{equation*}
\begin{split}
\ell_{1}& =p_{1}-\max\{p_{2},p^{\prime}_{2}\} =p_{1}-p_{2},
\\
\ell_{2}& =\min\{p_{2},p^{\prime}_{2}\}-\max\{p_{3},p^{\prime}_{3}\}
=p^{\prime}_{2}-p^{\prime}_{3}=0,
\\
\ell_{3} &=\min\{p_{3},p^{\prime}_{3}\}-\max\{p_{4},p^{\prime}_{4}\}
=p_{3}-p_{4},\\
\ell_{4} &=\min\{p_{4},p^{\prime}_{4}\}-\max\{p_{5},p^{\prime}_{5}\}
=p^{\prime}_{4}-p^{\prime}_{5}=0,\\
\ell_{5} &=\min\{p_{5},p^{\prime}_{5}\}=p_{5}.
\end{split}
\end{equation*}
Then the coefficient of
$X^{p^{\prime}_{1}}=X^{-\frac{1}{2}\hat{p}^{\prime}_{1}}
=X^{-\frac{1}{2}\hat{p}^{\prime\prime}_{1}}$ in the (finite) power
series expansion in $X$
\begin{equation*}
\begin{split}
&X^{\epsilon\epsilon^{\prime}\ell_{5}}
\prod^{4}_{i=1}\frac{X^{\ell_{i}+1}-X^{-\ell_{i}-1}}{X-X^{-1}}\\
= &X^{\epsilon\epsilon^{\prime}p_{5}}
\frac{X^{p_{1}-p_{2}+1}-X^{-(p_{1}-p_{2}+1)}}{X-X^{-1}}
\frac{X^{p_{3}-p_{4}+1}-X^{-(p_{3}-p_{4}+1)}}{X-X^{-1}}
\\
= &X^{\epsilon\epsilon^{\prime}p_{5}}
\sum^{p_{1}-p_{2}}_{i=0}\sum^{p_{3}-p_{4}}_{j=0}
X^{(p_{1}-p_{2})+(p_{3}-p_{4})-2(i+j)}
\end{split}
\end{equation*}
is equal to its multiplicity.

Then we have $\Lambda=
p_{0}y_{0}+p_{1}y_{1}+p_{2}y_{2}+p_{3}y_{3}+p_{4}y_{4}+\epsilon
p_{5}y_{5}\in{D(K, K_0)} $ with
$p_{0}=p^{\prime}_{0}=-3\hat{p}^{\prime}_{1} =6p^{\prime}_{1}\in
3{\mathbf Z}$.

\subsection{Eigenvalue computation}

Recall that the standard basis $\mathbf{e}_\alpha$
$(\alpha=0,1,\cdots, 5)$ of
$\mathfrak{t}=\{(\theta_0,\theta_1,\theta_2,\theta_3,\theta_4,\theta_5)
\mid \theta_\alpha\in \mathbf{R}\}$ corresponds to
$2\sqrt{-1}R(2e_1-e_2-e_3)\in \mathfrak{u}(1)$ and
$\sqrt{-1}R(e_2-e_3), D_{1,4}, D_{1,12}, D_{1,36}, D_{1,57}\in
spin(10)$, respectively. With respect to the inner product $\langle
u, v \rangle_{\mathfrak u}=-{\rm tr}uv$ for $u, v\in \mathfrak{k}
\subset \mathfrak{e}_6 \subset \mathfrak{gl}(H_3(\mathbf K)^{\mathbf
C})$,
\begin{equation*}
\langle \mathbf{e}_0, \mathbf{e}_0 \rangle=72, \quad \langle
\mathbf{e}_i, \mathbf{e}_i \rangle = 6, \quad \langle
\mathbf{e}_\alpha, \mathbf{e}_\beta \rangle = 0,
\end{equation*}
for $1\leq i \leq 5$ and $ 0\leq \alpha\neq \beta \leq 5$.
It follows that the inner products of the dual bases $\{ y_0, y_1,
y_2, y_3, y_4, y_5\}$ of $\mathfrak{t}^{*}$ corresponding to
$\{\mathbf{e}_0, \mathbf{e}_1, \mathbf{e}_2,\mathbf{e}_3,
\mathbf{e}_4, \mathbf{e}_5\}$ of $\mathfrak{t}$ are given by
\begin{equation*}
\begin{split}
& \langle y_\alpha, y_\beta \rangle =0, \quad (0\leq \alpha \neq \beta \leq 5),\\
& \langle y_0, y_0\rangle=\frac{1}{72},\quad \langle y_i, y_j
\rangle =\frac{1}{6}, \quad (1 \leq i\neq j \leq 5).
\end{split}
\end{equation*}

For
\begin{equation*}
\begin{split}
\Lambda&=
p_{0}y_{0}+p_{1}y_{1}+p_{2}y_{2}+p_{3}y_{3}+p_{4}y_{4}+\epsilon p_{5}y_{5}\in{D(K, K_0)},\\
\Lambda^{\prime}&=p_{0}y_{0}+\frac{p_{0}}{6}y_{1}+p_{2}^\prime
y_{2}+ p_{2}^{\prime}y_{3}
+p_{4}^{\prime}y_{4}+ p_{4}^{\prime} y_{5}\\
&= -\frac{p_0}{3} \hat{y}_1 + (p_2^{\prime}+p_4^{\prime})\hat{y}_2+ (p_2^{\prime}-p_4^{\prime})\hat{y}_3\in{D(K_2, K_0)},\\
\Lambda^{\prime\prime}&=-\frac{p_0}{3}\hat{y}_1+\hat{p}_2^{\prime\prime}\hat{y}_2
\in D(K_1, K_0),
\end{split}
\end{equation*}
the eigenvalue formulas of the Casimir operators
$\mathcal{C}_{K/K_0}$, $\mathcal{C}_{K_2/K_0}$ and
$\mathcal{C}_{K_1/K_0}$ with respect to the inner product $\langle
\, , \,\rangle_{\mathfrak u}$ are given respectively by
\begin{equation*}
\begin{split}
-c_{\Lambda} &= \frac{1}{72}p_{0}^{2}
+\frac{1}{6}\{(p_{1}+8)p_{1}+(p_{2}+6)p_{2}+(p_{3}+4)p_{3}
+(p_{4}+2)p_{4}+(p_{5})^{2}\},\\
-c_{\Lambda^\prime}&=\frac{1}{72}(p^{\prime}_{0})^{2}
+\frac{1}{6}\{(p^{\prime}_{1})^{2}+(p^{\prime}_{2}+6)p^{\prime}_{2}
+(p^{\prime}_{3}+4)p^{\prime}_{3} +(p^{\prime}_{4}+2)p^{\prime}_{4}
+(p^{\prime}_{5})^{2}\}\\
&=\frac{1}{72}(\hat{p}^{\prime}_{0})^{2}
+\frac{1}{6}\{(\hat{p}^{\prime}_{1})^{2}+(\hat{p}^{\prime}_{2}+6)\hat{p}^{\prime}_{2}
+(\hat{p}^{\prime}_{3}+4)\hat{p}^{\prime}_{3}
+(\hat{p}^{\prime}_{4}+2)\hat{p}^{\prime}_{4}
+(\hat{p}^{\prime}_{5})^{2}\}\\
&=\frac{1}{72}(p_0)^{2}
+\frac{1}{6}\{(\frac{1}{6}p_{0})^{2}+(p^{\prime}_{2}+6)p^{\prime}_{2}+(p^{\prime}_{2}+4)p^{\prime}_{2}
+(p^{\prime}_{4}+2)p^{\prime}_{4}+(p^{\prime}_{4})^{2}\},\\
-c_{\Lambda^{\prime\prime}}&=\frac{1}{72}(\hat{p}^{\prime\prime}_{0})^{2}
+\frac{1}{6}\{(\hat{p}^{\prime\prime}_{1})^{2}+(\hat{p}^{\prime\prime}_{2})^{2}
+(\hat{p}^{\prime\prime}_{3}+4)\hat{p}^{\prime\prime}_{3}
+(\hat{p}^{\prime\prime}_{4}+2)\hat{p}^{\prime\prime}_{4}
+(\hat{p}^{\prime\prime}_{5})^{2}\}\\
&=\frac{1}{6}\{(\frac{1}{3}p_{0})^{2}+(\hat{p}^{\prime\prime}_{2})^{2}\}.
\end{split}
\end{equation*}
Then for each $\Lambda\in{D(K,K_{0})}$, we have the following
eigenvalue formula
\begin{equation*}
\begin{split}
-c_{L}&=-12c_{\Lambda}+6c_{\Lambda^\prime}+3c_{\Lambda^{\prime\prime}}\\
&= 2\{(p_{1}+8)p_{1}+(p_{2}+6)p_{2}+(p_{3}+4)p_{3}
+(p_{4}+2)p_{4}+(p_{5})^{2}\}\\
& - \{(p^{\prime}_{2}+6)p^{\prime}_{2}
+(p^{\prime}_{2}+4)p^{\prime}_{2} +(p^{\prime}_{4}+2)p^{\prime}_{4}
+(p^{\prime}_{4})^{2}\} -\frac{1}{2}(\hat{p}^{\prime\prime}_{2})^{2}
\\
=& 2(p_{1}+8)p_{1}
+2((p_{2})^{2}-(p^{\prime}_{2})^{2})+12p_{2}-10p^{\prime}_{2}
+2(p_{3})^{2}+8p_{3}\\
& +2((p_{4})^{2}-(p^{\prime}_{4})^{2})+4p_{4}-2p^{\prime}_{4}
+2(p_{5})^{2}-\frac{1}{2}(\hat{p}^{\prime\prime}_{2})^{2}
\\
=& 2(p_{1}+8)p_{1}
+2((p_{2})^{2}-(p^{\prime}_{2})^{2})+2p_{2}+10(p_{2}-p^{\prime}_{2})
+2(p_{3})^{2}+8p_{3}\\
&
+2((p_{4})^{2}-(p^{\prime}_{4})^{2})+2p_{4}+2(p_{4}-p^{\prime}_{4})
+2(p_{5})^{2}-\frac{1}{2}(\hat{p}^{\prime\prime}_{2})^{2}
\\
\geq& 2(p_{1}+8)p_{1}+2p_{2}+
2(p_{3})^{2}-\frac{1}{2}(\hat{p}^{\prime\prime}_{2})^{2}+8p_{3}+2p_{4}
+2(p_{5})^{2}
\\
=& 2(p_{1}+8)p_{1}+2p_{2}+
(2(p^{\prime}_{5})^{2}-\frac{1}{2}(\hat{p}^{\prime\prime}_{2})^{2})
+8p_{3}+2p_{4}+2(p_{5})^{2}
\\
\geq& 2(p_{1}+8)p_{1}+2p_{2}+8p_{3}+2p_{4}+2(p_{5})^{2},
\end{split}
\end{equation*}
where the equalities hold if and only if $p_{2}=p^{\prime}_{2}$,
$p_{4}=p^{\prime}_{4}$,
$2p_{3}=2p_{4}=2p^{\prime}_{4}=2p^{\prime}_{5}
=\vert{\hat{p}^{\prime\prime}_{2}}\vert$ since we have
\begin{equation*}
\begin{split}
&p_{1}\geq p_{2}\geq p^{\prime}_{2}=p^{\prime}_{3} \geq p_{3}\geq
p_{4}\geq p^{\prime}_{4}=p^{\prime}_{5} \geq p_{5}\geq 0,
\\
&-2p^{\prime}_{4}=-2p^{\prime}_{5}=-(\hat{p}^{\prime}_{2}-\hat{p}^{\prime}_{3})
\leq \hat{p}^{\prime\prime}_{2} \leq
\hat{p}^{\prime}_{2}-\hat{p}^{\prime}_{3}
=2p^{\prime}_{5}=2p^{\prime}_{4}.
\end{split}
\end{equation*}

Notice that if $p_{1}=0$, then $-c_{L}=0$ and if $p_{1}\geq 2$, then
$-c_{L}\geq 40>30$. In case $p_{1}=\frac{3}{2}$, the possible
$\Lambda=(p_{0},p_{1},p_{2},p_{3},p_{4},p_{5})\in D(K,K_0)$ are
\begin{equation*}
\begin{split}
&(p_{0},\frac{3}{2},\frac{3}{2},\frac{3}{2},\frac{3}{2},\frac{3}{2}),
(p_{0},\frac{3}{2},\frac{3}{2},\frac{3}{2},\frac{3}{2},-\frac{3}{2}),
(p_{0},\frac{3}{2},\frac{3}{2},\frac{3}{2},\frac{3}{2},\frac{1}{2}),
(p_{0},\frac{3}{2},\frac{3}{2},\frac{3}{2},\frac{3}{2},-\frac{1}{2}),\\
&(p_{0},\frac{3}{2},\frac{3}{2},\frac{3}{2},\frac{1}{2},\frac{1}{2}),
(p_{0},\frac{3}{2},\frac{3}{2},\frac{3}{2},\frac{1}{2},-\frac{1}{2}),
(p_{0},\frac{3}{2},\frac{3}{2},\frac{1}{2},\frac{1}{2},\frac{1}{2}),
(p_{0},\frac{3}{2},\frac{3}{2},\frac{1}{2},\frac{1}{2},-\frac{1}{2}),\\
&(p_{0},\frac{3}{2},\frac{1}{2},\frac{1}{2},\frac{1}{2},\frac{1}{2}),
(p_{0},\frac{3}{2},\frac{1}{2},\frac{1}{2},\frac{1}{2},-\frac{1}{2}).
\end{split}
\end{equation*}
In these cases, the eigenvalue of the Casimir operator is given by
\begin{equation*}
\begin{split}
-c_{L}\geq&
2(p_{1}+8)p_{1}+2p_{2}+8p_{3}+2p_{4}+2(p_{5})^{2}\\
\geq& 2\cdot (\frac{3}{2}+8)\cdot
\frac{3}{2}+2\cdot\frac{1}{2}+8\cdot \frac{1}{2}+2\cdot \frac{1}{2}
+2\cdot (\frac{1}{2})^{2}\\
=&35>30.
\end{split}
\end{equation*}
Hence in order to decide the Hamiltonian stability, i.e., to compare
the first eigenvalue $-c_{L}$ and $30$, we can only concern on the
cases when $p_{1}=\frac{1}{2}$ or $1$.

It follows from the description of $D(K,K_0)$ that the element in
$D(K,K_0)$ when $p_1=\frac{1}{2}$ is given by
$$(p_{0},\frac{1}{2},\frac{1}{2},\frac{1}{2},\frac{1}{2},\frac{1}{2})
\text{ or }
(p_{0},\frac{1}{2},\frac{1}{2},\frac{1}{2},\frac{1}{2},-\frac{1}{2})
$$
and the element in $D(K,K_0)$ for $p_1=1$
is given by
\begin{equation*}
\begin{split}
&(p_{0},1,0,0,0,0), \quad (p_{0},1,1,0,0,0), \quad (p_{0},1,1,1,0,0),\\
&(p_{0},1,1,1,1,0), \quad (p_{0},1,1,1,1,1) \text{ or } \,
(p_{0},1,1,1,1,-1).
\end{split}
\end{equation*}

Using the branching laws, the descriptions of $D(K_2, K_0)$, $D(K_1,
K_0)$ and the eigenvalue formula given above, by direct computation
we get the following small eigenvalues in the above cases.

\begin{center}
\begin{tabular}
{|l|l|l|l|} \hline $\Lambda$ &$\Lambda^{\prime}$
&$\Lambda^{\prime\prime}$ &$-c_{L}$
\\
\hline
$3,\frac{1}{2},\frac{1}{2},\frac{1}{2},\frac{1}{2},\frac{1}{2}$
&$3,\frac{1}{2},\frac{1}{2},\frac{1}{2},\frac{1}{2},\frac{1}{2}$
&$0,-1,1,0,0,0$ &$15$
\\
\hline
$3,\frac{1}{2},\frac{1}{2},\frac{1}{2},\frac{1}{2},\frac{1}{2}$
&$3,\frac{1}{2},\frac{1}{2},\frac{1}{2},\frac{1}{2},\frac{1}{2}$
&$0,-1,-1,0,0,0$ &$15$
\\
\hline
$-3,\frac{1}{2},\frac{1}{2},\frac{1}{2},\frac{1}{2},-\frac{1}{2}$
&$-3,-\frac{1}{2},\frac{1}{2},\frac{1}{2},\frac{1}{2},\frac{1}{2}$
&$0,1,1,0,0,0$ &$15$
\\
\hline
$-3,\frac{1}{2},\frac{1}{2},\frac{1}{2},\frac{1}{2},-\frac{1}{2}$
&$-3,-\frac{1}{2},\frac{1}{2},\frac{1}{2},\frac{1}{2},\frac{1}{2}$
&$0,1,-1,0,0,0$ &$15$
\\
\hline $6,1,0,0,0,0$ &$6,1,0,0,0,0$ &$0,-2,0,0,0,0$ &$18$
\\
\hline $-6,1,0,0,0,0$ &$-6,-1,0,0,0,0$ &$0,2,0,0,0,0$ &$18$
\\
\hline $0,1,1,0,0,0$ &$0,0,0,0,0,0$ &$0,0,0,0,0,0$ &$32$
\\
\hline $0,1,1,0,0,0$ &$0,0,1,1,0,0$ &$0,0,0,0,0,0$ &$20$
\\
\hline $6,1,1,1,0,0$ &$6,1,1,1,0,0$ &$0,-2,0,0,0,0$ &$30$
\\
\hline $-6,1,1,1,0,0$ &$-6,-1,1,1,0,0$ &$0,2,0,0,0,0$ &$30$
\\
\hline $0,1,1,1,1,0$ &$0,0,1,1,0,0$ &$0,0,0,0,0,0$ &$36$
\\
\hline $0,1,1,1,1,0$ &$0,0,1,1,1,1$ &$0,0,0,0,0,0$ &$32$
\\
\hline $0,1,1,1,1,0$ &$0,0,1,1,1,1$ &$0,0,2,0,0,0$ &$30$
\\
\hline $0,1,1,1,1,0$ &$0,0,1,1,1,1$ &$0,0,-2,0,0,0$ &$30$
\\
\hline $6,1,1,1,1,1$ &$6,1,1,1,1,1$ &$0,-2,2,0,0,0$ &$32$
\\
\hline $6,1,1,1,1,1$ &$6,1,1,1,1,1$ &$0,-2,-2,0,0,0$ &$32$
\\
\hline $6,1,1,1,1,1$ &$6,1,1,1,1,1$ &$0,-2,0,0,0,0$ &$34$
\\
\hline $-6,1,1,1,1,-1$ &$-6,-1,1,1,1,1$ &$0,2,2,0,0,0$ &$32$
\\
\hline $-6,1,1,1,1,-1$ &$-6,-1,1,1,1,1$ &$0,2,-2,0,0,0$ &$32$
\\
\hline $-6,1,1,1,1,-1$ &$-6,-1,1,1,1,1$ &$0,2,0,0,0,0$ &$34$
\\
\hline
\end{tabular}
\end{center}
Here, $\Lambda=(p_0,p_1,p_2,p_3,p_4,p_5)\in{D(K,K_{0})}$,
$\Lambda^{\prime}=(p^{\prime}_0,p^{\prime}_1,p^{\prime}_2,p^{\prime}_3,p^{\prime}_4,p^{\prime}_5)\in{D(K_{2},K_{0})}$
and $\Lambda^{\prime\prime}
=(\hat{p}^{\prime\prime}_0,\hat{p}^{\prime\prime}_1,
\hat{p}^{\prime\prime}_2, \hat{p}^{\prime\prime}_3,
\hat{p}^{\prime\prime}_4, \hat{p}^{\prime\prime}_5)
\in{D(K_{1},K_{0})}$.

Since $\Lambda_1=(3,\frac{1}{2},\frac{1}{2},\frac{1}{2},\frac{1}{2},\frac{1}{2})$
corresponds to the complexified isotropy representation of EIII and
it is conjugate to $\Lambda_2=(-3,
\frac{1}{2},\frac{1}{2},\frac{1}{2},\frac{1}{2}, -\frac{1}{2})$, we
see that $\Lambda_1, \Lambda_2 \not\in D(K,K_{[\mathfrak{a}]})$.

Suppose that
$\Lambda=(p_{0},p_{1},p_{2},p_{3},p_{4},p_{5})=(6,1,0,0,0,0)\in D(K, K_0)$. Then
by the branching laws we get
$\Lambda^{\prime}=6y_{0}+y_{1}\in{D(K_{2},K_{0})}$,
$\Lambda^{\prime\prime}=-2\hat{y}_{1}\in{D(K_{1},K_{0})}$ and
$\Lambda^{\prime\prime\prime}=0\in{D(K_{0})}$.
Hence, the eigenvalue of the Casimir operator is $-c_{L}=18<30$.

On the other hand,
\begin{equation*}
\begin{split}
V_{\Lambda}\cong& \Bigl\{
\begin{pmatrix}
0&0&0\\
0&\xi_{2}&x_{1}\\
0&\bar{x}_{1}&\xi_{3}
\end{pmatrix}
{\ }\vert{\ }\xi_{2},\xi_{3}\in{\mathbf C}, x_{1}\in{\mathbf
K}^{\mathbf C}\Bigr\}
\cong {\mathbf C}^{10}\\
\supset&{\ }W_{\Lambda^{\prime}}
=U_{\Lambda^{\prime\prime}}=U_{\Lambda^{\prime\prime\prime}}
=(V_{\Lambda})_{K_{0}}
\end{split}
\end{equation*}
and $\rho_{\Lambda}=\mu_{6}\boxtimes \sigma_{{\mathbf C}^{10}}$,
where $\sigma_{{\mathbf C}^{10}}$ denotes the standard
representation of $SO(10)$, and for each $\phi(\theta)\in{U(1)}$,
\begin{equation*}
\mu_{6}(\phi(\theta))
\begin{pmatrix}
0&0&0\\
0&\xi_{2}&x_{1}\\
0&\bar{x}_{1}&\xi_{3}
\end{pmatrix}
= \theta^{-6}
\begin{pmatrix}
0&0&0\\
0&\xi_{2}&x_{1}\\
0&\bar{x}_{1}&\xi_{3}
\end{pmatrix},
\end{equation*}
where $\theta=e^{\sqrt{-1}t_{0}/2}$.
Since for any $\exp(\hat{t}_{0}\sqrt{-1}R(e_{1}-2e_{2}+e_{3}))\in S^1 \subset K_0$,
\begin{equation*}
\begin{split}
&\,\exp(\hat{t}_{0}\sqrt{-1}R(e_{1}-2e_{2}+e_{3}))\\
&= \exp(\hat{t}_{0}\frac{1}{2}\sqrt{-1}R(2e_{1}-e_{2}-e_{3}))
\exp(-\hat{t}_{0}\frac{3}{2}\sqrt{-1}R(e_{2}-e_{3}))\\
&\in U(1)\cdot Spin(2)\subset K,
\end{split}
\end{equation*}
we compute
\begin{equation*}
\begin{split}
&\rho_{\Lambda}(\exp (\hat{t}_{0}\sqrt{-1}R(e_{1}-2e_{2}+e_{3})))
\begin{pmatrix}
0&0&0\\
0&\xi_{2}&x_{1}\\
0&\bar{x}_{1}&\xi_{3}
\end{pmatrix}
\\
=& (\mu_{6}\boxtimes\sigma_{{\mathbf C}^{10}}) (\exp
(\hat{t}_{0}\sqrt{-1}R(e_{1}-2e_{2}+e_{3})))
\begin{pmatrix}
0&0&0\\
0&\xi_{2}&x_{1}\\
0&\bar{x}_{1}&\xi_{3}
\end{pmatrix}
\\
=&
\mu_{6}(\exp(\hat{t}_{0}\frac{1}{2}\sqrt{-1}R(2e_{1}-e_{2}-e_{3})))
\alpha_{23}(-\hat{t}_{0}\frac{3}{2})
\begin{pmatrix}
0&0&0\\
0&\xi_{2}&x_{1}\\
0&\bar{x}_{1}&\xi_{3}
\end{pmatrix}
\\
=& (e^{\sqrt{-1}\frac{1}{2}\hat{t}_{0}\frac{1}{2}})^{-6}
\begin{pmatrix}
0&0&0\\
0&e^{-\sqrt{-1}\hat{t}_{0}\frac{3}{2}}\xi_{2}&x_{1}\\
0&\bar{x}_{1}&e^{\sqrt{-1}\hat{t}_{0}\frac{3}{2}}\xi_{3}
\end{pmatrix}
\\
=& e^{-\sqrt{-1}\frac{3}{2}\hat{t}_{0}}
\begin{pmatrix}
0&0&0\\
0&e^{-\sqrt{-1}\hat{t}_{0}\frac{3}{2}}\xi_{2}&x_{1}\\
0&\bar{x}_{1}&e^{\sqrt{-1}\hat{t}_{0}\frac{3}{2}}\xi_{3}
\end{pmatrix}
\\
=&
\begin{pmatrix}
0&0&0\\
0&e^{-\sqrt{-1}3\hat{t}_{0}}\xi_{2}
&e^{-\sqrt{-1}\frac{3}{2}\hat{t}_{0}}x_{1}\\
0&e^{-\sqrt{-1}\frac{3}{2}\hat{t}_{0}}\bar{x}_{1}&\xi_{3}
\end{pmatrix}.
\end{split}
\end{equation*}
In particular,
\begin{equation*}
\rho_{\Lambda}(\exp(\hat{t}_{0}\sqrt{-1}R(e_{1}-2e_{2}+e_{3})))
\begin{pmatrix}
0&0&0\\
0&0&0\\
0&0&\xi_{3}
\end{pmatrix}
=
\begin{pmatrix}
0&0&0\\
0&0&0\\
0&0&\xi_{3}
\end{pmatrix}
\end{equation*}
for each $\hat{t}_{0}\in{\mathbf R}$. Hence,
\begin{equation*}
(V_{\Lambda})_{K_0} \cong {\ } \Bigl\{
\begin{pmatrix}
0&0&0\\
0&0&0\\
0&0&\xi_{3}
\end{pmatrix}
{\ }\vert{\ }\xi_{3}\in{\mathbf C}\Bigr\}.
\end{equation*}
But as a generator of ${\mathbf Z}_{4}$ of $K_{[\mathfrak{a}]}$, the
action of
$\alpha_{23}(\pi)(\alpha_{1},\alpha_{2},\alpha_{3})\in{K_{[{\mathfrak
a}]}}$ given by \eqref{eq:generator_EIII} is
\begin{equation*}
\begin{split}
&\rho_{\Lambda}(\alpha_{23}(\pi)(\alpha_{1},\alpha_{2},\alpha_{3}))
\begin{pmatrix}
0&0&0\\
0&0&0\\
0&0&\xi_{3}
\end{pmatrix}
\\
=& (\alpha_{23}(\pi))
\begin{pmatrix}
0&0&0\\
0&0&0\\
0&0&\xi_{3}
\end{pmatrix}
=
\begin{pmatrix}
0&0&0\\
0&0&0\\
0&0&-\xi_{3}
\end{pmatrix}.
\end{split}
\end{equation*}
Therefore $(V_{\Lambda})_{K_{[\mathfrak a]}}=\{0\}$ and
$\Lambda=6y_{0}+y_{1}\not\in{D(K,K_{[{\mathfrak a}]})}$. Similarly,
$\Lambda=-6y_{0}+y_{1}\not\in{D(K,K_{[{\mathfrak a}]})}$.

Suppose $\Lambda=(p_{0},p_{1},p_{2},p_{3},p_{4},p_{5})=(0,1,1,0,0,0)\in D(K, K_0)$.
Then
by the branching laws we get
\begin{eqnarray*}
\Lambda^{\prime}&=&(p^{\prime}_{0},p^{\prime}_{1},p^{\prime}_{2},p^{\prime}_{3},p^{\prime}_{4},p^{\prime}_{5})
=(0,0,1,1,0,0) \in D(K_2, K_0),\\
\Lambda^{\prime\prime}&=&(\hat{p}^{\prime\prime}_{0},\hat{p}^{\prime\prime}_{1},\hat{p}^{\prime\prime}_{2},
\hat{p}^{\prime\prime}_{3},\hat{p}^{\prime\prime}_{4},\hat{p}^{\prime\prime}_{5})
=(0,0,0,0,0,0)\in D(K_1, K_0).
\end{eqnarray*}
Here
$\rho^{\prime}_{\Lambda^{\prime}}
=\mathrm{Id}\boxtimes\mathrm{Id}\boxtimes\mathrm{Ad}^{\mathbf
C}_{Spin(8)}
=\mathrm{Id}\boxtimes\mathrm{Id}\boxtimes\mathrm{Ad}^{\mathbf C}_{SO(8)}
\in\mathcal{D}(K_{2})$.
Notice that
$W_{\Lambda^{\prime}}={\mathfrak o}(8)^{\mathbf C} ={\mathfrak o}(2)^{\mathbf C}\oplus{\mathfrak
o}(6)^{\mathbf C}\oplus M(2,6;{\mathbf R})^{\mathbf C}$, and the
subgroups $U(1)$ and $Spin(2)$ of $K_{2}=(U(1)\times (Spin(2)\cdot
Spin(8))/{\mathbf Z}_{4}$ acts trivially on ${\mathfrak
o}(8)^{\mathbf C}$. The subgroup $Spin(6)$ of $Spin(2)\cdot Spin(6)$
acts trivially on ${\mathfrak o}(2)^{\mathbf C}$, hence
$(W_{\Lambda^{\prime}})_{K_{0}}={\mathfrak o}(2)^{\mathbf C}$.
For $\alpha_{23}(\pi)(\alpha_{1},\alpha_{2},\alpha_{3})\in{K_{[{\mathfrak a}]}}$
a generator of ${\mathbf Z}_{4}$ given in \eqref{eq:generator_EIII},
$\alpha_{23}(\pi)$ and $(\alpha_{1},\alpha_{2},\alpha_{3})$ commute to each other.
$\alpha_{23}(\pi)\in{Spin(2)}$ acts trivially on ${\mathfrak o}(2)^{\mathbf C}$. $\alpha_{2}$ of
$(\alpha_{1},\alpha_{2},\alpha_{3})$ acts on ${\mathbf R}1+{\mathbf R}\mathbf{e}$ as $\left(
                                                    \begin{array}{cc}
                                                      0 & -1 \\
                                                      -1 & 0 \\
                                                    \end{array}
                                                  \right)$
and preserves the vector subspace orthogonally complementary to
${\mathbf R}1+{\mathbf R}\mathbf{e}$ in ${\mathbf K}\cong{\mathbf R}^{8}$.
Thus the $Spin(2)$-factor of $(\alpha_{1},\alpha_{2},\alpha_{3})$ in
$Spin(2)\cdot Spin(6)$ corresponds to
$\begin{pmatrix}
0&-1\\
-1&0
\end{pmatrix}
\in{O(2)}$.
Since its adjoint action of
on ${\mathfrak o}(2)^{\mathbf C}$ is $-\mathrm{Id}$, the adjoint
action of $(\alpha_{1},\alpha_{2},\alpha_{3})\in{Spin(8)}$ is not
trivial on ${\mathfrak o}(2)^{\mathbf C}$. Hence
$(W_{\Lambda^{\prime}})_{K_{[{\mathfrak a}]}}=\{0\}$
and in particular we obtain
$\Lambda=y_{1}+y_{2}\not\in{D(K,K_{[{\mathfrak a}]})}$.

Suppose $\Lambda=(p_{0},p_{1},p_{2},p_{3},p_{4},p_{5})=(6,1,1,1,0,0)\in D(K,K_0)$.
Then $\dim_{\mathbf C} V_{\Lambda}=120$.
By the branching laws we get
$\Lambda^{\prime}=6y_0+y_{1}+y_{2}+y_{3}=-2\hat{y}_1+\hat{y}_{2}+\hat{y}_{3}\in{D(K_{2},K_{0})}$,
$\Lambda^{\prime\prime}=-2\hat{y}_1 \in{D(K_{1},K_{0})}$ and
$\Lambda^{\prime\prime\prime}=0\in{D(K_{0})}$.
Hence, the eigenvalue of the Casimir operator is $-c_{L}=30$.

On the other hand,
$\rho^{\prime}_{\Lambda^{\prime}} =
\mathrm{Id}\boxtimes \mu_{-2}
\boxtimes \mathrm{Ad}^{\mathbf C}_{Spin(8)}
=\mathrm{Id}\boxtimes \mu_{-2} \boxtimes\mathrm{Ad}^{\mathbf C}_{SO(8)}
\in{\mathcal{D}(K_{2})}$.
Here $W_{\Lambda^{\prime}}={\mathfrak o}(8)^{\mathbf C} ={\mathfrak
o}(2)^{\mathbf C}\oplus{\mathfrak o}(6)^{\mathbf C}\oplus
M(2,6;{\mathbf R})^{\mathbf C}$.
Same as the previous case, we get
$(W_{\Lambda^{\prime}})_{K_{0}}={\mathfrak o}(2)^{\mathbf C}$.
Notice that for the generator
$\alpha_{23}(\pi)(\alpha_{1},\alpha_{2},\alpha_{3})$
of ${\mathbf Z}_{4}$ in $K_{[\mathfrak a]}$ given by \eqref{eq:generator_EIII},
the action of $\alpha_{23}(\pi)\in{Spin(2)}$ on $H_3(\mathbf{K}^{\mathbf C})$
is given by
\begin{equation*}
\left(
  \begin{array}{ccc}
    \xi_1 & x_3 & \bar{x}_2 \\
    \bar{x}_3 & \xi_2 & x_1 \\
    x_2 & \bar{x}_1 & \xi_3 \\
  \end{array}
\right) \mapsto \left(
  \begin{array}{ccc}
    \xi_1 & \sqrt{-1} x_3 & -\sqrt{-1}\bar{x}_2 \\
    \sqrt{-1}\bar{x}_3 & -\xi_2 & x_1 \\
   -\sqrt{-1} x_2 & \bar{x}_1 & -\xi_3 \\
  \end{array}
\right).
\end{equation*}
In particular, $\alpha_{23}(\pi)$ transforms $u_2$ to
$-\sqrt{-1}u_2$ and ${\mathbf e}u_2$ to $-\sqrt{-1}\mathbf{e}u_2$,
which says that $\alpha_{23}(\pi)$ acts on $\mathfrak{o}(2)\cong
\mathbf{R}1+\mathbf{R}\mathbf{e}$ as the matrix multiplication by
$\left(
                                   \begin{array}{cc}
                                     -\sqrt{-1} & 0 \\
                                     0 & -\sqrt{-1} \\
                                   \end{array}
                                 \right)
$. Thus $\mu_{-2}(\alpha_{23}(\pi))$ acts on $\mathfrak{o}(2)\cong
\mathbf{R}1+\mathbf{R}\mathbf{e}$ is just the matrix multiplication
by $-\mathrm{Id}$. On the other hand, $\alpha_{2}$ of
$(\alpha_{1},\alpha_{2},\alpha_{3})$ acts on ${\mathbf R}1+{\mathbf
R}e$ as $\begin{pmatrix}
0&-1\\
-1&0
\end{pmatrix}$.
Thus the $Spin(2)$-factor of $(\alpha_{1},\alpha_{2},\alpha_{3})$ in
$Spin(2)\cdot Spin(6)$ corresponds to
$\begin{pmatrix}
0&-1\\
-1&0
\end{pmatrix}
\in{O(2)}$.
Hence its
adjoint action
on $\mathfrak{o}(2)^{\mathbf C}$ is $-\mathrm{Id}$. Therefore,
$(V_{\Lambda})_{K_{[\mathfrak a]}}=\mathfrak{o}(2)^{\mathbf C}$, i.e.,
$\Lambda=6y_0+y_1+y_2+y_3\in D(K,K_{[\mathfrak
a]})=\mathfrak{o}(2)^{\mathbf C}$. Thus $\Lambda=6y_0+y_1+y_2+y_3\in
D(K,K_{[\mathfrak a]})$ with multiplicity $1$. Similarly,
$\Lambda=-6y_0+y_1+y_2+y_3\in D(K,K_{[\mathfrak a]})$ with
multiplicity $1$.

Suppose $\Lambda= (p_{0},p_{1},p_{2},p_{3},p_{4},p_{5})=(0,1,1,1,1,0)\in D(K,K_0)$.
Then $\dim_{\mathbf C} V_{\Lambda}=210$.
By the branching laws we get the following decomposition of $V_{\Lambda}$ into irreducible modules of $K_2$
and $K_1$:
\begin{equation*}
\begin{split}
&V_{\Lambda(0,1,1,1,1,0)}\\
=& W_{\Lambda^\prime_1(0,0,1,1,1,1)}\oplus W_{\Lambda^\prime_2(0,0,1,1,0,0)}\\
=&(U_{\Lambda_1^{\prime\prime}(0,0,0,0,0,0)}\oplus U_{\Lambda_1^{\prime\prime}(0,0,2,0,0,0)}
\oplus U_{\Lambda_1^{\prime\prime}(0,0,-2,0,0,0)})
\oplus U_{\Lambda_2^{\prime\prime}(0,0,0,0,0,0)}.
\end{split}
\end{equation*}
Then  the Casimir operator $-\mathcal{C}_L$ has eigenvalues $-c_{L}=32, 30, 30$ or $36$ along this decomposition.

On the other hand, $\Lambda_1^\prime=2\hat{y}_{2}\in{D(K_{2},K_{0})}$,
$W_{\Lambda_1^{\prime}} \cong S^{2}_{0}({\mathbf C}^{8}) \cong S^{2}_{0}({\mathbf K}^{8})$ and
\begin{equation*}
W_{\Lambda^\prime_1} \cap (V_{\Lambda})_{K_{0}}=
U_{\Lambda_1^{\prime\prime}(0,0,0,0,0,0)}\oplus
(U_{\Lambda_1^{\prime\prime}(0,0,2,0,0,0)})_{K_0}\oplus
(U_{\Lambda_1^{\prime\prime}(0,0,-2,0,0,0)})_{K_0}.
\end{equation*}

Recall that $\{1,c_{1},\cdots,c_{7}\}$ denote the standard basis of the
Cayley algebra $\mathbf K$ and $\mathbf{e}:=c_{4}$. Then
\begin{equation*}
3(1\cdot 1+\mathbf{e}\cdot \mathbf{e}) -(c_{1}\cdot c_{1}+c_{2}\cdot c_{2}+c_{3}\cdot
c_{3} +c_{5}\cdot c_{5}+c_{6}\cdot c_{6}+c_{7}\cdot c_{7})
\in{S^{2}_{0}({\mathbf K}^{\mathbf C})}.
\end{equation*}
For any $ A=
\begin{pmatrix}
\cos{t}&-\sin{t}\\
\sin{t}&\cos{t}
\end{pmatrix}
\in{SO(2)} $, $ A(1, \mathbf{e})=(1,\mathbf{e})
\begin{pmatrix}
\cos{t}&-\sin{t}\\
\sin{t}&\cos{t}
\end{pmatrix}
$.
Hence
\begin{equation*}
\begin{split}
A(1\cdot 1)
=&(\cos{t}1+\sin{t}\mathbf{e})\cdot(\cos{t}1+\sin{t}\mathbf{e})\\
=&\cos^{2}{t}(1\cdot 1)+\sin^{2}{t}(\mathbf{e}\cdot \mathbf{e})+2\sin{t}\cos{t}(1\cdot \mathbf{e}),\\
A(\mathbf{e}\cdot \mathbf{e})
=&(-\sin{t}1+\cos{t} \mathbf{e})\cdot (-\sin{t}1+\cos{t}\mathbf{e})\\
=&\sin^{2}{t}(1\cdot 1)+\cos^{2}{t}(\mathbf{e}\cdot \mathbf{e})-2\sin{t}\cos{t}(1\cdot \mathbf{e}),\\
A(1\cdot \mathbf{e})
=&(\cos{t}1+\sin{t} \mathbf{e})\cdot (-\sin{t}1+\cos{t}\mathbf{e})\\
=&-\frac{1}{2}\sin{2t}(1\cdot 1- \mathbf{e}\cdot \mathbf{e})+\cos{2t}(1\cdot \mathbf{e}).
\end{split}
\end{equation*}
In particular,
$A(1\cdot 1+ \mathbf{e}\cdot \mathbf{e})=1\cdot 1+ \mathbf{e}\cdot \mathbf{e}$
and
\begin{equation*}
\begin{split}
&A(3(1\cdot 1+ \mathbf{e}\cdot \mathbf{e}) -(c_{1}\cdot c_{1}+ c_{2}\cdot
c_{2}+c_{3}\cdot c_{3}+c_{5}\cdot c_{5}+ c_{6}\cdot c_{6}+ c_{7}\cdot c_{7}))\\
=& 3(1\cdot 1+\mathbf{e}\cdot \mathbf{e}) -(c_{1}\cdot c_{1}+c_{2}\cdot
c_{2}+c_{3}\cdot c_{3} +c_{5}\cdot c_{5}+c_{6}\cdot c_{6}+c_{7}\cdot
c_{7}).
\end{split}
\end{equation*}
Thus,
$3(1\cdot 1+\mathbf{e}\cdot \mathbf{e}) -(c_{1}\cdot c_{1}+ c_{2}\cdot c_{2}+ c_{3}\cdot
c_{3} +c_{5}\cdot c_{5}+c_{6}\cdot c_{6}+c_{7}\cdot c_{7}) \in
U_{\Lambda^{\prime\prime}(0,0,0,0,0,0)}$.
On the other hand,
$1\cdot 1-\mathbf{e}\cdot \mathbf{e}-2\sqrt{-1}(1\cdot \mathbf{e})$,
$1\cdot 1-\mathbf{e}\cdot \mathbf{e}+2\sqrt{-1}(1\cdot \mathbf{e})\in{S^{2}_{0}({\mathbf
K}^{\mathbf C})}$,
and we see that
\begin{equation*}
\begin{split}
&A(1\cdot 1-\mathbf{e}\cdot \mathbf{e}-2\sqrt{-1}1\cdot \mathbf{e})
=
\mathbf{e}^{\sqrt{-1}2t}(1\cdot 1-\mathbf{e}\cdot \mathbf{e}-2\sqrt{-1}1\cdot \mathbf{e}),\\
&A(1\cdot 1-\mathbf{e}\cdot \mathbf{e}+2\sqrt{-1}1\cdot \mathbf{e})
=
\mathbf{e}^{-\sqrt{-1}2t}(1\cdot 1- \mathbf{e}\cdot \mathbf{e}+2\sqrt{-1}1\cdot \mathbf{e}).\\
\end{split}
\end{equation*}
Hence,
$1\cdot 1-\mathbf{e}\cdot \mathbf{e}-2\sqrt{-1}1\cdot \mathbf{e} \in
U_{\Lambda^{\prime\prime}(0,0,2,0,0,0)}$,
$1\cdot 1-e\cdot \mathbf{e}+2\sqrt{-1}1\cdot \mathbf{e}
\in U_{\Lambda^{\prime\prime}(0,0,-2,0,0,0)}$.
Therefore,
\begin{equation*}
\begin{split}
&(V_{\Lambda})_{K_{0}}\cap W_{\Lambda^\prime_1}\\
=& {\ }
{\mathbf C} (3(1\cdot 1+ \mathbf{e}\cdot \mathbf{e}) -(c_{1}\cdot
c_{1}+c_{2}\cdot c_{2}+c_{3}\cdot c_{3} +c_{5}\cdot c_{5}+c_{6}\cdot
c_{6}+c_{7}\cdot c_{7}))\\
&\oplus
{\mathbf C}(1\cdot 1-\mathbf{e}\cdot \mathbf{e}-2\sqrt{-1}1\cdot \mathbf{e})\\
&\oplus {\mathbf C}(1\cdot 1-\mathbf{e}\cdot \mathbf{e}+2\sqrt{-1}1\cdot \mathbf{e}).
\end{split}
\end{equation*}
Since the action of the generator $\alpha_{23}(\pi)(\alpha_{1},\alpha_{2},\alpha_{3})$ is given by
\begin{eqnarray*}
&(\alpha_{23}(\pi)(\alpha_{1},\alpha_{2},\alpha_{3}))
(2\sqrt{-1}(1\cdot \mathbf{e})) =2(\sqrt{-1}\mathbf{e}\cdot (-1))
=-2\sqrt{-1}(1\cdot \mathbf{e}),\\
&(\alpha_{23}(\pi)(\alpha_{1},\alpha_{2},\alpha_{3})) (1\cdot
1-\mathbf{e}\cdot \mathbf{e})
=1\cdot 1-\mathbf{e}\cdot \mathbf{e},\\
&(\alpha_{23}(\pi)(\alpha_{1},\alpha_{2},\alpha_{3})) (1\cdot
1+\mathbf{e}\cdot \mathbf{e}) =-(1\cdot 1+ \mathbf{e}\cdot \mathbf{e}),
\end{eqnarray*}
we obtain
\begin{equation*}
(V_{\Lambda})_{K_{[{\mathfrak a}]}}\cap W_{\Lambda^\prime_1}={\mathbf C}(1\cdot 1-\mathbf{e}\cdot \mathbf{e}),
\end{equation*}
and thus
$\Lambda=y_{1}+y_{2}+y_{3}+y_{4}\in{D(K,K_{[{\mathfrak a}]})}$,
which has eigenvalue $30$ of $-\mathcal{C}_L$ with the multiplicity $1$.
Therefore,
\begin{equation*}
\begin{split}
n(L^{30})&=\dim_{\mathbf C} V_{(6,1,1,1,0,0)} + \dim_{\mathbf C}
V_{(-6,1,1,1,0,0)}
+\dim_{\mathbf C} V_{(0,1,1,1,0,0)}\\
&=120+120+210=450\\
&=\dim SO(32)-\dim U(1)\cdot Spin(10)=n_{kl}(L^{30}).
\end{split}
\end{equation*}
Then we conclude that
\begin{thm0} The Gauss image
$$
L^{30}=(U(1)\cdot Spin(10))/(S^{1}\cdot Spin(6) \cdot {\mathbf Z}_{4}) \subset Q_{30}({\mathbf C})
$$
is strictly Hamiltonian stable.
\end{thm0}


\end{document}